\documentclass[10pt,twoside,openright,a4paper]{report} 
\usepackage[utf8]{inputenc} 
\usepackage[T1]{fontenc}
\usepackage{lmodern}

\usepackage{amsmath}
\usepackage{amsfonts}
\usepackage{stixgras}
\usepackage{mathrsfs}
\usepackage{mathtools}
\usepackage{stmaryrd}
\usepackage[all]{xy}
\usepackage{makeidx}
\usepackage[british]{babel}
\usepackage{csquotes}

\usepackage{amssymb}
\usepackage{xspace}
\usepackage{amsthm}
\makeatletter\def\th@plain{\slshape}\makeatother
\makeatletter\patchcmd{\th@remark}{\itshape}{\slshape}{}{}\makeatother

\usepackage{xr}

\usepackage{titlesec}
\titleformat{\chapter}[hang]{\Huge\bfseries}{\thechapter.}{.5em}{\vspace*{.05em}}[\vskip 0em]
\renewcommand\thechapter{\Alph{chapter}}
\titlelabel{\thetitle.\quad}

\usepackage[url=true,isbn=false,doi=true,backend=biber,defernumbers=true,backref=true,safeinputenc,maxnames=4,sorting=anyt]{biblatex}

\newcounter{bidon}
\newcommand{\rdb}{\refstepcounter{bidon}}
\usepackage[french,nohints]{minitoc}

\theoremstyle{plain}

\newtheorem{theorem}{Theorem}[section]

\newtheorem{pstc}[theorem]{Concrete Positivstellensatz}
\newtheorem{pstf}[theorem]{Formal Positivstellensatz}

\newtheorem{lemma}[theorem]{Lemma}
\newtheorem{corollary}[theorem]{Corollary}
\newtheorem{proposition}[theorem]{Proposition}
\newtheorem{propdef}[theorem]{Proposition and definition}
\newtheorem{prpta}[theorem]{Expected proprerties}%[section]
\newtheorem{plcc}[theorem]{Concrete local-global principle}
\newtheorem{fact}[theorem]{Fact}

\newtheorem{theoremc}[theorem]{Theorem\etoz}

\newtheorem{corollaryc}[theorem]{Corollary\etoz}
\newtheorem{propositionc}[theorem]{Proposition\etoz}

\theoremstyle{definition}
 
\newtheorem{rstra}[theorem]{Admissible structural rules}

\newtheorem{definition}[theorem]{Definition}

\newtheorem{dfni}[theorem]{Informal definition}

\newtheorem{definota}[theorem]{Definition and notation} 
\newtheorem{definotas}[theorem]{Definitions and notations} 

\newtheorem{question}[theorem]{Question}
\newtheorem{questions}[theorem]{Questions}

\newtheorem{example}[theorem]{Example}
\newtheorem{examples}[theorem]{Examples}

\newtheorem{definitionc}[theorem]{Definition\etoz}

\theoremstyle{remark}
\newtheorem{remark}[theorem]{Remark}
\newtheorem{remarks}[theorem]{Remarks}

\newtheorem{notE}[theorem]{Note}

\newcommand\Subsection[1]{%\goodbreak
\rdb\addcontentsline{toc}{subsection}{#1} \subsection*{#1}}

\newcommand\Subsectio[2]{%\goodbreak
\rdb\addcontentsline{toc}{subsection}{#2} 
\subsection*{#1}}

\newcommand\Subsubsection[1]{%\goodbreak
\rdb\addcontentsline{toc}{subsubsection}{#1} \subsubsection*{#1}}

\newcommand\gui[1]{``{#1}''}

\renewcommand\paragraph[1]{

\rdb\addcontentsline{toc}{subsubsection}{#1} \medskip \noindent $\bullet$ \textbf{#1}}

\newcommand{\vou}{\MA{\tsbf{ op }}}
\newcommand{\Vou}{\MA{\tsbf{OP}}}
\newcommand \EXists[1] {\tsbf{Introduce }{#1\,}\tsbf{ such that }\,}
\newcommand \vet {\tsbf{,}\;}
\newcommand \Atcl {\mathrm{Atcl}}
\newcommand \Tcl {\mathrm{Tcl}}
\newcommand \Atclv {\mathrm{Atclv}}

\newcommand \Propeq {T.F.A.E.}

\newcommand \ssi {if, and only if,\xspace}

\newcommand \afr {$f$-ring\xspace}
\newcommand \afrs {$f$-rings\xspace}
\newcommand \afrdc {$2$-closed $f$-ring\xspace}

\newcommand \aftrs {strongly real rings\xspace}
\newcommand \alg {algebra\xspace}
\newcommand \algs {algebras\xspace}

\newcommand \alrds {residually discrete local rings\xspace}

\newcommand \cad {\textsl{i.e.}\xspace}
\newcommand \ie {\cad}
\newcommand \cdi {discrete field\xspace}

\newcommand \codi {discrete ordered field\xspace}

\newcommand \cof {constructive\xspace}

\newcommand \egt {equality\xspace}

\newcommand \elt {element\xspace}
\newcommand \elts {elements\xspace}
\newcommand \entrel {entailment relation\xspace}

\newcommand \eqv {equivalent\xspace}
\newcommand \eqve {\eqv}
\newcommand \eseq {essentially equivalent\xspace}
\newcommand \esid {essentially identical\xspace}
\newcommand \evc {vector space\xspace}

\newcommand \gnl {general\xspace}
\newcommand \grl {$\ell$-group\xspace}
\newcommand \grls {$\ell$-groups\xspace}
\newcommand \prmt {precisely\xspace}
\newcommand \prt {property\xspace}

\newcommand \nds {\textsl{non} discrete\xspace}
\newcommand \ndrcf {\nds \rcf}
\newcommand \ndrcfs {\nds \rcfs}
\newcommand \ndsof {\nds ordered field\xspace}
\newcommand \ndsofs {\nds ordered fields\xspace}

\newcommand \ralgs {Horn rules\xspace}
\newcommand \rdy {dynamical rule\xspace}
\newcommand \rdys {dynamical rules\xspace}
\newcommand \rsim {simplification rule\xspace}
\newcommand \rsims {simplification rules\xspace}

\newcommand \sad {dynamic algebraic structure\xspace}
\newcommand \sads {dynamic algebraic structures\xspace}
\newcommand \tho {theorem\xspace}
\newcommand \talg {Horn theory\xspace}
\newcommand \talgs {Horn theories\xspace}
\newcommand \tgm {geometric theory\xspace}

\newcommand \tdy {dynamical theory\xspace}

\newcommand \trdi {distributive lattice\xspace}
\newcommand \trdis {distributive lattices\xspace}

\newcommand \mcu {uniform continuity modulus\xspace}
\newcommand \mcus {uniform continuity moduli\xspace}

\newcommand \sagcs {semialgebraic closed set\xspace}

\newcommand \fsagc {continuous semialgebraic map\xspace}

\newcommand \dofs {discrete ordered fields\xspace}

\newcommand \rcf {real closed field\xspace}
\newcommand \rcfs {real closed fields\xspace}
\newcommand \coma {constructive mathematics\xspace}
\newcommand \clama {classical mathematics\xspace}

\newcommand \Kev {$\gK$-\evc}

\newcommand \QQlg {$\QQ$-\alg}

\newcommand \Rlgs {$\gR$-\algs}

\newcommand \RRlgs {$\RR$-\algs}

\newcommand \TwoRegles {\DeuxRegles}
\newcommand \TwoCols {\DeuxCols}

\newcommand \N {\NN}

\newcommand \oups {\\[}

\newcommand{\Tp}{{\sa T\,'}}

\newcommand \ClI[1] {[\,#1\,]}

\newcommand \thref[1] {Theorem~\ref{#1}}
\newcommand \lemref[1] {Lemma~\ref{#1}}
\newcommand \paref[1] {page~\pageref{#1}}
\newcommand \pstfref[1] {Formal Positivstellensatz~\ref{#1}}
\newcommand \pstref[1] {Positivstellensatz~\ref{#1}}

\newcommand \Note{\rdb
\noi{\sl Note. }}

\newcommand \comm{\rdb
\noi{\sl Comment. }}

\newcommand \rem{\rdb
\noi{\sl Remark. }}

\newcommand \rems{\rdb
\noi{\sl Remarks. }}

\input{MathMacrosRgeomEnglish.tex}

\makeindex

\usepackage[bookmarksopen=false,pdftex=true,breaklinks=true,%
 plainpages=false,%
 hyperindex=true,pdfstartview=FitH,%
 pdfpagelabels=true,
 linkcolor=blue,%
 citecolor=blue,urlcolor=red,
 colorlinks=true%
 ]%
 {hyperref}

\newcommand \sibrouillon[1]{}
\newcommand \Today {\sibrouillon{\vspace{-18em}\hspace{10cm}\fbox{\today}\vspace{15.6em}}}
\RequirePackage{etoolbox}

\patchcmd{\sectionmark}{\MakeUppercase}{}{}{}
\patchcmd{\chaptermark}{\MakeUppercase}{}{}{}

\begin{document} 

\title{Geometric theories for real number algebra without sign test or dependent choice axiom}
\author{Henri Lombardi and Assia Mahboubi}

\thispagestyle{empty}
% center
~
\vspace{1em}
\begin{center} 
{\bf \LARGE
Geometric theories for real number algebra
\oups.5em]
without sign test or dependent choice axiom}
\oups2em]
{\large
Henri Lombardi and Assia Mahboubi}
\oups.8em]\normalsize
First proposal
\oups.8em]
last version is available in\oups.3em] \url{http://hlombardi.free.fr/Real-Geom.pdf}

\vspace{2em}

  \today
\end{center}

\vspace{2em}

\centerline{\bf Abstract}

\medskip 
%: quotation
\begin{quotation} \label{quotation}

In this memoir, we seek to construct a dynamical theory that is as complete as possible to describe the algebraic properties of the real number field in constructive mathematics without a dependent choice axiom.

\medskip In the first part, we give a few general points about geometric theories and their dynamical version, dynamical theories.

\smallskip The second part is devoted to the study of a finitary geometric theory whose ambition is to describe exhaustively the algebraic properties of the real number field, and more generally of a \ndrcf, at least those that can be expressed in a restricted language close to the language of ordered rings. 
The result is a theory which, in classical mathematics, turns out to be the theory of local real closed rings. The theory of real closed rings is presented here in a constructive form as a purely natural equational theory, based on the virtual root maps introduced in earlier work.
All this constitutes a development, with some minor terminological modifications, of the ideas given in the article \cite{LM2017}.
Finally, we ask whether an infinitary axiom of archimedianity would provide a better understanding of the proposed finitary theory. 

\smallskip In the third part, we introduce a more ambitious theory in which continuous semialgebraic maps are given their own place: new sorts are created for them. This makes it possible to talk \und{inside the theory} about the \mcu of a continuous semialgebraic map on a bounded closed subset of $ \RR^n $. 
This new theory is resolutely infinitary. We then obtain a better description of the algebraic properties of $\RR$, but also a first outline for a constructive theory of certain o-minimal structures.

\end{quotation}

\rdb
%\addcontentsline{toc}{section}{Table des matières}
\setcounter{tocdepth}{1}
\setcounter{minitocdepth}{3}
\dominitoc

\newpage
\setcounter{page}{0}

\thispagestyle{empty}
~
\newpage

\small
\tableofcontents
\normalsize
\thispagestyle{empty}
%~
%
%
%\section*{\huge Préface}
%\addcontentsline{toc}{chapter}{Préface}
%\markboth{}{Préface}
%\vskip 1cm

\chapter*{Foreword}
\addstarredchapter{Foreword}
\markboth{Foreword}{Foreword}

The memoir we present here is an unfinished development of the article \cite{LM2017}. Compared to that paper, however, we have modified the definition of continuous semialgebraic maps (Definition \ref{defiFSAGC2}), in the same spirit in which Bishop defines a continuous real map as a uniformly continuous map on any bounded interval. 

Despite its unfinished nature and the many questions that we do not currently know how to answer, we hope that this paper will arouse interest for its original approach to the subject. 

\medskip 
The paper is written in the style of constructive mathematics à la Bishop, i.e.\ mathematics with intuitionistic logic (see \cite{Bi67,BB85,BR1987,CACM,MRR,CCAPM}).

Let us define \textsl{real algebra} as the study of the algebraic properties of real numbers, i.e., the properties of $\RR$ formulable in a first-order formal theory on the language of strictly ordered rings defined by the signature 
\Sigt{\Aso}{\cdot=0,\cdot\geq 0,\cdot>0\mathrel{;}\cdot+\cdot, \cdot\times\cdot,-\cdot, 0,1} 

\noindent with possibly all or some of the constructive reals as constants. We can also envisage introducing new function symbols for well-defined (from a constructive point of view) maps $ \RR^n\to\RR $ whose description is purely algebraic, such as the $ \sup $, $ \inf $ maps and many continuous semialgebraic maps defined on $\QQ$. 

\smallskip 
\textsl{Real constructive algebra} is not well understood! \textsl{Constructive analysis} ($ \simeq $ certified methods in numerical analysis) is much better studied.

From a constructive point of view, real algebra is far removed from the usual classical theory of real closed fields à la Artin-Schreyer-Tarski, in which we assume that we have a sign test for the reals.

Most algorithms in classical real algebra fail with real numbers, because they require \textsl{a sign test}. 

Even in constructive analysis, there could be interesting spin-offs from further study of real algebra. For example, it would help us to understand how to avoid using the axiom of dependent choice (which is common in Bishop's work).
 
\smallskip The understanding of constructive real algebra can also be a first step towards a constructive (and therefore algorithmic) theory of o-minimal structures (cf.~\cite{cos99}, \cite{vdD}). The real line and the $ \RR^n $ spaces studied from a purely algebraic point of view can be seen as constituting the simplest of o-minimal structures. The classical (non-algorithmic) theory of o-minimal structures yields pseudo-algorithms which, in order to work correctly, require at least one sign test on the reals (sorts must also be introduced for the definable parts of $ \RR^n $). And the theory of o-minimal structures has, a priori, a very important area of applications in analysis.

\medskip Thus we are looking for as complete a dynamical theory as possible to describe the algebraic properties of the real number  field in constructive mathematics without an axiom of dependent choice.

In the study we present here, we also avoid the use of negation. Fred Richman \cite{Ric2001} shows that constructive mathematics is more elegant when the axiom of dependent choice is dispensed with. We believe that they are also more elegant if negation is dispensed with.

\medskip In the first part, consisting of Chapters \ref{sectgmq} and \ref{subsecgeominfini} we give some general information on geometric theories and their dynamical version, dynamical theories. For the most part, we refer to sections 1 to 3 of the article \cite{LM2022}. 

\smallskip The second part is devoted to the study of a geometric theory whose ambition is to describe exhaustively the algebraic properties of the real number field, and more generally of a \ndrcf, at least those expressible in a restricted language, close to the language of ordered rings. This constitutes a development, with some minor terminological modifications, of the ideas given in the article \cite{LM2017}.

Chapter \ref{chapcoo} proposes a definition of the ordered field structure in the absence of a sign test.

 Chapter \ref{chap-afr} deals with $f$-rings and some derived structures.

 Chapter \ref{chapreelclos} tries to define the structure of a real closed ordered field in the absence of a sign test.

 Chapter \ref{secGeomReelsArchi} discusses an infinitary geometric theory when we add the axiom that the real number field is \textsl{archimedean}.

So, at the end of this second part, we propose for the coveted dynamical theory that of the archimedean local real closed ring structure. The theory of real closed rings is presented here in an elementary, purely equational form, in the style of \cite{Tre2007}. 

\smallskip In the third part, we add the sorts corresponding to continuous semialgebraic maps on bounded closed semialgebraic subsets. In this way, we hope to obtain a more precise description of real algebra and to be able to sketch a first constructively satisfactory theory for o-minimal structures.

\smallskip Throughout the text, theorems or lemmas in classical mathematics that have no known constructive proof, and often cannot have one, are indicated with a star. 

\smallskip Finally, the article \cite{Lom11} contains reflections, in a more philosophical framework, similar to those proposed here.

\medskip\noindent {\bf Acknowledgements} We would like to thank Michel Coste and Marcus Tressl for their patient answers to our many questions.

\bigskip 
\begin{flushright}
%\centerline{
Henri Lombardi, Assia Mahboubbi, \today
\end{flushright}
\part{Geometric theories}

\chapter*{Introduction}
\addstarredchapter{Introduction}

\rdb

This first part is the subject of a more detailed memoir in preparation \cite{Lom-tgac}, which can be found at: \url{http://hlombardi.free.fr/Theories-geometriques.pdf}

We give the main definitions and refer for the main part to sections 1 to 3 of the article \cite{LM2022}

A dynamical theory can be understood as a formalisation of a well-defined piece of intuitive mathematics. This intuitive mathematics, practised by the mathematical community, is studied in a completely computational form independent of any philosophical point of view. But where the classical point of view makes free use of \tsbf{LEM}\footnote{Law of Excluded Middle.} and the axiom of choice, dynamical theories replace these non-computational tools with the dynamical point of view of incompletely specified structures, which is the point of view of lazy evaluation in Computer Algebra. 

\smallskip 
Chapter \ref{chap-gmqfini} deals with finitary dynamical theories. 

A finitary geometric theory corresponds to what is known in classical mathematics as a coherent formal theory. But the geometric theory we are considering is governed by intuitionistic logic, whereas the coherent theory is generally governed by classical logic. 

What's more, the corresponding dynamical theory is a minimalist version of the geometric theory: it's pure computational machinery without logic, rather similar to Goodstein's recursive arithmetic.

A dynamical theory can also be seen as a partial version of natural deduction, in which the formulas examined are all of a very simple type, without the implication connector (hence without negation) and with very limited use of quantifiers. 

The surprise is that dynamical theories are nevertheless very expressive (in classical mathematics any first-order formal theory can be seen as a coherent theory) and that they erase the distinction between classical logic and intuitionistic logic.

In the frequent case where the signature is a countable set and the axioms form a decidable part of the language, the mathematical world outside the theory, which is where we situate ourselves in order to study a given structure, see how the formal system that describes it works, and establish theorems about it, has no interference with the dynamical theory itself. This is confirmed in a general way by the fundamental theorem \ref{thFond}: if we force a finitary geometric theory to behave in a classical way, the \rdys written in the initial language that are valid afterwards were already valid before.

This corresponds to the fact that Grothendieck's coherent topos, which are another form of coherent theories, have an intuitionistic internal logic, but that they can nevertheless be understood in different ways depending on whether or not we are in a constructive external mathematical world.\footnote{That is, essentially, whether or not we accept \tsbf{LEM} in this external world.} 

In constructive mathematics, only certain structures come under finite dynamical theories. For example, the discrete field structure, but not the Heyting field structure.\footnote{A discrete field is a non-trivial ring in which every element is zero or invertible, and a Heyting field is a local ring in which every non-invertible element is zero. Classical mathematics does not know the relevant distinction between the notion of a discrete field and that of a Heyting field.}\index{field!Heyting ---}\index{Heyting!field} In this respect, the restricted viewpoint of dynamical theories opens the way to a relevant classification, invisible in classical mathematics, concerning the degrees of complexity of mathematical beings invented by humans. 

\smallskip 
Chapter \ref{chap-gmqinfini} deals with infinitary dynamical theories, in which infinite disjunctions are allowed in the conclusion of a \rdy.

An essential restriction must be noted: the free variables present in such a disjunction must be specified in advance and in finite number.

Intuitively, such rules are used in the proof system of dynamical theories by ``opening the branches of computation corresponding to the infinite disjunction''. What does this mean precisely? It means that a conclusion will be declared valid if it is valid in each of the branches. 

 These theories are more expressive than finitary theories and make it possible to axiomatise a very large number of common mathematical structures. 
 
 Unlike finitary dynamical theories, the external mathematical world inevitably intervenes to certify the validity of a \rdy.
 
 Let's take a simple example and show what happens if the axioms contain an infinitary rule of the following type

\Regles{\lab {~} $ \Vdi{x_1,\dots,x_k} \Vou_{i\in I}\; \Gamma_i $ 
}

\noindent with an infinite set $ I $ and the $ \Gamma_i $ are lists of atomic formulas with no free variables other than those mentioned (i.e.\ $ \xk $). If for each $ i\in I $ we have a valid rule $ \,\,\Gamma_i\vd B(\ux) $, then we declare the rule $ \Vdi{x_1,\dots,x_k}B(\ux) $ to be valid. 

There is therefore necessarily an intuitive proof \und{external} to the dynamical theory to certify that the desired conclusion is valid in each of the branches. In fact, the computation system at work in the dynamical theory cannot handle such an infinite number of proofs. A purely mechanical computation cannot open up an infinite number of branches! For example, with $ I=\N $ the external intuitive proof could be a proof by induction.

Note, on the other hand, that the internal proof must show the validity of the desired conclusion according to the rules of proof \gui{without logic} of the dynamical theory.

\medskip\noindent 
{\bf Terminology.} 
Since we are dealing with constructive mathematics, terminological problems inevitably arise, simply because, for example, the same classical concept generally gives rise to several interesting constructive concepts which are not equivalent, but which are equivalent in classical mathematics. 

 Below are small tables comparing our terminology (in constructive mathematics) and the most common English terminology (in classical mathematics) for geometric theories. The one found in \cite[Chapter~D1]{Joh2-02}, \cite{Car2017} and in~\cite{BH2017}.

The comparison is somewhat biased by the fact that dynamical theories do not use logic as such. They are pure computational machines. Thus, although a finitary dynamical theory \gui{generate} a (first-order formal ) consistent theory and although every consistent theory admits a version \gui{finitary dynamical theory}, they are not the same formal objects. Witness the fact that a coherent theory does not work in the same way with classical logic and with intuitionistic logic, whereas a dynamical theory is insensitive to this distinction because, structurally, dynamic proofs are always constructive. 

%:\newpage
\newpage

%%%%%%%%%%%%%%%%%%%%%%%%%%%%%%%%%%%%%%%%%%%%%%%%%%%%%%%%%%%%%%%%%%%%
\begin{center}
 \tabcolsep0pt\renewcommand{\arraystretch}{0}%
\begin{tabular}{|c|}
\hline 
\Boite{.7}{12}{\large Finitary geometric theories}\\
\hline 
\end{tabular}
\begin{tabular}{|c|c|}
%\hline
\Boite{.6}{7.5}  {Our termnology}&
\Boite{.6}{4.5} {Elephant}\\
\hline
\Boite{.6}{7.5}  {geometric theory}&
\Boite{.6}{4.5} {geometric theory}\\
\hline
\Boite{.6}{7.5}  {dynamical theory}&
\Boite{.6}{4.5} {~}\\
\hline
\Boite{.6}{7.5} {purely equational}&
\Boite{.6}{4.5} {algebraic}\\
\hline
\Boite{.6}{7.5} {direct}&
\Boite{.6}{4.5} {}\\
\hline
\Boite{.6}{7.5}  {Horn}&
\Boite{.6}{4.5} {Horn}\\
\hline
\Boite{.6}{7.5}  {disjunctive}&
\Boite{.6}{4.5} {}\\
\hline
\Boite{.6}{7.5}  {propositional}&
\Boite{.6}{4.5} {propositional}\\
\hline
\Boite{.6}{7.5}  {existential, or regular}&
\Boite{.6}{4.5} {regular}\\
\hline
\Boite{.6}{7.5}  {existentially rigid}&
\Boite{.6}{4.5} {}\\
\hline
\Boite{.6}{7.5}  {existential existentialy rigid, or cartesian}&
\Boite{.6}{4.5} {cartesian}\\
\hline
\Boite{.6}{7.5} {rigid}&
\Boite{.6}{4.5} {disjunctive, \cite{Johnstone79}}\\
\hline
\Boite{.6}{7.5}  {finitary dynamical}&
\Boite{.6}{4.5} {}\\
\hline
\Boite{.6}{7.5}  {intuitionnist coherent}&
\Boite{.6}{4.5} {}\\
\hline
\Boite{.6}{7.5}  {classical coherent}&
\Boite{.6}{4.5} {coherent}\\
\hline
\end{tabular}
\end{center}

%\pagebreak	
%%%%%%%%%%%%%%%%%%%%%%%%%%%%%%%%%%%%%%%%%%%%%%%%%%%%%%%%%%%%%%%%%%%%
\begin{center}
\tabcolsep0pt\renewcommand{\arraystretch}{0}%
\begin{tabular}{|c|}
\hline 
\Boite{.7}{12}{\large General  (infinitary) geometric theories}\\
\hline 
\end{tabular}
\begin{tabular}{|c|c|}
\Boite{.6}{7.5}  {Theory}&
\Boite{.6}{4.5} {Theory}\\
\hline
\Boite{.6}{7.5}  {dynamical}&
\Boite{.6}{4.5} {}\\
\hline
\Boite{.6}{7.5}  {geometric}&
\Boite{.6}{4.5} {geometric intuitionist}\\
\hline
\Boite{.6}{7.5}  {geometric classical}&
\Boite{.6}{4.5} {geometric}\\
\hline
\end{tabular}
\end{center}
%%%%%%%%%%%%%%%%%%%%%%%%%%%%%%%%%%%%%%%%%%%%%%%%%%%%%%%%%%%%%%%%%%%%
\begin{center}
\tabcolsep0pt\renewcommand{\arraystretch}{0}%
\begin{tabular}{|c|c|}
\hline
\Boite{.7}{7.5}{Dynamical theories}&
\Boite{.7}{4.5}{Geometric theories}\\
\hline
\Boite{.6}{7.5}{identical (same signature)}&
\Boite{.6}{4.5}{equivalent}\\
\hline
\Boite{.6}{7.5}{essentially identical (same sorts)}&
\Boite{.6}{4.5}{~}\\
\hline
\Boite{.6}{7.5}{classically essentially identical (same sorts)}&
\Boite{.6}{4.5}{definitionally equivalent}\\
\hline
\Boite{.6}{7.5}{essentially equivalent}&
\Boite{.6}{4.5}{}\\
\hline
\Boite{.6}{7.5}{classically essentially equivalent?}&
\Boite{.6}{4.5}{Morita equivalent}\\
\hline
\end{tabular}
\end{center}
\newpage \thispagestyle{empty}

\chapter{Finitary geometric theories}\label{sectgmq}\label{chap-gmqfini}
\index{theory!finitary geometrical ---}\index{theory!geometrical ---}

\Today

%\vspace{-.5em}
\minitoc

\section{Coherent and finitary dynamical theories}\label{subsectdy}
%%%%%%%%%%%%%%%%%%%%%%%%%%%%%%%%%%%%%%%%%

%: Subsection{Coherent theories}
\Subsection{Coherent theories}\index{theory!coherent ---}
A \textsl{coherent theory} $\sa{T}=(\cL,\cA)$ is a first-order formal theory based on the language $\cL$ in which the axioms (the elements of $ \cA $) are all \gui{geometric}, i.e.\ of the following form:
%
%---- equation {eqAgeom} ----
\begin{equation} \label{eqAgeom}
\forall \und x \;\;\big(C\; \Longrightarrow \; \exists\, \und{y^1} \,D_1\;
\vee\;\cdots\;\vee\;\exists\,\und{y^m}\,D_m\big)
\end{equation}
%------end equation----
where $ C $ and the $ D_j $ are \textsl{conjunctions of atomic formulas} of the language $\cL$ of the formal theory, the $ \und{y^j} $ are lists of variables, and $ \und x $ is the list of other occuring variables (these lists may be empty). The variables in $ C $ are only in the $ \und x $ list. The variables in $ D_j $ are only in the disjoint lists $ \und x $ and $ \und{y^j} $. An empty disjunction in the second member can be replaced by the symbol $ \bot $ representing $ \False $.

We also say \textsl{finite geometric theory} instead of coherent theory when we use intuitionist logic.\index{coherent!theory}\index{theory!finitary geometric ---}

\Cadre{.9}{\noindent In the remainder of Chapter \ref{sectgmq}, we almost always omit the qualifier \gui{finitary} before \gui{geometric theory} or \gui{dynamical theory}}

%: Subsection{Théories dynamiques finitaires}
\Subsection{Finitary dynamical theories}\index{theory!dynamic ---}\index{dynamic!theory}\label{sectdyfinitaire}

Main reference \cite{CLR01}. This article introduces the notions of \gui{dyanamical theory} and \gui{dynamical proof}. See also: the article \cite[Bezem \& Coquand, 2005]{BC2005} which describes a number of advantages provided by this approach, and the precursor articles \cite[Prawitz 1971, sections 1.5 and 4.2]{pra1971}, \cite[Matijasevi\v c 1975]{Mat75} and \cite[Lifschitz, 1980]{Lif80}.

If \sa{T} is a (finitary) geometric theory, the corresponding \textsl{(finitary) dynamical theory} differs from it only by an extremely limited use of proof methods:
%-------begin item---
\begin{itemize}

\item 
Firstly, no formulas other than atomic formulas are ever used: no new predicates using logical connectors or quantifiers are ever introduced. Only lists of atomic formulas from the~$\cL$ language are manipulated.

\item 
Secondly, and in accordance with the previous point, axioms are not seen as true formulas, but as \textsl{deduction rules}: an axiom such as \pref{eqAgeom} is used as a rule \pref{eqRgeom}\label{NOTAvou}\index{dynamic!rule} 
%---- equation {eqRgeom} ----
\begin{equation} \label{eqRgeom}
\Gamma \vd \EXists{\und{y^1}} \Delta_1
\vou \cdots \vou \EXists{\und{y^m}} \Delta_m
\end{equation}
%------end equation----
Here the conjunctions of atomic formulas $ C $, $ D_1 $, \dots, $ D_m $ of \pref{eqAgeom} have been replaced by the corresponding lists $\Gamma$, $ \Delta_1 $, \dots, $ \Delta_m $.

\item 
Thirdly, we only prove \textsl{\rdys}, i.e.\ theorems which are in the form of the deduction rules above.\index{rule!dynamical ---}. 

\item 
Fourth, the only way to prove a \rdy is by a tree computation \gui{without logic}. At the root of the tree are all the hypotheses of the theorem we want to prove. The tree develops by applying the axioms according to pure algebraic computation machinery in the structure. See Examples \ref{exasaCd}. The precise formal definitions are given in \cite{CLR01}, we extend them to the case where there are several types of objects as in the theory of modules on a commutative ring with objects of type \gui{elements of the ring} and objects of type \gui{elements of the module}. 
\end{itemize}
%-------end item---

\smallskip When we apply an axiom such as \pref{eqRgeom}, we substitute arbitrary terms $ (t_i) $ from the language for the free variables $ (x_i) $ present in the rule. If the hypotheses, rewritten with these terms, are already proven, then branches of computation are opened in each of which fresh variables corresponding to the dummy variables $ \und{y^k} $ are introduced (their names may have to be changed to avoid conflict with the free variables present in the $ t_i $ terms) and each conclusion~$ B_k $ is valid in its branch.\footnote{$ B_k $ is the list $ \Delta_k $ in which the $x_i$ variables have been replaced by the $ t_i $ terms.} 

The very elementary examples \ref{exasaCd} show how to validate a dynamic rule in a given dynamical theory . We develop a computation tree using the axioms of the dynamical theory  as indicated above and we have won when, at each leaf of the tree, the conclusion is validated.

%: Example{exasaCd}
\begin{examples} \label{exasaCd}   
The dynamical theory  \SA{Cd} of \textsl{discrete fields} is based on the language of commutative rings and its axioms are those of non-trivial commutative rings (theory \Sa{Ac} in the example \ref{exaAc}) and 
the dynamic rule for discrete fields:

\Regles{\Lab{CD} $\vd x=0\vou\EXists {y} xy=1$}

\smallskip \noindent 1) To demonstrate the dynamic rule

\Regles{\Lab{ASDZ} $\,\,xy=0\vd x=0 \vou y=0$}

\noindent we open two branches in accordance with the axiom \tsbf{CD}.
In the first we have $x=0$ and the conclusion is proved.
In the second, we introduce a ``parameter'' (a fresh variable) $z$ with the relation \hbox{$xz=1$}. The axioms of commutative rings can then be used to prove the equalities $y=1\times y=(xz) y=(xy)z=0\times z=0$.
The conclusion is therefore validated for each of the two leaves of the tree. 

\smallskip \noindent 2) Then, for example, we deduce from the previous dynamic rule the rule 

\Regles{\Lab{Anz} $z^2=0 \vd z=0$}

\smallskip \noindent because this time both leaves of the tree have the same conclusion $z=0$.

\smallskip \noindent 3) The theory \SA{Al} of local rings is based on the language of commutative rings and its axioms are those of commutative rings (theory \Sa{Ac0} in Example \ref{exaAc}) and 
the dynamic rule for local rings

\Regles{\Lab{AL} $\,\, (x+y)z=1 \vd \EXists u \; xu=1 \;\vou\;\EXists u\;yu=1$}

\noindent To prove that a discrete field satisfies the rule \tsbf{AL},
we open two branches in accordance with the axiom \tsbf{CD}.
In the first, we have $x=0$ and the conclusion is proved because $(x+y)z=1$ gives $yz=1$.
In the second we introduce a ``parameter'' (a fresh variable) $v$ with the relation~\hbox{$xv=1$}.
The conclusion in the rule \tsbf{AL} is therefore proven at both leaves of the calculation tree.
\eoe
\end{examples}
%--------- end example ---------------------------------------- 

Note also that the validity of the following rule, which could be called \gui{Concrete existence implies formal existence}, is purely tautological. 

Let us consider a list $\Gamma(\ux,\uy)$ of atomic formulas in a dynamical theory $\sa{T}$. Let us denote $ \Gamma(\ux,\underline t)$ the list of these formulas in which we have substituted for each variable $y_j$ a term $t_j$ constructed on the $x_i$ and on the constants of the theory. Then the following existential rule is valid. 
\[
\Gamma(\ux,\underline t)\vd \Exists y_1,\dots,y_m \;\Gamma(\ux,\uy).
\]

%:paragraph{Logic replaced by computation}
\paragraph{Logic replaced by computation}~
 
\smallskip In practice, proving a \rdy within the framework of a dynamical theory always follows an intuitive natural reasoning, and this gymnastics can be seen as a simplified version of Gentzen's natural deduction. The symbol $\vou$ should be understood as an abbreviation for \gui{{\bf op}en (branches in the calculation)}.

\smallskip The symbols $\vou$ and $\;\EXists{\cdot} $ have been preferred to $\vuu$ and $\exists$, to make it clear that their use in deduction rules is not the use of new formulas constructed from atomic formulas. The symbol $\,\vd\,$ has been preferred to $\vdash$ to avoid confusion with the symbol used for entailment relations in distributional lattices. Note also that it does not have the same interpretation as the analogous symbol used in Gentzen-style sequence calculations. 

\smallskip Thus the language of a dynamical theory contains no logical symbols (connectors or quantifiers) that can be used to construct complicated formulas from atomic formulas. The \gui{logic} is replaced by the symbols $\vdi$, $\vou$ and $\;\EXists\cdot$ and by the separator \gui{$\vet$}, but these symbols are used to describe a machinery of arborescent calculations and not to form formulas. The non-logical part of a dynamical theory consists of symbols for variables, and the \textsl{signature}, which contains symbols for sorts, predicates and functions.

\Cadre{.9}{\noindent In the following, we replace \gui{\,$\EXists{\cdot} \dots$\,} with the less cumbersome \hbox{\gui{\,$\Exists\,\cdot\;\dots $\,}}, which is closer to and yet different from the traditional \gui{\,$\exists\,\cdot\;\dots $\,}.\label{NOTAExists}} 
 
%:paragraph{Equality predicate}
\paragraph{Equality predicate} ~

\smallskip 
In a dynamical theory each sort must be provided with an equality predicate $ \cdot=\cdot $ and we give the axioms which authorise the substitution of a term $t$ by a term $ t' $ when $ \vd t=t' $ is valid in the theory\footnote{This excludes the case where $t$ contains a variable $x$ under the dependence of an $ \Exists \,x $.} in any occurrence of an atomic formula present in a valid \rdy.

We could just as well not give any axioms relating to this substitution and consider that it is simply a legitimate calculation procedure. 
 
%\vspace{-.2em} 
\paragraph{Simple extension of a dynamical theory}
%d
%: Definition{defiextsimple}
\begin{definition} \label{defiextsimple}
It is said that the dynamical theory \textsl{$ \Tp=(\cL',\cA') $ is a simple extension of the dynamical theory $ \sa T=(\cL,\cA) $} if $ \cL\subseteq \cL' $ and $ \cA\subseteq \cA' $.
In this case the \rdys formulated in $\cL$ and valid (i.e.\ demonstrable) in \sa T  are valid in $\Tp$.\index{extension --- of a dynamical theory!simple} 
\end{definition}
%----------- end definition -------------------------------- 

%r
%: Remark{remdefiextsimple}
\begin{remark} \label{remdefiextsimple} 
 In the previous definition, the expression \gui{simple extension} can be questioned. If $ \cL,\cA,\cL',\cA' $ are finite sets, or if they are discrete countable sets, we can consider that everything is intuitively clear. However, it may happen that we wish to use more complicated sets, for example to introduce all the reals as constants in a theory of which one sort is intended to describe the real numbers. In such a case, the word \gui{simple extension} is questionable because there is no canonical monomorphism in Bishop's category~$ \slbSet $: in Bishop's conception, a part of a set corresponds to the categorical notion of a subobject. In this framework, therefore, \gui{simplicity} is not an objective notion, or if you prefer, it has no precise mathematical definition. \eoe
\end{remark}

\Subsection{Structural rules}

Here we give \textsl{admissible structural rules} for a dynamical theory. These are \textsl{external} deduction rules (different from \rdys, which are internal to the theory). They say that if certain \rdys are valid, then other \rdys are automatically valid.\index{rule!external deduction ---}\index{rule!admissible ---}\index{rule!structural ---}

Here are the admissible structural rules that we feel are the most important.
%: Structural rules
\begin{rstra} \label{rstr1}  ~
\begin{enumerate}\setcounter{enumi}{-1}
 
\item 
\textsl{Free variables, dummy variables}

\begin{enumerate}
 
\item \label{6rstr} 
\textsl{Substitution.} In a \rdy, you can replace all occurrences of a free variable with a term, provided that you never create a conflict between free and dummy variables. 
 
\item \label{5rstr} 
\textsl{Renaming.} In a \rdy, you can rename free variables or dummy variables (those present in the $ \Exists $) as long as you never create a conflict between free and dummy variables. 
\end{enumerate}

\item \textsl{Benefit from work already done}
 
\begin{enumerate} 
 
\item \label{12rstr} 
\textsl{Shortcuts}. Once the validity of a \rdy has been demonstrated, it can be added to the axioms of the theory.
 
\item \label{11rstr} 
\textsl{Simultaneous reinforcement of hypothesis and conclusions}. In a dynamical theory, we consider a valid rule 
\[
\Gamma \vd \Exists{\und{y^1}} \Delta_1 \vou \cdots\vou \Exists{\und{y^m}} \Delta_m
\]
Let $A$ be an atomic formula which does not involve any of the existential variables of the second member. Let $\Gamma'$ be the list $\Gamma$ followed by $A$ and $\Delta'_i$ the list $\Delta_i$ followed by~$A$. Then the following rule is also valid:
\[
\Gamma' \vd \Exists{\und{y^1}}\, \Delta'_1 \vou \cdots\vou \Exists{\und{y^m}}\,\Delta'_m
\]
\end{enumerate}%
 
\item 
\textsl{Lists as finite sets}
\begin{enumerate} 
 
\item \label{1rstr} 
\textsl{Permutation of atomic formulas appearing in a list.} 
 
\item \label{2rstr} 
\textsl{Contraction} If two identical atomic formulas appear in a list, one of the two can be deleted. \\
Conversely, you can duplicate an atomic formula in an arbitrary list.
 
\item \label{3rstr} 
\textsl{Monotony}. Atomic formulas can be added as required to the list to the left of $ \vd $. 
 
\item \label{4rstr} 
\textsl{Permutation, contraction and monotony for the $\vou$ to the right of the $ \vd $}. 

\end{enumerate} 

\item \label{10rstr} 
\textsl{Lists of atomic formulas as conjunctions}

\begin{enumerate}
 
\item \label{10rstr1} 
\textsl{To prove a list of atomic formulas is to prove each of them.}
In a theory, consider a \rdy $ \Gamma\vd (A_1\vet \dots\vet A_n) $. This \rdy is valid if, and only if, the rules $ \, \Gamma\vd A_k \; (k\in\lrbn) $ are valid. 
 
\item \label{10rstr2}
\textsl{Distributivity of $\vou$ on the implicit \gui{and} in the lists.}
In a theory, we consider a \rdy 
\[
\Gamma\vd (A_1\vet \dots\vet A_n) \vou \Exists{\und{y^1}}\,\Delta_1\vou \dots \vou \Exists{\und{y^m}}\,\Delta_m.\] 
This \rdy is valid if, and only if, the following \rdys are valid 
\[
\Gamma\vd A_k \vou \Exists{\und{y^1}}\,\Delta_1\vou \dots \vou \Exists{\und{y^m}}\,\Delta_m\quad \quad (k\in\lrbn).
\] 
\end{enumerate}
 
\item \textsl{Transitivity and variants}
\begin{enumerate}
 
\item \label{7rstr} 
\textsl{Transitivity.} We give an example, leaving it to the reader to give the general formulation. Let us suppose that we have valid \rdys in a dynamical theory 
\[ 
\begin{array}{rcl} 
\Gamma(\ux) & \vd & \Exists y,z\; \Delta_1(\ux,y,z)\vou \Exists u \;\Delta_2(\ux,u), \oups.3em] 
\Gamma(\ux),\Delta_1(\ux,y,z) & \vd & \Exists r,s,t\; \Delta_3(\ux,y,z,r,s,t), \oups.3em] 
\Gamma(\ux),\Delta_2(\ux,u) & \vd & \Exists v\; \Delta_4(\ux,u,v)
\vou \Exists w\; \Delta_5(\ux,u,w). 
 \end{array}
\]
Then the rule
\[
\Gamma(\ux)\vd\Exists y,z,r,s,t\; \Delta_3(\ux,y,z,r,s,t)\vou \Exists u,v\; \Delta_4(\ux,u,v)\vou \Exists u,w\; \Delta_5(\ux,u,w).
\]
is also valid.
 
\item \label{8rstr} \textsl{Cut.} 
Consider lists of atomic formulas $\Gamma(\ux),\Delta_0(\ux),\Delta_1(\ux),\dots,\Delta_m(\ux)$  $(m\geq 1) $ in a dynamical theory $ \sa{T} $. If the two \rdys 
\[
\Gamma\vd\Delta_0\vou\Delta_1\vou\dots\vou\Delta_m \quad \hbox{ and }\quad \Gamma,\Delta_0\vou\Delta_1\vou\dots\vou\Delta_m 
\]
are valid in $\sa{T}$, then the rule $\Gamma\vd\Delta_1\vou \dots\vou \Delta_m$ is also valid.
 
\item \label{8rstr2} \textsl{Cut with existence.} 
A more general version is as follows. 
Consider lists of atomic formulas $\Gamma(\ux),\Delta_0(\ux,\und{y^0}),\Delta_1(\ux,\und{y^1}),\dots,\Delta_m(\ux,\und{y^m})$ $(m\geq 1)$ dans une \tdy $\sa{T}$ $(m\geq 1)$ in a dynamical theory $ \sa{T}$.
If the two \rdys 
\[\hspace{-1.5em}
\Gamma\vd\Exists{\und{y^0}}\,\Delta_0\vou \Exists{\und{y^1}}\,\Delta_1\vou \dots\vou \Exists{\und{y^m}}\,\Delta_m \; \hbox{ and }\; \Gamma,\Delta_0\vd\Exists{\und{y^1}}\,\Delta_1\vou \dots\vou \Exists{\und{y^m}}\,\Delta_m 
\]
are valid in $\sa{T}$, then the rule
$\Gamma\vd\Exists{\und{y^1}}\,\Delta_1\vou \dots\vou \Exists{\und{y^m}}\,\Delta_m $ 
is also valid.
\end{enumerate}
\end{enumerate}
\end{rstra}

\Subsection{Collapsus}\index{collapsus}%
\index{rule!collapse ---}\label{NOTABot}

A \rdy is called a \textsl{collapse rule} when the second member is \gui{$ \Faux $}, which we note $\Bot$. We can also see $\Bot$ as designating the empty disjunction. Once $\Bot$ has been proved, the universe of discourse collapses, and every atomic formula is then deemed to be \gui{true}, or at least \gui{valid}. This is the application of the rule \gui{ex falso quod libet}, which is the relevant intuitive meaning of $ \False $ in constructive mathematics. Thus in dynamical theories the rules 

\Regles{\lAb{False$ _{P} $} $ \,\,\Bot\vd P $ \quad (ex falso quod libet)} 

\noindent are valid for all atomic formulas.

In the language, we also give the logical constant $\Top$ for \gui{$ \True $}, with the following Horn rule  as its axiom. \label{NOTATop} 

\Regles {\lAb{True} $ \,\,\vd \Top $}

We can also see $\Top$ as designating the empty conjunction.\footnote{When there's nothing to prove, let's prove nothing and everything will be OK. Moreover, in a dynamical theory with at least one sort $ \iS $, $\Top$ is equivalent to $ x=_\iS x $.}
The constants $\Bot$ and $\Top$ are the only logical symbols in dynamical theories.

When a dynamical theory has no collapse rule, it always admits the model reduced to a point\footnote{If there are several sorts, each sort is reduced to a point.} where all atomic formulas are evaluated true. This is the final object in the category of models of the theory.

We can say that a dynamical theory without a collapse rule collapses if all the atomic formulas are valid, with the exception of $\Bot$.

A dynamical theory with a collapse rule is said to collapse when $\Bot$ is provable, and consequently so are all the \rdys. In this case the theory admits no model.
 
To consider collapse in the sense of a single model reduced to a point, rather than in the sense of pure nothingness, is merely a matter of taste which changes nothing in the essence of things.%
\footnote{In fact, one of the authors must have a horror of the void~:{\tiny )}, the silence of this infinite space frightens him~:{\tiny (}. Moreover, if total disappearance into nothingness is the true meaning of $ \False $, the fact remains that, even before forbidding the existence of models, $ \False $ begins by reducing them to a single point, which satisfies all the predicates. As Boris Vian's song says: \gui{on est descendu chez Satan et en bas c'était épatant!}.} 

Instead of saying that a collapsing dynamic algebraic structure has no model, we say (without negation) that any model of this dynamic algebraic structure is trivial, reduced to a point, and that \gui{everything in it is true}.

To formally reconcile these two points of view, the best solution seems to be the following: each sort $\iS$ introduced is accompanied by at least two constants in this sort, say $0_\iS$ and $1_\iS$ to fix ideas, with the axiom \fbox{$ 0_\iS=_\iS1_\iS\vd\Bot $}. In what follows, this is what would normally happen for the theory of non-trivial distributive lattices and the theory of non-zero commutative rings, as well as in all their extensions. But we prefer to use the following convention.

\Cadre{.9}{\noindent Throughout this memoir, in the case of a ring or a distributive lattice, we consider that the collapse is always given in the form $1=0$ or a formula of the same style, for example $0>0$ for an ordered field. One disadvantage of using the symbol $\Bot$ is that it takes us out of the realm of Horn theories when we could be staying there. Readers who so wish can add an axiom of the type $ 1=0\vd \Bot $.}

\Subsection{Classification of dynamical theories} 

\paragraph{Horn theories}~

\smallskip\noindent A \rdy which does not contain to the right of the symbol $\vd\,$ either $\vou$, or $\,\Exists\,$ or $\Bot$ is called a \textsl{Horn rule}. A dynamical theory is said to be \textsl{Horn} when it contains only Horn rules as axioms. In the french version of this paper and in the paper \cite{CLR01}, Horn theories are called \textsl{théories algébriques}. A special case is provided by purely equational theories, which are Horn theories with a single sort and the only predicate being the equality predicate.%.
\index{Horn!rule}\index{Horn!theory}\index{rule!Horn ---}%
\index{theory@Horn ---}

\smallskip
\Note We use the following terminology from \cite{CLR01}. A Horn rule is said to be \textsl{direct} when, to the left of the symbol $ \vd\, $, there are only atomic formulas relating to variables different from each other or to constants. The other Horn rules are called \textsl{\rsims}. For example, in the two examples below, the first is direct, the second is not.%
\index{direct!rule}\index{simplification!rule}\index{rule!simplification ---}\index{rule!direct ---}

\DeuxRegles
{\labu $ x=0\vet y=0\vd x+y=0 $}
{\labu $ x+y= 0\vet x=0\vd y=0 $}

\paragraph{Disjunctive theories}~

\smallskip \noindent A dynamical theory is said to be \textsl{disjunctive} if in the axioms there are no $ \Exists $ to the right of the~$ \vd $ \index{theory!disjunctive ---}\index{rule!disjunctive ---}. 

\paragraph{Existential theories}~

\smallskip\noindent A \rdy is said to be \textsl{simple existential} if the second member (the conclusion) is of the form $ \Exists \uy\; \Delta $ where $\Delta$ is a finite list of atomic formulas.\index{rule!simple existential ---}

A dynamical theory is said to be \textsl{existential} if its axioms are all simple algebraic or existential rules (a Horn rule  can also be considered as a special case of a simple existential rule). A typical existential theory is the theory of \textsl{Bézout rings} (any finitely generated ideal is principal). In the English literature on categorical logic (studied in the context of classical mathematics), an existential theory is called a \textsl{regular theory}. In the french version of this paper an existential theory is called a \textsl{théorie existentielle}%.
\index{theory!existential ---}\index{theory!regular ---}

\paragraph{Existentially rigid, cartesian theories}~

\smallskip\noindent 
Existentially rigid theories are dynamical theories in which the existential axioms are simple and correspond to unique existences. This generalises (very slightly) the disjunctive theories. 

\smallskip 
An existentially rigid theory is said to be cartesian.
This generalises (very slightly) Horn theories. 

\paragraph{Rigid theories}~

\smallskip\noindent A dynamical theory is said to be \textsl{rigid} if it is existentially rigid and if the disjuncts of the second member in the disjunctive axioms are two by two incompatible. For example, the theory of discrete fields is rigid, but the theory of local rings is not. The theory of real closed discrete fields can be stated rigidly, but the theory of algebraically closed discrete fields cannot. See \cite{Johnstone79}, who uses the terminology \gui{disjunctive theory} where we use \gui{rigid theory}.

\paragraph{Propositional theories}~

\smallskip\noindent
The (classical or intuitionistic) logic of propositions has a very abstract character, which may seem useless from a dynamic point of view, since it is already present in the form of some of the admissible structural rules \ref{rstr1}. However, it is useful for the definition of distributive lattices associated with dynamic algebraic structures (Section \ref{subsectrdisad}.)
 
 The logic of propositions can be presented in a minimal dynamic form as follows, without any sort, which implies that the constants must be interpreted as pure abstract truth values (in classical logic they only hesitate between $\True$ and $\False$). 
 
The constants are therefore $\Bot$, $\Top$, and \textsl{propositional constants} or \textsl{propositions}. To define such a theory $(\cL,\cA)$, we give a set $G$ of propositional constants.\footnote{This fixes the language $\cL$ via the signature $\so{G,\Top,\Bot}$} and a set $ \cA $ of axioms which are disjunctive rules on the language $\cL$) 

First we have the axioms $ \Bot\vd p $ and $ p \vd \Top $, and the axioms which handle equality in $G$: $ p\vd q $ each time $ p=_G q $. If $G$ is not a discrete set, these axioms reflect the structure of the set $G$ in the informal category \sa{Set}.

The additional axioms given in $ \cA  $ are of the type $ p_1\vet p_n\vd q_1 \vou \dots \vou q_m $ where $ p_i $ and $ q_j $ are constants in $G$ (with possibly $ m=0 $ or $ n=0 $).

Two constants $p$ and $ q $ are said to be \textsl{opposite} or \textsl{complementary} if they satisfy the axioms of negation

\DeuxRegles
{\labu $ \vd p\vou q $}
{\labu $ p\vet q\vd \Bot $}

While classical propositional logic can be interpreted as given by dynamical theories without any sort of the type described above, the same cannot be said for intuitionistic logic. Indeed, the~$\Rightarrow$ connector cannot be described by restricting ourselves to dynamical theories. Obviously, this connector can be introduced into the language, but the external structural rule used to introduce implication in natural deduction cannot be formulated as a \rdy. 

%%%%%%%%%%%%%%%%%%%%%%%%%%%%%%%%%%%%%%%%%
\section{Dynamic algebraic structures}\label{subsubsecSAD}

References: \cite{CLR01}, \cite{Lom06}, \cite{LM2022}.
 
\smallskip The dynamic algebraic structures are explicitly named in~\cite{Lom06}. 
In~\cite{CLR01}, they are implicit, but explicit in the form of their presentation. They are also implicit in~\cite{Lom02}, and, last but not least, in \cite[{D5}, 1985]{D5}, which has been an essential source of inspiration: one can compute safely in the algebraic closure of a discrete field, even when it is not possible to construct this algebraic closure. It is therefore sufficient to consider the algebraic closure as a dynamic algebraic structure \gui{à la D5} rather than as a usual algebraic structure: \textsl{lazy evaluation in D5 provides a constructive semantics for the algebraic closure of a discrete field}.

\Subsection{Definitions, examples} 

% : Example{exaAc}
\begin{example} \label{exaAc} 
Our first example is the purely equational theory of commutative rings (with only one sort, called $ \Ac $) in which most of the calculations are entrusted to machines outside the formal theory. This possibility is based on the fact that the elements of the ring $ \ZZxn $ can be reduced to a predefined normal form. This implies that the equality of two terms is equivalent to the identity of their normal forms. Consequently, the binary equality predicate can be replaced by the equality to~$0$ predicate.

The \textsl{theory \SA{Ac0} of commutative rings} is written on the following signature. There is only one sort, called $ \Ac$.
\Sigt{\Ac}{\cdot=_\Ac0\mathrel{;}\cdot+\cdot,\cdot\times \cdot,-\,\cdot,0_\Ac,1_\Ac} \label{NOTASigAc}

\noindent The only axioms are the following (these are direct rules):\footnote{The names of the rules are written as follows: for the direct rules, all lower case, for the other Horn rules (the \rsims), the first letter in upper case, and finally the other \rdys, all upper case.} 

\DeuxRegles{
\Lab{ac0} $ \vd 0_\Ac =_\Ac0 $ 
\Lab{ac2} $ \,\,x=_\Ac0\Vdi{x,y:\Ac} x\times y=_\Ac0 $ 
}
{
\Lab{ac1} $ \,\, x=_\Ac 0\vet y=_\Ac 0\Vdi{x,y:\Ac} x+y=_\Ac0 $ 
}

\smallskip 
The term \gui{$ x-y $} is an abbreviation for \gui{$ x+(-y) $} and the binary predicate \gui{$ \cdot=\cdot $} is \textsl{defined} by convention: \gui{$ x=y $} is an abbreviation for \gui{$ x-y=0 $}.

\smallskip \rdb 
We often consider the theory $\SA{Ac}$ of \textsl{\und{non-trivial} commutative rings}, which is obtained from $ \sa{Ac0} $ by adding the collapse axiom

\Regles{\lAb{CL$_{\Ac}$} $\,\,1=_\Ac0\vd \Bot$ \label{AxCLnqAc}}

\smallskip\noindent \textbf{Explanations.}\label{Ac-comments}

\noindent \textsl{1.} The rules that define the \sa{Ac0} theory of commutative rings must be understood precisely as follows. Any term of the theory can be seen as a polynomial with integer coefficients in the present variables. We then use the computational machinery of commutative polynomials with integer coefficients (\gui{external} to the theory), which rewrites any term (formed over constants and variables) as a polynomial with integer coefficients in a predefined normal form.

\noindent The distributivity rule $ x(y+z)=_\Ac xy+xz $, for example, is then entrusted to an automatic calculation which reduces to $0$ the term
 $ x(y+z)-(xy+xz) $.

\noindent Similarly, the transitivity of binary equality is handled by the rule \Tsbf{ac1} and by the automatic calculation which reduces the term 
 $ (x-y)+(y-z) $ to $ (x-z) $.

\smallskip\noindent \textsl{2.} In the three rules \Tsbf{ac0}, \Tsbf{ac1} and \Tsbf{ac2} we recognise the axioms of ideals, which make it possible to create a quotient ring structure, and which signify the compatibility of equality with addition and multiplication. In the \sa{Ac0} theory, any atomic formula is of the form \gui{$ t(\xn)=_\Ac0 $} where $x_i$'s are variables and $t$ a term of the language. Any atomic formula is therefore immediately equivalent to an atomic formula in which $t$ is an element of the ring $\ZZ[\xn]$, written in the agreed normal form. The \sa{Ac0} theory is therefore the \gui{the theory of algebraic identities}, in the old sense of the expression. Precisely, it is easy to check that the validity of a simple Horn rule such as

\Regles{\labu $ \,\,p_1=_\Ac0\vet\dots\vet p_m=_\Ac0\Vdi{\xr:\Ac} q=_\Ac0 $}

\noindent means exactly that the polynomial $q\in \ZZxr$ is in the ideal generated by the polynomials $p_1,\dots,p_m$ of $\ZZxr$. This property is rather difficult to decide.\footnote{This follows for example from Theorem VIII-1.5 in \cite{MRR}.} 

\noindent The dynamical theory of a purely equational theory does not provide any additional tool to the purely equational theory itself. There is therefore nothing really \gui{dynamic} about purely equational dynamical theories. The really interesting dynamical theories are obtained by adding dynamical axioms to Horn theories.

\noindent For the theory \sa{Ac0} the validity of more complicated rules than those considered above is handled by the structural rule \ref{10rstr1} \paref{10rstr1}. 

\smallskip\noindent \textsl{3.} The theory \sa{Ac0} as it is presented does not seem \gui{purely equational} at first sight because the axioms are not simple equalities between terms. This is due to our decision to replace equality with the unary predicate \gui{$ \cdot=0 $} accompanied by the external computational machinery of polynomials with integer coefficients. This approach has the advantage, in our opinion, of showing the true logical structure of the theory by reducing it to three very simple axioms and by entrusting to an automatic calculation what can be entrusted to it, which has little to do with logic proper. The same remark will subsequently apply to many theories that we will describe as purely equational. \eoe
\end{example}
%--------- fin example ---------------------------------------- 

%d
%:     Definition{defiSAD}
\begin{definition} \label{defiSAD}
If $\sa{T}=(\cL,\cA)$ is a dynamical theory, a \textsl{dynamic algebraic structure of type~\sa{T}} is given by~a set $G$ of \textsl{generators} and~a set $ R $ of \textsl{relations}. A \gui{relation} is by definition an atomic formula $ P(\und{t}) $ constructed on the language $ \cL\cup G $ with closed terms $ t_i $ in this language. Such a relation is associated with the axiom \gui{$ \vdi P(\und{t}) $} of the dynamic algebraic structure. So, this \sad is the \tdy $(\cL\cup G,\cA\cup R)$, also denoted by $\big((G,R),\sa{T}\big)$. 
\end{definition}
%----------- fin definition -------------------------------- 

%: Example{exaSaCd}
\begin{example} \label{exaSaCd} 
For example, we obtain a dynamic algebraic structure for a discrete field 

\snic{\gK=\big((G,R),\Sa{Cd}\big)}

\noindent by taking $ G=\so{a,b} $ and $ R=\so{105=0,\,a^2+b^2-1=0}. $ This dynamic discrete field corresponds to any discrete field of characteristic $ 3 $ or $ 5 $ or $ 7 $ generated by two elements $\alpha$ and $ \beta $ satisfying $ \alpha^2+\beta^2=1 $. 

\noindent In addition to the \rdys valid in all discrete fields, there are now those obtained by extending the language with the constants taken from $G$ and by adding to the axioms the relations taken from $ R $.
For example the disjunctive rule

\Regles{\labu $3=0 \vou 5=0  \vou 7=0$}

\noindent is valid, and so is the  \rdy

\Regles{\labu $\Exists z \;15z=1 \vou \Exists z \; 21z=1  \vou \Exists z \;35z=1$ \eoe}
\end{example}
%--------- end example ---------------------------------------- 

%d
%: Definition{defiFact} 
\begin{definotas} \label{defiFact}  \label{notasadreglevalide}~\\
Let $ \gS=\big((G,R),\sa{T}\big) $ be a dynamic algebraic structure of type $\sa{T}=(\cL,\cA)$.
\begin{itemize}
 
\item We will indicate that the rule \gui{$\,\Gamma\vd \dots $} is valid in the dynamic algebraic structure $\gS$ in the following abbreviated form: \gui{$ \,\Gamma\vdi_{\gS} \dots $}. We could also use the notation \gui{$ \,R\vet \Gamma\vdi_{\sA{T}} \dots $}, which means that the proof can use a finite list of axioms extracted from $ R $.
 
\item The set of closed terms of $\gS$, i.e.\ the terms built on $\cL\cup G$, is denoted by $\Tcl(\gS)$. The set of closed atomic formulas is denoted by $\Atcl(\gS)$. 
  
\item An Horn rule $\vd P$ for~$P\in\Atcl(\gS)$ is called \textsl{a fact in} $\gS$. The set of valid facts in $\gS$ is called $\Atclv(\gS)$. A fact only concerns syntactically definable objects in the structure. It is clear that $\gS$ proves exactly the same \rdys as the dynamic algebraic structure $ \wi\gS=\big((\Tcl(\gS),\Atclv(\gS)),\sa{T}\big) $.
\end{itemize}
\end{definotas}
%----------- fin definition -------------------------------- 

Concrete algebra very often consists of proving facts or dynamic rules in particular dynamic algebraic structures. It is a little more general than the (inexhaustible) theory of algebraic identities, i.e.\ the universal algebra behind a large proportion of the great theorems of abstract algebra.

\smallskip 
In the case of a Horn theory \sa{T}, a dynamic algebraic structure of type \sa{T} gives a usual algebraic structure, defined by generators and relations, satisfying the required Horn rules. 

\smallskip 
The dynamic method is often a practical way of constructing algebraic identities (\gui{Positivstellensätze} for example), following as closely as possible the paths indicated in the proofs given in classical mathematics. 
 
\smallskip In a dynamic algebraic structure a fact $ P(\und t) $ is \textsl{absolutely true} if it is provable (i.e.\ if the rule \gui{$ \vdi P(\und t)\,$} is valid). It is \textsl{absolutely false}, or more precisely \textsl{catastrophic} if \gui{$ P(\und t)\vd \Bot $} is valid. There are many possibilities in between these two cases: a dynamic algebraic structure does not have a single fixed model, but represents all the possible ideal realisations of the structure in the potential state (this notion remains deliberately vague). Adding a catastrophic fact as an axiom amounts to eliminating all models.\footnote{In the variant where the collapse reduces all models to a singleton: \dots\ amounts to allowing only the trivial model.}

%e
%: Example{exasdz}
\begin{example} \label{exasdz} 
We consider a presentation $(G,R)$ in the language of  \Sa{Ac}. Let \sa T be a dynamical theory which extends the theory \Sa{Ac} without extending the language, for example the theory \sa{Cd} of discrete fields. Any closed term of the dynamic algebraic structure $ \big((G,R),\sa T\big) $ is rewritten as a polynomial $ f(\ux)\in\ZG $ with integer coefficients in the \gui{constants} $ x_i\in G $ of the dynamic algebraic structure. The elements of $ R $ are relations $ f(\ux)=0 $, so that by a slight abuse of language, we can consider $ R $ as a set of elements of $ \ZG $. 

\noindent We are therefore studying the ring $ \gA=\aqo\ZG R $, or more precisely what happens to this ring when we ask it to satisfy certain new axioms. We will note~$ \sa T(\gA) $ the dynamic algebraic structure $ \big((G,R),\sa T\big) $. 

\noindent In many examples, the theory collapses if, and only if, $\gA$ is trivial. For example, the dynamic algebraic structure $ \Sa{Cd}(\gA) $ collapses if, and only if, $ 1=_\gA0 $. In classical mathematics we say: indeed a non-trivial ring has a prime ideal $\fp$, and the field of fractions of the integral ring~$ \gA/\fp $ is a non-trivial model of $ \Sa{Cd}(\gA) $. More simply, without using model theory or the axiom of the prime ideal, we transform a proof of $ 1=0 $ in~$ \Sa{Cd}(\gA) $ into a proof of $ 1=0 $ in~$ \Sa{Ac}(\gA) $ (proof analogous to that of \cite[Theorem~2.4]{CLR01}).

\noindent Concerning the facts $\theta=0$\footnote{With $\theta=t(\uxi)\in \gA$, where $t\in\ZZ[G]$ and the $\xi_k$ are the $x_k$ seen in the quotient $\gA$ of $\ZZ[G]$.} valid in the theory $\sa T(\gA)$, the situation is a little more complicated.\\
The theory of local rings \Sa{Al} proves $\theta=0$ exactly when $\theta=_\gA0$, hence the great importance of local rings in commutative algebra.\index{ring!local ---}
\\
The theory \sa{Cd} proves $\theta=0$ exactly when $ \theta \in\!\sqrt[\gA]0 $. This corresponds to the reduced quotient of~$\gA$: if $ \theta =0 $ in $ \sa{Cd}(\gA) $ then $ \theta $ is nilpotent in $\gA$. This is a (relatively weak) abstract form of the Nullstellensatz. The proof is elementary. Naturally, if two theories prove the same facts, they can differ in terms of \rdys that are more general tha Horn rules. \eoe
\end{example}
%--------- fin example ---------------------------------------- 

\Subsection{Positive diagram of an algebraic structure} 

%: Definition{defidiagramme}
\begin{definota} \label{defidiagramme}~
\begin{enumerate}
 
\item Let $ \sa T_1=(\cL_1,\cA_1) $ be a dynamical theory and $\gA$ an algebraic structure over a language $ \cL\subseteq \cL_1 $. We call \textsl{positive diagram of $\gA$ for the language $\cL$}, a presentation $(G,R)$ of~$\gA$ as an algebraic structure over the language $\cL$. Such a diagram is denoted~$ \Diag(\gA,\cL) $. \\
We then denote $ \sa T_1(\gA) $ the dynamic algebraic structure $ \big(\Diag(\gA,\cL),\sab{T}_{\!1}\big) $.%.
\index{positive diagram!of an algebraic structure on a given language} 
 
\item Let $ \sa T=(\cL,\cA) $ be a dynamical theory, $\gB$ a model of $ \sa T $, and $ \sab{T}_{\!1}=(\cL_1,\cA_1) $ a simple extension of $ \sa{T} $. Consider a presentation $(G,R)$ of $\gB$ as a dynamic algebraic structure of type \sa{T}. Such a diagram is called \textsl{positive diagram of $\gB$ for the dynamical theory $ \sa T $} and it is denoted~$ \Diag(\gB,\sa T) $. We then note $ \sab{T}_{\!1}(\gB) $ the dynamic algebraic structure $ \big(\Diag(\gB,\sa T),\sab{T}_{\!1}\big) $.%.
\index{positive diagram!of a model of a dynamical theory}
\end{enumerate}
\end{definota}
Item 1 can be seen as a special case of item 2, where $ \cA=\emptyset $.

%r
%: Remark{remdefidiagramme}
\begin{remarks} \label{remdefidiagramme}~

\noindent 1) Here is a typical example for Item 1. The theory $\sa T_1$ is a simple extension of the theory \sa{Ac} where equality is the only predicate. Let $\cL$ be the language of~$\sa{Ac}$ and let~$\gA$ be a commutative ring. For generators of $\Diag(\gA,\cL)$ we can take the elements of the set underlying~$\gA$, and for relations we can restrict ourselves to the equalities $0_\gA=0$, $1_\gA=1$, $-a=b$, $a+b=c$ and $ab=c$ when they are satisfied for elements of $\gA$. This positive diagram does not contain any $a\neq b$ inequalities for the simple reason that they are not part of the language of \sa{Ac}. This is why we call it a \gui{positive} diagram.

\smallskip\noindent 2) 
An element $a$ of $\gA$ does not always have a canonical representative in a set à la Bishop, even if the set is discrete. In such a case, to return to the definition of the set underlying $\gA$ according to Bishop, we can take a different constant $ x_b $ for each representative $b$ of the element~$a$. We then find in the positive diagram of $\gA$ a relation $ x_b=x_c $ each time $ b=_\gA c $.
\eoe 
\end{remarks}
%----------- fin remark ---------------------------------- 

%: Subsubsection{Constructive versus classical models}
\Subsection{Constructive versus classical models}\label{subsecmodelescofsSAD}
%-----------
Consider a dynamic algebraic structure $ \gA=\big((G,R),\sa T\big) $ of type \sa{T} with one or more sorts. To simplify the notation, we assume a single sort. A \textsl{model of~$\gA$} is a usual (static) algebraic structure~$ M $ described in the language associated with~$\gA$ and verifying the axioms of $\gA$ (those of $ \sa{T} $ and those given by the presentation of $\gA$).

When $\gA$ is defined by the empty presentation, we speak of \textsl{models of \sa T}.

\smallskip 
The notion of model is therefore based a priori on an intuitive notion of \textsl{algebraic structure} à la Bourbaki. We can describe these algebraic structures as \gui{static} in contrast to the general dynamic algebraic structures. Note that here the set \gui{underlying} the structure is a \gui{naive set} (or several naive sets if there are several sorts) structured by giving predicates and functions (in the naive sense) subject to certain axioms.
 
From a constructive point of view, the models must satisfy the axioms by respecting the intuitive sense of \gui{or} and \gui{there exists}: to prove that a particular algebraic structure satisfies the axioms, we allow only intuitionistic logic. Note also that the set theory to which we refer is a priori Bishop's informal set theory. 

%%%%%%%%%%%%%%%%%%%%%%%%%%%%%%%%%%%%%%%%%%%%%%%%%%%%%%%%%%%%%%%%%%%% 
%: Subsection{Morphismes entre structures algébriques dynamiques de même type}
\Subsection{Morphisms between dynamic algebraic structures of the same type}
%-----------

Consider a disjunctive theory \sa{T}.
In this case, a possible natural notion of morphism from a dynamic algebraic structure $ \gA=\big((G, R), \sa{T}\big) $ to a dynamic algebraic structure $ \gA'=\big((G', R'), \sa{T}\big) $ for the disjunctive theory \sa{T} is as follows.

 An element of $ \Sad_{\sA T}(\gA,\gA') $ is given by a map $ \varphi\colon G\to \Tcl(\gA') $ which interprets the elements of $G$ by closed terms of $\gA'$. This map uniquely extends into a map $\Tcl(\gA)\to \Tcl(\gA')$ respecting the construction of the terms by means of the function symbols present in $\cL$. In addition, the elements of $ R $ must give valid facts in $\gA'$ according to this interpretation.
 
The equality between two elements $\varphi$ and $\psi$ of the set $M=\Sad_{\sA T}(\gA,\gA')$ is defined as follows: one has $ \varphi=_M\psi $ if, and only if, for all $x\in G$, the equality $ \varphi(x)=\psi(x) $ is valid in $\gA'$.
 
The composition of morphisms is defined in a natural way for three dynamic algebraic structures $\gA$, $\gB$ and $\gC$. 

\smallskip This gives us a very interesting (informal) category. The objects are dynamic algebraic structures of type \sa{T}. The set of arrows from $\gA$ to $\gA'$ is $ \Sad_{\sA T}(\gA,\gA') $. 

This category has arbitrary limits and colimits, constructed very naively at the level of the presentations $(G,R)$, based on the naive intuitive set theory that we consider in the ambient mathematical world. 

For example, the product in the category $ \Sad_{\sA T} $ of $ \gA=\big((G, R), \sa{T}\big) $ and $ \gA'=\big((G', R'), \sa{T}\big) $ is the dynamic algebraic structure of type \sa{T} whose presentation is given by $ (G\times G',R\times R') $. When~\sa{T} is the theory of discrete fields, in the disjunctive version where a discrete field is defined as a connected reduced zero-dimensional ring (a function symbol must be introduced for the quasi-inverse). The previous product is a new dynamic algebraic structure of a discrete field, and $ \gA\times \gA' $ is provided with a usual algebraic structure of a reduced zero-dimensional ring. This situation seems analogous to that of bundles of discrete fields (according to the semantics of Kripke-Joyal), which are discrete fields only in the fibres.

\smallskip \rem If we are dealing with an existentially rigid theory, we can reduce ourselves to the case of a disjunctive theory by skolemisation of the rigid existential rules. However, it seems that in the case of a dynamical theory with non-rigid existential axioms, things are not very clear. \eoe 

\smallskip 
Sometimes we are interested in a more restrictive notion of morphism between two dynamical algebraic structures $\gA$ and $\gA'$ of the same type \sa{T}, for example the notion of local morphism between commutative rings, adapted to a specific context. In such a case, we would like the new morphisms from $\gA$ to $\gA'$ to be treated as given by dynamic algebraic structures for a certain dynamical theory defined from \sa{T}, $\gA$ and $\gA'$ (as may be the case for local morphisms).
 
%:Subsection{Examples}
\Subsection{Examples}

\noindent 1) The disjunctive theory \SA{Asdz0} (resp. \SA{Asdz}) of \textsl{rings without zerodivisor} is obtained 
from the theory \Sa{Ac0} (resp.\ \Sa{Ac}) by adding the \rdy%
\index{ring!without zerodivisor ---}%

\Regles {\lab{ASDZ} $\,\,xy=0\vd x=0 \vou y=0$}

\smallskip \noindent 2) Reference \cite[section VIII-3]{ACMC}. 
The existential theory \SA{Alsdz0} (resp. \SA{Alsdz}) of the \textsl{rings locally without zerodivisor} is obtained by adding to the theory \Sa{Ac0} (resp. \Sa{Ac}) the axiom \tsbf{LSDZ}:%
\index{ring!locally without zerodivisor ---}

\Regles {\Lab{LSDZ} $\,\,xy=0\vd \Exists u,v \;(ux=0\vet vy=0\vet u+v=1)$}

\smallskip \noindent 3) Dynamical theory \SA{Ai} of \textsl{integral rings} (or \textsl{domains}).
With the signature
\Sigt{\Ai}{\cdot=0,\cdot\neq0\mathrel{;}\cdot+\cdot,\cdot\times \cdot,-\,\cdot,0,1} \label{NOTASigAi}
\noindent the theory  of integral rings is obtained 
from the theory \Sa{Ac0} by adding as axioms the following dynamical rules%.
\index{ring!integral ---} 

\DeuxRegles 
{
\labu $\,\,x\neq 0\vet y= 0\vd x+y\neq0$
\labu $\,\,x\neq 0\vet y\neq 0 \vd xy\neq0$
\labu $\,\,x\neq 0\vet xy=0 \vd y=0$
\lAb{ED\inq} $\vd x=0\vou x\neq0$ \label{AxEdnq}
}
{
\labu $\vd 1\neq0$
\labu $\,\,xy\neq 0 \vd x\neq0$
\lab{} 
\lAb{col\inq} $\,\,0\neq 0 \vd 1=0$ \label{Axcolnq}
}

Note the significant difference between the theories \sa{Asdz} and \sa{Ai}, which corresponds to an important distinction in \coma but invisible in \clama.

\smallskip \noindent 4) The theory \SA{Al1} of local rings with units has the signature: 
\Sigt{\Alu}{\,\cdot=0,\U(\cdot)\mathrel{;}\cdot+\cdot,\cdot\times \cdot,-\,\cdot,0,1\,}\label{NOTASigAlu}
\noindent This theory is an extension of the theory of commutative rings. A predicate~$\U(x)$ is defined as the invertibility predicate with the two convenient axioms.

\DeuxRegles 
{\Lab{Uv} $\,\, xy=1\Vd \U(x)$
}
{\Lab{UV} $\,\, \U(x) \Vd \Exists y\;xy=1$
}

\noindent We add the proper
axiom \tsbf{AL1} of local rings. 

\Regles 
{\Lab{AL1} $\,\, \U(x+y) \Vd \U(x) \vou \U(y)$}

\noindent The valid rule $\,\,\U(0)\Vd 1=0$ is seen as a collapse rule.

\smallskip \noindent 5) Theory \SA{Alrd} of \textsl{\alrds}. 

\noindent A local ring is sait \textsl{residually discrete} when the residual ring is a discrete field. 
The notion of residual ring (the quotient of the ring by its Jacobson radical) is not obvious to introduce in the framework of dynamical theories. However, for residually discrete local rings, this happens properly. The dynamical theory of residually discrete local rings is obtained from the theory \Sa{Al1} by adding a predicate $\Rn$ (for residually zero elements) as a predicate opposite to the invertibility predicate by means of the axioms\index{residually discrete!local ring}\index{local ring!residually discrete ---}

\DeuxRegles{
\Lab{Alrd} $\,\, \U(x)\vet \Rn(x)\vd 1=0$
}
{
\Lab{ALRD} $\vd \U(x) \vou \Rn(x)$
}

\noindent  
So we have the following signature.
\Sigt{\Alrd}{\cdot=0,\U(\cdot),\Rn(\cdot)\mathrel{;}\cdot+\cdot,\cdot\times \cdot,-\,\cdot,0,1}
\label{NOTASigAlrd}

\noindent
In classical mathematics, any non-trivial local ring is residually discrete. It is significant that the difference between the two notions (present in constructive mathematics) is natural at the level of dynamical theories.
 
\noindent 
The theory \sa{Alrd} can also be described from \Sa{Ac} by adding as axioms the following direct rules (see \cite{Lom1997})

\vspace{-.2em}
\DeuxRegles{
 \Lab{alrd0} $\,\,x=0 \vd \Rn(x)$ 
 \Lab{alrd2} $\,\,\Rn(x) \vd \Rn(xy)$
 \Lab{alrd4} $\vd \U(1)$
}
{
 \Lab{alrd1} $\,\,\Rn(x)\vet \Rn(y) \vd \Rn(x+y)$
 \Lab{alrd3} $\,\,\Rn(x)\vet \U(y) \vd \Iv(x+y)$
 \Lab{alrd5} $\,\,\U(x)\vet \U(y) \vd \U(xy)$
}

\noindent then the \rsim $\,\,\U(xy) \vd \U(x)$ and the \rdys \Tsbf{IV}, \Tsbf{AL1},  \Tsbf{ALRD} 
\eoe

%:Subsubsection{Primitive récursive arithmetic}
\Subsection{Primitive récursive arithmetic}
\label{PrimRec}
This subsection shows the interest of using sorts for maps that can be defined constructively in an algebraic structure (for example here the semi-ring of natural numbers) when the language in which the structure is defined does not allow the introduction of function symbols corresponding to these maps in the corresponding geometric theory.

This example motivates us to introduce sorts for certain continuous semialgebraic maps and their continuity moduli in the theory of \ndrcfs. 

\smallskip Long before the machinery of dynamical theories was set up, R.~L.~Goodstein explained how to treat a large part of the usually practised mathematics by means of purely computational formal systems, \textsl{without logic}.

In the book \cite{Goo1957} the author, following a suggestion by Skolem, shows how a calculus system \gui{without logic, and without quantifiers} makes it possible to develop a very important part of \gui{arithmetic}, understood in the sense of a formal theory of the natural integers.
 
The only problem, and it's a major problem that is likely to put off many a mathematician, is that the usual mathematical statements have to be encoded in the form of primitive recursive maps. This may seem to take us back to the realm of second-order arithmetic and Reverse Mathematics. 

In a second book (\cite{Goo1961}) Goodstein extends his study to recursive analysis. We also highly recommend \cite{Goo1979}.

%%%%%%%%%%%%%%%%%%%%%%%%%%%%%%%%%%%%%%%%%%%%%%%%%%%%%%%%%%%%%%%%%%%%
%:Goodstein's formal system
\Subsubsection {Goodtein's formal system}
If we restrict ourselves to primitive recursive arithmetic,\footnote{Goodtein's book studies broader systems of calculation that include maps defined by multiple recurrences, such as the Ackerman function.} we can describe the formal system proposed by Goodstein as follows.

As in the formal theory \sa{Peano}, the variables and constants represent natural numbers. There is a single constant, $0$, and a single relation symbol, which is the equality $ x=y $. 

For any primitive recursive map $ f\colon \N^{k+1}\to\N $, defined by simple recurrence using the equations
\[ 
\begin{array}{rcl} 
f(\xk,0) & = & g(\xk) \oups.3em] 
f(\xk,y+1) & = & h(\xk,y,f(\xk,y)) \xk,y,f(\xk,y) 
 \end{array}
\] 
where $g$ and $h$ are previously defined primitive recursive maps, we introduce a function symbol corresponding to this definition of $f$. 

Similarly, for any primitive recursive map defined by composition of previously defined primitive recursive maps, we introduce a function symbol corresponding to this definition.

The function symbols which \gui{initialise} the system are $S$ for the successor map, $0_1$ for the null map in one variable \hbox{and $\pi_{n,k}$} for the $k$-th coordinate map $\N^n\to\N$ ($k\in\lrbn$, $n\in\N$).

In this way we obtain a function symbol of arity $r$ for each definition of a primitive recursive map $\N^r\to\N$.

Note that we have a formal name $\und n$, an abbreviation of $S(S(\dots(S(0))\dots))$, for each integer~$n$.

Calculations in primitive recursive arithmetic à la Goodstein consist of establishing \gui{identities} between two functions defined in this way corresponding to two function symbols $f_i$ and $f_j$ of the same arity.
\[ 
\forall \xk\;\; f_i(\xk)=f_j(\xk)
\]
which we can write in the form of a valid rule in the proposed system: 

\Regles{ \lab{Eq$ _{i,j} $} $ \vd f_i(\xk)=f_j(\xk) $}

However, this is not a dynamical theory because the valid rules do not result from a simple axiomatic system of \rdys. 

In fact, we must obviously take as axioms all the equalities mentioned earlier, which are used to define arbitrary primitive recursive maps, but this is clearly not enough. 

Indeed, if the equality $ (\und m+\und n)+\und p=\und m+(\und n+\und p) $ can be established for arbitrary $ m,n,p $ integers by simply using the definition of $ + $ by recurrence, the corresponding rule

\Regles {\lab{~} $ \vd (x+y)+z=x+(y+z) $}

\noindent does not result in a purely finitary way from the axioms of definition of $+$.

The same applies, for example, to an equality $ f_i(\und m,\und n)=f_j(\und m,\und n) $ which could be found for all integers $ m,n $ by simply applying the axioms defining $f_i$ and $f_j$, whereas the corresponding rule \tsbf{Eq$_{i,j}$} cannot in general be established in a purely finitary way if the definitions of $f_i$ and $f_j$ use the simple induction scheme.

The computational system allows us to validate composition of maps,\footnote{For example, $g=f_1\circ (f_2\circ f_3)$ and $h=(f_1\circ f_2)\circ f_3$ which correspond to two different definitions, give rise to the identity $ \vd g(x)=h(x) $ because according to the definitions \hbox{$ \vd g(x)=f_1(f_2(f_3(x)))$} and $\vd h(x)=f_1(f_2(f_3(x)))$.} because it is subject to the axioms of equality. 
But to complete primitive recursive arithmetic, we need the \textsl{external rules} corresponding to definitions by recurrence. Specifically, for example, from the following two valid rules 

\Regles {\labu $ \vd u(x,0)=v(x,0) $ 
\labu $ \,\,u(x,y)=v(x,y)\vd u(x,y+1)=v(x,y+1) $}

\noindent we deduce the validity of the rule

\Regles {\labu $ \vd u(x,y)=v(x,y) $} 
 
\noindent (here $u$ and $v$ are two terms containing $x$ and $y$ as free variables)

The use of these external rules avoids recourse to the corresponding axiom, which can be formulated in a first-order formal theory, but not in a finitary dynamical theory:
\[
\big[\forall x\;u(x,0)=v(x,0)\,\vii\, \forall x,y\;\big((u(x,y)=v(x,y)\Rightarrow u(x,y+1)=u(x,y+1)\big)\big] \Rightarrow \forall x,y\;u(x,y)=v(x,y)
\]

%:subsubsection{PRA tgm finitary} 
\Subsubsection{A finitary geometric theory for primitive recursive arithmetic} \label{secPRA}

In this subsection we show how to treat Goodstein-style primitive recursive arithmetic within the framework of a dynamical theory.

We now explain how the use of sorts for maps allows us to avoid recourse to these external rules and to define a simple formal system (a finitary Horn theory) for primitive recursive arithmetic. 
\begin{enumerate}
 
\item \textsl{Sorts} \\
For each integer $k$ we introduce the sort $\iF_k$ of the primitive recursive maps $ f\colon \iN^k\to\iN $. The sort $ \iF_0 $ is the sort of integers denoted $\iN$.
 
\item \textsl{Predicates.} \\
For each sort $\iF_k$, there is a corresponding equality symbol $ \cdot=_k\cdot $.
 
\item \textsl{Constants.}
\begin{enumerate}
 
\item 
The basic constants are (names for) 
\begin{itemize}
 
\item $0$ of sort $\iN$,
 
\item $ 0_k $ of sort $\iF_k$ (for the constant null map, $ k\geq 1 $),
 
\item the successor map $ S $ of sort $ \iF_1 $, 
 
\item for $ n\geq 1 $, the $ n $ coordinate maps\footnote{Note that 
$\pi_{1,1}$ is the constant which designates the identity in $\iF_1$.} 
$\pi_{n,k}$ of sort $ \iF_n $ ($ 1\leq k\leq n $). 
\end{itemize}

\item For any primitive recursive map $f\colon \N^{k}\to\N$ ($k\geq 1 $), defined by simple recurrence using the equations
\[ 
\begin{array}{rcll} 
f(\ux,0) & = & g(\ux) &g\colon \N^{k-1}\to\N\oups.3em] 
f(\ux,S(y)) & = & h(\ux,y,f(\ux,y))\qquad &h\colon \N^{k+1}\to\N %\[.3em] 
 \end{array}
\] 
where $g$ and $h$ are previously defined primitive recursive maps, we introduce a name for~$f$ as a constant of sort $ \iF_{k} $.\footnote{For $k=2$ for example this could be the name $R2(G,H)$ if $G$ and $H$ are names for $g$ and $h$.}

\item A name is also introduced for any primitive recursive map defined by composition of previously defined primitive recursive maps.
\end{enumerate}
 
\item \textsl{The other function symbols.}
The function symbols given in 4b and 4c can be used, if desired, to avoid creating the constants given in 3b and 3c.
\begin{enumerate}
 
\item \textsl{Evaluations}. For each $ \ell\geq 1 $ there is a function symbol for the evaluation $ \Ev_{\ell} $ of the constant $ f\in \iF_\ell $ in a $k$-uplet of integers. It is a symbol of the type $ \iF_\ell\times \iN^\ell\to \iN $. We abbreviate $ \Ev_{\ell}(f,x_1,\dots,x_\ell) $ to $ f(x_1,\dots,x_\ell) $. 
 
\item \textsl{Simple recurrence}. For $ k\geq 1 $ we have a function symbol $ \rR_{k} $ for the element $ f\in \iF_k $ defined \gui{by simple recurrence} from an element $ g\in \iF_{k-1} $ and an element~\hbox{$ h\in \iF_{k+1} $} (as in 3b, but here, $f$, $g$ and $h$ are variables). It is a function symbol of the type~\hbox{$ \iF_{k-1}\times \iF_{k+1}\to \iF_{k} $}.

\item \textsl{Compositions}. For $ k\geq 1 $ and $ \ell\geq 1 $ we have a function symbol $ \rC_{\ell,k} $ for the \gui{composition} of the element $ f\in \iF_\ell $ with $\ell$ elements $ g_i\in \iF_k $. It is a symbol of the type $ \iF_\ell\times \iF_k^{\,\ell}\to \iF_k $. We abbreviate $ \rC_{\ell,k}(f,g_1,\dots,g_\ell) $ to $ f\circ(g_1,\dots,g_\ell) $ or $ f(g_1,\dots,g_\ell) $. We can see $ \Ev_{\ell} $ as the special case $ \rC_{\ell,0} $.
\end{enumerate}
\item \textsl{The axioms} are as follows.
\begin{enumerate}
 
\item The usual equality axioms for $ =_k $ relations.
 
\item A collapse axiom $ \,\,S(0)=0\vd \Bot $. We note $ 1 $ rather than $ \und 1 $ for $ S(0) $.
 
\item The axioms that establish equalities linking constants and other function symbols. For example: 

\Regles{
\lab{~} $ \Vdi{f:\iF_k} f\circ (\pi_{k,1},\dots,\pi_{k,k})=f $ 
\lab{~} $ \Vdi{f_1,\dots,f_k:\iF_\ell} 0_k(f_1,\dots,f_k)=0_\ell $ 
\lab{~} $ \Vdi{f_1,\dots,f_k:\iF_\ell} \pi_{k,i}(f_1,\dots,f_k)=f_i $ 
\lab{prod} $ \Vdi{x,y:\iN} \Prod=\rR_2(0_1,\mathrm{\Sum}\circ (\pi_{3,3},\pi_{3,1}) $ 
}

The last rule gives the recurrence definition of the product (the constant $ \Prod $ in~$ \iF_2 $) from the addition (the constant $ \mathrm{\Sum} $ in $ \iF_2 $).

\item The axioms for the associativity of compositions, including the cases of evaluations. \\
For example:

\Regles{
\lab{asC$ _{1,1,1} $} $ \Vdi{f,g,h:\iF_1} f\circ (g\circ h)=(f\circ g)\circ h $ 
\lab{asC$ _{1,1,0} $} $ \Vdi{f,g:\iF_1;x:\iN} (f\circ g)(x)=f(g(x)) $ 
\lab{asC$ _{2,1,1} $} $ \Vdi{f:\iF_2;g_1,g_2,h_1,h_2:\iF_1} f\circ (g_1\circ h_1,g_2\circ h_2)=(f\circ (g_1,g_2))\circ (h_1,h_2) $ 
}
 
\item The axioms for definitions by recurrence.
For example, $ \rR_{2} $:

\Regles{
\lab{Rec$_{2,\mathrm{ini}} $} $ \,\,f=\rR_2(g,h)\Vdi{f:\iF_2;g:\iF_1;h:\iF_3} f\circ (\pi_{1,1},0_1)=g $ 
\lab{Rec$_{2,\mathrm{rec}} $} $ \,\,f=\rR_2(g,h)\Vdi{f:\iF_2;g:\iF_1;h:\iF_3} f\circ (\pi_{2,1},S\circ \pi_{2,2})=h\circ (\pi_{2,1},\pi_{2,2},f) $ 
}
 
\item The axioms for proofs by induction (one for each arity).
For example, $ \tsbf{REC}_2 $:

\Regles{
\lab{REC$ _2 $}
 $ \left.
\begin{array}{rcll} 
f_1\circ (\pi_{1,1},0_1)&=&f_2\circ (\pi_{1,1},0_1)\vet \oups.3em] 
f_1\circ (\pi_{2,1},S\circ \pi_{2,2}) & = & h\circ (\pi_{2,1},\pi_{2,2},f_1)\vet \oups.3em] 
f_2\circ (\pi_{2,1},S\circ \pi_{2,2}) & = & h\circ (\pi_{2,1},\pi_{2,2},f_2) 
 \end{array} \right\}
\Vdi{f_1,f_2:\iF_2;h:\iF_3} f_1=f_2
 $ 
}
\end{enumerate}
\end{enumerate}

Note that in (e) the axioms assert that the map $ f=\rR_2(g,h) $ verifies the properties expected of a definition by recurrence, whereas in (f) the axiom asserts the uniqueness of the map verifying these properties. 

Let's now look at how a usual proof by induction translates into the \gui{language of maps} that we have set up. For example, the distributivity of multiplication over addition. In the usual proof, we prove the equality $ x\times (y+z)=(x\times y)+(x\times z) $ by induction on $z$ as follows:
\begin{itemize}
 
\item \textsl{Initialisation.} 
\[
x\times (y+0)\,\eqdf{1}\, x\times y\,\eqdf{2}\, (x\times y)+0 \,\eqdf{3}\,(x\times y)+(x\times 0)
\]
with: $ 1: $ initialisation of $ a+\cdot $, $ 2: $ initialisation of $ a+\cdot $, $ 3: $ initialisation of $ a\times \cdot $ 
 
\item \textsl{Induction.} 
\[ 
\begin{array}{rccclll} 
x\times (y+S(z)) & \eqdf{4} & x\times S(y+z)& \eqdf{5} & (x\times (y+z))+x& \eqdf{6}\oups.3em] 
((x\times y)+(x\times z))+x & \eqdf{7} & (x\times y)+((x\times z)+x) 
 & \eqdf{8} & (x\times y)+(x\times S(z))
 \end{array}
\] 
with $ 4: $ induction of $ a+\cdot $, $ 5,8: $ induction of $ a\times \cdot $, $ 6: $ induction hypothesis, $ 7: $ associativity of $ + $ (demonstrated earlier),

%\item 
\end{itemize}

\smallskip Let's translate all this into the language of maps, for the elements of $ \iF_3 $ 
\[
f=\Prod(\pi_{3,1},\mathrm{\Sum}(\pi_{3,2},\pi_{3,3}))\;\hbox{ and }\; g=\mathrm{\Sum}(\Prod(\pi_{3,1},\pi_{3,2}),\Prod(\pi_{3,1},\pi_{3,3}).
\]

To validate $ f=g $, we use the $ \tsbf{REC}_3 $ principle. To do this, we validate the three hypotheses. First of all the initialisation, which is $ f(\pi_{2,1},\pi_{2,2}, 0_1)=g(\pi_{2,1},\pi_{2,2}, 0_1) $. 
\begin{itemize}
 
\item We have $ f(\pi_{2,1},\pi_{2,2}, 0_1)=\Prod(\pi_{2,1},\mathrm{\Sum}(\pi_{2,2},0_1)) $. Since $ \mathrm{\Sum}(\pi_{2,2},0_1)=\pi_{2,2} $ according to the initialisation in the recurrence definition of $ \mathrm{\Sum} $, we obtain 
\[
f(\pi_{2,1},\pi_{2,2}, 0_1)=\Prod(\pi_{2,1},\pi_{2,2}).
\]
 
\item We have $ g(\pi_{2,1},\pi_{2,2}, 0_1)=\mathrm{\Sum}(\Prod(\pi_{2,1},\pi_{2,2}),\Prod(\pi_{2,1},0_1)) $. Since $ \Prod(\pi_{2,1},0_1)=0_1 $ from the recurrence definition of $ \Prod $, we obtain 
\[
g(\pi_{2,1},\pi_{2,2}, 0_1)=\mathrm{\Sum}(\Prod(\pi_{2,1},\pi_{2,2}),0_1),
\] 
then $ g(\pi_{2,1},\pi_{2,2}, 0_1)=\Prod(\pi_{2,1},\pi_{2,2}) $ from the recurrence definition of $ \Sum $.
\end{itemize}

\smallskip Next we need to validate the passage from $ n $ to $ S(n) $, i.e.\ find a suitable element $ h\in \iF_4 $, i.e.\ satisfying the equalities
\[ 
\begin{array}{rcl} 
f(\pi_{3,1},\pi_{3,2},S(\pi_{3,3})) & = & h(\pi_{3,1},\pi_{3,2},\pi_{3,3},f) \oups.3em] 
g(\pi_{3,1},\pi_{3,2},S(\pi_{3,3})) & = & h(\pi_{3,1},\pi_{3,2},\pi_{3,3},g) \end{array}
\] 

\begin{itemize}
 
\item We have $ f(\pi_{3,1},\pi_{3,2},S(\pi_{3,3})) = \Prod(\pi_{3,1},\Sum(\pi_{3,2},S(\pi_{3,3})) $ which gives according to the induction in the definition by recurrence of $ \Sum $ 
\[
f(\pi_{3,1},\pi_{3,2},S (\pi_{3,3}))=\Prod(\pi_{3,1},S(\Sum(\pi_{3,2},\pi_{3,3}))),
\]
then according to the recurrence definition of $ \Prod $ 
\[
f(\pi_{3,1},\pi_{3,2},S (\pi_{3,3}))=\Sum(\Prod(\pi_{3,1},\Sum(\pi_{3,2},\pi_{3,3})),\pi_{3,1})=\Sum(f,\pi_{3,1}).
\] 
 
\item In the same way (using the associativity of addition) 
\[
g(\pi_{3,1},\pi_{3,2},S (\pi_{3,3}))=\Sum(\Sum(\Prod(\pi_{3,1}, \pi_{3,2}),\Prod(\pi_{3,1}, \pi_{3,3})),\pi_{3,1})=\Sum(g,\pi_{3,1}).
\]
 
\item We have therefore validated the hypotheses with the element $ h\in \iF_4 $ defined by 
\[
h=h(\pi_{4,1},\pi_{4,2},\pi_{4,3},\pi_{4,4}):=\Sum(\pi_{4,4},\pi_{4,1}). 
\]
\end{itemize}

We call \SA{PRA} the dynamical theory of primitive recursive arithmetic that we have just defined. This dynamical theory demonstrates exactly the same statements as the system developed by Goodstein.
 
%%%%%%%%%%%%%%%%%%%% SECTION %%%%%%%%%%%%%%%%%%%%%%%%%%%%
%%%%%%%%%%%%%%%%%%%%%%%%%%%%%%%%%%%%%%%%%%%%%%%%%%%%%%%%%%%%%%%%%%%%
\section{Conservative extensions of a dynamical theory} \label{secconservative}

\Subsection{Essentially equivalent extensions}

%: Definition{defithconserv}
\begin{definition} \label{defithconserv} A dynamical theory~$ \sab{T}' $ is said to be a \textsl{conservative extension of the theory~\sa{T}}
if it is an extension of \sab{T} and if the \rdys formulable in \sab{T} and valid in~$ \sab{T}' $ are valid in~\sa{T}.\footnote{The reciprocal is clear.}\index{extension --- of a dynamical theory!conservative}
\end{definition}
%----------- end definition -------------------------------- 

Two dynamical theories \und{on the same language} are said to be \textsl{identical} when they prove exactly the same \rdys. In other words, the axioms of one are valid \rdys in the other. Each is obviously a conservative extension of the other.

\smallskip The simplest case of conservative extension when the language has grown is that of extensions which are essentially equivalent in the following meaning.

%:     Definition{defi-exteseq}
\begin{definition} \label{defi-exteseq}
An extension $\Tp$ de la \tdy \sa T is said \textsl{\eseq}
if it is obtained, up to renamings, by means of the following procedures, used iteratively, each time giving the appropriate axioms (see the details in \cite[section 2.3]{LM2022}). 

\noindent \textsl{Essentially identical} extensions are those which are obtained without adding new sorts.%.
\index{extension!identical ---}%
\index{extension!essentially identical ---}%
\index{extension!equivalent ---}%
\index{extension!essentially equivalent ---}%
\begin{itemize}
 
\item Add abbreviations.
 
\item Addition of predicates expressing the conjunction or disjunction of already existing predicates. This amounts to accepting $ \vii $ and $\vuu$ as logical symbols for constructing compound formulas, i.e.\ accepting a bit of intuitionist logic in the language.\rdb\label{NOTAvii}\label{NOTAvuu} 
 
\item Add predicates translating a formula $\exists x P$ where $P$ is an existing predicate and $x$ is a variable. Same comment as for the previous point.\rdb\label{NOTAexists} 
 
\item \rdb Add a function symbol when a unique existence is valid under certain hypotheses.\label{skolemunique} 
 
\item Add a sort that is the product of several sorts.
 
\item Add a sort that is the disjoint union of several sorts.
 
\item Add a subsort of an existing sort, defined as the elements satisfying an existing predicate. 
 
\item Add a quotient sort of an already existing sort, defined by a binary predicate, which is provably an equivalence relation, and which defines the new equality in the quotient sort. 
 
\item Add a sort whose objects are (certain) morphisms of one sort into another that shares some algebraic structure with the first. 
\end{itemize}
\end{definition}

An \eseq extension is intuitively equivalent in the following meaning.
 
%: definition{defextintequiv}
\begin{dfni} \label{defextintequiv}
Consider a dynamical theory \sa{T} and an extension $ \Tp $ of \sa{T}. We say that $ \Tp $ is an \textsl{intuitively equivalent} extension of  \sa T if the following three properties are verified.\index{extension!intuitively equivalent} 
\begin{enumerate}
\item  $ \Tp $ is a conservative extension of \sa{T}.
 
\item [\textsl{2}.] Any \rdy formulated in the language of $ \Tp $ is equivalent\footnote{The equivalence in question is an external rule, like the structural rules described above. It may depend on the logic used in the external world.} to a family of \rdys formulated in the language of \sa{T}.
 
\item For any presentation $(G,R)$ in the language of \sa{T}, the dynamic algebraic structures $ \gA=\big((G,R),\sab{T}\big) $ and 
 $ \Tp(\gA):=\big((G,R),\Tp\big) $ have the same models (in constructive mathematics as in classical mathematics).
\end{enumerate}
\end{dfni}

\smallskip In the rest of this section we give two very important cases of conservative extensions (Theorems~\ref{thFond} and \ref{thFondExists}), relating to the use of classical logic, which do not fall into the previous simple case. Before that, we give Theorem~\ref{thFond0} relating to the use of constructive (intuitionistic) logic.

%: Subsection{Comparison with intuitionistic logic}
\Subsection{Comparison with intuitionistic logic}\label{subsubsecthFond0}

Dynamical theories can be considered to be nothing more than truncated versions of intuitionistic natural deduction, in which neither the $ \Rightarrow $ connector nor the~$ \forall $ quantifier is introduced.

This is precisely the strength of dynamical theories: not being encumbered with \gui{complicated} formulas such as $ (A\Rightarrow B)\Rightarrow C $, or $ \forall x\;\exists y\;\forall z \dots $, makes it possible to see things more clearly and to simplify a certain number of non-trivial results, when they can be demonstrated at the basic level of natural deduction, i.e.\ with the \gui{logic-free} system of dynamic proofs. 

%: --- Theorem{thFond0}------- 
\begin{theorem}[conservativity of intuitionistic logic with respect to dynamic proofs] \label{thFond0} 
Consider a finitary dynamical theory \sa{T}. If a consistent formula on the language of~\sa{T} is proved using the axioms of \sa{T} and intuitionistic logic, the corresponding \rdy can be proved valid directly in the dynamical theory \sa{T}.
\end{theorem}
%--- end-theorem-----------

%
\begin{proof}
See \cite[Coquand, 2005]{Coq2005}. 
\end{proof}

Thierry Coquand intuitively interprets the proof as the construction of a certain type of model (a \textsl{generic model}) of the dynamic algebraic structure under consideration. The proof in \cite{Coq2005} is more direct and more intuitive than that of the slightly stronger conservativity result given in Theorem~\ref{thFond}. 

%%%%%%%%%%%%%%%%%%%%%%%%%%%%%%%%%%%%%%%%%
%: Subsection{Fundamental theorem of dynamical theories}
\Subsection{Fundamental theorem of dynamical theories}\label{subsubsecthFond}

To a dynamical theory \sa{T} corresponds a coherent theory, or finitary geometric theory, obtained by replacing the \rdys by the corresponding formulas according to the scheme given on \paref{subsectdy} at the beginning of Section \ref{subsectdy}. This coherent theory can be treated according to classical logic or intuitionistic logic. Let us note $ \sab{T}^{\rm c} $ and $ \sab{T}^{\rm i} $ respectively. 

We have the fundamental theorem \ref{thFond} below (cf. for example Theorem 1 in \cite{CLR01}).
This theorem is already given for purely equational theories in \cite[Prawitz, 1971]{pra1971}, and this kind of result is omnipresent in the contemporary literature, in more or less varied forms. We recommend the recent progress on this theme described in \cite{DN2015,Dyc2015} which shows that, properly treated, classical proofs do not provide significantly longer constructive proofs. The proof in \cite{CLR01} is constructive and relatively intuitive, but leads to an explosion in the size of proofs.

It is based on the following lemma which explains the harmlessness, \und{in certain circumstances}, of the rule \tsbf{LEM} of excluded thirds.

%: Lemma{lemNegValide}
\begin{lemma}[elimination of classical negation] \label{lemNegValide} ~ \\
Let \sa{T} be a finitary dynamical theory, and $P(.,.)$ be a predicate forming part of the signature (we have taken it here of arity 2 by way of example). Let us introduce \gui{the predicate opposed to $P$}, let us note it $Q(.,.)$, with the two \rdys which define it in classical mathematics:\footnote{The definition of the predicate opposed to a predicate $P$ in constructive mathematics is not the same, and it cannot be treated in the framework of dynamical theories, except in the case where the predicate is decidable. The constructive meaning of $P$ is $P\Rightarrow\Bot$ and the constructive implication cannot be treated by the dynamical method alone.}

\DeuxRegles{
\lab{In-non$ _P $} $ \vd P(x,y) \;\vou\; Q(x,y) $ 
}{
\lab{El-non$ _P $} $ \,\, P(x,y),\, Q(x,y) \vd \Bot %1=0
 $ 
}

\smallskip\noindent Then, the new dynamical theory is a conservative extension of \sa{T}. 
\end{lemma}
%--------- end lemma -----------------------------------

Note that this time some constructive models of the first theory may no longer be constructive models of the second. Nevertheless, this is not too serious, as the lemma indicates, and is generalised in the following fundamental theorem.

%: --- Theorem{thFond}------- 
\begin{theorem}[elimination of cuts] 
\label{thFond}  ~\\
As far as finitary dynamical theories are concerned, logic, including classical logic (and in particular the \tsbf{LEM}) only serves to shorten proofs. More precisely
a \rdy is valid in a dynamical theory \sa{T} if, and only if, it is valid in the corresponding classical coherent theory (the one with the same signature and axioms as~\sa{T}): connectors, quantifiers and classical first-order logic are used in the coherent theory.
\end{theorem}
%--- end-theorem----------- %r
%: Remark{remprogHilbert}
\begin{remark} \label{remprogHilbert} 
The preceding theorem, and the following one concerning skolemisation, show that the use of dynamical theories allows Hilbert's programme to be partly realised, by providing a constructive semantics for certain uses of  \tsbf{LEM} and the axiom of choice. \eoe
\end{remark}
%----------- fin remark ---------------------------------- 

%: Subsection{Skolemisation}
\Subsection{Skolemisation}
%-----------

We now look at the general process of skolemisation, which consists of getting rid of the $ \,\Exists\, $ in some valid rules of a dynamical theory by replacing the existing $ \,\Exists\, $ with functions symbols.

We have already indicated the case where this operation is harmless, according to the following informal remark: when the existent in a valid rule is provably unique, it doesn't hurt to replace the dummy variable which designates the existent by a function symbol. 

On the other hand, replacing the dummy variable that designates the existing with a function symbol when the existing is not provably unique is more problematic. This is called skolemisation. Some constructive models before skolemisation may no longer be suitable after skolemisation, and the new theory may no longer have any known constructive model. And even in classical mathematics, if the models are \gui{almost} the same, it is on condition that the axiom of choice is assumed.

%: --- Theorem{thFondExists}------- 
\begin{theorem}[skolemisation] 
\label{thFondExists} 

Consider a dynamical theory \sa{T}. We denote $ \sab{T}' $ the \gui{skolemised} theory, where we have skolemised all the existential axioms by replacing the $ \Exists $ with the introduction of function symbols. Then~$ \sab{T}' $ is a conservative extension of \sa{T}.
\end{theorem}
%--- end-theorem-----------
 
%
\begin{proof}
A proof in classical mathematics using an axiom of choice consists in noting that the two theories have \gui{the same models}. A syntactic and constructive proof is given in \cite[Bezem \& Coquand, 2019]{BC2019}.
\end{proof}

\section[Distributive lattices associated with a dynamic algebraic structure]{Distributive lattices and spectral spaces associated with a dynamic algebraic structure} \label{subsectrdisad}

For this section, we refer to \cite[Chapters XI and XIII]{CACM} \cite[sections 1 and 3]{LM2022}

\Subsection{Distributive lattices and entailment relations}

A particularly important rule for distributive lattices, called \textsl{cut}, is the following
\begin{equation}\label{coupure1}
\hbox{if } \, x\vi a \leq b \, \hbox{ and } \, a \leq x\vu b 
\,\hbox{ then } \, a \leq b.
\end{equation}

If $ A\in\Pfe(\gT) $ (set of finitely enumerated parts of~$\gT$) we will note
\[ 
\ndsp \Vu A:=\Vu _{x\in A}x\qquad \hbox{ and }\qquad \Vi A:=\Vi _{\!x\in A}x.
\]
We denote $ A \vda B $ or $ A \vdash_\gT B $ the relation defined as follows on the set $ \Pfe(\gT) $:

\snic{A \vda B \; \; \equidef\; \; \Vi A\;\leq \;
\Vu B.}

This relationship verifies the following axioms, in which we write $x$ for $ \{x\} $ and $ A, B $ for $ A\cup B $.

\vspace{-.5em}
%--------------------begin array---------------
\[\arraycolsep3pt\begin{array}{rcrclll}
& & x &\vda& x &\; &(R) \oups1mm]
 \hbox{if } A \vda B & \hbox{then} & A,A' &\vda& B,B' &\; &(M) \oups1mm]
\hbox{if } (A,x \vda B) \hbox{ and } (A \vda B,x) 
& \hbox{then} & A &\vda& B &\;
&(T).
\end{array}
\]
%---------------------end array--------------
The relationship is said to be \textsl{reflexive}, \label{remotr} \textsl{monotonic} and \textsl{transitive}. The third rule (transitivity) can be seen as a rewriting of the rule (\ref{coupure1}) and is also called the \textsl{cut} rule.

%: --- Definition{defEntrel}-------------
\begin{definition}
\label{defEntrel}
For an arbitrary set $S$, a binary relation on $ \Pfe(S) $ which is reflexive, monotonic and transitive is called an {\sl entailment relation}. 
\end{definition}
%--- end-definition------------------------------------

The following theorem is fundamental. It states that the three properties of entailment relations are exactly what is needed for the interpretation of an entailment relation as the trace of that of a distributive lattice to be adequate.

%: Theorem{thEntRel1}----------------
\begin{theorem}[fundamental theorem of entailment relations]
\label{thEntRel1} {\rm \cite[Satz 7]{Lor1951}, \cite{CC00}, \cite[\hbox{XI-5.3}]{ACMC}}.
Let $S$ be a set with an entailment relation $ \vdash_S $ on $ \Pfe(S) $. Consider the distributive lattice~$\gT$ defined by generators and relations as follows: the generators are the elements of $S$ and the relations are the

\snic {A \, \vdash_\gT\, B}

\noindent each time $ A\, \vdash_S \, B $. Then, for all $A$, $ B $ in $ \Pfe(S) $, we have

\snic {\hbox{if }A\, \vdash_\gT \, B
\hbox{ then } A\, \vdash_S \, B.}

\noindent 
In particular, two elements $x$ and $y$ of $S$ define the same element of $\gT$ if, and only if, we have $ x \vdash_S y $ and $ y \vdash_S x $.
\end{theorem}
%--- end-theorem------------------------------------

%:Subsection{The spectrum of a distribution lattice}
\Subsection{The spectrum of a distributive lattice} 

In classical mathematics a {\sl prime ideal} $\fp$ of a distributive lattice $ \gT\neq \Un $ is an ideal whose complementary $ \fv $ is a filter (which is then a {\sl prime filter}). We then have $ \gT/(\fp=0,\fv=1)\simeq\Deux $. It is the same to give a prime ideal of $\gT$ or a morphism of distributive lattices $ \gT \rightarrow \Deux $.

It is easy to check that if $S$ is a generating part of the distributive lattice $\gT$, a prime ideal~$\fp$ of $\gT$ is completely characterised by its trace on $S$ (cf. \cite{CC00}).

%d
%: Definition{defiSpecTrdi}
\begin{definition} \label{defiSpecTrdi}
 The \textsl{spectrum of a distributive lattice $\gT$} is the set $ \Spec \,\gT $ of its prime ideals, with the following topology: a basis of opens is given by the 
\[
\DT(a)\eqdefi\sotq{\fp\in\Spec,\gT}{a\notin\fp},\quad a\in \gT.
\]
\end{definition}
%----------- end definition -------------------------------- 
Classical mathematics verifies that
%--- equation eqDa --------
\begin{equation} \label{eqDa}
\left.
\begin{array}{rclcrcl}
 \DT(a\vi b) & = & \DT(a)\cap \DT(b),&\quad & \DT(0) & = & 
\emptyset ,\\
 \DT(a\vu b) & = & \DT(a)\cup \DT(b),&& \DT(1) & = & 
\Spec\,\gT. \end{array}
\right\}
\end{equation}
%---------------------end equation--------------

The complementary of $ \DT(a) $ is a closed set which we denote $ \VT(a) $.

The notation $ \VT(a) $ is extended as follows: if $ I\subseteq\gT $, then $ \VT(I)\eqdefi\Vi_{x\in I}\VT(x) $. If $ I $ generates the ideal $ \fII $, then $ \VT(I)=\VT(\fII) $. It is sometimes said that $ \VT(I) $ is \textsl{the variety associated~with~$ I $}.

%d
%: Definition{ofSpectralSpace}
\begin{definition} \label{defiSpectralSpace}
 A topological space homeomorphic to a space $ \Spec(\gT) $ is called a \textsl{spectral space}. 
 \end{definition}
%----------- end definition -------------------------------- 

The spectral spaces come from the study \cite[Stone, 1937]{Sto37}.

\cite{Joh1986} calls these spaces \textsl{coherent spaces}. \cite{BW74} calls them \textsl{Stone spaces}. This terminology is obsolete, because since \cite{Joh1986} Stone spaces are the spectral spaces associated with distributive lattices which are Boolean algebras.

The name \textsl{spectral space} was given by \cite[Hochster, 1969]{Hoc1969}, who popularised them in the mathematical community after a prolonged hibernation since 1937.

With classical logic and the axiom of choice, the space $ \Spec (\gT) $ has \gui{enough points}: we can find the lattice $\gT$ from its spectrum. 

\smallskip 
A point $\fp$ in a spectral space $ X $ is said to be the \textsl{generic point of the closed set $ F $} \hbox{if $ F=\ov{\so{\fp}} $}. This point (when it exists) is necessarily unique because spectral spaces are Kolmogoroff spaces. In fact, the $ \ov{\so{\fp}} $ closed sets are exactly all the irreducible closed sets of $ X $. The order relation $ \fq\in\ov{\so{\fp}} $ will be denoted $ \fp\leq_X \fq $, and we have the equivalences
%: equation label {eqOrdreSpec}
\begin{equation} \label {eqOrdreSpec}
\fp\leq_X \fq\;\Longleftrightarrow\; \ov{\so{\fq}}\subseteq \ov{\so{\fp}}\,.
\end{equation}
The closed points of $ \Spec(\gT) $ are the maximal ideals of $\gT$. When $ X=Spec(\gT) $ the relation $ \fp\leq_X \fq $ is simply the usual inclusion relation $ \fp\subseteq \fq $ between prime ideals of the distributive lattice $\gT$.

In the category of spectral spaces we define \textsl{spectral morphisms} as maps such that the reciprocal image of any quasi-compact open is a quasi-compact open (in particular they are continuous).

\Subsection{Stone's antiequivalence}

Stone's antiequivalence asserts (in modern language) that in classical mathematics the category of distributive lattices is antiequivalent to the category of spectral spaces.

Although spectral spaces have invaded contemporary abstract algebra, it is only in constructive mathematics that this anti-equivalence of classical mathematics is given the attention it deserves. 

The aim is to correctly define the distributive lattices corresponding to the spectral spaces in the classical literature, and, if possible, to decipher the classical discourses using spectral spaces into constructive discourses concerning the corresponding distributional lattices. 

\Subsection{The Zariski lattice and spectrum of a commutative ring}

The Zariski lattice of a commutative ring can be obtained from rules valid in different extensions of the theory \sa{Ac} of commutative rings.

We choose the theory of local rings because of their fundamental role in Grothendieck schemes.

We consider precisely the dynamical theory of \textsl{local rings with units} \Sa{Al1}

Let $\gA$ be a commutative ring. Consider the entailment relation $ \vdash_{\gA,\mathrm{Zar} }$ on the set underlying $\gA$ defined by the following equivalence 
 % equation label {eqZarclass}
\begin{equation} \label {eqZarclass}
\begin{aligned} 
 a_1,\dots,a_n &\,\vdash_{\gA,\mathrm{Zar}} c_1,\dots,c_m 
 \qquad\quad \equidef \oups.3em] 
\U(a_1)\vet \dots\vet \U(a_n) & \Vdi{\sA{Al1}(\gA)} \U(c_1)\vou \dots\vou \U(c_m) 
 \end{aligned}
\end{equation}

We define the \textsl{Zariski lattice of $\gA$}, denoted $ \ZarA $ or $ \Zar(\gA) $, as the \trdi generated by the entailment relation $ \vdash_{\gA,\mathrm{Zar}} $.
 
The corresponding map $ \DA:\gA\to\ZarA $ is called the \textsl{Zariski support of $\gA$}. When $\gA$ is fixed by the context we simply note $ \rD $.

The usual \textsl{Zariski spectrum} is the dual spectral space of this distributive lattice.

Note that since $ \rD(a_1)\vi\dots\vi\rD(a_n)=\rD(a_1\cdots a_n) $, the elements of $ \ZarA $ are all of the form 
 $ \rD(c_1,\dots,c_m):= \rD(c_1)\vu\dots\vu\rD(c_m) $.

We have the following equivalences. Equivalence between (1) and (2) essentially copies the definition of~$ \ZarA $. The equivalence with~(3) is the subject of a \textsl{formal Nullstellensatz}. The Hilbert Nullstellensatz itself is a more difficult result. 
%t
%: Theorem{thNstFormel}
\begin{theorem}[Nullstellensatz formel] \label{thNstFormel} ~\\
Let $\gA$ be a commutative ring, et $ \an,c_1,\dots c_m\in\gA $. \Propeq
\vspace{-.5em}
\[ 
\begin{aligned} 
(1)\qquad\quad\; \rD(a_1),\dots,\rD(a_n) &\,\vdash_{\ZarA} \rD(c_1),\dots,\rD(c_m) \oups.3em] 
(2)\qquad\quad \U(a_1)\vet \dots\vet \U(a_n) &\Vdi{\sA{Al1}(\gA)} \U(c_1)\vou \dots\vou \U(c_m) \\ 
(3)\qquad\; \exists k>0\;\;(a_1 \cdots a_n)^k&\,\in\gen{c_1,\dots,c_m} 
 \end{aligned} \label {eqNstFormel}
\] 
\end{theorem}
%----------- fin theorem ----------------------------- 

We can therefore identify the element $ \rD(c_1,\dots,c_m) $ of $ \ZarA $ with the ideal $ \sqrt[\gA]{\gen{c_1,\dots,c_m}} $. Modulo this identification, the order relation is the inclusion relation.
%c
%: Corollary{corZarA}
\begin{corollary} \label{corZarA}
The lattice $ \ZarA $ is generated by the smallest entailment relation on (the set underlying~to)~$\gA$ satisfying the following relations.

\DeuxRegles
{
\labu $ \;\;0\vdash 0_\gT $ 
\labu $ \;\;ab\vdash a $ 
\labu $ \;\;a+b\vdash a,b $ 
}
{
\labu $ \;\; 1\vdash 1_\gT $ 
\labu $ \;\;a,b\vdash ab $ 
} 

\noindent In other words, the map $ \rD:\gA\to\ZarA $ satisfies the relations 
\[
\rD(0)=0,\;\rD(1)=1,\;\rD(ab)=\rD(a)\vi\rD(b),\;\rD(a+b)\leq \rD(a)\vu \rD(b),
\]
and any other map $ \rD':\gA\to \gT $ which satisfies these relations factorises via $ \ZarA $ with a unique morphism of distributive lattices $ \ZarA\to \gT $.
\end{corollary}
%--------- fin corollary ------------------------------- 

\Subsection{Other examples}

Remember that a disjunctive rule is a \rdy without the $ \Exists $ symbol, and that a simple disjunctive rule is a \rdy of the following form, with $ m,n\geq 0 $.%.
\index{rule!disjunctive ---}\index{rule!disjunctive!simple ---}
 
%---- equation {eqAxds} ----
\begin{equation} \label{eqAxds}
C_1\vet \ldots \vet C_n \vd D_1\vou \ldots \vou D_m\end{equation}
%------end equation----
where $ C_i $'s and $ D_j $'s are atomic formulas. (Simple) Horn rules are special cases of (simple) disjunctive rules. 

\paragraph{First example.} Consider a dynamic algebraic structure $ \gA=\big((G,R),\sa{T}\big) $ for a dynamical theory $\sa{T}=(\cL,\cA)$. If $ P(x,y) $ is a binary predicate in the signature, and if $ Tcl=\Tcl(\gA) $ is the set of closed terms of $\gA$, we obtain an entailment relation $ \vdash_{\gA,P} $ on $ Tcl \times Tcl $ by defining
 
\vspace{-.8em}
% equation label {eq1}
\begin{equation} \label {eq1}
\begin{aligned} 
 (a_1,b_1),\dots,(a_n,b_n) &\,\vdash_{\gA,P} (c_1,d_1),\dots,(c_m,d_m) 
 \qquad\quad \equidef \oups.3em] 
 P(a_1,b_1)\vet \dots\vet P(a_n, b_n) & \Vdi{\gA} P(c_1, d_1)\vou \dots\vou P(c_m, d_m) 
 \end{aligned}
\end{equation}
% end-equation

Intuitively, the distributive lattice generated by this entailment relation is the lattice of the \gui{truth values} of the predicate~$P$ in the dynamic algebraic structure $\gA$.
 
\paragraph{More generally} Consider a dynamic algebraic structure $ \gA=\big((G,R),\sa{T}\big) $ for a dynamical theory $\sa{T}=(\cL,\cA)$. Let $S$ be a set of closed atomic formulas of $\gA$. We define \textsl{the entailment relation on $S$ associated with $\gA$} as follows: 
 
\vspace{-.5em}
% equation label {eq2}
\begin{equation} \label {eq2}
\begin{aligned}
 A_1,\dots,A_n &\,\vdash_{\gA,S} B_1,\dots,B_m 
 \qquad\quad \equidef \oups.2em] 
 A_1\vet \dots\vet A_n &\Vdi{\gA} B_1\vou\dots\vou B_m 
 \end{aligned}
\end{equation}
% end-equation
We can denote $ \Zar(\gA,S) $ the distributive lattice generated by this entailment relation.

In particular, the lattice $ \Zar(\gA,\Atcl(\gA)) $ is called the \textsl{absolute Zariski lattice of the dynamic algebraic structure $\gA$}.\index{absolute Zariski lattice of a dynamic algebraic structure}.
 
\paragraph{Case of an extension \sa{T$ _1 $} which reflects valid disjunctive rules.} Let \sa{T$ _1 $} be an extension of a dynamical theory \sa{T} which proves exactly the same disjunctive rules (for example a conservative extension). Let $ \gA=\big((G,R),\sa{T}\big) $ and $ \gA_1=\big((G,R),\sa{T$ _1 $}\big) $. Let $S$ be a set of closed atomic formulas of $\gA$. Then the Zariski lattices $ \Zar(\gA,S) $ and $ \Zar(\gA_1,S) $ are isomorphic. 

In particular, when \sa{T$ _1 $} is an essentially equivalent extension of \sa{T} the absolute Zariski lattices of $\gA$ and $ \gA_1 $ are isomorphic.

\paragraph{Zariski lattices, however, give a lesser image} of a dynamic algebraic structure. On the one hand, in these Zariski lattices nothing is taken into account a priori that corresponds to valid \rdys when they are not disjunctive. On the other hand, adding classical logic and skolemising a dynamical theory do not change the lattices corresponding to $ \Atcl(\gA) $, but in this case, the absolute Zariski lattice of $ \gA_1 $ is the Boolean algebra generated by $ \Zar(\gA,\Atcl(\gA)) $. To rediscover the richness of dynamical theories seen from a constructive point of view, it is necessary to call upon the theory of bundles or topos. 

%%%%%%%%%%%%%%%%%%%%%%%%%%%%%%%%%%%%%%%%%%%%%%%%%%%%%%%%%%%%%%%%%%%%
\section{Model theory}\label{subsecModeles}

\Subsection{Theories that share certain rules}

We consider two dynamical theories $\sa{T}_1$ and $\sa{T}_2$ whose signatures extend the same signature  $\Sigma$. 

%d
%:     Definition{defiColSim}
\begin{definition} \label{defiColSim} We assume that each of the two theories has a collapse axiom.
We say that the dynamical theories $\sa{T}_1$ and $\sa{T}_2$ \textsl{collapse simultaneously} (for $\Sigma$) when, for any presentation $(G,R)$ on $\Sigma$, the dynamical algebraic structures $\gA_1=\big((G,R),\sa T_1\big)$ and $\gA_2=\big((G,R),\sa T_2\big)$ collapse simultaneously.
\end{definition}
%----------- fin definition -------------------------------- 
  
%d
%:     Definition{defiMemesRalgs}
\begin{definition} \label{defiMemesRalgs}
The dynamical theories $\sa{T}_1$ and $\sa{T}_2$ are said to prove the same algebraic rules (for $\Sigma$) when, for any presentation $(G, R)$ on $\Sigma$, the dynamical algebraic structures $\gA_1=\big((G,R),\sa T_1\big)$ and $\gA_2=\big((G,R),\sa T_2\big)$ prove the same \ralgs.
\end{definition}
%----------- fin definition -------------------------------- 

It is the same to say that the theories prove the same facts (Definition \ref{defiFact}).

%: Subsection{Completeness theorem}
\Subsection{Completeness theorem, simultaneous collapse}
%-----------

First of all, here is the completeness theorem in its minimal form: its intuitive interpretation in classical mathematics is that classical logic gives exhaustive rules for reasoning in accordance with \gui{absolute truth}, based on an ideal mathematical universe in which no doubt is ever allowed, \tsbf{LEM} is absolutely true and the axiom of choice just the same.
 
%: Theorem{thGodel1}
\begin{theoremc}[Gödel's completeness theorem, first form] \label{thGodel1} ~\\
A dynamic algebraic structure that does not collapse admits a non-trivial model. 
\end{theoremc}
%--------- end theorem ----------------------------------- 

\comm An equivalent form of the completeness theorem is
the following special case (Krull's lemma): \textsl{any non-trivial commutative ring has a non-trivial integral quotient}.

\noindent
 The constructively acceptable form of Krull's lemma is the following easy result: 
 \textsl{when we add the \rdy \gui{$\,\,xy=0 \vd x=0 \vou y=0\,\,$} to the theory of commutative rings, a dynamical algebraic structure collapses in the former theory \ssi it collapses in the latter}.

\noindent
In other words, the theories \Sa{Ac} and \Sa{Asdz} collapse simultaneously. 
\eoe

%: Theorem{thGodel2}
\begin{theoremc} [Gödel's completeness theorem, second form] \label{thGodel2} ~\\
Consider a dynamical theory \sa{T} and a dynamic algebraic structure $\gA$ of type~\sa{T}. A fact is valid in $\gA$ if, and only if, it is satisfied in all models of $\gA$. 
\end{theoremc}
%--------- fin theorem ----------------------------------- 
%
A dynamical theory that extends another (by adding sorts and/or predicates and/or axioms) proves \textsl{a priori} more results. An interesting case is when it proves the same results while offering greater facilities for proofs. This was the essence of the fundamental theorems \ref{thFond} and \ref{thFondExists}. A variant in model theory, but only in classical mathematics, is given by the following theorems. 
 
%
%: Theorem{thcolsimcomp}
\begin{theoremc}[simultaneous collapses and non-trivial models] \label{thcolsimcomp} 
Let \sa{T} be a dynamical theory and $ \sab{T}' $ an extension which collapses simultaneously with \sa{T}. If a dynamic algebraic structure of type~\sa{T} admits a non-trivial model, it also admits a non-trivial model
as a dynamic algebraic structure of type~$ \sab{T}' $. More precisely, if~$ M $ is a non-trivial model of \sa{T}, the dynamic algebraic structure $ (\Diag(M),\sab{T}') $ admits a non-trivial model.
\end{theoremc}
%--------- fin theorem ----------------------------------- 

\comm
A good constructive version of the completeness theorem is 
a pure tautology: if a dynamic algebraic structure does not collapse, then \dots\ it does not collapse. Or again: if a first-order formal theory doesn't collapse, then \dots\  it doesn't collapse. And for  \thref{thcolsimcomp}: if \sa{T} and $\Tp$ collapse simultaneously, then \dots\ they collapse simultaneously. The same would apply to Theorems \ref{thGodel2} and~\ref{thcolsimralg}. 

\noindent
 Indeed, what we call a constructive version of a \gui{doubtful} classical theorem  is a statement, correct in constructive mathematics, which, in practice, 
i.e.\ to demonstrate concrete results, provides the same services as the classical theorem. Indeed, in practice, all these \gui{abstract} theorems  are only used, to arrive at concrete results, only in reasoning by the absurd, which uses fictitious models to conclude that they cannot exist. The concrete result, on the other hand, is much closer to the hypothesis of the abstract theorem that has been invoked. A detailed analysis of the whole proof then generally shows that one has tautologised in circles without realising it (see for example \cite{Lom98} for Hilbert's 17th problem). This is one of the reasons why classical mathematics is so often constructive, contrary to the appearance given by its demonstrations.
\eoe

%%%%%%%%%%%%%%%%%%%%%%%%%%%%%%%%%%%%%%%%%%%%%%%%%%%%%%%%%%%%%%%%%%%%
%: Subsection[Representation theorem, theories proving the same Horn rules] 
\Subsection {Representation theorem, theories proving the same Horn rules}
%-----------

%: Theorem{thcolsimralg}
\begin{theoremc}[representation theorem] \label{thcolsimralg} 
Consider a dynamical theory $\sab{T}'$ which extends a Horn theory~\sa{T} and proves the same Horn rules. 
Any algebraic structure of type \sA{T} is a subdirect product of algebraic structures of type $\sAb{T}'$
\end{theoremc}
%--------- end theorem ----------------------------------- 

For example \sa{T} is the theory of $\ell$-groups and $ \sab{T}' $ is the theory of linearly ordered abelian groups. The important and constructive result is that these two theories prove the same Horn rules. The intuitive interpretation in classical mathematics is that every lattice group is a lattice subgroup of a product of linearly ordered groups. When Paul Lorenzen proved this result, he generalised Krull's analogous result that the integral closure of an integral ring $\gA$ is the intersection of the valuation rings of its fraction field that contain $\gA$. 

\medskip 
The following intuitive theorems, which will be useful to us, concern extensions which prove the same Horn rules; they are proved in \cite{Lom-tgac}. 

%: Theorem{thEseqMemesfaits}
\begin{theorem} \label{thEseqMemesfaits}
Let $ \sab{T}_{\!2} $ be a dynamical theory which is a simple extension of a $ \sab{T}_{\!1} $ theory and which proves the same Horn rules. Consider an essentially equivalent extension $ \sab{T}_{\!1}' $ of $ \sab{T}_{\!1} $ obtained without the addition of existential predicates. We assume that there is no syntactic interference at the level of language extensions between $ \sab{T}_{\!1}' $ and $ \sab{T}_2 $. We can therefore construct an essentially equivalent extension $ \sab{T}_{\!2}' $ of $ \sab{T}_{\!2} $ by copying for $ \sab{T}_{\!2} $ what was done for $ \sab{T}_{\!1} $. In these conditions, $ \sab{T}_{\!2}' $ proves the same Horn rules as 
 $ \sab{T}_{\!1}' $.
\end{theorem}
%----------- end theorem ----------------------------- 

%: Theorem{thEseqMemesfaitsbis}
\begin{theorem} \label{thEseqMemesfaitsbis}
Let $ \sab{T}_{\!2} $ be a dynamical theory which is a simple extension of a $ \sab{T}_{\!1} $ theory and which proves the same Horn rules. Consider an extension $ \sab{T}_{\!1}' $ of $ \sab{T}_{\!1} $ obtained by adding a family of axioms which are all Horn rules. Consider the extension $ \sab{T}_{\!2}' $ of $ \sab{T}_{\!2} $ obtained by adding the same axioms. Under these conditions, $ \sab{T}_{\!2}' $ proves the same Horn rules as $ \sab{T}_{\!1}' $.
\end{theorem}

\newpage \thispagestyle{empty}

\chapter{Infinitary geometric theories}\label{subsecgeominfini}\label{chap-gmqinfini}
\index{theory!infinitary geometric ---}\index{theory!geometric ---}
%-----------
\Today

\minitoc

\section{General}
A very useful general notion of geometric theory is defined, which is not necessarily expressed in finitary terms. This is known as \textsl{infinitary geometric theory}. In an infinitary geometric theory, we allow \rdys that have an infinite disjunction in the second member. There is one essential restriction: the free variables present in such a disjunction must be specified in advance and in finite number.

Intuitively, we use such rules in the proof system of dynamical theories by \gui{opening the branches of calculation corresponding to the infinite disjunction}. What does this mean precisely? It means that a conclusion will be declared valid if it is valid in each of the branches. 

Let's take a simple example, and show what happens if we have in the axioms an infinitary rule of the type

\Regles{\lab {~} $ \Vdi{x_1,\dots,x_k} \Vou_{i\in I}\; \Gamma_i $ 
}

\noindent with an infinite set $ I $ and the $ \Gamma_i $ are lists of atomic formulas with no free variables other than those mentioned (i.e.\ $ \xk $). If for each $ i\in I $ we have a valid rule $ \,\,\Gamma_i\vd B(\ux) $, then we declare the rule $ \Vdi{x_1,\dots,x_k}B(\ux) $ to be valid. 

There is therefore necessarily an intuitive proof \und{external} to the dynamical theory to certify that the desired conclusion is valid in each of the branches. Indeed, the system of calculation \gui{without logic} at work in the dynamical theory cannot handle such an infinity of deductions. A purely mechanical calculus cannot open up an infinite number of branches! For example, with $ I=\N $ the external intuitive proof could be a proof by induction.

Note, on the other hand, that the internal proof must show the validity of the desired conclusion according to the deduction rules \gui{without logic} of the dynamical theory.

\smallskip The above presentation is only a sketch. All this deserves a more formal definition of what is the legal functioning of an infinitary geometric theory; even if there is an unavoidable informal aspect in the recourse to \gui{external} proofs in intuitive mathematics. 

We should also say a few words about the operation of the formal intuitionist and classical theories that extend the infinite dynamical theory (by adding the $ \Rightarrow $ connector and the universal quantifier in the intuitionist case, and by adding \tsbf{LEM} in the classical case).

As in the case of finitary geometric theories, we will reserve the name of dynamical theory for proofs whose internal part is \gui{without logic}. 

%: Subsec{Example: nilpotent elements, Krull dimension,
\Subsection{Example: nilpotent elements, Krull dimension} 

 An element $x$ of a ring is nilpotent if there exists an $ n\in\N^+ $ such that $ x^n=0 $. If we introduce a predicate $ \mathrm{Z}(x) $ for \gui{$x$ is nilpotent}, it will be subject to the natural axioms 

\DeuxRegles{
\lab{nil1} $ \vd \mathrm{Z}(0) $ 
\lab{nil2} $ \,\, \mathrm{Z}(x),\,\mathrm{Z}(y)\vd \mathrm{Z}(x+y) $ 
\lab{NIL1} $ \,\,Z(x)\vd \Exists z\;z(1+x)=1 $ 
}{
\lab{nil3} $ \,\,\mathrm{Z}(x)\vd \mathrm{Z}(xy) $ 
\lab{Nil} $ \,\,\mathrm{Z}(x^2)\vd \mathrm{Z}(x) $ 
}

\smallskip In the corresponding dynamical theory, the only $t$ terms for which we will be able to demonstrate~$ \mathrm{Z}(t) $ will be those for which we will be able to demonstrate $ t^n=0 $ for a~$ n>0 $. However, there is no guarantee that in a model of the theory, the predicate~$ \mathrm{Z}(x) $ corresponds to \gui{$x$ is nilpotent}.

The only way to be sure is to introduce the infinitary \rdy

\UneRegle{NIL}{$ \,\,\mathrm{Z}(x)\vd\Vou_{n\in\N^+} x^n=0 $}\label{AxNIL}

\smallskip This concern is directly related to the Krull dimension of commutative rings. The Krull dimension of a distributive lattice can be formulated in a dynamical theory as follows.

%: Definition{defiDDKTRDI}
\begin{definition}\label{defiDDKTRDI}~
\begin{enumerate}
 
\item Two sequences $ (\xzn) $ and $ (\bzn) $ in a distributive lattice $\gT$ are said to be \textsl{complementary} if
%--- equation eqC2G --------
\begin{equation}\label{eqC2G}
\left.\arraycolsep3pt
\begin{array}{rcl}
 b_0\vi x_0& = & 0 \\
 b_1\vi x_1& \leq & b_0\vu x_0 \\
 \vdots~~~~& \vdots &~~~ \vdots \\
 b_n\vi x_n & \leq & b_{n -1}\vu x_{n -1} \\
 1& = & b_n\vu x_n
\end{array}
\right\}
\end{equation}
%---------------------end equation--------------
A sequence which has a complementary sequence is said to be \textsl{singular}.
 
\item For $ n\geq0 $ we will say that the distributive lattice $\gT$ is \textsl{of Krull dimension $ \leq n $} if any sequence $ (\xzn) $ in $\gT$ is singular. Furthermore, the distributive lattice $\gT$ is of Krull dimension $ -1 $ if it is trivial, i.e.\ if $ 1_\gT=0_\gT $.
\end{enumerate}
\end{definition}

For example, for $ n=2 $ the equalities and inequalities \pref{eqC2G} correspond to the following drawing in $\gT$.
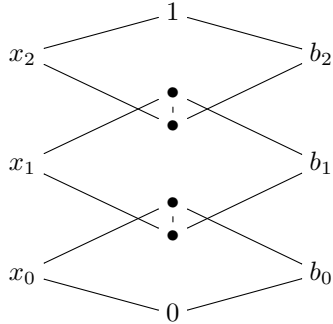
\[\SCO{x_0}{x_1}{x_2}{b_0}{b_1}{b_2}\]

And the dimension $ \leq 2 $ corresponds to the following existential axiom.

\UneRegle{KDIM2}{$ \,\,\vd \Exists b_0,b_1,b_2 \;(x_2\vu b_2=1,\,x_2\vi b_2\leq x_1\vu b_1,\,x_1\vi b_1\leq x_0\vu b_0, \,x_0\vi b_0=0 ) $}

\smallskip For the Krull dimension of rings, we need to involve the distributive lattice formed by the radicals of finitely generated ideals and we express for example the dimension $ \leq 2 $ as follows, noting $ \DA(x,y)=\sqrt{\gen{x,y}}=\sotq{z\in\gA}{\exists n\in\N^+, z^n\in\gen{x,y}} $:

For all $ x_0 $, $ x_1 $, $ x_2 \in \gA $ there exist $ b_0 $, $ b_1 $, $ b_2\in \gA $ such that
%--- equation eqCG --------
\begin{equation}\label{eqCG}
\left.\arraycolsep2pt
\begin{array}{rcl}
\DA(b_0x_0)& = &\DA(0) \\
\DA(b_1x_1)& \subseteq & \DA(b_0,x_0) \\
\DA(b_2 x_2 )& \subseteq & \DA(b_{1},x_{1}) \\
\DA(1)& = & \DA(b_2,x_2 )
\end{array}
\right\}
\end{equation}
%---------------------end equation--------------

Note that $ \DA(b_2 x_2 ) \subseteq \DA(b_{1},x_{1}) $ means that there exist $ a_1,y_1\in\gA $ and $ n\in\N^+ $ such that
\[
(b_2 x_2)^n=a_1b_1+y_1x_1.
\] 

We can therefore express \gui{$ \Kdim\gA\leq 2 $} in the context of an infinitary geometric theory. And for example the theorems of Serre or Foster-Swan (\cite[Chapter XIV]{CACM}) with the Krull dimension as an assumption can be treated entirely within the framework of geometric theories~(\cite{CL05}).

%: Subsection{The Barr Theorem}---- 
\section{Barr's Theorem}

The fundamental theorem of dynamical theories \ref{thFond} is a solid basis for the constructive decoding of classical proofs. In classical mathematics, we show that a coherent theory proves a \rdy by looking at what happens in the models of the theory, which we study with powerful but dubious tools such as the excluded third, the axiom of choice and sometimes even the full power of \sa{ZFC}. However, \thref{thFond} assures us that if the rule in question is provable in formal theory with classical logic, it is also provable by the elementary methods \gui{without logic} that constitute dynamic proofs.

The essential part of decoding is therefore to check that the classical proof can be formalised in classical first-order logic. This is not always easy, because after all, \sa{ZFC} theory can be used to prove results that are much more  \gui{strange}  than Gödel's completeness theorem, and why not results that are downright false if \sa{ZFC} is inconsistent. But in practice, in classical mathematics, even the overuse of ultrafilters or the continuum hypothesis always seems to hide simpler arguments.

\smallskip {\bf Barr's theorem}, established in classical mathematics (and apparently impossible to prove in constructive mathematics), states that for geometric theories, any result proved with classical logic can also be proved with constructive logic. This is a generalisation of \thref{thFond} which is confirmed in practice, even if it is not completely certain from the constructive point of view. A recent study of the problem is made by Rathjen in the article \cite{Rat2016} published in the book~\cite{CPM2016}.

Barr's theorem gives us good reason to believe that the type of decryption provided by \thref{thFond} also applies for infinitary geometric theories, with the same caveats as we indicated for finitary theories.
The reader can find examples of this type in \cite[Sections XV-6 and XV-7]{CACM}.\label{inThGeomR}

\smallskip At the end of this chapter we illustrate how Barr's theorem should not be understood.
%%%%%%%%%%%%%%%%%%%%%%%%%%%%%%%%%%%%%%%%%%%%%%%%%%%%%%%%%%%%%%%%%%%%

\Subsection{An infinitary geometric theory for primitive recursive arithmetic}\label{PrimRecOmega}

We now consider the infinitary geometric theory  \sa{PRA$\omega $} obtained from the finitary geometric theory  \Sa{PRA} by adding the following axiom which forces the sort $ \iN $ to contain only usual integers.

\Regles{
\Lab{Nat} $ \Vdi{x:\iN} \Vou_{n\in\iN}\; x=\und n $ 
}

In the theory thus obtained, we see that to prove $\vd f(x)=g(x)$ with $f,g$ of sort~$F_1$, it is sufficient to know how to show that for each $n\in\NN$ the rule
$\vd f(\und n)=g(\und n)$ is valid. Indeed, we then deduce $\,\,x=\und n\vd f(x)=g(x)$ for each concrete integer $n$, and by using the rule~$\tsbf{Nat}$, we obtain $\vd f(x)=g(x)$.

Now for maps $f$ and $g$ defined in \sa{PRA} (i.e.\ two arbitrary primitive recursive maps), $f(\und n)$ and $g(\und n)$ are two explicit usual integers. 
Thus two primitive recursive maps are \gui{provably everywhere equal} in the theory if they take the same values in any integer, i.e.\ if they are concretely equal, i.e.\ if we have a proof for the fact that they are equal. 

But this proof is not always formalisable within the theory (it may, for example, use a double induction). In any case, it is supposed to be produced in intuitive mathematics in the intuitive mathematical world of the natural integers.

An example is provided by the primitive recursive map $C_\PRA:\N\to\NN$ which is everywhere zero if the theory \sa{PRA} is consistent (which we are intimately convinced). According to Gödel's incompleteness theorem, the theory \sa{PRA} cannot prove $\vd C_\PRA(x)=0$. However, it does prove $\vd C_\PRA (\underline n)=0$ for all $n$.

In the same way it seems probable that the \tgm $\sa{PRA}\omega$, although it is capable of proving $\vd C_\PRA(x)=0$, cannot however prove $\vd C_\PRA=0_1$.

\smallskip The infinitary axiom \tsbf{Nat} that we add therefore authorises up to a certain point, but only up to a certain point, the use of the $\omega$-rule for the \egt of primitive recursive maps.

\Subsection{Conclusion}

The study we have just made casts a shadow over Barr's theorem, because it would seem to assert in this case that any primitive recursive map proved null in classical mathematics would be provably null in constructive mathematics.

However, it may be that the intuitive proof of classical mathematics uses dubious principles, such as those formalised in the theory \sa{ZF}, in which case it may lead to quite questionable results.

For example, consider the primitive recursive map $\mathrm{consisZ}$, which always takes the value $0$ until such time as we eventually find a proof of $0 = 1$ in \sa{Z}, at which point the map takes the value $1$. Thus $\mathrm{consisZ}$ is a well-defined constant of sort $\iF_1$ in \sa{PRA}.

In classical mathematics with a sufficiently strong intuitive set theory (e.g.\ \sa{ZF}), it can be shown that $\mathrm{consisZ}$ is identically zero. And this demonstrates in classical mathematics (using \sa{ZF} in the external intuitive mathematical world) the rule $\vd \mathrm{consisZ}(x)=0$ in the geometric theory $\sa{PRA}\omega$.

Now this result clearly escapes any constructive proof. And there can be no constructive proof of the rule $\vd \mathrm{consisZ}(x)=0$ in the geometric theory $\sa{PRA}\omega$.

\smallskip In fact we need to clarify the statement of Barr's theorem. It does not say that the framework of classical mathematics is conservative for geometric properties in a geometric theory. It only says that when we use the same mathematics outside the formal infinitary theory, adding the connectives, quantifiers and rules of classical logic inside the infinitary geometric theory, does not allow us to prove new properties formulable as geometric rules.

%%%%%%%%%%%%%%%%%%%%%%%%%%%%%%%%%%%%%%%%%%%%%%%%%%%%%%%%%%%%%%%%%%%%

\vspace{3cm} 
% {\bf ajouter une conclusion de la première partie}

%
%\chapter*{Conclusion}
%\addstarredchapter{Conclusion}

\newpage \thispagestyle{empty}

\part{Finitary geometric theories for real algebra}

\chapter*{Introduction}
\addstarredchapter{Introduction}

This second part is devoted to the development of a finitary dynamical theory whose ambition is to describe exhaustively the algebraic properties of the real number field, and more generally of a \ndrcf\footnote{In this text, a negation is italicised when the corresponding statement, true in classical mathematics, implies in constructive mathematics a well-recorded non-constructive principle, such as \tsbf{LPO} or even \tsbf{MP}.}
 at least those that are expressible in a restricted language, close to the language of ordered rings. This constitutes a development, with some minor terminological modifications, of the ideas given in the article \cite{LM2017}. The axiom of archimedianity, introduced in the last chapter, takes us out of the realm of finitary geometric theories.

\smallskip Chapter \ref{chapcoo} proposes a first definition of the ordered field structure in the absence of a sign test. It also discusses the possibility of a suitable axiomatic for  discrete real closed fields, such as the field of real numbers.

Section \ref{seccodi} gives some reminders about the dynamical theory of discrete ordered fields and that of discrete real closed fields.
 
Section \ref{secPstFormels} discusses the decisive consequences of formal Positivstellensätze, in our framework.

Section \ref{secConondisc} proposes an axiomatic for \ndsofs (Definition \ref{defiConondiscret}). 

Section \ref{condna} describes an example of a non-archimedean discrete Heyting ordered field.

Section \ref{secCRCnondis} gives a first discussion on acceptable dynamical theories for $\RR$ as a \ndrcf.

\smallskip Chapter \ref{chap-afr} deals with dynamical theories which admit extensions essentially equivalent to the theory \Sa{Co} of \ndsofs. 

We start (Section \ref{sectrdisad}) with the theory of distributive lattices (a \textsl{non} discrete ordered field is a distributive lattice for its order relation). 

In Section \ref{secgrl} we deal with $\ell$-groups of lattice groups (purely equational theory, valid for addition on the reals).

Then (Section \ref{secfrings}) we move on to $f$-rings ($f$-rings in french litterature), a theory inspired by rings of continuous real maps. 

Section \ref{secArftr} describes dynamical theories in which the predicate $\, \cdot>0 $ is added to the signature (strict $f$-rings and variants).

Section \ref{secCOG} proposes a return to the theory \Sa{Co} by confronting it with suitable extensions of the theory of strict $f$-rings.

\smallskip Chapter \ref{chapreelclos} proposes a definition of the structure of a real closed ordered field in the absence of a sign test.

Section \ref{subsecclotrlRR} explains how to introduce square roots of the elements~\hbox{$\geq 0$} in a \ndsof. This is done as a warm-up to the more general notion of virtual roots.

Section \ref{secCoVR} introduces virtual root maps and some corresponding dynamical theories: in particular $f$-rings with virtual roots and \ndsofs with virtual roots,

Section \ref{secArc} deals with real closed rings and Section \ref{secCrc2} proposes a definition for \ndrcfs as local real closed rings. The theory of real closed rings is presented here in an elementary, purely equational form, in the style of \cite{Tre2007}. 

\smallskip Chapter \ref{secGeomReelsArchi} deals with an infinitary geometric theory where we add the axiom that the field of real numbers is \textsl{archimedean}.

\smallskip Thus, we propose for the coveted dynamical theory the structure of a local real closed ring (possibly archimedean if that proved useful). 

\smallskip In a concluding chapter, we summarise the situation we have arrived at, specifying the important questions, from a constructive point of view, which we do not know how to answer satisfactorily today. 

\newpage \thispagestyle{empty}

\chapter{Ordered fields} \label{chapcoo}
\Today
\minitoc
%%%%%%%%%%%%%%%%%%%%%%%%%%%%%%%%%%%%%%%%%%%%%%%%%%%%%%%%%%%%%%%%%%%%

\section*{Introduction}
\addtocontents{toc}{\vskip0.8em}
\addcontentsline{toc}{section}{Introduction}
\rdb

This chapter offers a first constructive approach to the classical theory of real closed fields. In fact, the classical theory applies only to real closed fields for which we have a sign test on any element of the field, if it is given in accordance with the definition. In other words, the usual classical theory is a theory of discrete real closed fields. But it is well known that matrix numerical analysis, used in applications of the theory to concrete situations, never uses such a sign test. A constructive approach to a theory of algebraic properties of the real number field requires a dynamical theory of \ndrcfs. 

Section \ref{seccodi} gives some reminders on the dynamical theory of discrete ordered fields and that of discrete real closed fields.

Section \ref{secPstFormels} explains the great utility of formal Positivstellensätze.

Section \ref{secConondisc} proposes an axiomatic for \ndsofs (Definition \ref{defiConondiscret}). We must abandon the axioms of total order in their usual discrete formulation and replace them with \rdys relevant to $\RR$. We then find that many well-defined continuous rational maps on $\QQ$, such as the map $ \mathrm{sup} $, need to be introduced into the language.

Section \ref{condna} describes an example of a \textsl{non} discrete non-archimedean Heyting ordered field.

Section \ref{secCRCnondis} gives a first discussion of acceptable dynamical theories for~$\RR$ as a \ndrcf. We are guided by the continuity extension principle \ref{thRRcomplet} which can be seen as an algebraic version of the completeness of $\RR$. In this framework  \thref{thParamcontFsagc0} plays a fundamental role for a relevant definition of continuous semialgebraic maps, by reducing the definition to the case of continuous semialgebraic maps in the discrete framework of $\RRa$.

\section{About discrete ordered fields}\label{seccodi}

A \textsl{discrete ordered field} is a discrete field $\gR$ equipped with a suitable order relation, which can be described by the data of the subset $P$ formed by the elements  $\geq  0$. This subset $P$ must satisfy the following properties.
\begin{enumerate}
\item $P\cup -P=\gR$.
\item $P\cap -P =\so0$.
\item $P+P:=\sotq{x+y}{x,y\in P}\subseteq P$.
\item $P\cdot P:=\sotq{x\cdot y}{x,y\in P}\subseteq P$.
\item $\gK^2:=\sotq{x^2}{x\in \gR}\subseteq P$.
\end{enumerate}
The order relation is defined by  $x\geq y\equidef x-y\in P$, and the strict order is defined by $x>0\equidef x\geq0 \vii x\neq 0$. 

From a constructive point of view, the \gui{or} hidden in the  $\cup$ of Item 1 must be explicit. So we have explicitly the trichotomy \fbox{$x=0\vuu x> 0\vuu x< 0$}.

In terms of dynamical theory, all this is translated by predicates and laws given in the language of \Sa{Aso}. In order to demonstrate the Positivstellensatz constructively, we have chosen axioms suitably ordered, starting with the direct rules.

%:Subsection A natural dynamical theory for discrete ordered fields
\Subsection{A natural dynamical theory for discrete ordered fields}\label{subseccodi}

We recall here the dynamical theory of \textsl{discrete ordered fields} \SA{Cod} given in \cite{CLR01}.\index{discrete ordered field}%.

\vspace{-.2em}
\Sigt{\Aso}{\cdot=0,\cdot\geq 0,\cdot>0\mathrel{;}\cdot+\cdot, \cdot\times\cdot,-\cdot, 0,1} \label{NOTASigAso}

\noindent  If we want to give a dynamic discrete ordered field, i.e.\ a dynamic algebraic structure of type \Sa{Cod}, we add to the signature a presentation by generators and relations of the dynamic algebraic structure considered. For example, this can be the empty presentation, or a countable set of generators, without any relations, or it can be based on an existing algebraic structure in which certain relations are required to be preserved, for example all the equality relations between terms constructed on the elements of the structure. Thus any ring defines a dynamic discrete ordered field.

\bni {\bf Abbreviations}

\vspace{-1em} \DeuxCols{
\begin{itemize}
\itbu $x\neq0 $ signifie $x^2>0$
\itbu $x = y $ signifie $ x - y = 0$
\itbu $x > y $ signifie $ x - y > 0$
\itbu $x < y $ signifie $ y > x$
\end{itemize}
}
{\begin{itemize}
\itbu $x \geq y$ signifie $x-y\geq 0$ $\phantom {x^2>0}$
\itbu $x \neq y $ signifie $ x - y \neq 0$
\itbu $x \leq y $ signifie $ y \geq x$
\end{itemize}
} 

\bni {\bf Axioms}

\medskip\noindent {\sl Direct rules}

\smallskip First we put the axioms of commutative rings, then the rules concerning $ \cdot=0 $ and $ \cdot\geq 0 $, then the rules involving $ \cdot>0 $.\footnote{The rules \Tsbf{ga0} and \Tsbf{ga1} have been introduced with namaes \Tsbf{ac0} and \Tsbf{ac1} in Example \ref{exaAc}.}

\DeuxRegles{
\laB{ga0} $\vd 0=0$
\laB{ac2} $\,\,x=0\vd xy=0 $
}
{
\laB{ga1} $\,\, x=0\vet y=0\vd x+y=0$
}

\DeuxRegles{
 \laB{gao0} $\vd 0 \geq 0$
 \Lab{ao1} $\vd x^2 \geq 0$
}
{
 \laB{gao1} $\,\, x \geq 0\vet y \geq 0 \vd x + y \geq 0$
 \Lab{ao2} $\,\, x \geq 0\vet y \geq 0 \vd x y\geq0$ 
}

\DeuxRegles{
 \Lab{aso1} $ \vd 1> 0$ \phantom{$x^2$}
 \Lab{aso2} $\,\, x> 0 \vd x \geq 0$
}
{
 \Lab{aso3} $\,\, x > 0\vet y \geq 0 \vd x + y > 0$
 \Lab{aso4} $\,\, x > 0\vet y > 0 \vd xy > 0$
}

\medskip\noindent {\sl Collapse} 

\Regles{\lAb{col$ _> $} $ \,\,0> 0 \vd 1=0 $ 
\label{Axcoligt}}

\medskip\noindent {\sl Simplification rules}

\DeuxRegles{
\laB{Gao} $\,\, x\geq 0\vet x\leq 0 \vd x = 0$
}
{
\Lab{Iv} $\,\, xy = 1 \vd x\neq  0$
}

\medskip\noindent {\sl Dynamical rules}

\DeuxRegles{
\Lab{IV} $\,\, x\neq 0 \vd \Exists y\; xy = 1$

\item [\Edinq]  $\vd x=0 \vou \,x\neq 0$
}
{
\Lab{OT} $ \vd x \geq 0 \vou x\leq 0$
}

\smallskip 
The dynamical theory \SA{Crcd} of \textsl{discrete real closed fields} is obtained from the theory \Sa{Cod} by adding as axioms the \rdys \tsbf{RCF$_n$}.\footnote{A theorem essentially equivalent to these rules is proved by Bishop for the field $\RR$, but using the axiom of dependent choice.}\index{real closed!discrete field}

\Regles{
\lAb{RCF$_n$} {$ \,\, a< b\vet P(a)P(b)<0 \vd \Exists x\, \big(P(x)=0,\,a<x<b\big) $ \quad ($ P(x)=\sum_{k=0}^na_kx^k $) }\label{RCFn}}

\medskip The rules \tsbf{gao0} and \tsbf{gao1} express, in the context of abelian groups, the reflexivity and transitivity of the order relation (compatible with the group law). The rule \tsbf{Gao} corresponds to antisymmetry for the order relation.

\smallskip The rules \Edinq\  and \Tsbf{OT} express that the equality is discrete and the order total. They are not satisfied constructively for~$\RR$. For Bishop's reals, the rule \tsbf{ED$ _> $} is equivalent to the omniscience principle \tsbf{LPO} and the rule \Ednq\ is equivalent to the principle~\tsbf{LLPO}. Note also that the principle \gui{any regular element of $\RR$ is invertible} is equivalent to the Markov principle\footnote{Equivalence suggested by Fred Richman.} \tsbf{MP}.

\smallskip Given the form \gui{without negation} adopted here for collapse, the trivial ring is a discrete ordered field, and the collapse axiom \tsbf{col$ _> $} is a consequence of \tsbf{IV}.

\smallskip By means of the direct rules alone we see that $\,1=0\vd (x=0\vet x\geq 0\vet x>0)$. This justifies taking $1=0$ as a substitute for $\Bot$. 

%%%%%%%%%%%%%%%%%%%%%%%%%%%%%%%%%%%%%%%%%%%%%%%%%%%%%%%%%%%%%%%%%%%%
\Subsection{Some valid rules in \Sa{Cod}}
%: Subsection{Quelques règles dérivées}

\noindent  {\sl Four valid \rsims} \label{regsimpcod}

\DeuxRegles{
\laB{Anz} $ \,\, x^2= 0 \vd x = 0 $ 
\Lab{Aonz} $ \,\, c\geq 0\vet x(x^2+c)\geq 0 \vd x\geq 0 $ 
}
{
\Lab{Aso1} $ \,\, x> 0\vet xy\geq 0 \vd y\geq 0 $ 
\Lab{Aso2} $ \,\, x\geq 0\vet xy> 0 \vd y> 0 $ 
}

\medskip\noindent {\sl Two valid \rdys}

\DeuxRegles{
\Lab{OTF} $ \,\, x+y> 0 \vd x > 0\vou y>0 $ 
}
{
\lAb{OTF $ \eti $} $ \,\, xy< 0 \vd x <0 \vou y<0 $ \label{AxOTFx}
}

\smallskip\noindent Note that the rule \tsbf{Aso1} implies that elements $ >0 $ are regular.

%l
%: Lemma{lemthRRco}
\begin{lemma} \label{lemthRRco} The following rule is provable with direct axioms.

\Regles{\Lab{Aonz2} $ \,\, c\geq 0\vet x(x^2+c)\geq 0\vet x<0 \vd 0> 0 $}
 
\end{lemma}
%----------- fin lemma ----------------------------------- 

%: Theorem{thRRco}
\begin{theorem} \label{thRRco}
With the exception of the rules \tsbf{ED$ _> $} and \tsbf{OT}, all the above rules are constructively valid for $\RR$, without using the axiom of dependent choice. 
\end{theorem}
%----------- fin theorem ----------------------------- 
%
\begin{proof}
Everything is clear except perhaps the rule \Tsbf{Aonz}. For $ x\in\RR $, we can prove $ x\geq 0 $ by reducing $ x<0 $ to the absurd. This is what \Tsbf{Aonz2} does (take $ c=0 $).
\end{proof}
%

%r
%: Remark{remOTFx}
\begin{remark} \label{remOTFx} 
 The rules \Edinq\ and \OTFx  imply the rule $ \vd x=0 \,\vou\, x<0,\vou\, x>0 $, and a fortiori~\Tsbf{OT}. See also~\ref{lemAfrsdz}, \ref{lemAsrs} and \ref{lemArftr}. 
 \eoe
\end{remark}

If $\gK$ is a discrete ordered field, we denote $ \Sa{Cod}(\gK) $ the dynamic algebraic structure of type \Sa{Cod} having for presentation the \textsl{positive diagram of $\gK$}. A non-trivial model of $ \Sa{Cod}(\gK) $ is a non-trivial discrete ordered field $ \gL $ given with a morphism $ \gK\to\gL $. Similarly, if $\gA$ is a commutative ring (or an ordered ring), we denote $ \Sa{Cod}(\gA) $ the dynamic algebraic structure~of type \Sa{Cod} having for presentation the \textsl{positive diagram of $\gA$}. A model of $ \Sa{Cod}(\gA) $ is a discrete ordered field $\gL$ given with a morphism $ \gA\to\gL $ (of commutative ring, or of ordered ring).

%%%%%%%%%%%%%%%%%%%%%%%%%%%%%%%%%%%%%%%%%
\Subsection{Weaker dynamical theories} 
%: Subsection{Weaker dynamical theories}

The rule \Tsbf{Aonz} implies $ x^3\geq 0\vd x\geq 0 $, therefore also, under \Tsbf{Gao}, $ x^3= 0\vd x= 0 $, and a fortiori~\Tsbf{Anz}.

%: Definition{defisaAor}
\begin{definition} \label{defisaAor}~
\\
A) Theories based on the language of ordered rings\index{ring!ordered ---} \sIgt{\Ao}{\cdot=0,\cdot\geq 0\mathrel{;}\cdot+\cdot, \cdot\times\cdot,-\,\cdot,0,1}. \label{NOTASigAo}
\begin{enumerate}
 
\item The direct theory \SA{Apo} of \textsl{preordered rings}.\index{ring!preordered ---}
The axioms are those of commutative rings and the direct rules \Tsbf{gao0}, \Tsbf{gao1}, \Tsbf{ao1}, \Tsbf{ao2}.%.
 
\item The Horn theory \SA{Ao} of \textsl{ordered rings}. The axioms are those of pre-ordered rings and the \rsim \Tsbf{Gao}.%
 
\item The Horn theory \SA{Aonz} of \textsl{strictly reduced ordered rings}\footnote{The rule \tsbf{Aonz} is stronger than the rule \tsbf{Anz}, so \gui{strictly reduced} is used rather than \gui{reduced}. However, see Item \textsl{3} of Lemma \ref{lemAtonz}.} is obtained by adding the \rsim \Tsbf{Aonz} to the theory \sa{Ao}.\index{strictly reduced!ordered ring}\index{ring!strictly reduced ordered ---}%.
 
\item The dynamical theory \SA{Ato} of \textsl{linearly ordered rings}\footnote{An order relation is \textsl{linear} or \textsl{total} when two elements are always comparable.} is obtained by adding the \rdy \Tsbf{OT} to the theory \sa{Ao}.%
 
\item The dynamical theory \SA{Atonz} of \textsl{reduced linearly ordered rings} is obtained by adding the \rdy \Tsbf{Anz} to the theory \sa{Ato}. This theory proves the rules \tsbf{Aonz} and \tsbf{ASDZ} (Lemma~\ref{lemAtonz})%.
\end{enumerate}
B) Theories based on the language of strictly ordered rings: we add the predicate $ \cdot> 0 $.
\begin{enumerate}\setcounter{enumi}{5}
 
\item The direct theory \SA{Apro} of \textsl{proto-ordered rings} (cf. \cite{CLR01}). The axioms are those of commutative rings, all the direct rules stated for \Sa{Cod} (\Tsbf{gao0}, \Tsbf{gao1}, \Tsbf{ao1}, \Tsbf{ao2}, \Tsbf{aso1} to \Tsbf{aso4}) and the collapsus \coligt.%.
 
\item The Horn theory \SA{Aso} of \textsl{strictly ordered rings} is the theory \sa{Apro} to which we add the \rsims \Tsbf{Gao}, \Tsbf{Aso1} and \Tsbf{Aso2}. It can also be seen as constructed from \sa{Ao} by adding the predicate $\, \cdot> 0 \,$ in the language, the direct rules \Tsbf{aso1} to \Tsbf{aso4} and the \rsims \Tsbf{Aso1} and \Tsbf{Aso2}.\index{strictly ordered!ring}\index{ring!trictly ordered ---}%.
 
\item The Horn theory \SA{Asonz} of \textsl{reduced strictly ordered rings} (\gui{quasi-ordered rings} in \cite{CLR01}) is obtained by adding the \rsim \Tsbf{Aonz} to~\sa{Aso}. It can also be seen as the theory \sa{Apro} to which we add the \rsims \Tsbf{Gao}, \Tsbf{Aonz}, \Tsbf{Aso1} and \Tsbf{Aso2}.%.
 
\item The dynamical theory \SA{Asto} of \textsl{strictly linearly ordered rings} is the theory \sa{Aso} to which we add the \rdy \Tsbf{OT}. It can also be seen as constructed from \sa{Ato} by adding the predicate $ \cdot> 0 $ in the language, the direct rules \Tsbf{aso1} to \Tsbf{aso4} and the \rsims \Tsbf{Aso1} and \Tsbf{Aso2}.%.
 
\item The dynamical theory \SA{Aito} of \textsl{linearly ordered integral rings} is obtained by adding the \rdy \tsbf{ED$ _> $}  to \sa{Asto}. The rules \tsbf{Aonz}, \tsbf{OTF} and \tsbf{OTF $ \eti $} are valid in this theory.%
\end{enumerate} 
\end{definition}
%--------- end definition --------------------------------

In Items 6, 7 and 8, the meaning of $ x>0 $ is not fixed a priori. This can range from \gui{$x$ is regular and $\geq 0$} to \gui{$x$ is invertible and $\geq 0$}.

The direct theory \Sa{Apro} is the one in which the collapse is the clearest, directly given by an algebraic certificate, as specified in the following lemma. 

Recall that in a ring, a \textsl{cone}\index{cone} is a part $ C $ which contains squares and which is stable by addition and product: $ C+C\subseteq C $, $ C\times C\subseteq C $.

%l
%: Lemma{lemColApo}
\begin{lemma}[algebraic certificate of collapse] \index{collapse}\label{lemColApo}~\\
Let $\gK$ be a dynamic algebraic structure of type \Sa{Apro} given by a presentation $ (G; R_{>0}, R_{\geq 0}, R_{=0}) $ with the following meaning: $G$ is the set of generators of the structure, $R_{>0}$, $R_{\geq 0}$ and $R_{=0}$ are three parts of $\ZZ[G]$, the elements of $R_{>0}$ (resp. $R_{\geq 0}$, $R_{=0}$) are assumed $> 0$ (resp. $\geq 0$, $=0$) in $\gK$.\\ 
The dynamic algebraic structure $\gK$ collapses if, and only if, we have in $\ZZ[G]$ an equality 
\[\fbox{$s+p+z=0$}\] 
where $s$ is in the multiplicative monoid generated by $R_{>0}$, $p$ is in the cone generated by $R_{>0}\cup R_{\geq 0}$ and $z$ in the ideal generated by $R_{=0}$. 
\end{lemma}
%----------- end lemma ----------------------------------- 

%l
%: Lemma{lemAtonz}
\begin{lemma} \label{lemAtonz} \emph{(See also Lemma \ref{lemAfrsdz}.)}
\begin{enumerate}
 
\item The dynamical theory \sa{Aso} proves the \rsim \Tsbf{Iv} and the collapsus \coligt.
\item The dynamical theory \sa{Ato} proves the following \rsims. 

\TwoRegles{
\Lab{Ato1} $ \,\,y\geq 0 \vet xy=1\vd x\geq 0 $ 
}
{
\Lab{Ato2} $ \,\, c\geq 0\vet x(x^2+c)\geq 0 \vd x^3\geq 0 $ 
}

\item The theory \sa{Atonz} proves the rules \Tsbf{Aonz} and \Tsbf{ASDZ}.
\end{enumerate}
 
\end{lemma}
%----------- end lemma ----------------------------------- 
%
\begin{proof}
Items \textsl{1} and \textsl{2} are easy (use \Tsbf{Gao}). For Item \textsl{3}, \tsbf{Aonz} follows from~\tsbf{Ato2}. Let's look at \tsbf{ASDZ}. Let $ a,b $ be such that $ab=0$; if $ \abs a\leq \abs b $, we have $ 0\leq {\abs a}^2\leq \abs {ab}=0 $ so $ a=0 $, and in the case where $ \abs b\leq \abs a $ we get $b=0$. 
\end{proof}

In the \sa{Aso} theory, the rule \tsbf{Ato1} is a weakened variant of \Tsbf{Aso2}. In the \sa{Ao} theory, the rule \tsbf{Ato2} is a weakened variant of \Tsbf{Aonz}.

\smallskip In a linearly ordered ring, if we define \gui{$x>0$} by \gui{$x$ is regular \hbox{and $\geq 0$}}, all the rules that define \Sa{Aso} are satisfied (and $x\neq 0$ is a priori stronger than the simple negation of $x=0$). This explains the interest of the Horn theory  \Sa{Aso}. A lattice variant, the \Sa{Asr} theory of strict $f$-rings, will be defined later.

%%%%%%%%%%%%%%%%%%%%%%%%%%%%%%%%%%%%%%%%%
\Subsection{An example with nilpotents}

%e
%: Example{exatotordnonreduced}
\begin{example}[a non-reduced linearly ordered ring] \label{exatotordnonreduced}  The example we now give is the one that should be kept in mind in order to fully understand the difference between linearly ordered rings and linearly ordered domains.

\noindent This is the linearly ordered ring $ \QQ[\alpha] $ where $ \alpha>0 $ and $ \alpha^6=0 $ ($\alpha$ is an infinitesimal~\hbox{$ >0 $} nilpotent). Let $ c $ be an element such that $ c^2=0 $ (for example $ c=\alpha^5 $). The system of constraints 

\snic{x^2=c^2,\;x\geq 0,}

\noindent which could be suggested to characterise $ \abs c $ without using the sign test in the case of an ordered field, now admits an infinite number of solutions: all $ r\alpha^3+y\alpha^4 $ where $ r>0 $ in $\QQ$ and $y$ arbitrary in $ \QQ[\alpha] $. \eoe
\end{example}
%--------- end example ---------------------------------------- 

%%%%%%%%%%%%%%%%%%%%%%%%%%%%%%%%%%%%%%%%%%%%%%%%%%%%%%%%%%%%%%%%%%%%
%:Subsection Add function symbol $\vu$
\Subsection{Adding the function symbol $\vu$ for the lub}\label{secCodisup}

In a linearly ordered set, and even more so in a discrete ordered field, every pair of elements has a least upper bound (lub): the greater of the two. We therefore change nothing essential in the theory of \sa{Cod} by adding a function symbol $\cdot\vu\cdot$ subject to the three axioms which define the sup of two elements, when it exists for a of given order relation.

%: Definition{deficodisup}
\begin{definition} \label{deficodisup}
 The dynamical theory of \textsl{discrete ordered fields with sup}, denoted \SA{Codsup}, is the dynamical theory of discrete ordered fields to which we add a function symbol~$ \cdot\vu\cdot $ and for axioms the following Horn rules \tsbf{sup1}, \tsbf{sup2} and \tsbf{Sup}.\index{discrete!ordered field with sup}\index{discrete ordered field!with sup}%

\TwoRegles{
\laB{sup1} $ \vd x\vu y\,\geq x $ 
\laB{Sup} $ \,\, z\geq x\vet z\geq y\vd z\geq x\vu y $ 
}
{
\laB{sup2} $ \vd x\vu y\,\geq y $ 
}

\noindent From the theories \Sa{Ato}, \Sa{Asto}, \Sa{Aito} and \Sa{Crcd} the theories \SA{Atosup}, \SA{Astosup}, \SA{Aitosup} and \SA{Crcdsup} are defined in the same way.
\end{definition}
%--------- end definition --------------------------------

In the case of discrete ordered fields and discrete real closed fields we could also have replaced the \rsim \Tsbf{Sup} by the following direct rule \tsbf{sup}, so as to add only direct rules to \Sa{Cod}.

\Regles{
\laB{sup} $ \vd \big((x\vu y)- x\big)\,\big((x\vu y)- y\big)=0 $ 
}
 
We can also think of the theory \Sa{Crcdsup} as the theory \Sa{Codsup} to which we add the axioms of real closure \RCFn.
 
%r
%: Remark{rem}Crdsup
\begin{remark} \label{remCrdsup} 
The theories \Sa{Cod} and \Sa{Codsup} are essentially identical. The same applies to the other pairs of theories in Definition \ref{deficodisup}. 
\eoe\end{remark}
%----------- end remark ---------------------------------- 

%%%%%%%%%%%%%%%%%%%%%%%%%%%%%%%%%%%%%%%%%%%%%%%%%%%%%%%%%%%%%%%%%%%%
\section{Formal Positivstellensätze}\label{secPstFormels}
%: Subsection{Formal Positivstellensätze}

The formal Positivstellensatz of classical mathematics (\cite[Theorem~4.4.2]{BCR}) admits the following constructive version (see \cite{CLR01}).
%\index{Positivstellensatz!formal --- !for ordered fields, 3}
%: pstf{Pst1}
\begin{pstf}[formal Positivstellensatz for ordered fields, 1] \label{Pst1}
\index{Positivstellensatz!formal --- !for ordered fields, 1}~

\noindent  
On considère the dynamic algebraic structures in the language of strictly ordered rings.
\begin{enumerate}
 
\item The dynamical theories \Sa{Apro}, \Sa{Aso}, \Sa{Cod} and \Sa{Crcd} collapse simultaneously. 
 
\item The dynamical theories \Sa{Asonz}, \Sa{Aito}, \Sa{Cod} and \Sa{Crcd} prove the same Horn rules.
\end{enumerate}
\end{pstf}
%--------- fin pstf ----------------------------------- 

The strength of this theorem lies in the fact that the collapse of a dynamic algebraic structure of type \Sa{Apro} is given by an \textsl{algebraic certificate of collapse} (Lemma \ref{lemColApo}), which we call a \textsl{Positivstellensatz}.%
\index{Positivstellensatz}\index{algebraic certificate}

As a special case, for a commutative ring $\gR$, the dynamic algebraic structure $ \Sa{Crcd}(\gR) $ collapses if, and only if, $ -1 $ is a sum of squares in $\gR$.

In classical mathematics, thanks to Gödel's completeness theorem \ref{thGodel1}, we deduce from previous Item~\textsl{1} the abstract formal \textsl{Positivstellensatz} in the following form (see~\cite{CLR01}).

\smallskip\noindent \textsl{A system of sign conditions imposed on elements of a ring $\gA$ admits an algebraic certificate of impossibility if, and only if, the only model of $ \Sa{Cod}(\gA) $ is trivial, if, and only if, the only model of $ \Sa{Crcd}(\gA) $ is trivial. \index{Positivstellensatz!abstract formal ---}} 

\smallskip Item \textsl{2} of \ref{Pst1} also admits \gui{abstract} classical versions  via model theory, in application of   \thref{thcolsimralg} (see \cite{CLR01}). 

\smallskip We now examine what happens to the previous results in the absence of the predicate~\hbox{\gui{$ \cdot>0 $}} in the presentation of a dynamic algebraic structure.

%: pstf{Pst1bis}
\begin{pstf}[formal Positivstellensatz, 1bis]\label{Pst1bis}%
\index{Positivstellensatz!formal --- !for ordered fields, 1bis} 

\noindent 
We consider dynamic algebraic structures in the language of ordered rings.
\begin{enumerate}
 
\item The dynamical theories \Sa{Ao}, \Sa{Apro}, \Sa{Ato}, \Sa{Cod} and \Sa{Crcd} collapse simultaneously. 
 
\item The dynamical theories \Sa{Aonz}, \Sa{Atonz}, \Sa{Cod} and \Sa{Crcd} prove the same Horn rules.
\end{enumerate}
NB. The language used for the presentation must not mention the predicate $ \cdot>0 $. There is no collapse axiom in \Sa{Ao} and \Sa{Ato}, and a dynamic algebraic structure of type \Sa{Ao} or \Sa{Ato} is said to collapse when it proves $1=0$.
\end{pstf}
%--------- end pstf -----------------------------------
 
%
\begin{proof}

\noindent \textsl{1}. Consider a dynamic algebraic structure $\gA$ for \Sa{Ao}. The same presentation gives a dynamic algebraic structure~$\gA'$ for \Sa{Apro}. Suppose that $\gA'$ proves $ 1=0 $. The collapse for~$\gA'$ (of type~\Sa{Apro}) has the form of a very precise algebraic certificate (a Positivstellensatz). This certificate is written $ 1+p=0 $, where $p$ is $\geq 0$ by virtue of the presentation and the axioms of~\Sa{Ao}. We therefore have both $ 1\geq 0 $ and $ 1\leq 0 $ in $\gA$, so $ 1=0 $ by virtue of \Tsbf{Gao}. Conversely, if $\gA$ proves~$ 1=0 $, then a fortiori the same will be true for $\gA'$.
\\ Finally, we apply \pstref{Pst1} and note that \Sa{Ato} is an intermediate theory between \Sa{Ao} and \Sa{Cod}.

\sni\textsl{2}. Consider a dynamic algebraic structure $\gA$ for \Sa{Aonz}. It suffices to prove the result for a fact of the form $x\geq0$ (because $x=0$ is equivalent to $ x\geq 0 $ and $ x\leq 0 $). This fact is valid in~\Sa{Cod} if, and only if, the fact $x<0$ collapses the dynamic algebraic structure $ \Sa{Cod}(\gA) $. According to \pstref{Pst1}, this corresponds to an algebraic certificate of the form $ x^{2n}+p=xq $, where~$p$ and~$ q $ are $\geq 0$ by virtue of the presentation and the axioms of \Sa{Ao}. This gives $x(x^{2n}+p)\geq0$, then $ x^k(x^{2k}+p_1)\geq 0 $ for a suitable odd integer $k$. The rule \Tsbf{Aonz} tells us that $ x^k\geq 0 $ in $\gA$. This same rule shows that $ x^3\geq 0 $ implies~\hbox{$ x\geq 0 $}, and consequently $ x^k\geq 0 $ implies $ x\geq 0 $ for all odd~$k$.
\end{proof}

A consequence of Item \textsl{1} in classical mathematics (via  \thref{thcolsimcomp}) is that a field~$\gK$ in which $-1$ is not a sum of squares can be ordered. On the other hand, the only known \und{computational} meaning of this result of classical mathematics is that the theory~$\Sa{Cod}(\gK)$ collapses if, and only if, $-1$ is a sum of squares in $\gK$. 

From a classical point of view, since the theory $ \Sa{Cod}(\RR) $ does not collapse, we can provide $\RR$ with an order relation which extends the usual order relation and which is a total order. But the only constructive meaning of this result of classical mathematics is that $-1$ is not a sum of squares in $\RR$. 

Another consequence of Formal \pstref{Pst1bis} is the following corollary. A more direct proof seems a challenge. A direct proof will be given in the theory \Sa{Afrnz}.

%c
%:     Corollary{corPst1bis}
\begin{corollary} \label{corPst1bis}
The following rule is valid in the theory \sa{Aonz}.

\Regles{
\Lab{Aonz3} $\,\, a\geq 0\vet b\geq 0\vet a^2=b^2 \vd a=b$%
}
 
\end{corollary}
%--------- fin corollary ------------------------------- 

%%%%%%%%%%%%%%%%%%%%%%%%%%%%%%%%%%%%%%%%%%%%%%%%%%%%%%%%%%%%%%%%%%%%
\Subsection{The demonstrative force of formal Positivstellensätze}
%: Subsection{The demonstrative power of formal Positivstellensätze}

The dynamical theories we explore in the following to describe the algebraic properties of real numbers are extensions of \Sa{Asonz} (if the predicate $ \cdot > 0 $ is present) or \Sa{Aonz} (in the opposite case). Moreover, the theories explored are always weaker than \Sa{Crcd}. And any Horn rule valid in the dynamical theory \Sa{Crcd} is valid in \sa{Asonz} (in \sa{Aonz} if the predicate $ \cdot > 0 $ is absent).

Now $\RR$ constitutes a constructive model of the \Sa{Asonz} theory for the language based on the signature $ (\cdot=0,\cdot>0,\cdot\geq0\cdot\mathrel{;}\cdot+\cdot, \cdot\times \cdot,-\cdot,0,1) $ (\thref{thRRco}).

Thus from the point of view of Horn rules alone, the formal Positivstellensätze tell us that the theory \Sa{Crcd} is entirely satisfactory, including for $ \RR$, which does however satisfy neither \coligt \ nor~\Tsbf{OT}. However, to temper this optimistic statement, here is the precise result. Note also that it applies only to Horn rules, not to other \rdys.

%t
%: Theorem{thRRtalg}
\begin{theorem} \label{thRRtalg}
Consider a Horn rule formulated in the dynamic algebraic structure \hbox{$ \gR=\Sa{Asonz}(\RR) $}. If the constants involved in the rule are in a discrete subfield $ \gR_0 $ of $\RR$, for the rule to be valid in $\gR$, it is sufficient for it to be valid in $ \Sa{Crcd}(\gR_0) $. 
\end{theorem}
%----------- end theorem ----------------------------- 
The existential rules satisfied in $\RR$ and introduced in the dynamical theories under consideration will as far as possible be treated in the framework of provably unique existences and can therefore be skolemised without damage, providing theories without existential axioms which are essentially equivalent to those which would have required existential axioms.

%: Remark{remRRtdy0}
\begin{remark} \label{remRRtdy0} 
 We can also apply  \thref{thFond}  with the dynamical theory $ \Sa{Co--}(\RR) $ 
(Definition \ref{defiCo0}). We will then introduce a predicate $ x\succeq y $ opposed to $ x<y $. The new dynamical theory will treat $\RR$ as a discrete ordered field and any \rdy proved in the new theory but not using $x\succeq y$ will also be valid in $\RR$. The disadvantage is of course that $\RR$ is not a constructive model of the new theory. Another drawback is the mysterious status of the new predicate $x\succeq y $, which is weaker than $x\geq y$ in the new dynamical theory. In conclusion, the advantage that \thref{thFond} seems to provide (the use of classical logic is harmless) does not seem to go beyond the considerations we have developed on the proper use of the formal Positivstellesatz.
\eoe\end{remark}
%----------- end remark ---------------------------------- 

%%%%%%%%%%%%%%%%%%%%%%%%%%%%%%%%%%%%%%%%%%%%%%%%%%%%%%%%%%%%%%%%%%%%
\Subsection{Concrete Positivstellensatz}
%: Subsection{Concrete Positivstellensatz}

First, we recall Tarski's fundamental theorem. For a simple Cohen-Hormander proof, see \cite[Section 1.4]{BCR} or \cite[Lemma 3.12]{CLR01}. Some instructive comments can be found in \cite[Theorems 10, 11, 12]{LR90}.

%t
%: Theorem{thTarski}
\begin{theorem}[elimination of quantifiers] \label{thTarski} \label{NOTARa}
The first-order intuitionistic formal theory associated with the dynamical theory \Sa{Crcd} admits the elimination of quantifiers. It is complete and decidable. In particular, it exhaustively describes all the purely algebraic properties of $\RRa$ (those formulated to first-order in the language of ordered rings). 
\end{theorem}
%----------- end theorem ----------------------------- 

This paragraph gives a theorem equivalent to Krivine-Stengle's Positivstellensatz, stated here in the language of dynamic algebraic structures. 

A constructive proof of \pstref{thPstStengle} can be found in \cite{CLR01} or \cite{Lom91}. It is based on the formal Positivstellensatz and on  Lemma 3.12 of \cite{CLR01}, a variant of Tarski's theorem.

For a more conceptual approach and better complexity bounds see \cite{LPR2015}. For the construction of the real closure of a discrete ordered field see \cite{LR91,LR90}.
%: pstc{thPstStengle}
\begin{pstc} \label{thPstStengle}%
\index{Positivstellensatz!concrete ---} ~
\\
Let $\gK$ be a discrete ordered field and $\gR$ be a discrete real closed field containing $\gK$, (for example the real closure of $\gK$). Let $\gA=\big((G,Rel),\sa{Cod}(\gK)\big)$ be a dynamic algebraic structure where $G=(\xn)$ and where $Rel$ is finite. 
\begin{enumerate}
 
\item The dynamic algebraic structure $\gA$ collapses if, and only if, it is impossible to find a model of $\gA$ contained in $\gR$. 
 
\item The collapse if it occurs is given by an algebraic certificate according to Item~1 of   \thref{Pst1} and Lemma \ref{lemColApo}.
 
\item We have an algorithm which decides whether $\gA$ collapses and which in the case of a negative answer gives the description of a system $(\xin)$ in $\gR^n$ which satisfies the constraints given in the relations $Rel$.
\end{enumerate} 
\end{pstc}
%--------- end pst ----------------------------------- 

This statement is not valid in this general form if we take $\gK=\gR=\RR$ because there is no sign test in $\RR$ and the algorithms which explicite  
\pstref{thPstStengle}\footnote{These algorithms are provided by the constructive proof of the theorem.} make crucial use of this sign test. 

\smallskip Here's a small example of the problems we run into. On $\RR$, as on an arbitrary local ring\footnote{For the constructive treatment of local rings, the Jacobson radical and Heyting fields see for example \cite[section IX-1]{CACM}.} in which $x\neq0$ denotes the invertibility predicate, we have the equivalence
% equation label {eqTiersexclu1}
\begin{equation} \label {eqTiersexclu1}
\exists y\ x^2y=x \,\iff\; x=0\; \mathrm{ or }\; x\neq 0.
\end{equation}
% end-equation
Let's assume that $x(1-xy)=0$. If $xy$ is invertible, then $x$ is invertible, and if $1-xy$ is invertible, then $x=0$. This proof translates formally into the corresponding dynamical theory by establishing the following three valid rules: 

\Regles{
\labu $ \,\, x^2y=x\vd x=0 \; \vou\; x\neq 0 $, 
\labu $ \,\, x=0\vd\Exists y\; x^2y=x $,
\labu $ \,\, x\neq0\vd\Exists y\; x^2y=x $.
}

This simple case of eliminating the quantifier $ \exists $ shows that the calculations lead to dead ends from the point of view of decidability, since \gui{$ x=0 \;\mathrm{ or }\;  \hbox{ invertible} $} is undecidable in~$\RR$.

\smallskip Nevertheless, in the final section of the article \cite{GL93}, we find a fully satisfactory constructive form for the 17th Hilbert problem on $\RR$. And other cases of constructively provable Positivstellensätze on $\RR$ are also treated.

%%%%%%%%%%%%%%%%%%%%%%%%%%%%%%%%%%%%%%%%%%%%%%%%%%%%%%%%%%%%%%%%%%%%%
\section{\textsl{Non} discrete ordered fields}\label{secConondisc}
%\Today

As a first approximation, and following a suggestion by Heyting, we could choose as a first-order formal theory for the algebraic properties of $\RR$ the~\Sa{Asonz} theory (seen as a first-order formal theory) to which we add the geometric axioms \Tsbf{IV} and \Tsbf{OTF} as well as the following axiom \tsbf{HOF}, non-geometric and therefore undesirable. \index{ordered field!Heyting ---}\index{Heyting!ordered field}

\Regles{
\Lab{HOF} $ \quad (x> 0 \,\Rightarrow 1=0) \,\Rightarrow\, x\leq 0 $ \qquad (Heyting axiom for ordered field)
}

This amounts to replacing the axioms \coligt \ and \Tsbf{OT} by the axioms~\Tsbf{OTF} and~\Tsbf{HOF}. We then have a local ring structure, because the rules \Tsbf{Iv} and \Tsbf{IV} imply that \gui{$ x\neq0 $} means \gui{$x$ is invertible}, so \tsbf{OTF} implies that for all $x$, $x$ \hbox{or $ 1-x $} is invertible. 
In this context, the axiom~\tsbf{HOF} means that the Jacobson radical is reduced to $0$.

%r
%: Remark{remHOH}
\begin{remark} \label{remHOF} 
Note that the axiom \tsbf{HOF}, which can be formulated to first-order even though it is not part of dynamical theories, is satisfied indirectly in the following form: \textsl{in a dynamic algebraic structure of type \Sa{Asonz}, if a closed term $t$ verifies $\,t>0\vd\Bot $, then it also verifies $ \vd t\leq 0 $.}. This follows from the formal Positivstellensatz. 
\\
In fact we even have: \textsl{if a closed term $t$ verifies $\,t\geq 0\vd\Bot $, then it also verifies $ \vd t< 0 $}. This means that Markov's principle, which is expressed on $\RR$ by the implication $ \,\lnot (t\geq 0)\,\Rightarrow\,t<0 $ holds as an \textsl{external} deduction rule\footnote{We must add the word \textsl{external} here, as this is not a valid rule in the dynamic algebraic structure itself. It refers to the fact that we deduce the validity of one rule from that of another rule.} admissible in the dynamical theory $ \sa{Asonz}(\RR)$.\\
The same remarks apply to dynamical theories which extend \sa{Asonz} while simultaneously collapsing.
\eoe
\end{remark}
%----------- end remark ---------------------------------- 

%%%%%%%%%%%%%%%%%%%%%%%%%%%%%%%%%%%%%%%%%%%%%%%%%%%%%%%%%%%%%%%%%%%%
%:Subsection
\Subsection{A first dynamical theory}\label{subsecCo0}
 
 Apart from the undesirable character of \Tsbf{HOF}, the formal theory considered at the beginning of the section has a major drawback, which is that it cannot prove the existence of the upper bound of two elements: see on this subject~\cite{Coq07}.

It is therefore legitimate to explore the possibilities offered by the addition of a law for this upper bound, with the appropriate rules. We now propose a minimalist dynamical theory for \ndsofs by introducing the function symbol $\vu$ into the language.
 
%d
%: Definition{defiCo0}
\begin{definition} \label{defiCo0}
A first minimal \tdy for \ndsofs, denoted \SA{Co--}, is based on the following signature. There is only one sort, named $ \CoO $.
\Sigt{\CoO}{\cdot=0,\cdot\geq 0,\cdot>0\mathrel{;}\cdot+\cdot, \cdot\times\cdot,\cdot\vu\cdot,-\,\cdot,0,1} \label{NOTASigCoO}
\noindent The axioms are those of \Sa{Asonz}, the axioms \Tsbf{IV} and \Tsbf{OTF}, and the natural axioms for $\vu$: \Tsbf{sup1}, \Tsbf{sup2}, \Tsbf{Sup}, \Tsbf{grl} and~\Tsbf{afr}. They are all listed below ($ x^+ $ is an abbreviation of $ x\vu 0 $).%.

\DeuxRegles{
\laB{ga0} $ \vd 0=0 $ 
\laB{ac2} $ \,\,x=0\vd xy=0 $ 
\laB{gao0} $\vd 0 \geq 0 $ 
\laB{ao1} $ \vd x^2 \geq 0 $ 
}
{
\laB{ga1} $ \,\, x=0\vet y=0\vd x+y=0 $ 
\laB{}
\laB{gao1} $ \,\, x \geq 0\vet y \geq 0 \vd x + y \geq 0 $ 
\laB{ao2} $ \,\, x \geq 0\vet y \geq 0 \vd x y\geq0 $ 
}

\DeuxRegles{
\laB{aso1} $ \vd 1> 0 $ \phantom{$ x^2 $}
\laB{aso2} $ \,\, x> 0 \vd x \geq 0 $ 
\laB{sup1} $ \vd x\vu y\,\geq x $ 
\laB{sup2} $ \vd x\vu y\,\geq y $ 
}
{
\laB{aso3} $ \,\, x > 0\vet y \geq 0 \vd x + y > 0 $ 
\laB{aso4} $ \,\, x > 0\vet y > 0 \vd xy > 0 $ 
\laB{grl} $ \vd x+(y\vu z)=(x+y)\vu(x+z) $ 
\laB{afr} $ \vd x^+\, (y\vu z)=(x^+\, y)\vu(x^+\, z) $ 
}

\DeuxRegles{
\laB{Gao} $ \,\, x\geq 0\vet x\leq 0 \vd x = 0 $ 
\laB{Anz} $ \,\, x^2= 0 \vd x = 0 $ 
\laB{Aonz} $ \,\, c\geq 0\vet x(x^2+c)\geq 0 \vd x\geq 0 $ 
\laB{Sup} $ \,\, z\geq x\vet z\geq y\vd z\geq x\vu y $ 
\laB{IV} $ \,\, x> 0 \vd \Exists y\, xy = 1 $ 
}
{
\laB{Iv} $ \,\, xy = 1 \vd x^2> 0 $ 
\laB{Aso1} $ \,\, x> 0\vet xy\geq 0 \vd y\geq 0 $ 
\laB{Aso2} $ \,\, x\geq 0\vet xy> 0 \vd y> 0 $ 
\laB{}
\laB{OTF} $ \,\, x+y> 0 \vd x >0 \vou y>0 $ 
}

\end{definition}
%----------- fin definition -------------------------------- 

%r
%: Remark{remCo0}
\begin{remarks} \label{remCo0}~\\ 
 1) Note that the collapse \gui{$ \,\,0>0\vd 1=0 $ \,\,} is deduced from \Tsbf{IV}. 

\smallskip \noindent 2) If we add the axioms \Tsbf{OT} and \coligt \ to the theory \Sa{Co--} we find a theory which is essentially identical to \Sa{Codsup} or \Sa{Cod}. In fact, we just need to add the axiom \tsbf{ED$ _> $}: see Lemma \ref{lemAfrsdz}. 
\eoe\end{remarks}
%----------- end remark ---------------------------------- 

%e
%: Example{exacorpsnondiscret}
\begin{examples} \label{exacorpsnondiscret} 
 Many natural subfields of $\RR$ are \nds, for example the enumerable field $ \RR_{\tsbf{PR}} $ of real numbers computable in primitive recursive time, or the enumerable field $ \RR_{\tsbf{Ptime}} $ of real numbers computable in polynomial time, or the \textsl{non} enumerable field $\RR_{\tsbf{Rec}}$ of recursive real numbers. 
A satisfactory dynamical theory for the algebraic properties of the real numbers will have to accept as models these natural subfields of $\RR$.
\eoe
\end{examples}
%--------- end example ---------------------------------------- 

The subfields $ \RR_{\tsbf{PR}} $ and $ \RR_{\tsbf{Ptime}} $ can be handled on a machine in a nicer way than the field $ \RR_{\tsbf{Rec}} $. These are enumerable fields (albeit \nds), whose elements do not need to be accompanied by \gui{certificates} external to the dynamical theory under consideration (a general recursive map exists constructively only if it is accompanied by a \gui{certificate}: a constructive proof of the fact that it is total). 

Note that the \gui{complete} character of $\RR$ seems to come more from analysis than from algebra. Note also that the \gui{field} of Puiseux series on $\RR$ does not seem to satisfy \Tsbf{OTF} (for any attempt at a reasonable definition for the order relation).

%%%%%%%%%%%%%%%%%%%%%%%%%%%%%%%%%%%%%%%%%%%%%%%%%%%%%%%%%%%%%%%%%%%%
%:Subsection The convexity axiom and the theory \Sa{Co}
\Subsection{The convexity axiom and the theory \Sa{Co}}\label{subsecCocvx}

In addition to the lub map, other \gui{rational} maps pose the same kind of problem. 

\smallskip In the theory of real closed rings, in classical mathematics, (see the articles~\cite{Sch84,PN2002} and Section \ref{secArc}), the  following axiom \gui{of convexity}\footnote{There are two very distinct uses of the term \gui{convex} in the present text. On the one hand, an ordered ring can be declared convex, as here.\index{convex!ordered ring}\index{ordered ring!convex ---} On the other hand, a subgroup of an ordered group can be declared convex as \paref{subsecgrlsolide}.} is satisfied

\Regles{\Lab{CVX} $ \,\,0\leq a\leq b\vd \Exists z\; zb=a^2 $ \quad (convexity)}

\noindent Note that if $ zb=a^2 $ then $ (z\vi a)b=zb\vi ab=a^2\vi ab=a(a \vi b)=a^2 $. Similarly $ (z\vu 0)b=a^2 $. Consequently, an equivalent axiom which ensures the uniqueness of $z$ is given by the following \rdy:

\Regles{
\Lab{FRAC} $ \,\,0\leq a\leq b\vd \Exists z\; (zb=a^2\vet 0\leq z\leq a) $ 
}\label{RFRAC}

\noindent This rule is valid for $\RR$ because it defines the element $z$ as a continuous function from  $\sotq{(a,b)}{a,b\in\RR,\,0\leq a\leq b}$ to $\RR$.

%l
%: Lemma{lemUniqFRAC}
\begin{lemma} \label{lemUniqFRAC}
The uniqueness of $z$ (when it exists) in the rule \tsbf{FRAC} is guaranteed in the Horn theory \Sa{Atonz} and in~\Sa{Co--}. The same calculation shows that uniqueness is guaranteed in the Horn theory \Sa{Afrnz} (Lemma~\ref{lemAfrnzFRAC}).
\end{lemma}
%----------- end lemma ----------------------------------- 
%
\begin{proof}
 Suppose $ yb=zb=a^2 $, $ 0\leq z\leq a $ and $ 0\leq y\leq a $. We have $ {(y-z)}\,b=0 $, $ \abs{y-z}\leq a\leq b $ and \hbox{thus $ \abs{y-z}^2\leq\abs{y-z}\,b=0 $}. 
\end{proof}
%

%: Lemma{lemCo0FRAC}
\begin{lemma} \label{lemCo0FRAC}
The addition of the axiom \tsbf{FRAC} to the theory \Sa{Co--} can be replaced by the introduction of a function symbol $ \Fr $ with the axioms

\DeuxRegles{
\Lab{fr1} $ \vd \mathrm{Fr}(a,b)\, \abs b=(\abs a\!\vi\! \abs b)^2 $ 
}
{
\Lab{fr2} $ \vd 0\leq {\mathrm{Fr}(a,b)}\leq \abs a\!\vi\! \abs b \phantom{(\abs a\!\vi\! \abs b)^2} $ 
}

\noindent The same applies to the theories \Sa{Atonz} and \Sa{Afrnz}.
\end{lemma}
%----------- end lemma ----------------------------------- 

%
\begin{proof} The rule \tsbf{FRAC} is equivalent to the following rule

\Regles{\labu $ \vd \Exists z\; \big(z\abs b = (\abs a\vi\abs b)^2\vet\, 0\leq z\leq \abs a\vi\abs b\big) $}

\noindent We have therefore simply skolemised an existential rule in the case of unique existence (Lemma~\ref{lemUniqFRAC}). This gives us an essentially identical extension (see \paref{skolemunique}). Moreover, once the function symbol~$ \mathrm{Fr} $ and the axioms \tsbf{fr1} and \tsbf{fr2} are added, the rule \tsbf{FRAC} clearly becomes unnecessary.
\end{proof}
%

%: Definition{defiConondiscret}
\begin{definition} \label{defiConondiscret} 
 The dynamical theory \SA{Co} of \ndsofs\footnote{In \cite{LM2017}, a slightly different, slightly stronger definition was given, see Remark \ref{remCo}.} is the extension of the theory \Sa{Co--} obtained by adding the function symbol $\mathrm{Fr}$ and the axioms \Tsbf{fr1} and \Tsbf{fr2}.\index{ordered field!\textsl{non} discrete ---}
\end{definition}
%----------- end definition -------------------------------- 

This is a one sort dynamical theory with the signature
\Sigt{\Co}{\cdot=0,\cdot\geq 0,\cdot>0\mathrel{;}\cdot+\cdot, \cdot\times\cdot,\cdot\vu\cdot,-\,\cdot,\Fr(\cdot,\cdot),0,1}
\label{NOTASigCo}
\noindent The \nds subfields of $\RR$ of Example \ref{exacorpsnondiscret} are models of~\sa{Co}, and also of certain extensions of \sa{Co} which we define later, such as~\Sa{Crc1} or~\Sa{Corv}. 

%r
%: Remark{remPst1Co}
\begin{remark} \label{remPst1Co} 
In the Formal \pstref{Pst1} statement, we can add the \sa{Codsup} theory which is essentially identical to the \sa{Cod} theory, and the \sa{Co} theory which is intermediate between \sa{Asonz} and \sa{Codsup}. For more information, see Formal \pstref{Pst2}.\eoe 
\end{remark}
%----------- end remark ---------------------------------- 

\Subsection{Other continuous operations}\label{subsecratcon}

Here is another paradigmatic example with a continuous function defined everywhere
\vspace{-.5em}
% equation label {eqratcont}
\begin{equation} \label {eqratcont}
f(x,y)=\frac{(ax+by)xy}{x^2+y^2}
\end{equation}
This rational map \footnote{It is a priori defined for $ (x,y)\neq (0,0) $ and extends by continuity into $ f(0,0)=0 $. We can then see that it is uniformly continuous on any cube $ [-a,+a]^4 $.} is the prototype of a family, parametrised by $a,b$, of continuous real maps $\RR^2\to \RR$ (or of a continuous real map $\RR^4\to \RR$).

A \rdy defines this map:
% equation label {eqratcont2}
\begin{equation} \label {eqratcont2}
\vd \Exists z\quad \big(z(x^2+y^2)=(ax+by)xy ,\;
\abs z\leq \abs{ax+by}\big)
\end{equation}
% end-equation
and it does not seem valid in the basic theory \Sa{Co--}.

In this example, if $a=b=1$, the fraction is of the type $z=u/v$ with $u^2\leq v^3$. It is characterised by the relationships $zv=u$ and $\abs z^2\leq \abs v $. The following \rdys are satisfied for $\RR$, and also for discrete real closed fields: 

\Regles{\lAb{FRAC$_n$} $ \,\,\abs u^n\leq \abs v^{n+1}\vd \Exists z\; (zv=u\vet \abs z^{n}\leq \abs v) $ \quad ($ n\geq 1 $) }\label{AxFRACn}

\smallskip Intuitively, this rule means that the fraction $ u/v $ is well-defined. In the case of a discrete ordered field we reason case by case: if $ v\neq 0 $ it is clear, if $ v=0 $ the rule forces $ z=0 $. More generally we check that existence, if assumed, is uniquely proved in the theory \Sa{Afrnz} (\paref{theorieAfrnz}) as follows.
\\
If $ zv=u=yv,\, \abs z^{n}\leq \abs v,\, \abs {y}^{n}\leq \abs v $, we pose $ w=\abs{z-y} $ and we obtain 

\snic{w\abs v=0,\,\leq \abs z +\abs{y}\leq 2 {\abs v}^{\frac1 n},\,w^n\leq 2^n {\abs v},\,0\leq w^{n+1}\leq 2^n\abs v w=0
,}

\noindent therefore $ w^{n+1}=0 $ and finally $ w=0 $.

\noindent Another argument is that the \rsim

\Regles {\labu $ \,\,zv=u\vet yv=u\vet \abs z^{n}\leq \abs v\vet \abs {y}^{n}\leq \abs v \vd z=y $}

\noindent is satisfied in \Sa{Cod}, and that \Sa{Afrnz} and \Sa{Cod} prove the same Horn rules (Formal \pstref{Pst2}).

\smallskip In the Horn theory \Sa{Aonz} the rule \tsbf{FRAC} follows from \tsbf{FRAC$_1$} by posing~\hbox{$ u=a^2 $} \hbox{and $ v=b $}. 

Conversely, the rules  \FRACn \ can be deduced from the rule \tsbf{FRAC} in a fairly general framework (see Lemma \ref{lemfracfrac2}). In the following we will only use the rule \tsbf{FRAC}. 

%r
%: Remark{remCo}
\begin{remark} \label{remCo} 
It would have been more logical to ask, in the definition of the theory \Sa{Co}, in addition to the validity of the rule \tsbf{FRAC}, that of the rules $ \tsbf{FRAC}_n $ 
(this was the choice made in the article \cite{LM2017}, definition 2.13). More generally, for any map $ f\colon \QQ^n\to \QQ $ which extends by continuity\footnote{More precisely: the fraction $ h/p $ with coefficients in $\QQ$, defined for $ p(x)\neq 0 $, must extend into a continuous map $ \QQ^n\to\QQ $.} a fraction $ h/p $, where $h$ and $p$ are semipolynomials,\footnote{A sup-inf combination of polynomials.} the zeros of $p$ being of empty interior, we should ask that the rule analogous to $ \tsbf{FRAC}_n $ which says that \gui{$f$ exists} be valid, and more precisely introduce a corresponding function symbol with the appropriate axioms. But we did not want to complicate too much the definition of the theory \Sa{Co} of  \ndsofs insofar as we have essentially in view the theory of  \ndrcfs, in which the rule \tsbf{FRAC} is sufficient.
\eoe\end{remark}
%----------- end remark ---------------------------------- 

%: sous section supplémentaire
\section{A non-archimedean \ndsof}
\label{condna}

In this section we describe an example of a \nds non-archimedean Heyting ordered field.

Let $ \vep $ be an indeterminate. Let $ \gZ=\QQ[[\vep]] $ be the ring of formal series with rational coefficients and $ \gQ=\QQ((\vep)):=\QQ[[\vep]][1/\vep] $. In classical mathematics $\gZ$ is a local integral henselian ring and $\gQ$ is its field of fractions. Let $\gZ$ have the order relation for which $ \vep $ is an infinitesimal $ >0 $ (i.e.\ $ 0< \vep $ and $ \vep<r $ for any rational $ r>0 $). Let $\gZ$ be the linearly ordered ring thus obtained. The localised $ \gZ[1/\vep] $ with the order relation compatible with that of $\gZ$ will again be denoted~$\gQ$. This is probably the simplest example of a non-archimedean Heyting ordered field.

These classical statements are still valid in constructive mathematics, provided that suitable definitions are used. For example, we will see that, modulo suitable definitions, $\gQ$ is a model of the~\Sa{Co} theory.

From a constructive point of view, there is no sign test on $\gZ$ or on $\gQ$. And it is not immediate to define an order structure corresponding to the intuition given by classical mathematics.

\smallskip Here's how to treat this example constructively.
An element of $\gZ$ is given by a formal series $ \xi=\sum_{j=0}^{+\infty} x_j \vep^j $ with $ x_j\in\QQ $. Any $c\in\QQ$ can be considered as an element of $\gZ$ according to the usual procedure. The coefficient $x_j$ is denoted $\rc_j(\xi)$. Conventionally, $\rc_{j}(\xi)=0$ is given for $j<0$ in $\ZZ$ and for all $\xi\in\gZ$. 

For a series $ \xi=\sum_{j=0}^\infty x_j \vep^j $ in $\gZ$ we define for each $ k\geq 0 $ a \textsl{potential sign in exponent} $k$, denoted $ \kappa_k(\xi)\in\so{-1,0,1} $ as follows, by induction on $k$: 
\begin{itemize}
 
\item $\kappa_{0}(\xi) $ is the sign of $ x_{0} $;
 
\item if $ \kappa_k(\xi)\neq 0 $ then $ \kappa_{k+1}(\xi)=\kappa_k(\xi) $, otherwise $ \kappa_{k+1}(\xi) $ is the sign of $ x_{k+1} $;
 
\item we conventionally pose $ \kappa_{j}(\xi)=0 $ for $ j<0 $ in $ \ZZ $ and for all $ \xi\in\gZ $.
\end{itemize}

\smallskip At least intuitively we have the following result: if $ \kappa_k(\xi)=1 $, then \hbox{$ \xi>0 $}; if $ \kappa_k(\xi)=-1 $, then~\hbox{$ \xi<0 $}; if $ \kappa_k(\xi)=0 $, then the sign of $ \xi $ is a priori unknown.

The set $\gZ$ has the usual ring structure (for formal series) and the equality $ \xi=\zeta $ occurs exactly when the series are identical. This is equivalent to $ \forall k\geq 0\; \kappa_k(\xi-\zeta)=0 $. This ring is the projective limit of the sequence of surjective morphisms $ \pi_{k+1,k}\colon\aqo{\QQ[\vep]}{\vep^{k+1}}\to \aqo{\QQ[\vep]}{\vep^k} $ ($ k\in\N $), via the natural morphisms $ \pi_k\colon\gZ\to \aqo{\QQ[\vep]}{\vep^k} $ obtained by truncation of the series to order $k$. 

\smallskip
The foundations of the constructive theory of residually discrete henselian local rings, including the construction of the henselisation of a residually discrete local ring, are treated in the article~\cite{ALP08}. 

\smallskip Everything necessary for the constructive treatment of the ordered ring~$\gZ$ is now introduced in detail.

\begin{enumerate}
 
\item We define $ \xi>0 $ by \fbox{$ \exists k \;\kappa_k(\xi)=1 $} and $ \xi\geq 0 $ by \fbox{$ \forall k\;\kappa_k(\xi)\geq 0 $}. 
\\
Then we have:
\begin{itemize}
 
\item $ \xi=0 $ if, and only if, $ \xi\geq 0 $ and $ \xi\leq 0 $;
 
\item the rules \Tsbf{OTF} and \OTFx are valid in $\gZ$;
 
\item Heyting's axiom \gui{$ \lnot (x>0)\Rightarrow -x\geq 0 $} is satisfied. 
%
%\item 
\end{itemize}

\item \textsl{Absolute value and map $ \sup $}. We define the map \gui{absolute value} $ \xi\mapsto\abs\xi $ by posing \fbox{$ \rc_k(\abs\xi):=\kappa_k (\xi) \rc_k(\xi) $} for all~$k$. Finally \fbox{$ \xi\vu\zeta:=(\xi+\zeta+\abs{\xi-\zeta})/2 $}.
 
\item We then check that $\gZ$ is a strict $f$-ring (theory \sa{Asr}, Chapter \ref{chap-afr}), in other words, by adding the fact that $\gZ$ is reduced, all the Horn rules valid in the theory \sa{Codsup} are satisfied in $\gZ$ (see Formal Positivstellensatz \ref{Pst2}, Item \ref{i4Pst2}). 
 
\item \textsl{Valuation}.\label{avalnd}\index{valuation} 
\begin{enumerate}
 
\item We define $ \rv(\xi)= k \equidef (\kappa_k(\xi)=\pm1,\; \kappa_{k-1}(\xi)=0) $.
 
\item We define $ \rv(\xi)> k \equidef \kappa_k(\xi)=0 $.
 
\item Intuitively, we read $ \rv(\xi)> k $ as \gui{the valuation of $ \xi $ is $ >k $}. In fact, $ \rv(\xi) $ is not an element of $ \N $ but of a suitable compactification of $ \N $ containing $ +\infty $.\footnote{This is the metric space obtained by taking on $\N$ the metric $d(k,\ell)=\sup(2^{-k},2^{-\ell})$ and completing (this sends $\infty$ to $0 $).} 
 
\item We have precisely the following description for $ \alpha<\beta $;
% equation label {eqcondna1}
\begin{equation} \label {eqcondna1}
\alpha<\beta\;\Longleftrightarrow\; \exists k\;
\formule{
(\kappa_k(\alpha)=-1,\; \kappa_k(\beta)\geq 0) & \vu\\
(\kappa_k(\alpha)\leq 0,\; \kappa_k(\beta)=+1) & \vu\\
(\rv(\alpha)=\rv(\beta)=k, \; \rc_k(\alpha)<\rc_k(\beta))}
\end{equation}
% end-equation

\item We have $ \rv(\xi^2)=2\rv(\xi) $ (equality defined by $ \rv(\xi^2)>2k-1 \Leftrightarrow \rv(\xi^2)>2k \Leftrightarrow \rv(\xi)>k $).
 
\item We deduce that $\gZ$ is a reduced ring (it is thus a constructive model of the \Sa{Asrnz} theory).
 
\item Finally, $\gZ$ is a \textsl{valuation ring} in the following sense: it is a reduced strict $f$-ring in which two strictly positive elements $\alpha$ and $ \beta $ are always comparable for divisibility. In other words, the following rule \tsbf{Val1} is valid.\footnote{We give here a definition for the case of a strict $f$-ring. A more general definition could be given for a residually discrete local ring with a suitable $ \cdot \neq 0 $ predicate.}\index{valuation!ring} 

\Regles
{\Lab{Val1} $ \,\,\alpha>0\vet \beta>0 \vd \Exists \xi\; \alpha\xi=\beta \vou \Exists \xi\; \beta\xi=\alpha $}

A neighbouring rule that is also valid in $\gZ$ is the following.

\Regles
{\Lab{Val2} $ \,\, \beta\geq \alpha\geq 0\vet\beta>0 \vd \Exists \xi\; \alpha=\beta\xi $ 
}

Let's prove these rules. From $ \beta>0 $ we deduce that there is a $k$ such that $ \beta=\vep^k\gamma $ with $ \rv(\gamma)=0 $. We will see in Item \ref{serconv} that $ \gamma\in\gZ\eti $. Therefore $ \vep^k=\gamma^{-1}\beta $. For \tsbf{Val1} we also have a $\ell$ such that $ \alpha=\vep^\ell\delta $ with $ \delta\in\gZ\eti $. Hence the disjunction depending on whether $ \ell\geq k $ or $ k > \ell $. For \tsbf{Val2}, from $ 0\leq \alpha\leq \beta $ we deduce $ \kappa_{k-1}(\alpha)=0 $, so $ \alpha= \vep^k\delta= \beta\gamma^{-1}\delta $ for a $ \delta\in\gZ $.

\noindent Note also that the following implication is satisfied.
\[
\rv(\alpha)=\rv(\beta)=k \Rightarrow \exists \xi\in\gZ\eti\;\; \alpha=\beta\xi
\]
 
\item The \textsl{valuation group} is the ordered group defined as the symmetrisation of the monoid of divisibility $ \gZ\eti/\gZ^+ $ where $ \gZ^+=\sotq{\alpha\in\gZ}{\exists k>0\;\rv(\alpha)=k} $. This valuation group is isomorphic to $ (\ZZ,+,\geq) $, and is generated by (the class of) $ \vep $. In the usual terminology, we say that $\gZ$ is a \textsl{discrete valuation ring} (DVR),
but here the word \gui{discrete} does not have the usual meaning given to it in constructive mathematics.%
\index{valuation!group of ---}\index{discrete valuation} 

\end{enumerate}

\item \label{serconv} \textsl{Convergent series.} Il we have an infinite sequence $ (\xi_n)_{n\in\N} $ in $\gZ$ and if the sequence of $ \rv(\xi_n) $ tends to $ +\infty $, then the infinite sum $ \sum_{n\in\N}\xi_n $ is well-defined. \\
In particular, if $ \rv(\xi)>0 $ the sum $ \zeta=1+\sum_{n\in\N}\xi^n $ is well-defined and we have $ \zeta(1-\xi)=1 $. \\
From this we can deduce the following properties.
\begin{enumerate}
 
\item \label{iaserconv} The ring $\gZ$ is a local ring, whose residual field is discrete, isomorphic to $\QQ$. 
 
\item \label{ibserconv} We have $\gZ\eti=\sotq{\xi\in\gZ}{\kappa_0(\xi)=\pm1}$ and $\Rad(\gZ)=\sotq{\xi\in\gZ}{\kappa_0(\xi)=0}$. 
 
\item \label{icserconv} If $ \rc(\beta)=k $, we write $ \beta=\vep^k\gamma $ with $ \gamma=c_0(1- \alpha) $ and $ \rv(\alpha)>0 $, therefore: $\vep^k=\beta \gamma^{-1}=\beta c_0^{-1}(1+sum_{n\in\N}\alpha^n) $.
 
\item \label{idserconv} The ring $\gZ$ is henselian. Precisely, if $ P\in\gZ[X] $ satisfies the conditions $ \rv(P(0))>0 $ and $ \rv(P'(0))=0 $, there exists a (unique) $ \xi\in\gZ $ such that $ P(\xi)=0 $ and $ \rv(\xi)>0 $. To construct the series $ \xi $, we use Newton's method.
 
\item \label{ieserconv} The ring $\gZ$ is the henselisation of the residually discrete local ring $ (\QQ[\vep])_{1+\vep\QQ[\vep]} $. 
\end{enumerate}
 
\item Finally, we show that the rule \Tsbf{FRAC} is valid in $\gZ$. The hypothesis is given by two elements $ \xi,\zeta\in\gZ $ which verify $ 0\leq \xi\leq \zeta $ and we look for a $ \rho $ such that $ 0\leq \rho\leq \xi $ \hbox{and $ \rho\zeta=\xi^2 $}. According to Lemma \ref{lemAfrnzFRAC}, the uniqueness of~$ \rho $ (if existence) is guaranteed in strict $f$-rings, as in the theory \sa{Co--} (Lemma \ref{lemUniqFRAC}).
\\ 
We note that $ 0\leq \xi\leq \zeta $ implies that $ \rv(\xi)\geq \rv(\zeta) $. We define $ c_k(\rho) $ as follows. 
\begin{itemize}
 
\item If $\kappa_k(\zeta)=0$, then $\kappa_k(\xi)=0$, which forces $\kappa_k(\rho)=0$, so $\rc_k(\rho)=0$.
 
\item If $ \rv(\zeta)=k $, we have $ \zeta= z_k \vep^k(1+\vep \alpha) $ with $ z_k>0 $ and $ \alpha\in\gZ $, and $ \xi=\vep^k\beta $ with $ \beta\in\gZ $. Then we must have the equality \fbox{$ \rho=z_k^{-1}\beta^2\vep^k(1+\vep \alpha)^{-1} $}. As this equality implies $ \kappa_{k-1}(\rho)=0 $, it is compatible with the coefficients of $ \rho $ calculated up to exponent $ k-1 $. This equality makes it possible to define $ \rc_{k+m}(\rho) $ for all $ m>0 $.
 
\item Finally if $ \kappa_{k-1}(\zeta)=\kappa_k(\zeta)=1 $, we look for the exponent $ \ell<k $ such that $ \kappa_{\ell-1}(\zeta)=0 $ and 
$\kappa_\ell(\zeta)=1$, and we are brought back to the previous case via the calculation of the series $\rho$. 
\end{itemize}
So $ \rho $ is well-defined.
\end{enumerate}

Let's summarise the results.
%p
%: Proposition{propQ[[T]]}
\begin{proposition} \label{propQ[[T]]}
The ring $\gZ$ is an  henselian residually discrete local ring and a reduced strict $f$-ring. Moreover it satisfies the rules \tsbf{OTF}, \tsbf{FRAC}, \tsbf{Val1} and \tsbf{Val2}. 
\end{proposition}
%----------- end of proposal ----------------------------- 

The test on $\gZ$ for $ \forall \alpha (\alpha^2>0 \vuu \alpha=0) $ is equivalent to \tsbf{LPO}.

Nor can we prove constructively that $\gZ$ is a ring without zerodivisor: the hypothesis $ \xi\zeta=0 $ is equivalent to $ \abs\xi\vi\abs\zeta=0 $, but the implication $ \abs\xi \vi\abs \zeta=0\Rightarrow ((\abs\xi=0) \vuu(\abs\zeta=0)) $ is equivalent to the principle \tsbf{LPPO}. 

Finally, we cannot prove that every regular element $\geq 0$ is strictly positive: this is equivalent to the Markov principle \tsbf{MP}. The total ring of fractions of $\gZ$ is therefore a somewhat mysterious object, a ring which contains $\gQ$ and which fortunately is of no obvious mathematical interest.

\smallskip
\Note The articles \cite{KL00,KLP03} propose a constructive theory of valuation rings (without order relation) only in the case of integral rings with a decidable divisibility relation. It would be useful to generalise the results (obtained constructively) to valuation rings in the sense given for $\gZ$, and to other similar cases (we need a separation relation on the ring)\footnote{See also the article \cite{ALP08}.} can be used as a basis. \eoe

\smallskip It is easy to deduce the following theorem from Proposition \ref{propQ[[T]]}.
 
%p
%: theorem{propQ((T))}
\begin{theorem} \label{propQ((T))}
The ring $ \gQ=\gZ[1/\vep] $ satisfies all the axioms of the theory \Sa{Co} as well as the ordered Heyting axiom. It is a residually discrete local ring with $ \Rad(\gQ)=0 $ (thus a Heyting field in the terminology of \cite{CACM} or \cite{MRR}). In short, it is a non-archimedean Heyting field,  and a  \textsl{non} discrete ordered field in the sense of \sa{Co} theory.
\end{theorem}
%----------- end of proposition ----------------------------- 

\noindent \textsl{Note.} An element of $ \gQ=\gZ[1/\vep] $ is written $ \vep^{j_0}\alpha $ with $ \alpha\in\gZ $ and $ j_0\in\ZZ $, it can be encoded in the form $ \gamma=(j_0,\alpha) $. The equality $ (j_0,\alpha)=(j'_0,\alpha') $ is defined as follows: 
\begin{itemize}
 
\item if $ j_0\leq j'_0 $, $ \vep^{j'_0-j_0}\alpha=\alpha' $;
 
\item si $ j'_0\leq j_0 $, $ \vep^{j_0-j'_0}\alpha '=\alpha $. 
\end{itemize}
We then define, for $ \gamma=(j_0,\alpha) $:
\begin{itemize}
 
\item $\kappa_k(\gamma)=\kappa_{k-j_0}(\alpha) $ (so $ \kappa_k(\gamma)=0 $ for $ k<j_0 $) ;
 
\item $ \rc_k(\gamma)=\rc_{k-j_0}(\alpha) $ (so $ \rc_k(\gamma)=0 $ for $ k<j_0 $) ;
 
\item $ \rv(\gamma)=\rv(\alpha)-j_0 $ (so $ \rv(\gamma)\geq -j_0 $). \eoe
\end{itemize}
 
%%%%%%%%%%%%%%%%%%%%%%%%%%%%%%%%%%%%%%%%%%%%%%%%%%%%%%%%%%%%%%%%%%%
%%%%%%%%%%%%%%%%%%%%%%%%%%%%%%%%%%%%%%%%%%%%%%%%%%%%%%%%%%%%%%%%%%%

\section{\textsl{Non} discrete real closed fields: position of the problem}
\label{secCRCnondis}

Our dream is to repeat the feat that Artin, Schreier and Tarski achieved for the description of the algebraic properties of $\RR$ through the theory of discrete real closed fields, but in a constructive framework, in intuitionistic logic without \tsbf{LEM}, taking into account the fact that $\RR$ \textsl{is not} discrete, and avoiding the axiom of dependent choice.

%r

%r
%: Remark{remRRtdy} 
\begin{remark} \label{remRRtdy} 
We can consider that our quest is the following: to fix a signature~$ \Sigma $ which allows us to describe as precisely as possible the structure of a \ndrcf, to describe on this signature a dynamical theory which is essentially equivalent to a theory weaker than \Sa{Crcd}, while being the strongest possible among the dynamical theories which admit $\RR$ as a constructive model, without using the axiom of dependent choice. This Holy Grail seems out of reach in absolute terms, as there is no clear criterion for knowing whether a \rdy is constructively satisfied on $\RR$.\footnote{Moreover, the axiom of dependent choice is not allowed in proofs.}
\eoe
\end{remark}
%----------- end remark ---------------------------------- 

%%%%%%%%%%%%%%%%%%%%%%%%%%%%%%%%%%%%%%%%%%%%%%%%%%%%%%%%%%%%%%%%%%%%
%%%%%%%%%%%%%%%%%%%%%%%%%%%%%%%%%%%%%%%%%%%%%%%%%%%%%%%%%%%%%%%%%%%%
\Subsection{The principle of extension by continuity}
%: Subsection{The principle of extension by continuity}

The \gui{completion} property of $\RR$ is expressed naturally in the following form, without interference from the axiom of dependent choice.

%t
%: Theorem{thRRcomplet}
\begin{theorem} \label{thRRcomplet}
If a map $ f\colon \QQ^n\to\RR $ is uniformly continuous on all bounded subsets it extends uniquely into a map $ \wi f\colon \RR^n\to\RR $ uniformly continuous on all bounded subsets. 
\end{theorem}
%----------- fin theorem ---------------- 

This theorem is a theorem of analysis and cannot be expressed directly in the context of a dynamical theory which aims at the algebraic properties of $\RR$, because the property \gui{to be uniformly continuous} is not geometric.

Nevertheless, it is essentially this theorem that guides us in our quest expressed in Remark \ref{remRRtdy}. We will replace the property \gui{be uniformly continuous} by a formulation where uniform continuity is controlled a priori and no longer hides $ \forall\exists $.

Moreover, the only maps that we can envisage inside in a purely algebraic framework are semialgebraic maps.

We must therefore rely on relevant properties of continuous semialgebraic maps, which we develop in the following paragraph. 
 
%%%%%%%%%%%%%%%%%%%%%%%%%%%%%%%%%%%%%%%%%%%%%%%%%%%%%%%%%%%%%%%%%%%%
%:Subsection
\Subsection{Continuous semialgebraic maps}

 First of all we recall that the uniform continuity over any bounded subset of a continuous semialgebraic map $ \gR^n\to\gR $, where $\gR$ is a discrete real closed field, is controlled à la \L{o}jasiewicz precisely as follows.

%t
%: fact{factfsagcLoja}
\begin{lemma} \label{factfsagcLoja}
Let $\gR$ be a discrete real closed field and $ K\subseteq \gR^n $ be a bounded semialgebraic closed subset. 
\begin{enumerate}
 
\item 
Let $ g\colon K\to\gR $ be a continuous semialgebraic map. Then $g$ has a \mcu which is expressed à la {\L}ojasiewicz  as follows (with $ c\in\gR $ and $\ell$ integer $ \geq 1 $)
% equation label {eqfactfsagcLoja1}
\begin{equation} \label {eqfactfsagcLoja1}
\forall \uxi,\uxi'\in K\;\abs{g(\uxi)-g(\uxi')}^\ell
 \leq \abs c \, 
 \,\norm{\uxi-\uxi'}.
\end{equation}
 
\item 
Let $ f\colon \gR^n\to\gR $ be a continuous semialgebraic map. Then $f$ has a \mcu over any bounded subset which is expressed 
à la {\L}ojasiewicz as follows (with $ c\in\gR $ and integers $ k,\ell \geq 1 $) 
% equation label {eqfactfsagcLoja}
\begin{equation} \label {eqfactfsagcLoja}
\forall \uxi,\uxi'\in \gR^n\;\abs{f(\uxi)-f(\uxi')}^\ell
 \leq \abs c \, \big(1+\norm{\uxi}+\norm{\uxi'}\big)^k 
 \,\norm{\uxi-\uxi'}.
\end{equation}
\end{enumerate}
\end{lemma}
%----------- end theorem ----------------------------- 
\begin{proof} This is a consequence of Theorem 2.6.6 of \cite{BCR} which states that on a locally closed semialgebraic set, if there are two continuous semialgebraic maps~$ F $ and~$G$ such that $G$ vanishes at the zeros of $ F $, there exists an exponent $ N $ and a continuous semialgebraic map $h$ such that $ G^N=hF $. In the compact case, $h$ is bounded by a constant; in the general case, $h$ is bounded by a polynomial map. We apply this with $ F(\ux,\ux')=\norm{\ux-\ux'} $ \hbox{and $ G(\ux,\ux')=\abs{f(\ux)-f(\ux')} $}. 
\end{proof}

%:paragraph{Continuous parameterisation of continuous semialgebraic maps
\paragraph{Continuous parametrisation of continuous semialgebraic maps}~

\noindent We now present a parametrisation result saying that, \textsl{from the point of view of continuous semialgebraic maps, everything comes continuously from what happens on the subfield $\RRa$ of algebraic real numbers}. In other words, any continuous semialgebraic map~\hbox{$ \gR^n\to\gR $} is a point with coordinates in $\gR$ of an equicontinuous family defined on $\RRa$.

The idea is in fact a simple generalisation of the following remark. The family of univariate polynomials $ f(x)=ax^2+bx+c $ (family parametrised by $ (a,b,c)\in\RR^3 $) is never just the polynomial in four variables $ g(a,b,c,x)=ax^2+bx+c $ {defined on $\QQ$}, where we take $ (a,b,c) $ as parameters and $x$ as variable, all in $\RR$: so we don't have to worry too much about the \nds character of $\RR$, since everything is defined on $\QQ$. Each individual map $ f(x) $ (depending on parameters taken from $ \RR^3 $) is a real point of a family defined on $\RRa$. This real point comes from the extension by continuity at $ \RR^4 $ of a continuous map $ \RRa^4\to\RRa $. 

%t
%: Theorem{thParamcontFsagc0}
\begin{theorem} \label{thParamcontFsagc0} 
Let $\gR$ be a discrete real closed field and $ f\colon \gR^n\to\gR $ be a continuous semialgebraic map. There exists an integer $ r\geq 0 $, a continuous semialgebraic map $ g\colon \gR^{r+n}\to\gR $ defined on~$\RRa$, and an element $ \uy\in\gR^r $ such that
\[
\forall \xn\in\gR\;\; f(\xn)=g(\yr,\xn).
\] 
\end{theorem}
%----------- fin theorem ----------------------------- 

This result seems to be part of folklore. We give here a proof inspired by the advices of Michel Coste and Marcus Tressl. However, it is not entirely constructive. This would require, for example, a constructive re-reading of Chapter 7 of \cite{BCR}. See Question \ref{questthParamcontFsagc0}.
\begin{proof}
The map $f$ has a  closed graph $F$ which is a semialgebraic union of basic semialgebraic closed sets $F_i=\sotq{(\ux,y)\in\gR^{n+1}}{p_{i}(\ux,y)\geq 0}$. The coefficients of $p_{i}$ are in $\gR $ but can be seen as specialisations of parameters $c_{k}$ ($ k\in\lrbm $) so that we have polynomials $P_{i}(\uc,\ux,y)$ with parameters $c_k$. The inequalities $ P_{i}(\uc,\ux,y)\geq 0 $ define for $i$ a fixed semialgebraic closed set $ G_i\subseteq \gR^{m+n+1} $. The union of $G_i $'s, denoted $G$, is a semialgebraic which is not sufficiently relevant. We add a parameter $c_0$ and we will now restrict the domain of variation of $c_k$ to a \gui{suitable} semialgebraic set $S$. Suitable means that the following formula is satisfied
\[ 
\begin{array}{ccc} 
\hspace{-.5em} \forall \uxi,\uxi'\in \gR^n \,\forall \zeta,\zeta'\in \gR\;\;((\uc,\uxi,\zeta)\in G,\,(\uc,\uxi',\zeta')\in G) \;\Rightarrow\;\abs{\zeta-\zeta'}^\ell \leq \abs {c_0} \, \big(1+\norm{\uxi}+\norm{\uxi'}\big)^k \,\norm{\uxi-\uxi'}.
 \end{array}
\]
where $k$ and $\ell$ are integers for which the map $f$ satisfies these inequalities (for a certain specialisation of $ c_0 $). Note that $ S\subseteq \gR^{m+1} $. It is clear that the semialgebraic set $S$ is defined on $\RRa$. Let $H$ be the semialgebraic subset of $ \gR^{m+n+2} $ formed by the points of $G$ whose $ m+1 $ first coordinates (the parameters) form a point of $S$. The semialgebraic set $H$ is the graph of a map $ h\colon S\times \gR^n\to \gR $, which is seen as a family of maps $ \gR^n\to\gR $ parametrised by $S$. For any point $ s\in S $ the corresponding graph $ H_s $ is that of a continuous semialgebraic map $ f_s:\gR^n\to\gR $ whose \mcu is controlled by $ \abs{c_0} $,~$k$ and~$\ell$. The initial map $f$ corresponds to a point $ s_0\in S $ with coordinates in $\gR$. By means of a cylindrical algebraic decomposition of $\gR^{m+1}$ adapted to $S$, we insert $ s_0 $ in a cell~$\Gamma$ defined on $\RRa$ semialgebraically homeomorphic to $ \gR^q $ for a $ q\geq 0 $ ($ q=0 $ implies that $ s_0 $ has coordinates in $\RRa$). Moreover the homeomorphism is defined on $\RRa$. We then obtain a semialgebraic map $ \varphi\colon \gR^{q+n}\to\gR $ defined on $\RRa$ which has the following properties:
\begin{itemize}
 
\item There is an element $ \gamma=(\gamma_1,\dots,\gamma_q)\in \gR^q $ such that $ \varphi(\gamma,\uxi)=f(\uxi) $ for all $ \uxi\in\gR^n $.
 
\item For any element $ \alpha=(\alpha_1,\dots,\alpha_q)\in \gR^q $, the map $ \uxi\mapsto \varphi(\alpha,\uxi) $ is continuous semialgebraic. 
The map $\varphi$ is locally bounded. 
\end{itemize}
Under these hypotheses, Remark 7.4.9 of \cite{BCR}, assures us that there exists a semialgebraic partition $ A_1\cup\dots\cup A_k $ of the space of parameters $ \gR^q $, defined on $\RRa$, such that the map $\varphi$ restricted to each of the $ A_i\times \gR^n $ is continuous. For example the point $ \gamma=(\gamma_1,\dots,\gamma_q) $ belongs to $ A_1 $. By means of a cylindrical algebraic decomposition of $ \gR^{q} $ adapted to $ A_1 $, we insert $ \gamma $ in a cell $\Delta$ defined on $\RRa$ semialgebraically homeomorphic to $ \gR^r $ for an $ r\geq 0 $. Moreover the homeomorphism is defined on $\RRa$. This provides the continuous semialgebraic map $ g\colon\gR^{r+n}\to \gR $ defined on $\RRa$ requested in the statement.
\end{proof}

\paragraph{A reasonable definition}  

\noindent We therefore propose the following definition in constructive mathematics, made legitimate by  \thref{thParamcontFsagc0}.
%d
%: Definition{defiFSAGC2}
\begin{definota} \label{defiFSAGC2} 
Let $\gR$ be an ordered subfield\footnote{Precisely $\gR$ is a subobject of $\RR$ for the \ndsof defined by the theory~\sA{Co}.
Moreover, for the simple existential rule \Tsbf{IV}, we require that an element of $\gR$ invertible in $\RR$ be invertible in~$\gR$.} 
 of $\RR$ containing the field of algebraic reals $\RRa$ and a map $f\colon \gR^n\to \gR$.%.
\index{continuous semialgebraic map!on an ordered subfield of $\RR$ containing $\RRa$}
\begin{enumerate}
 
\item (elementary case) The map $f$ is semialgebraic continuous if there exists a continuous semialgebraic map $ g\colon \RRa^{n}\to\RRa $ of which $f$ is the extension by continuity. Precisely we must have the following two properties: $f$ is an extension of $g$, and $f$ has the same uniform modulus of continuity as $g$, given in Item \textsl{2} of Lemma \ref{factfsagcLoja}.
 
\item (general case) The map $f$ is semialgebraic continuous if there exist an integer \hbox{$r\geq 0$}, elements $ y_1,\dots,y_r\in \gR $ and a map $ h\colon \gR^{r+n}\to \gR $ which belongs to the previous elementary case such that 
\[
\forall \xn\in\gR\;\; f(\xn)=h(\yr,\xn).
\]
\end{enumerate}
We denote $ \Sac_n(\gR) $ the ring of these maps (it is a reduced strict $f$-ring for the natural order relation). 
\end{definota}
%----------- end definition -------------------------------- 

Some important properties of these function spaces will be established in Section \ref{PropGenFsagcs}.

%%%%%%%%%%%%%%%%%%%%%%%%%%%%%%%%%%%%%%%%%%%%%%%%%%%%%%%%%%%%%%%%%%%%
%:Subsection
\Subsection{Rational dynamical theories for real number algebra} \label{secCrcdesirable}

Recall that the field $\RRa$ of the algebraic reals is a discrete real closed field in the constructive sense.
 
Following on from Remark \ref{remRRtdy} and Definition \ref{defiFSAGC2}, here are the properties we have in mind for a dynamical theory \sa{Crc} of (\nds) real closed  fields, described here in a rather informal way.% %p

%p
%: Proriety{prptaCrcdesirable}
\begin{prpta} \label{prptaCrcdesirable} ~
\begin{enumerate}
 
\item The theory \sa{Crc} is an extension of \Sa{Co}.
 
\item The fields $\RR$, $ \RR_{\tsbf{PR}} $, $ \RR_{\tsbf{Ptime}} $ and $ \RR_{\tsbf{Rec}} $ (cf.\ Example \ref{exacorpsnondiscret}) are constructive models of \sa{Crc} (without using the axiom of dependent choice).
 
\item The theory \sa{Crc} becomes essentially equivalent to \Sa{Crcd} when we add to it the axiom~\tsbf{ED$ _> $}.
 
\item \label{Ra} The continuous semialgebraic maps $ \RRa^n\to\RRa$ are nicely defined in the language of \sa{Crc} and the Horn rules they satisfy are valid in the theory. 
 
\item \label{prolcont} Continuity extension principles (as broad as possible) are satisfied in a suitable form in the dynamical theory. 
 
\item \label{prirec} Gluing principles (the broadest possible) for maps defined on a finite covering by semialgebraic opens, or by semialgebraic closed subsets, are satisfied in a suitable form in the dynamical theory.
\end{enumerate}

\noindent The following points are open to discussion. Item \textsl{\ref{lsalg}} will be dropped if we want to describe more \gui{structure} on $\RR$, for example for an o-minimal structure. Item \textsl{\ref{lprimrec}} will be abandoned for example if we wish to introduce all the reals as constants of the theory: in a general dynamical theory $ \sa T=(\cL,\cA) $, $\cL$ and $ \cA $ are only supposed to be naive sets (à la Bishop). 
\begin{enumerate}\setcounter{enumi}{6}
\item \label{lsalg} All function symbols in \sa{Crc} define on $\RR$ continuous semialgebraic maps of their variables (Definition \ref{defiFSAGC2}). 
 
\item \label{lprimrec} The language of \sa{Crc} is enumerated in a natural way and in this framework the axioms are decidable in a primitive recursive way. 
\end{enumerate}
 
\end{prpta}
%----------- end of statement ----------------------------- 
A slightly crude way of getting a relatively satisfactory answer is to take seriously Item \textsl{\ref{Ra}} above. 
The result is as follows.

%d
%: Definition{defiCrc1}
\begin{definition} \label{defiCrc1} The dynamical theory \SA{Crc1} is obtained from the dynamical theory \Sa{Co} by adding a function symbol and suitable axioms for each continuous semialgebraic map \hbox{$ f\colon \RRa^n\to\RRa $}.%
\index{real closed!non dis@\textsl{non} discrete --- field}
\end{definition}
%----------- end definition -------------------------------- 

\noindent \textsl{Explanation.} More precisely, we proceed as follows. We know from the finiteness theorem (\cite[Theorem~2.7.1]{BCR}) that the graph $G_f=\sotq{(\ux,y)}{\ux\in\gR^n,y=f(\ux)}$ of $f$ (which is assumed to be semialgebraically continuous) is a \sagcs of~$ \RRa^{n+1} $ which can be described as the zero set of  a \textsl{semipolynomial map} $ F\colon \RRa^{n+1}\to \RRa $, \textsl{i. e.} a map written in the form
\[
\sup\nolimits_i\,(\inf\nolimits_{ij}\,p_{ij})  \quad \hbox{ where }p_{ij}\in\RRa[\xn,y]
\]
We can decide whether such a \sagcs $G_f$ described in this way is that of a continuous semialgebraic map, and if so calculate a \mcu 
à la {\L}ojasiewicz. Whenever such a (description of) semipolynomial map defines a continuous semialgebraic map, we introduce a function symbol $ \fsac_F $ with the corresponding axiom \label{Notafsa}

\UneRegle{Df$_{F}$} {$ \vd F(\ux,\fsac_F(\ux))=0 $.}

Furthermore, for an arbitrary term $ t(\ux) $ in the language thus defined, if this term defines a map everywhere zero on $ \RRa^n $ ($n\geq 0$), we introduce the corresponding axiom $\vd t(\ux)= 0$. 

For example, for the map $\fsac_F$ we will have an axiom of continuity which repeats that which is satisfied for the algebraic reals.

\UneRegle{Cont$_F$}{$\vd\abs{\fsac_F(\ux)-\fsac_F(\ux')}^\ell
 \leq  \abs c \, \big(1+\norm{\ux}+\norm{\ux'}\big)^k 
  \,\norm{\ux-\ux'}.$}

\noindent Indeed, an inequality between two terms, $t_1\leq t_2$, is equivalent to making the term $(t_2-t_1)^-$ equal to $0$. 
\eoe 

\smallskip  
Naturally, such a dynamical theory is frustrating at first sight, because it is not very natural and it is undoubtedly difficult to practise from an effective point of view.

However, we shall see that a more natural way, with the addition of few function symbols, which we propose later, leads to the theory \Sa{Corv} essentially identical to~\Sa{Crc1}.

All this is closely related to the theory of real closed rings and its rewriting in concrete form in \cite{Tre2007}.
%%%%%%%%%%%%%%%%%%%%%%%%%%%%%%%%%%%%%%%%%%%%%%%%%%%%%%%%%%%%%%%%%%%%

\section{General properties of continuous semialgebraic maps}\label{PropGenFsagcs}

In this section we give some remarkable properties of the rings $ \Sac_n(\gR) $ (continuous semialgebraic maps $ \gR^n\to\gR $ according to Definition \ref{defiFSAGC2}) over the field $ \gR=\RR $. More generally we can consider an ordered subfield $\gR$ of $\RR$ containing $\RRa$ and in which any continuous semialgebraic map defined on $\RRa$ takes its values in $\gR$ at the points whose coordinates are in $\gR$, for example $ \RR_{\tsbf{PR}} $, $ \RR_{\tsbf{Ptime}} $ or $ \RR_{\tsbf{Rec}} $.

These properties known for discrete real closed fields are extended to $\RR$ because we take the precaution of only taking properties whose formulation does not imply the discrete nature of the order. 

\Subsection{Stability by composition} 

 For example we compose $ f,g\in\Sac_3(\gR) $ with $ h\in\Sac_2(\gR) $. Suppose that 
\begin{itemize}
 
\item $f$ is given in the form $ f(x,y,z)=\wi f(a,b,x,y,z) $, for $ a,b\in\gR $ and $ \wi f\colon \gR^5\to\gR $ extends by continuity $ \ov f\colon \RRa^5\to\RRa $,
 
\item $g$ is given by the form $ g(x,y,z)=\wi g(c,x,y,z) $, for $ c\in\gR $ and $ \wi g\colon \gR^4\to\gR $ extends by continuity $ \ov g\colon \RRa^4\to\RRa $,
 
\item $h$ is given by the form $ h(u,v)=\wi h(d,u,v) $, for $ d\in\gR $ and $ \wi h\colon \gR^3\to\gR $ is a continuous extension of $ \ov h\colon \RRa^3\to\RRa $,
 
\item then $ h\circ(f,g):\gR^3\to\gR $ is of the form $ k(x,y,z)=\wi k(a,b,c,d,x,y,z) $ for $ (a,b,c,d)\in\gR^4 $ and $ \wi k:\gR^7\to\gR $ extends by continuity the map $ \ov k:\RRa^7\to\RRa $ defined by 
\[
\ov k(a,b,c,d,x,y,z)=h(d,f(a,b,x,y,z),g(c,x,y,z)).
\]
\end{itemize}

\Subsection{Stability by upper bound}

 For example we have $ f\in \Sac_4(\gR) $ and we want to show that there is a $ g\in \Sac_2(\gR) $ such that $ g(x,y)=\sup_{z,t\in\ClI{0,1}}f(x,y,z,t) $. If $f$ is given in the form $ f(x,y,z,t)=\wi f(a,b,x,y,z,t) $, for $ a,b\in\gR $, where $ \wi f\colon \gR^6\to\gR $ extends by continuity $ \ov f\colon \gRa^6\to\gRa $, consider the continuous semialgebraic map $ \ov g\colon \RRa^4\to\RRa $ defined by  $ \ov g(a,b,x,y)=\sup_{z,t\in\ClI{0,1}}\ov f(a,b,x,y,z,t) $.\footnote{Here, $ a,b $ are variables.} It extends by continuity into a map $ \wi g\colon \gR^4\to\gR $ and we define $ g(x,y)=\ov g(a,b,x,y) $. The fact that $g$ is indeed the desired lub is due to the fact that the lub on a compact is a continuous function of the parameters and that $\RRa$ is dense in $\gR$. 
 
\Subsection{Finiteness properties}

 In classical mathematics, any continuous semialgebraic map $ f\colon \RR\to\RR $ has a finite table of signs and variations. But this table does not depend continuously on the parameters, for example when $ f(x)=g(\an,x) $ for parameters $ \an\in\RR^n $, where $g$ is the continuity extension of a continuous semialgebraic map $ \ov g\colon \RRa^{n+1}\to\RRa $.

However, when $f$ is a monic polynomial, a constructive approach to the question is to use virtual root maps. For example we can see Items \textsl{\ref{ivrBudan}}, \textsl{\ref{ivrTVI}}, \textsl{\ref{ivrExtrema}} and \textsl{\ref{ivrMinAbs}} of  \thref{thVirtualRoots} as well as Proposition \ref{factCompleteTable}. 

Analogous results should be established in constructive mathematics for arbitrary continuous semialgebraic maps, at least for Proposition \ref{factCompleteTable}, but restricted to maps on the interval $ \ClI{-1,1} $ (on $\RR$ this would not be possible). It may be necessary to use an infinite $ \Vou $. 

\section{Some questions} 

%q
%: Question{questthParamcontFsagc0}
\begin{question} \label{questthParamcontFsagc0} 
Give a complete constructive proof of \thref{thParamcontFsagc0}.
\end{question}
%----------- end{question} ----------------------------- 

%q
%: Question{questeqratcont}
\begin{question} \label{questeqratcont} Show that the Horn rule \Tsbf{FRAC} is not valid in \Sa{Co--}. Similarly, show that the Horn rule \pref{eqratcont2} which is equivalent to the existence of the map \pref{eqratcont} \paref{subsecratcon} is not valid in~\Sa{Co--}. 
\end{question}
%----------- end question ----------------------------- 

%q
%: Question{quest17thRR}
\begin{question} \label{quest17thRR} 
Determine which algebraic properties of $\RR$ allow us to prove the constructively satisfying forms of Positivstellensätze proved for~$\RR$ in~\cite{GL93}. See in particular Question \ref{quest17H}. 
\end{question}
%----------- end question ----------------------------- 

\Subsection{Continuous variations}

Continuous semialgebraic maps $ \RR^n\to\RR $ could have been defined as follows. This was the definition adopted in \cite[Definition 3.3]{LM2017}.

%d
%: Definition{defiAlgAfrnz}
\begin{definition} \label{defiAlgAfrnz}
Let $\gR$ be a commutative ring. A map $ f\colon  \gR^n\to\gR $ is said to be \textsl{algebraic on $ \Rxn=\Rux $} if there is a polynomial \hbox{$ g(\ux,y)=\sum_{k=0}^mg_k(\ux)y^k\in\gR[\ux, y] $}, with at least one of the coefficients of a $ g_k(\ux) $  invertible, such that $ g(\uxi,f(\uxi))=0 $ for all $ (\uxi)\in\gR^n $.
\end{definition}
%----------- end definition -------------------------------- 

%d
%: Definition{defiFSAGC}
\begin{definition}[alternative definition to \ref{defiFSAGC2}] \label{defiFSAGC} 
 Let $\gR$ be a 
subfield of $\RR$.
A map $ f\colon \gR^n\to \gR $ is said to be \textsl{semialgebraic continuous} if, on the one hand, it is algebraic on $ \Rux $ and if, on the other hand, it has a \mcu on everything bounded à la {\L}ojasiewicz, given by an inequality \ref{eqfactfsagcLoja} as in Lemma \ref{factfsagcLoja}. \index{continuous semialgebraic map!on an ordered subfield of $\RR$, alternative definition}
\end{definition}
%----------- end definition -------------------------------- 

%\vspace{-.3em}
This definition is legitimate for subfields of $\RR$ because 
%i
\begin{itemize}
 
\item it is valid in classical mathematics, 
 
\item it has a clear constructive meaning,
  
\item semialgebraic maps which are continuous in the sense of Definition \ref{defiFSAGC2} are also continuous in the sense of Definition \ref{defiFSAGC}.
\end{itemize}

%q
%: Question{questRRfsagc1}
\begin{question} \label{questRRfsagc1} 
 If a map $ f\colon \RR^n\to \RR $ is algebraic on $ \RR[\ux] $ and if it is uniformly continuous on all bounded subsets, does it have a \mcu à la {\L}ojasiewicz, as in Lemma \ref{factfsagcLoja}?
\\
NB: the answer is positive in classical mathematics, but it seems much trickier in constructive mathematics.
\end{question}
%----------- end question -----------------------------

%q
%: Question{questRRfsagc2}
\begin{question} \label{questRRfsagc2} Is a semialgebraic map that is continuous in the sense of Definition \ref{defiFSAGC} also continuous in the sense of Definition \ref{defiFSAGC2}? 
\\
Yes in classical mathematics, but the problem arises in constructive mathematics, and seems very difficult. It may be that, by preferring the Definition \ref{defiFSAGC2} to Definition~\ref{defiFSAGC}, we are in a situation similar to that which led Bishop to define the continuity of a map $ \RR\to\RR $ as meaning uniform continuity on any bounded interval.
\end{question}
%----------- end of question -----------------------------

\newpage \thispagestyle{empty}

%%%%%%%%%%%%%%%%% %%%%%%%%%%%%%
%%%%%%%%%%%%%%%%% CHAPITRE %%%%%%%%%%%%%

\chapter{$f$-rings}\label{chap-afr}
\Today
\minitoc

\section*{Introduction}
\addtocontents{toc}{\skip0.8em}
\addcontentsline{toc}{section}{Introduction}
\rdb

This chapter takes up the problem of \ndsofs from scratch. 

All the theories introduced admit extensions essentially equivalent to the theory \Sa{Co} of \ndsofs (Definition \ref{defiConondiscret}). 

We start (Section \ref{sectrdisad}) with the theory of distributive lattices (a \ndsof is a distributive lattice for its order relation). 

In Section \ref{secgrl} we deal with $\ell$-groups or lattice groups (purely equational theory, valid for addition on the reals).

Then (Section \ref{secfrings}) we move on to $f$-rings ($f$-rings in french litterature), a theory inspired by rings of continuous real maps. 

Section \ref{secArftr} describes dynamical theories in which we add the
predicate $\, \cdot>0 \,$  (strict $f$-rings and variants).

Section \ref{secCOG} proposes a return to the theory \Sa{Co} by confronting it with suitable extensions of the theory of strict $f$-rings. 

\smallskip In this chapter we say \gui{group} for \gui{abelian group}. And the rings are commutative unitary as throughout the memoir.

\section{Distributive lattices}\label{sectrdisad}

References: \cite{CC00,CL05,CLQ2006,Lor1951}

\Subsection{Distributive lattice theory}\label{subsecTrdi}

The theory of lattices \SA{Tr0} with the only sort $\TR$ is a purely equational theory based on the following signature,\footnote{More precisely, we could prefer $ \cdot=_\TR\cdot $, $ \cdot\vi_\TR\cdot $, $ \cdot\vu_\TR\cdot $, $ 0_\TR $ and $ 1_\TR $.}
\Sigt{\TR}{\cdot=\cdot\mathrel{;}\cdot \vi \cdot, \cdot \vu \cdot, 1, 0}
\label{NOTASigTr}

\vspace{-1em}

\noindent Symbols $\vu$ et $\vi$ we use for the binary least upper bound 
and greatest lower bound must not be confused with the symbols~$\vuu$ and $\vii$ for the logical disjunction and conjunction.

 In addition to the axioms of equality we have the following axioms

\TwoRegles{
 \labu $ \vd 0\vi x=0 $ 
 \labu $ \vd x\vi x=x $ 
 \labu $ \vd x\vi y=y\vi x $ 
 \labu $ \vd (x\vi y)\vi z=x\vi (y\vi z) $ 
 \labu $ \vd (x\vi y)\vu x=x $ 
}
{
 \labu $ \vd 1\vu x=1 $ 
 \labu $ \vd x\vu x=x $ 
 \labu $ \vd x=y\vu x $ 
 \labu $ \vd (x\vu y)\vu z=x\vu (y\vu z) $ 
 \labu $ \vd (x\vu y)\vi x=x $ 
} 

\noindent We define $ x\geq y $ as an abbreviation of $ x=x\vu y $. This is an extension of the theory of ordered sets.

\smallskip The following rules are easily proved

\DeuxRegles{
 \labu $\vd 0\leq  x\leq 1$
}
{
 \labu $\,\,1=0\vd  x=0$
 }

\smallskip The theory \SA{Tr} of \textsl{non-trivial lattices} is obtained by adding the collapse axiom

\Regles {\lAb{CL$_{\Tr}$} $\,\,1=_{\Tr} 0\vd \Bot$}

\smallskip 
The theory \SA{Trdi} of \textsl{distributive lattices} is obtained by adding the following distributivity axiom (the dual axiom, obtained by reversing the order, is deduced from this)

\Regles { \labu $ \vd (x\vi y)\vi z=(x\vi z)\vu (y\vi z) $}

The theory \SA{Etob} of \textsl{bounded linearly ordered sets}
is obtained from \sa{Trdi} by adding the axiom of total order 

\Regles {\labu $\vd x=x\vi y\vou y=x\vi y$} 

\smallskip  
The theory \SA{AgB} of \textsl{Boolean algebras} is obtained from \sa{Trdi} by adding the axiom 

\Regles {\labu $\vd \Exists y\,(x\vi y=0\vet x\vu y=1)$}

\smallskip  
The theory \SA{ABo} of \textsl{Boolean rings} is obtained from \sa{Ac} en by adding the axiom 
 
\Regles {\labu $\vd x^2=x$}

It is well known that theories \sa{AgB} and \sa{ABo} are \esid.

\smallskip 
The theory \SA{AgBo} of the Boolean algebra \hbox{$\BB=\so{\zB,\uB}$} is obtained from \sa{Trdi} by adding the axiom 

\Regles {\labu $\vd x=0\vou x=1$}

%%%%%%%%%%%%%%%%%%%%%%%%%%%%%%%%%%%%%%%%%%%%%%%%%%%%%%%%%%%%%%%%%%%%
%:subsection{Ideals and filters in a distributional lattice}---- 
\Subsection{Ideals and filters in a distributive lattice}\label{subsecTrdiIdeFi}
%-----------
 An \textsl{ideal} $ \fb $ of a distributive lattice $ (\gT,\vi,\vu,0,1) $ is a part that satisfies the constraints:
%--- equation eqIdeal --------
\begin{equation}\label{eqIdeal}
\left.
\begin{array}{rcl}
  & & 0 \in \fb  \\
x,y\in \fb & \Longrightarrow  & x\vu y \in \fb  \\
x\in \fb ,\; z\in \gT& \Longrightarrow  & x\vi z \in \fb  \\
\end{array}
\right\}
\end{equation}
%---------------------end equation--------------
We denote $ \gT/(\fb=0) $ the quotient lattice obtained by forcing the elements of $ \fb $ to be zero. Ideals can also be defined as kernels of morphisms.

A  \textsl{principal ideal} is an ideal generated by a single element $a$, and is denoted by $\dar a $. We have $\dar a=\sotq{x\in\gT}{x\leq a}$. The ideal $\dar a$, subject to the laws $\vi$ and $\vu$ of $\gT$ is a distributive lattice in which the maximum element is $a$. The canonical injection $ \dar a\to \gT$ is not a morphism of distributive lattices because the image of $a$ is not equal to $ 1_\gT $. On the other hand, the map $ \gT\rightarrow \dar a,\;x\mapsto x\vi a $ is a surjective morphism, which therefore defines $\dar a$ as a quotient structure $\gT/(a=1)$.

The notion of \textsl{filter} is the opposite notion (obtained by reversing the order relation) to that of ideal.

Let $ \fa $ be an ideal and $ \ff $ a filter of $\gT$. We say that $(\fa,\ff) $ is a \textsl{saturated pair} in $  \gT $ if
\[
(g\in \ff,\; x\vi g \in \fa) \Longrightarrow x\in \fa,
\hbox{ and }(a\in \fa,\; x\vu a \in \ff) \Longrightarrow x\in \ff.
\]
A saturated pair is a pair $ (\varphi^{-1}(0),\varphi^{-1}(1)) $ for a morphism $ \varphi\colon \gT\to\gT' $ of distributive lattices. When $(\fa,\ff) $ is a saturated pair, we have the equivalences
\[
1\in \fa\; \;\Longleftrightarrow\;\; 0\in \ff
\Longleftrightarrow;\ (\fa,\ff)=(\gT ,\gT )
\]

If $A$ and $ B $ are two parts of $\gT$ we note
%--- equation eqvuvi --------
\begin{equation}\label{eqvuvi}
A\vu B=\sotq{ a\vu b}  {a\in A,\,b\in B}  \; \hbox{ and } \; A\vi
B=\sotq{ a\vi b}  {a\in A,\,b\in B}.
\end{equation}
%---------------------end equation--------------
Then the ideal generated by two ideals $ \fa $ and $ \fb $ is equal to $ \fa\vu\fb $. The set of ideals of $\gT$ itself forms a distributive lattice\footnote{In fact it is necessary to introduce a restriction to really obtain a set, so that we have a well-defined procedure for constructing the ideals concerned. For example, we can consider the set of ideals obtained from the principal ideals by certain predefined operations, such as countable meetings and intersections.} for inclusion, with the lower bound of $ \fa $ and $ \fb $ being the ideal $ \fc=\fa\vi\fb $. Thus the operations $\vu$ and $\vi$ defined in (\ref{eqvuvi}) correspond to the sup and inf in the lattice of ideals.

When we consider the lattice of filters, we must pay attention to what the inversion of the order relation produces: $ \ff\cap\ffg=\ff\vu\ffg $ is the inf of the filters $ \ff $ and $ \ffg $, whereas their sup is the lattice generated by $ \ff\cup \ffg $, equal to $ \ff\vi\ffg $.

%:--- SUBsection{Quotients}-----------
\Subsection{Quotients}

A \textsl{quotient distributive lattice $\gT'$ of $\gT$} is given by a binary relation
 $ \preceq $ over $  \gT $ satisfying the following properties:
%--- equation eqPreceq --------
\begin{equation} \label{eqPreceq}
\left.
%--------------------begin array---------------
\begin{array}{rcl}
a\leq b&  \Longrightarrow  & a\preceq b   \\
a\preceq b,\,b\preceq c&  \Longrightarrow  & a\preceq c   \\
a\preceq b,\,a\preceq c&  \Longrightarrow  & a\preceq b\vi c   \\
b\preceq a,\,c\preceq a&  \Longrightarrow  & b\vu c\preceq a
\end{array}
%---------------------end array--------------
\right\}
\end{equation}
%---------------------end equation--------------

%--- Proposal{propIdealFiltre}----
\begin{proposition}\label{propIdealFiltre}
Let $\gT$ be a distributive lattice and $ (J,U) $ be a pair of parts of $\gT$. Consider the quotient $\gT'$ of $\gT$ defined by the relations $ x=0 $ for the $ x\in J $ and $ y=1 $ for the $ y\in U $. Then we have $ a\leq_{\gT'}b $ if, and only if, there exists a finite part $ J_0 $ of $ J $ and a finite part $ U_0 $ of $ U $ such that:
%--- equation eqpropIdealFilter --------
\begin{equation} \label{eqpropIdealFiltre}
a \vi \Vi U_0 \; \leq_\gT\; b \vu \Vu J_0
\end{equation}
%---------------------end equation--------------
We will note $ \gT/(J=0,U=1) $ this quotient lattice $ \gT'.$ 
\end{proposition}
%--- end-proposition----------------------------------------

\Subsection{Representation theorem}

The following constructive theorem gives in classical mathematics the representation theorem which says that any distributive lattice is a sub-distributive-lattice of the Boolean algebra of the parts of a set.
Or the one which says that any distributive lattice is a subdirect product of bounded linearly ordered sets.

%:     Theorem{thTrdiAgb}
\begin{theorem} \label{thTrdiAgb}
The theory \Sa{Trdi} of \trdis, the theory \Sa{Etob} of bounded linearly ordered sets, the theory \Sa{AgB} of Boolean algebras ans the theory \Sa{AgBo} of Boolean rings prove the same \ralgs. 
\end{theorem}
%----------- fin theorem ----------------------------- 
%:2025 preuve réécrite
\begin{proof} We must demonstrate that a fact in a dynamic algebraic structure of type \sa{Trdi}, that is, in a distributive lattice $\gT_0$, is provable in one of the 4 theories if, and only if, it is provable in \sa{Trdi}. It is enough to see this for \sa{Trdi} and \sa{AgBo} because the others are intermediate theories. For a dynamic algebraic structure of type \sa{AgBo}, a proof uses a computation tree governed by the additional axiom $\vd x=0\vou x=1$. At a node $\nu$ of such a tree, the accumulated hypotheses give a distributive lattice $\gT_\nu$ quotient of $\gT_0$ obtained by forcing certain elements to be equal to 0 or 1. If $\nu$ is not a leaf, below the node $\nu$ we have two branches: a certain element $y\in\gT_0$ defines an element $\ov y\in\gT_r$, in the first branch we force $\ov y=0$ and in the other we force $\ov y=1$. If a fact $a=b$ of $\gT_\nu$ is proven in the two new branches, it is also provable in $\gT_\nu$ because the canonical morphism $\gT_\nu\to\gT_\nu/(\ov y=0)\times \gT_\nu/(\ov y=1)$ is injective.
\end{proof}

A corollary is the following proposition.

%p
%:     Proposition{propTriTrdi}
\begin{proposition} \label{propTriTrdi} Let $\xn$ be \elts in a \trdi.\\
We define $\Tri(\ux)=[\Tri_1(\ux), \Tri_2(\ux),\ldots,\Tri_n(\ux)]$, where 
\[\ndsp
\Tri_k(\xn) =\Vi_{I\in \cP_{k,n}}\big(\Vu_{i\in I}x_i\big) \quad(k\in\lrbn)\footnote{We note $\cP_{k,n}$ the set of $k$ \elts subsets of $\cP_{n}:=\so{1,\dots,n}$.}.
\]
\noindent 
We have the following results.
\begin{enumerate}
\item $\Tri_k(\xn) =\Vu_{J\in \cP_{n-k+1,n}}\big(\Vi_{j\in J}x_j\big)$, $(k\in\lrbn)$.
\item $\Tri_1(\ux)\leq \Tri_2(\ux)\leq \cdots \leq \Tri_n(\ux).$
\item If $x_i$'s are comparable, the list $\Tri(\ux)$
is the list of~$x_i$'s in nondecreasing order (it is not necessary that the lattice be discrete).
\end{enumerate} 
\end{proposition}
%----------- fin proposition ----------------------------- 
%
\begin{proof}
This is clear in the case of a bounded linearly ordered set.
\end{proof}

\Subsection{Coverings and gluings of \trdis}

In commutative algebra, if $ \fa $ and $ \fb $ are two ideals of a ring $\gA$ 
we have an \gui{exact sequence} of $\gA$-modules (with $ j $ and $p$ ring homomorphisms)
\[
0\to\gA/(\fa\cap\fb)\vers{j}(\gA/\fa) \times 
(\gA/\fb)\vers{p}\gA/(\fa+\fb)\to 0
\]
which can be read in everyday language: the system of congruences $ x\equiv a\mod\fa $, $ x\equiv b\mod\fb $ has a solution if, and only if, $ a\equiv b\mod\fa+\fb $ and in this case the solution is unique modulo $ \fa\cap\fb $. It is remarkable that this Chinese remainder theorem generalises to a system of congruences if, and only if, the ring is \textsl{arithmetic} (\cite[Theorem XII-1.6]{CACM}), i.e.\ if the lattice of ideals is distributive (the \gui{contemporary} Chinese remainder theorem concerns the special case of a family of two-by-two comaximal ideals, and it works without any hypothesis on the base ring).

Other epimorphisms in the category of commutative rings are localisations. And there is a covering principle analogous to the Chinese remainder theorem for localisations, which is extremely fruitful (the local-global principle).

 In the same way we can recover a distributive lattice from a finite number of its quotients, if the information they contain is \gui{sufficient}. This can be seen either as a procedure for covering (going from local to global), or as a version of the Chinese remainder theorem for distributive lattices. Let's take a closer look. 
%--- Definition{defRecouvTD}---------
\begin{definition}
\label{defRecouvTD}
Let $\gT$ be a distributive lattice, $ (\fa_i)_{i=1,\ldots n} $ (resp. $ (\ff_i)_{i=1,\ldots n} $) a finite family of ideals (resp. filters) of $\gT$. We say that the ideals $ \fa_i $ \textsl{cover $\gT$} if $ \Vi_i\fa_i=\so{0} $. Similarly we say that the filters $ \ff_i $ \textsl{cover $\gT$} if $ \Vi_i\ff_i=\so{1} $.
\end{definition}
%--- end-definition----------------------------

%--- theorem{propRecouvTD}---------------
\begin{theorem}[coverings of a \trdi]
\label{propRecouvTD}~
\begin{enumerate}
\item Let $\gT$ be a distributive lattice, $ (\fa_i)_{i=1,\ldots n} $ be a finite family of principal ideals of~$\gT$ and $ \fa=\Vi_i\fa_i $.

\begin{enumerate}
 
\item If $ (x_i) $ is a family of elements of $\gT$ such that for each $ i,j $ we have $ x_i\equiv x_j\;\mod\;(\fa_i\vu\fa_j=0) $, then there exists a unique $x$ modulo $\fa=0$ satisfying: $x\equiv x_i\;\mod\;(\fa_i=0)$ for each $i$.
 
\item Let us note $ \gT_i=\gT/(\fa_i=0) $, $ \gT_{ij}=\gT_{ji}=\gT/(\fa_i\vu\fa_j=0) $, $ \pi_i:\gT\to\gT_i $ and $ \pi_{ij}:\gT_i\to\gT_{ij} $ the canonical projections. If $(\fa_i)_{i=1}^n$ covers $\gT$, then $(\gT,(\pi_i)_{i=1,\ldots n})$ is the projective limit of the diagram $ ((\gT_i)_{1\leq i\leq n},(\gT_{ij})_{1\leq i<j\leq n}; (\pi_{ij})_{1\leq i\neq j\leq n})$ (see figure below). 
\end{enumerate}

\item Now let $ (\ff_i)_{i=1,\ldots n} $ be a finite family of principal filters and $\ff=\Vu_i\ff_i$.  
\begin{enumerate}
\item If $(x_i)$ is a family of elements of $\gT$ such that for each $i,j$ we have $x_i\equiv x_j\;\mod\;(\ff_i\vu\ff_j=1) $, then there exists a unique $x$ modulo $\ff=1$ satisfying: $x\equiv x_i\;\mod\;(\ff_i=1)$ fo each $i$.
\item Let $ \gT_i=\gT/(\ff_i=1) $, $ \gT_{ij}=\gT_{ji}=\gT/(\ff_i\cup\ff_j=1) $, $ \pi_i: \gT_i\to\gT_i $ and $ \pi_{ij}:\gT_i\to\gT_{ij} $ the canonical projections. If $\ff_i$'s cover $\gT$, $ (\gT,(\pi_i)_{i=1,\ldots n}) $ is the projective limit of the diagram $ ((\gT_i)_{1\leq i\leq n},(\gT_{ij})_{1\leq i<j\leq n};(\pi_{ij})_{1\leq i\neq j\leq n}) $ (see figure below).
\end{enumerate}
\end{enumerate}
%-----------------end enum------------------
\end{theorem}
%--- end-theorem-----------------------------------------
 
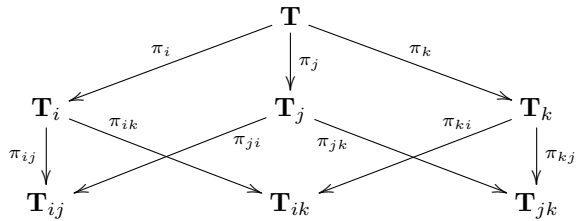
{\hspace*{10em}{
\xymatrix @R=2em @C=7em{
 & \gT \ar[rd]^{\pi _{k}}\ar[d]^{\pi _{j}}\ar[ld]_{\pi _{i}}\\
 \gT _i\ar[d]_{\pi _{ij}}\ar@/-0.75cm/[dr]^{\pi _{ik}} &
 \gT _j\ar@/-1cm/[dl]^{\pi _{ji}}\ar@/-1cm/[dr]_{\pi _{jk}} &
 \gT _k\ar@/-0.75cm/[dl]_{\pi _{ki}}\ar[d]^{\pi _{kj}} &
\\
 \gT _{ij} & 
 \gT _{ik} & 
 \gT _{jk} 
}
}}

\medskip There is also a procedure for gluing properly speaking (voir \cite[proposition 1.2.7]{CLQ2006}%, traduction anglaise \cite{}
).
%:2025 traductoion anglaise ?

%--- theorem{thRecolTD}-------
\begin{theorem}[gluing \trdis]
\label{thRecolTD}
Suppose given a linearly ordered finite set~$I$ and in the category of distributive lattices a diagram

\snic{\big((\gT_i)_{i\in I},(\gT_{ij})_{i<j\in I},(\gT_{ijk})_{i<j<k\in I};
(\pi_{ij})_{i\neq j},(\pi_{ijk})_{i< j, j\neq k\neq i}\big)}

\noindent 
as in the figure below, as well as a family of elements 

%$$
\snic
{(s_{ij})_{i\neq j\in I}\in \prod\nolimits_{i\neq j\in I}\gT_{i}}
%$$

\noindent satisfying the following conditions:
\begin{itemize}
\item the diagram is commutative ($\pi_{ijk}\circ \pi_{ij}=\pi_{ikj}\circ \pi_{ik}$ for all $i$, $j$, $k$ distinct), 
\item for $i\neq j$, $\pi_{ij}$ is a morphism of passage to the quotient by the ideal $\dar s_{ij}$,
\item for $i$, $j$, $k$ distinct, $\pi_{ij}(s_{ik})=\pi_{ji}(s_{jk})$ and  $\pi_{ijk}$ is a morphism of passage to the quotient by the ideal $\dar\pi_{ij}(s_{ik})$.   
\end{itemize}

\smallskip 
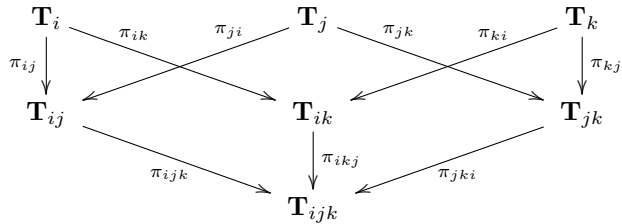
{\hspace*{10em}
\xymatrix @R=2em @C=7em{
 \gT_i\ar[d]_{\pi _{ij}}\ar@/-0.75cm/[dr]^{\pi _{ik}} &
     \gT_j\ar@/-.8cm/[dl]_{\pi _{ji}}\ar@/-.8cm/[dr]^{\pi _{jk}} &
        \gT_k\ar@/-0.75cm/[dl]_{\pi _{ki}}\ar[d]^{\pi _{kj}} &
\\
 ~\gT_{ij}~ \ar[rd]_{\pi _{ijk}} & 
    ~\gT_{ik}~  \ar[d]^{\pi _{ikj}} & 
      ~\gT_{jk}~  \ar[ld]^{\pi _{jki}} 
\\
   &  ~\gT_{ijk}~ 
\\
}
}

\smallskip \noindent Then if $\big(\gT\,;\,(\pi_i)_{i\in I}\big)$ is the projective limit of the diagram, the $\pi_i:\gT\to \gT_i$ form a covering of $\gT$ by quotients by principal ideals, and the diagram is isomorphic to that obtained in \thref{propRecouvTD}.
More precisely, there exist $s_i$'s $\in\gT$ such that each $\pi_i$ is a  morphism of passage to the quotient by the ideal $\dar s_i$ and $\pi_i(s_j)=s_{ij}$ for all $i\neq j$.

\noindent The analogous result is valid for quotients by principal filters.
\end{theorem}
%--- end-proposition----------------------------------------

The previous covering and gluing propositions have analogous versions for the category of modules over a commutative ring and for that of lattice groups of Krull dimension $\leq 1$. On the other hand, in the category of commutative rings, only the covering procedure (by localisation in comaximal elements) is valid, and it was necessary for Grothendieck to invent the category of schemes to have a gluing procedure: a quasi-compact quasi-separated scheme is nothing other than the gluing of a finite number of affine schemes along quasi-compact open sets, which corresponds to an  \gui{abstract}  gluing of commutative rings in the style of \thref{thRecolTD} for localisation epimorphisms (see \cite{CLS2009}).

\section{$\ell$-groups}\label{secgrl}

\Subsection{Ordered and linearly ordered abelian groups}

The \SA{Gao} theory of ordered (abelian) groups is defined as follows. There is only one sort, called $\Gao$.\index{abelian group!ordered ---}

\vspace{-.5em}
\Sigt{\Gao}{\cdot=0,\cdot\geq 0\mathrel{;}\cdot+\cdot,-\,\cdot,0}
\label{NOTASigGao}

\noindent {\bf Abbreviations}

\smallskip \noindent \textsl{Predicates}

\vspace{-1em} \DeuxCols{
\begin{itemize}
\itbu $x = y $ means $ x - y = 0$
\itbu $x \leq y $ means $ y\geq x$
\end{itemize}
}
{\begin{itemize}
\itbu $x \geq  y $ means $ x - y \geq  0$
\end{itemize}
} 

\medskip \noindent {\bf Axioms}

\smallskip \noindent {\it Direct rules of abelian groups}

\DeuxRegles{
\Lab{ga0} $\vd 0=0$
\Lab{ga2} $\,\,x=0\vd -x=0 $
}
{
\Lab{ga1} $\,\, x=0\vet y=0\vd x+y=0$
}

\smallskip \noindent NB. The rules \tsbf{ga0}, \tsbf{ga1} and \tsbf{ga2} concern the purely equational theory \SA{Ga} of abelian groups. We must then replace, in the explanation given on \paref{Ac-comments} for commutative rings (example \ref{exaAc}), the computational machinery of commutative rings $\ZZxn$ (freely generated by the $x_i$'s) by that of abelian groups $\ZZ^{\{\xn\}}$ (freely generated by the $x_i$'s).

\smallskip \noindent {\it Règles pour la relation d'ordre}

\DeuxRegles{
\Lab{gao0} $\vd x \geq 0$
\Lab{Gao} $\,\, x\geq 0\vet x\leq 0\vd x=0$
}
{
\Lab{gao1} $\,\, x\geq 0\vet y\geq 0\vd x+y\geq 0$
}

The reflexivity and transitivity of $\cdot\geq \cdot$, and the rule \tsbf{gao2} $\,\,x\geq y\vd x+z\geq y+z\,$ esult directly from the definitions. The \tsbf{Gao} rule gives the antisymmetry of $\cdot\geq \cdot$.

The \SA{Gto} theory of \textsl{linearly ordered (abelian) groups} is obtained from  \Sa{Gao} by adding as an axiom the dynamic rule \rdy \Tsbf{OT}.

\Subsection{Definition of the purely equational theory of \grls}

The theory \SA{Grl} of \textsl{$\ell$-groups} (or \textsl{reticulated groups}, or \textsl{lattice groups}) is defined as follows. There is only one sort, named $ \GRL $.\index{l-group@$\ell$-group}\index{l-group@$\ell$-group!lattice group}

\vspace{-.5em}
\Sigt{\GRL}{\cdot=0\mathrel{;}\cdot+\cdot,-\,\cdot,\cdot\vu\cdot,0}
\label{NOTASigGrl}
 
\smallskip\noindent {\bf Abbreviations}

\smallskip\noindent \textsl{Function symbols}
 
\vspace{-1em} \TwoCols{
\begin{itemize}
\labu $ x\vi y $ means $ - (-x\vu -y) $ 
\labu $ \abs{x} $ means $ x \vu -x $ 
\end{itemize}
}
{
\begin{itemize}
\labu $ {x}^+ $ means $ x \vu 0 $ 
\labu $ {x}^- $ means $ -x \vu 0 $ 
\end{itemize}
}

\medskip\noindent \textsl{Predicates}

\vspace{-1em} \TwoCols{
\begin{itemize}
\labu $ x = y $ means $ x - y = 0 $ 
\labu $ x \perp y $ means $ \abs x \vi \abs y =0 $ 
\end{itemize}
}
{\begin{itemize}
\labu $ x \geq y $ means $ x \vu y = x $ 
\labu $ x \leq y $ means $ y\geq x $ 
\end{itemize}
} 

\medskip\noindent {\bf Axioms}

\smallskip After the axioms of abelian groups, we add the following axioms.

\smallskip\noindent {\it Rule for the compatibility of $\vu$ with equality}

\Regles{
\lAb{sup$_=$} $\,\, x=0\vet y=0\vd (u+x)\vu (v+y)= u\vu v$ \label{Axsup=}
}
 
\smallskip\noindent {\it Equality rules}

\smallskip 
The following identities express the fact that $\vu$ defines an unbounded sup-lattice as well as the compatibility of $\vu$ with $ + $ (the fact that translations are  morphisms for the law $\vu$).\rdb

\DeuxRegles{
\Lab{sdt1} $\vd x\vu x=x $
\Lab{sdt2} $\vd x\vu y=y\vu x$
}
{
\Lab{sdt3} $\vd (x\vu y)\vu z=x\vu (y \vu z)$
\Lab{grl} $\vd x+(y\vu z)=(x+y)\;\vu\;(x+z)$
}

\smallskip We thus obtain an  (abelian) $\ell$-group with all the related geometric rules (see \cite{BKW}, \cite[Chapter 2]{Zaa97}, [Bourbaki, Algebra, Chapter 6], and \cite[Section XI-2, results 2.2-2.5, 2.11 and 2.12]{CACM}). 

\smallskip A theory essentially identical to \sa{Grl} is obtained from \Sa{Gao}  by adding the axioms which say that two elements $x$ and $y$ of the ordered group always have an upper bound, which we denote by $x\vu y$.\footnote{The existence of the upper bound is indeed a unique existence due to antisymmetry.}.\label{GrlGao}

\DeuxRegles{
\Lab{sup1} $ \vd x\vu y\,\geq x $
\Lab{Sup} $ \,\, z\geq x\vet z\geq y\vd z\geq x\vu y$
}
{
\Lab{sup2} $ \vd x\vu y\,\geq y $
}

%%%%%%%%%%%%%%%%%%%%%%%%%%%%%%%%%%%%%%%%%%%%%%%%%%%%%%%%%%%%%%%%%%%%
\Subsection{Some rules derived in \sa{Grl}}
%: Subsection{Some rules derived in \sa{Grl}}

\DeuxRegles{
\Lab{grl1} $\cramped{\vd x\vu(y_1\vi y_2)= (x\vu y_1)\vi (x\vu y_2)}$
\Lab{grl2} $\cramped{\vd x\vi(y_1\vu y_2)= (x\vi y_1)\vu (x\vi y_2)}$
\Lab{grl3} $\vd (x\vi y)\vu x=x$
\Lab{grl4} $\vd (x\vu y)\vi x=x$
\Lab{grl5} $\vd (x\vi y)+(x\vu y) = x+y$
\Lab{grl6} $\vd x=x^+-x^-$
\Lab{grl7} $\vd \abs{x}=x^++x^-=x^+\vu x^-$
}
{
\laB{Sup} $ \,\, z\geq x\vet z\geq y\vd z\geq x\vu y$
\laB{Gao} $\,\, x\geq 0\vet x\leq 0 \vd x=0$

\Lab{Grl1} $\cramped{\,\, y\geq 0\vet z\geq 0\vet y\perp z \vd (y-z)^+=y}$
\Lab{Grl2} $\,\, x\leq z \vd (x\vi y)\vu z=x\vi(y\vu z) $
\lAb{Grl3$_n$} \label{AxGrl3} $\,\, nx\geq0\vd x\geq0\quad (n\in\NN, \,n>1) $
\lAb{Grl4$_n$} $\,\, nx\geq \Vi_{k=1}^n(ky+(n-k)x) \vd x\geq y $ \label{AxGrl4}
}

\vspace{-.7em}
\Regles{\Lab{Gauss} $\,\, x\geq 0\vet y\geq 0\vet z\geq 0\vet x\perp y\vet x\leq y+z \vd x\leq z $
}

%%%%%%%%%%%%%%%%%%%%%%%%%%%%%%%%%%%%%%%%%
% subsubsection{Quotient structures}
\Subsection{Quotient structures}\label{subsecgrlsolide}
%: Subsection{Quotient structures}

 The kernels of morphisms of ordered (abelian) groups are the \textsl{convex subgroups}: a subgroup $H$ is convex if, and only if, it verifies the property
\[
 (x\in H,\,y\in G,\, 0\leq y\leq x)\Rightarrow y\in H.
\]
If a subgroup is convex, the order relation \gui{pass to quotient} in $ G/H $ \index{convex!subgroup (in an ordered group).}.

The kernels of $\ell$-group morphisms are the \textsl{solid subgroups}. A subgroup is solid if, and only if, it is a convex  $\ell$-subgroup, or convex and stable by $ x\mapsto \abs x $ (\cite[theorem 2.2.1]{BKW})\index{solid!subgroup (in an $\ell$-group)}.

\smallskip The solid subgroup generated by an element $a$ is $ \cC(a):=\sotq{x\,}{\exists n\in\N,\abs x\leq n\abs a} $.

Solid finitely generated subgroups are all principal: $ \cC(\abs a + \abs b)=\cC(\abs a \vu \abs b) $ is the solid subgroup generated by $a$ and $b$. The principal solid subgroups form a distributive lattice (with $ \cC(a)\cap\c(b)=\cC(\abs a \vi\abs b) $), except that a maximum element is missing, which can be added formally. 

The Krull dimension of this distributive lattice is called the \textsl{dimension, or height, of the $\ell$-group}. This is a constructive definition equivalent to the classical definition in classical mathematics, but does not require the existence of prime convex subgroups (see \cite[section XIII-6]{CACM} for the Krull dimension of distributive lattices). In the case of linearly ordered groups, this corresponds to the \textsl{rank} of the group.

%%%%%%%%%%%%%%%%%%%%%%%%%%%%%%%%%%%%%%%%%
\Subsection{Representation theorem} 
%: Subsection{Representation theorem}
In classical mathematics, any lattice group is a subgroup of a product of linearly ordered groups.

The method of proof explained in \cite[Principle XI-2.10]{CACM} gives a constructive equivalent of this property: to prove a concrete fact in a lattice group, we can always act as if we were in the presence of a product of linearly ordered groups.

\smallskip In fact, we have a {better} (more formal) formulation in the language of dynamical theories: \textsl{both dynamical theories (with and without the axiom of total order) prove the same Horn rules}. Let's look at this in more detail.

%
%: Definition{defiGto}
\begin{definition} \label{defiGtosup}
The dynamical theory \SA{Gtosup} of \textsl{linearly ordered groups with sup}\footnote{Or totally ordered groups with sup.} is the dynamical theory of $ \ell$-groups to which we add the \rdy \tsbf{OT} saying that the order is total.

\Regles{
\laB{OT} {$\vd x\geq 0\;\vou\;x\leq 0 $}} 
\end{definition}
%--------- end definition -----------------

\vspace{-.5em}
Note that compared with the usual theory of linearly ordered groups \Sa{Gto} we have introduced into the signature the law $ \cdot\vu\cdot $ which is well-defined. The \Sa{Gtosup} theory is essentially identical to the \Sa{Gto} theory.

%: theorem thfairecommesi0
\begin{pstf} [for $\ell$-groups] \label{thfairecommesi0} \index{Positivstellensatz!formal --- !for $\ell$-groups}  ~\\
The dynamical theories \Sa{Grl} and \Sa{Gtosup} prove the same Horn rules. 
\end{pstf}
\begin{proof}
The reader can refer to the proof of Formal Positivstellensatz \ref{thfairecommesi}, and change the very little that needs to be changed.
\end{proof}

 For example, the reader can easily prove the rules \Tsbf{Grl2} and \Grlqn\ using  \pstref{thfairecommesi0}, which would otherwise be much less simple.

A corollary in classical mathematics of  \pstref{thfairecommesi0} is the  the following theorem (as a special case of  \thref{thcolsimralg}).
%: Corollary{corthfairecommesi0}
\begin{corollaryc}[representation theorem]
 \label{corthfairecommesi0} 
\emph{See \cite[Lorenzen, 1939]{Lor1939}, and the developments \cite{Lor1950,Lor1953} commented in \cite{CLN2019b}.}
Any $\ell$-group $G$ is a subdirect product of linearly ordered groups\footnote{Any $\ell$-group $G$ is a substructure of a product of linearly ordered quotient groups of $G$. 
In other words, there is a lattice subgroup of a product of linearly ordered groups which, as a lattice group, is isomorphic to the original lattice group.} quotients of $G$.
\end{corollaryc}
%--------- fin corollary ------------------------------- 

%r
%: Remark{remLLM01}
\begin{remark} \label{remLLM01} 
 The theory of algorithmic complexity in the space of continuous real maps on the interval $\ClI{0,1} $ makes natural use of the divisible  $\ell$-group structure (2-divisibility is sufficient). This space of functions is seen essentially as a Riesz space, and the multiplication of maps is relegated to the background. See for example \cite[definition 3.2.1]{LLM2001}. Note also that in this theory  formulas are replaced by circuits (a short circuit can encode a very long formula). In this case we are in analysis rather than abstract algebra. 
 \eoe\end{remark}
%----------- fin remark ---------------------------------- 

An example of applying the formal Positivstellensatz for $\ell$-groups is given in \cite[Fact XI-2.12]{CACM} which we reproduce below.

%%: Fact{factGpRtcl}
\begin{fact}[other identities in $\ell$-groups]\label{factGpRtcl} ~\\
Let $x$, $y$, $ x' $, $ y' $, $z$, $ t\in G $, $ n\in\N $, $ x_1 $, \dots, $ x_n\in G $.

\vspace{-.10em} 
\begin{enumerate}%\itemsep=1pt
 
\item \label{i1factGpRtcl} $ x+y =\abs{x-y} +2(x\vi y) $ 
 
\item $ (x\vi y)^+=x^+\vi y^+ $, $ (x\vi y)^-=x^-\vu y^- $, $ (x\vu y)^+=x^+\vu y^+ $, $ (x\vu y)^-=x^-\vi y^- $.
 
\item $ 2(x \vi y)^+ \leq (x+y)^+ \leq x^++y^+ $.
 
\item $ \abs{x+y} \leq \abs{x}+\abs{y};:\; $ 
 $ \abs{x}+\abs{y}=\abs{x+y}+2(x^+\vi y^-) +2( x^-\vi y^+) $.
 
\item $ \abs{x-y} \abs{x}+\abs{y};:\; $ 
 $ \abs{x}+\abs{y}=\abs{x-y}+2(x^+\vi y^+) +2( x^-\vi y^-) $.
 
\item $ \abs{x+y}\vu\abs{x-y}=\abs{x}+\abs{y} $.
 
\item $ \abs{x+y}\vi\abs{x-y}=\abS{\abs{x}-\abs{y}} $.

\item $ \abs{x-y}=(x\vu y)-(x\vi y) $.
 
\item $ \abs{(x\vu z)-(y\vu z)}+\abs{(x\vi z)-(y\vi z)}= \abs{x-y}. $ 
 
\item $ \abs{x^+ - y^+} + \abs{x^- - y^-} = \abs{x-y} $.
 
\item \label{i11factGpRtcl} $ x\leq z \;\Longrightarrow\; (x\vi y)\vu z= x\vi (y\vu z) $.
 
\item $ x+y=z+t \;\Longrightarrow\; x+y=(x\vu z)+(y\vi t) $.
 
\item \label{i13factGpRtcl} $ n\, x\geq \Vi_{k=1}^n (k y+(n-k)x) \,\Longrightarrow\,x\geq y $.
 
\item $ \Vu_{i=1}^nx_i = \sum_{k=1}^n(-1)^{k-1}
 \sum_{I\in \cP_{k,n}}\Vi_{i\in I}x_i\big) $\,\, ($n=2$, $x\vi y + x\vu y =x+y$).
 
\item $ x\perp y\,\Longleftrightarrow\, \abs{x+y}=\abs{x-y}
\,\Longleftrightarrow\, \abs{x+y}=\abs{x}\vu \abs{y} $.
 
\item $ x\perp y\,\Longleftrightarrow\, \abs{x+y}=\abs{x}
+\abs{y}=\abs{x}\vu \abs{y} $.
 
\item \label{i15bisfactGpRtcl} $ (x'\perp y,\,x'\perp y,\,x'\perp y',\,x'\perp y',\,x+y=x'+y') \,\Longrightarrow\, (x=x', \, y=y') $.
 
\end{enumerate}
Suppose $u$, $v$, $ w\in G^+ $.

\begin{enumerate}\setcounter{enumi}{17}\itemsep=1pt
 
\item \label{i18factGpRtcl} $ u\perp v\,\Longleftrightarrow\, u+v=\abs{u-v} $.
 
\item $ (u+v)\vi w \leq (u\vi w)+(v\vi w) $.
 
\item $ (x+y)\vu w \leq (x\vu w)+(y\vu w) $.
 
\item $ v\perp w \,\Longrightarrow\,(u+v)\vi w = u\vi w $.
 
\item $ u\perp v \,\Longrightarrow\,(u+v)\vi w = (u\vi w)+(v\vi w) $.
\end{enumerate}
\end{fact}
\begin{proof}
All this is more or less immediate in a linearly ordered group, reasoning case by case. We conclude with the formal Positivstellensatz.
\end{proof}

\Subsection{The \grl generated by an ordered group} 
 
We have a natural morphism between \talgs, from $\Sa{Gao}$ to $\Sa{Grl}$, because we have defined $x\geq 0$ as a simple abbreviation in the theory $\sa{Grl}$.

If $(G,\geq)$ is an ordered group, as $\sa{Grl}$ is an \talg, the dynamic algebraic structure $\sa{Grl}(G)$ defines a usual algebraic structure $H$ of \grl.
This is the \grl \textsl{generated by} $G$ in the usual sense.

A precise description of the usual algebraic structure defined by $\sa{Grl}(G)$ has been given by Lorenzen. See \cite{Lor1951,CLN2019b}. Let us recall these results.

\begin{theorem}\label{thmgogrlfree} \emph{(\cite[Theorem 4.15]{CLN2019b})}
If $G$ is an ordered group, we can construct an \grl $H$ with a morphism 
\(\varphi\colon G\to H\) such that 
\[
0\leq_H\varphi(a) \iff \exists n\in \NN\etl, \;0\leq_G\varphi(na)
\]
More \prmt, $H$ is the \grl freely generated by $G$  
(in the sense of the left adjoint functor to the forgetting functor).
\\
Here is the construction. We consider the distributive lattice $\gT$ generated by the \entrel on $G$ defined as follows.
\begin{enumerate}
\item $\varphi(a)\vdash \varphi(b) $ for $(a,b)$'s such that $a\leq b$,
\item 
$\varphi(a_1),\dots,\varphi(a_k)\vdash \varphi(b_1),\dots,\varphi(b_\ell)
$
\ssi there exist integers \(n_1,\dots,n_k,m_1,\dots,\alb m_\ell\geq 0\) such that 
\begin{itemize}
\item \(n_1+\cdots+n_k=m_1+\dots+m_\ell\geq1\) and
\item  \(n_1a_1+\dots+n_ka_k\leq_Gm_1b_1+\dots+m_\ell b_\ell\).
\end{itemize}
\end{enumerate}
On this \trdi $\gT$, the group law of~$G$ is uniquely extended, which defines the \grl~$H$.\\ 
So $\varphi(a_1),\dots,\varphi(a_k)\vdash \varphi(b_1),\dots,\varphi(b_\ell)$
means \fbox{$\varphi(a_1)\vi\dots\vi\varphi(a_k)\leq_H\varphi(b_1)\vu\dots\vu\varphi(b_\ell)$}.
\end{theorem}

Thus the morphism $\varphi$ is injective \ssi for all $a\in G$ and all integers $n\geq 1$, the inequality $na\geq 0$ implies $a\geq 0$.
This  \cof \tho implies in \clama the Lorenzen-Clifford-Dieudonné \tho according to which an abelian group can be equipped with an ordered group structure \ssi the equality $na= 0$ implies $a= 0$.

Furthermore Lorenzen highlighted a crucial \rsim valid in the \sa{Grl} theory which is the \textsl{regularity} (see \cite{Lor1953} and \cite[Section 2]{CLN2019b})

The result of Lorenzen is that an ordered group $G$ satisfies the \prt of regularity \ssi it identifies with an ordered subgroup of an \grl 
(or, what amounts to the same thing, of the \grl it generates).

\Regles{\Lab{Reg}
$\,\,x+a\geq 0\vet y+b\geq 0\vd x+b\geq 0\vou y+a\geq 0 $
}

\medskip \noindent {\bf Another way to define the \grl generated by an ordered group}

\smallskip The \grl structure on an abelian group $H$ can be defined from the unary law $x\mt\abs{x}$. Indeed 
$2x^+=x+\abs x$, $a\vu b=b+(a-b)^+$, and $x\geq 0$ is equivalent to $x=\abs x$.

We now indicate which axioms must satisfy this unary law so that it defines a \grl law $a\vu b$.

First of all, the group must be \textsl{$2$-divisible}, that is, verify the following axioms.\label{grl-abs}

\DeuxRegles
{
\Lab{2div1} $\,\,2y=0\vd y=0$
}    
{
\Lab{2div2} $\vd \Exists x\;2y=x $
}

In which case, we can define unambiguously $\frac 1 2\,x$.
Then, the law $x\mt \abs x$ must satisfy the following axioms. 

\DeuxRegles
{
\lAb{abs$_=$}$\,\, x=0\vd \abs {y+x}=\abs{y}$
\Lab{abs1}$\vd \abs x=\abs{-x}$
\Lab{abs3}$\vd\abs x+ x= \abS{\abs{x}+ x}$
}    
{
\Lab{abs0}$\vd \abs{0}=0$
\Lab{abs2}$\vd\abs x + \abs y= \abS { \abs x + \abs y}$
\Lab{Abs1}$\,\,z\geq x\vet z\geq -x\vd z\geq\abs x$
}

When these conditions are met, we use the following abbreviations.
\begin{itemize}
\item $x\geq 0$ means $\abs x=x$.
\item $a\geq b$ means $a-b\geq 0$.
\item $x^+:=\frac 1 2\,(x+\abs x) $. 
\item $a\vu b:=b+(a-b)^+$. 
\end{itemize}

\smallskip The validity of the following rules must be verified

\DeuxRegles{
\laB{gao0} $\vd 0 \geq 0$
\laB{Gao} $\,\, x\geq 0\vet x\leq 0\vd x=0$
\laB{sup1} $ \vd x\vu y\,\geq x $
\laB{Sup} $ \,\, z\geq x\vet z\geq y\vd z\geq x\vu y$
}
{
\laB{gao1} $\,\, x\geq 0\vet y\geq 0\vd x+y\geq 0$
\item[ ]
\laB{sup2} $\vd x\vu y\,\geq y $
\laB{grl} $\vd x+(y\vu z)=(x+y)\;\vu\;(x+z)$
}

\begin{enumerate}
\item A small calculation gives  \fbox{$a\vu b = \frac{a+b}2+\frac{\abs{a-b}}2 $}. Using \tsbf{abs1} this shows that $a\vu b =b\vu a$.
\item We get $2\abs u= \abs{2u}$ from \tsbf{abs2} taking $u=x=y$. We have $2\abs u=2u\vd \abs u=u$ from \tsbf{2div1}. Thus  \fbox{$\abs {2u}=2u\vd \abs u=u$}. Thus $\abs {z}=z\vd \abs {\frac z 2}=\frac z 2$ and 
\fbox{$\abs {\frac z 2}=\frac{\abs z} 2 $}. So we get $2u\geq 0\vd u\geq 0$, and \fbox{$2a\geq 2b\vd a\geq b$}.
\item Note that the compatibility with the equality of the predicates $x\geq 0$, $a\geq b$ and the laws $x\mt x^+$ and $(x,y)\mt x \vu y$ is an automatic consequence of the fact that the law $\abs x$ respects the equality.
\item \tsbf{gao0} means $\abs0=0$ which is \tsbf{abs0}.
\item \tsbf{gao1} means $\abs x= x\vet\abs y =y\vd \abs{x+y}=x+y$. this results from \tsbf{abs2} since the hypotheses imply we may replace $\abs x$ and $\abs y$ with  $x$ et~$y$.  
\item \tsbf{Gao} means $\abs x= x\vet\abs x =-x\vd x=0$. Using  \tsbf{abs1} and symmetry and transitivity of \egt,
we deduce $x=-x$. We conclude with~\tsbf{2div1}. 
\item \tsbf{sup1} means $y+(x-y)+-x\geq 0$, \ie $(x-y)^+-(x-y)\geq 0$. It is sufficient to show $a^+-a\geq 0$, \ie $\frac 1 2\,(a+\abs a)-a\geq 0$, \ie $\frac 1 2\,(\abs a -a)\geq 0$. We concluded with \tsbf{abs3} et $\abs {\frac z 2}=\frac{\abs z} 2 $. 
\item \tsbf{sup2} is deduced from \tsbf{sup1}  because $a\vu b=b\vu a$.  
\item  \tsbf{Sup} means $\abs{z-x}=z-x,\vet \abs{z-y}=z-y\vd z\geq \abs{x+y}$. \\
Using $\abs u = \abs {-u}$ the hypotheses give $2z-(x+y)=\abs{z-x}+\abs{y-z}$. As $\abs a \geq a$ we get with \tsbf{gao1} $2z-(x+y)\geq y-x$. Symmetrically $2z-(x+y)\geq x-y$. Thus  \tsbf{Abs1} gives $2z-(x+y)\geq \abs{y-x}$, \ie $2z\geq 2 (y\vu x)$ and we conclude with $2a\geq 2b\vd a\geq b$.
\item \tsbf{grl} means $2x+(y+z)+\abs{y-z}=(x+y)+(x+z)+\abs{(x+y)-(x+z)}$, which is clear.  
\end{enumerate}

%%%%%%%%%%%%%%%%%%%%%%%%%%%%%%%%%%%%%%%%%%%%%%%%%%%%%%%%%%%%%%%%%%%%
\section{$f$-rings}\label{secfrings}

References: \cite{BKW}, \cite[Section V-4]{Joh1986}, \cite{BP56,DM1995,Mad10}. 

\smallskip The french terminology of \cite{BKW} is \gui{$f$-anneau} or \gui{anneau de fonctions}. They study the case of non-commutative and non-unitary rings, for which the results are more subtle and delicate than those given here for the commutative unitary case. The terminology \gui{$f$-anneau} can be found in Bourbaki's exercises (Algebra, Chapter VI, Exercises, \S2, exercise 5). 

\smallskip The  (commutative unitary) $f$-rings are defined by a purely equational theory.\index{f-ring@$f$-ring} The axioms are those of commutative rings, those of $\ell$-groups for addition, and finally the equality rule \Tsbf{afr} which expresses a form of compatibility of $\vu$ with multiplication.\footnote{Compared to the theory \sa{Grl}, 
we added the law \gui{$ \cdot\times \cdot $} and the rules \tsbf{ac2} and \tsbf{afr}. In addition, the computational machinery that reduces any term on the variables $\Xn$ to its canonical writing in the free abelian group $\ZZ^{\{\Xn\}}$ has been replaced with the computational machinery that reduces any element of $\ZZXn$ to a normal form.}
Here's everything in detail.

%%%%%%%%%%%%%%%%%%%%%%%%%%%%%%%%%%%%%%%%%%%%%%%%%%%%%%%%%%%%%%%%%%%%
\Subsectio{Definition of the purely equational theory \sa{Afr}}{Purely equational theory of $f$-rings}

The \SA{Afr} theory is defined as follows.

\vspace{-.5em}
\Sigt{\AfR}{\cdot=0\mathrel{;}\cdot+\cdot, \cdot\times\cdot,\cdot\vu\cdot,-\,\cdot,0,1}
\label{NOTASigAfr}

\noindent {\bf Abbreviations} (as for $\ell$-groups)

\smallskip\noindent \textsl{Function symbols}

\vspace{-1em} \TwoCols{
\begin{itemize}
\labu $ x\vi y $ means $ - (-x\vu -y) $ 
\labu $ \abs{x} $ means $ x \vu -x $ 
\end{itemize}
}
{
\begin{itemize}
\labu $ {x}^+ $ means $ x \vu 0 $ 
\labu $ {x}^- $ means $ -x \vu 0 $ 
\end{itemize}
}

\medskip\noindent \textsl{Predicates}

\vspace{-1em} \TwoCols{
\begin{itemize}
\labu $ x = y $ means $ x - y = 0 $ 
\labu $ x \perp y $ means $ \abs x \vi \abs y =0 $ 
\end{itemize}
}
{\begin{itemize}
\labu $ x \geq y $ means $ x \vu y = x $ 
\labu $ x \leq y $ means $ y\geq x $ 
\end{itemize}
} 

\medskip\noindent {\bf Axioms}

\smallskip\noindent \textsl{Rules of commutative rings}

\TwoRegles{
\lab{ga0} $ \vd 0=0 $ 
\lab{ac2} $ \,\, x=0\vd xy=0 $ 
}
{
\lab{ga1} $ \,\, x=0\vet y=0\vd x+y=0 $ 
}

\smallskip\noindent \textsl{Rule for compatibility of $\vu$ with equality}

\Regles{
\lAb{sup$_=$} $\,\, x=0\vet y=0\vd (u+x)\vu (v+y)= u\vu v$ \label{Axsup=}
}

\smallskip\noindent \textsl{Equality rules}

\TwoRegles{
\lab{sdt1} $ \vd x\vu x=x $ 
\lab{sdt2} $ \vd x=y\vu x $ 
\lab{sdt3} $ \vd (x y)\vu z=x (y z) $ 
}
{
\lab{grl} $ \vd x+(y\vu z)=(x+y)\vu(x+z) $ 
\Lab{afr} $ \vd x^+\, (y\vu z)=(x^+\, y)\vu(x^+\, z) $ 
}

%:     note \label{noteAfrp}
\begin{notE} \label{noteAfrp}
If we choose the signature

\Sigt{\AfR'}{\cdot=0,\cdot\geq 0\mathrel{;}\cdot+\cdot, \cdot\times\cdot,\cdot\vu\cdot,-\,\cdot,0,1}
\label{NOTASigAfr'}

\vspace{-.8em}
\noindent we give the three rules \Tsbf{sup1}, \Tsbf{sup2} et \Tsbf{Sup} (\paref{Axsup1})
in order to connect $\cdot\geq 0$ and $\cdot\vu\cdot$. We name this theory  \SA{Afr'}, it is \esid to \sa{Afr}. \eoe
\end{notE}

%%%%%%%%%%%%%%%%%%%%%%%%%%%%%%%%%%%%%%%%%%%%%%%%%%%%%%%%%%%%%%%%%%%%
\Subsection{Note on $\ell$-rings}
%: Subsection{Note on $\ell $-rings}

The theory \SA{Arl} of $\ell$-rings (or lattice rings)  is defined by replacing the rule \Tsbf{afr} by the rules \Tsbf{ao1} and \Tsbf{ao2} of ordered rings, valid in~\Sa{Afr}.\index{l-ring@$\ell$-ring}
\label{ao1}

\DeuxRegles{
\Lab{ao1} $\vd \,x^2\geq 0$
}
{
\Lab{ao2} $\,\,x\geq 0\vet y\geq 0\vd xy\geq 0$ \phantom{$,a^2\geq 0$}
}

%: Lemma{lemaoafr}
\begin{lemma} \label{lemaoafr}
In the theory of $\ell$-rings, the following rules are all equivalent.

\DeuxRegles{
\laB{afr}  $\vd  a^+\, (b\vu c)=(a^+\, b)\vu(a^+\, c)$
\Lab{afr'}  $\vd  a^+\, (b\vi c)=(a^+\, b)\vi(a^+\, c)$
\Lab{afr0} $\vd b^-\vi  a^+b^+= 0  $
\Lab{afr1} $\vd a^+\ a^-= 0  $
\Lab{afr2} $\vd \abs a\,\abs b= \abs{ab}  $
\Lab{afr3a} $\vd  (ab)^+=a^+b^++a^-b^-$
\Lab{afr3b} $\vd  (ab)^-=a^+b^-+a^-b^+$
\Lab{afr4} $\vd  c^+\abs a= \abs{c^+a}$
\Lab{afr5} $\vd  (a\vi b)(a\vu b)=ab $
}
{
\Lab{Afr}  $\,\,a \geq 0\vd  a(b\vu c)=ab\vu ac\phantom{a^+}$
\Lab{Afr'}  $\,\,a \geq 0\vd  a(b\vi c)=ab\vi ac\phantom{a^+}$
\Lab{Afr0}  $\,\,b\vi c = 0\vet a \geq 0 \vd  b\vi ac = 0\phantom{a^+}$
\Lab{Afr1} $\,\,   a \vi b =0\vd  ab=0  $ 
\Lab{afr6a} $\vd  a^2=(a^+)^2+(a^-)^2$
\Lab{afr6b} $\vd  a^2= {\abs a}^2$
\Lab{sup} $ \vd  \big((x\vu y)- x\big)\,\big((x\vu y)- y\big)=0$
\Lab{Afr2} $\,\, b\perp c\vd  ab\perp ac$
%\Lab{Afr3} $\,\,b\vi c=0\vet x\geq 0\vd  b\vi xc=0$
}
\end{lemma}
%--------- end lemma ----------------------------------- 

\vspace{-.5em}
In other words, each of these rules can be used to define $f$-rings by adding it to the theory \Sa{Arl}. The fact that \tsbf{afr} implies \tsbf{ao1}, \tsbf{ao2} and the rules indicated in Lemma \ref{lemaoafr} results from Formal Positivstellensatz \ref{thfairecommesi}. 

In the case of a non-unitary $f$-ring the rule \tsbf{afr0} is stronger than the others (see \cite{BKW}, proposition 9.1.10\footnote{The book deals more generally with ordered rings which are not necessarily commutative or unitary. The condition \tsbf{afr0} must then be split to take account of the non-commutativity.}).

%%%%%%%%%%%%%%%%%%%%%%%%%%%%%%%%%%%%%%%%%%%%%%%%%%%%%%%%%%%%%%%%%%%%
\Subsectio{Some derived rules in the theory \sa{Afr}}{Some derived rules}
%: Subsection{Some derived rules in \sa{Afr}}

 In addition to the rules derived for $\ell$-groups and those indicated in Lemma \ref{lemaoafr}, here are some very useful classical rules in which multiplication is involved.

\DeuxRegles{
\laB{Ato1} $\,\,   b\geq 0 \vet  ab=1\vd  a\geq 0 \vphantom{\abs{a}^2} $}
{
\laB{Ato2} $\,\, c\geq 0\vet a(a^2+c)\geq 0\vd  a^3\geq 0 \vphantom{\big)}$
}

\vspace{-1em}
\Regles{
\Lab{afr7} {$\vd  ab^+ = (ab \vi (a^2 +1)b) \vu (-(a^2 +1)b\vi 0)$}  
}

\smallskip \rem The rule \tsbf{afr7} is used to demonstrate the possibility of writing terms in a simplified form in a  free $f$-ring: see Lemma \ref{lemAfrReecriture}. 
\eoe

%%%%%%%%%%%%%%%%%%%%%%%%%%%%%%%%%%%%%%%%%
\Subsection{Quotient structures}
%: Subsection{Quotient structures}

\paragraph{Solid ideals (or $\ell$-ideals)}~

\smallskip By definition, the kernels of $f$-ring morphisms are called \textsl{solid ideals or $\ell$-ideals}.\index{solid!ideal (in an $f$-ring)}\index{l-ideal@$\ell$-ideal} 

An ideal is solid if, and only if, it is solid as a subgroup. 

The solid ideal generated by an element $a$ is 
\[
\cI(a):=\sotq{x\,}{\,\exists y,\,\abs x\leq \abs {ya}}.
\] 
We have $ \cI(a)=\cI(\abs a) $ and $ \cI(a)\cap\cI(b)=\cI(\abs a \vi \abs b) $. Finally, the $\ell$-ideal generated by $ \an $ is 
\[
\cI(\an)=\cI(\abs {a_1}+\dots+\abs {a_n})= \cI(\abs {a_1} \vu \cdots \vu \abs {a_n}). 
\]

\paragraph{Irreducible $\ell$-ideals}~

\smallskip We say that \textsl{a solid ideal $ \fa  $ of an $f$-ring $\gA$ is \textsl{irreducible}} if the quotient $f$-ring is linearly ordered. In other words, for any $ x\in\gA $, $ x^+\in \fa  $ or $ x^-\in \fa  $.\index{irreducible!solid ideal (in an $f$-ring)} 

By Lemma \ref{lemAfrsdz}, every prime solid ideal is irreducible.

Moreover, a convex prime ideal (as an additive subgroup) $\fp$ is solid: we must see that it is stable by $\vu$. If $a,b\in \fp$ we have $(a-b)^+$ or $(a-b)^-\in\fp$. And the identities $ b+(a-b)^+=a\vu b=a+(a-b)^- $ are valid in $\ell$-groups (and a fortiori in $f$-rings) because they are valid in linearly ordered groups (Formal Positivstellensatz~\ref{thfairecommesi0}).

%%%%%%%%%%%%%%%%%%%%%%%%%%%%%%%%%%%%%%%%%
\Subsectio{Formal Positivstellensatz and representation theorem for $f$-rings}{Formal Positivstellensatz and representation theorem} 

Recall that the dynamical theory of linearly ordered rings with sup is the dynamical theory of linearly ordered rings to which we add a function symbol $\cdot\vu\cdot$ which must satisfy the following Horn rules.%.

\TwoRegles{
\laB{sup1} $ \vd x\vu y\,\geq x $ 
\laB{Sup} $ \,\, z\geq x\vet z\geq y \vd z\geq x\vu y $ 
}
{
\laB{sup2} $ \vd x\vu y\,\geq y $ 
}

\smallskip\noindent We can also see \sa{Atosup} as the theory of $f$-rings to which we add as an axiom the \rdy \Tsbf{OT} (saying that the order is total).

\Regles{
\laB{OT} $ \vd x\geq 0\;\vou\;x\leq 0 $ 
}

Given the unique existence of the lub in a linearly ordered ring, the theories \Sa{Ato} and \Sa{Atosup} are essentially identical. In particular, they prove the same \rdys (when formulated without using $\vu$). 
%\end{lemma}
%--------- end lemma ----------------------------------- 

The theorem for $f$-rings analogous to Positivstellensatz \ref{thfairecommesi0} is as follows. It is a result of the same type as Item \textsl{2} of Positivstellensatz \ref{Pst1bis}.
%: theorem thfairecommesi

\begin{pstf}[for $f$-rings] 
\index{Positivstellensatz!formal --- !for $f$-rings} \label{thfairecommesi}~\\
We consider \sads on the signature\\ 
\centerline{\sIgt{\AfR'}{\cdot=0,\cdot\geq 0\mathrel{;}\cdot+\cdot, \cdot\times\cdot,\cdot\vu\cdot,-\,\cdot,0,1} }\\
The theories \sa{Afr}, \sa{Ato} and \sa{Atosup} prove the same Horn rules.
\end{pstf}

\begin{proof} First \Sa{Ato} and \Sa{Atosup} are essentially identical. Let us consider a Horn rule proved in the dynamical theory $ \sa{Atosup}\! $. We can assume without loss of generality that the conclusion of the rule is an equality $ t=0 $ for a suitable term $t$. In the corresponding calculation, in the presence of a term $u$, we are authorised by \Tsbf{OT} to open two branches. One where $ u\geq 0 $, the other where $ u\leq 0 $. At each node of the dynamic proof, we are in fact working in an $f$-ring defined by generators and relations: the generators are given in the presentation and in the hypotheses of the Horn rule to be proved; the same applies to the relations, with the addition of those which we have added, in the branch we are in, to the branches which precede the node. Suppose that at a given moment, for two terms $a$ and $b$, we have opened a branch where $a\geq b$ and another where $a\leq b$. Let's put $c=b-a$. In the first branch we have added the hypothesis $c^-=0 $, in the second the hypothesis $c^+=0$. If in each of the branches we can prove $t=0$, this means that in the $f$-ring corresponding to the node in question, we have on the one hand $t\in \cI(c^-)$, and on the other hand $ t\in \cI(c^+)$. Now in an $f$-ring we have $\cI(c^+)\cap \cI(c^-)= \cI(c^+\vi c^-)=\so 0$.
\end{proof}

Let's use \pstref{thfairecommesi} for proving Horn Rules \Tsbf{Ato1} and \Tsbf{Ato2}.

\TwoRegles{
\lab{Ato1} $ \,\, y\geq 0 \vet xy=1\vd x\geq 0\phantom{x^3} $ 
}
{
\lab{Ato2} $ \,\, c\geq 0\vet x(x^2+c)\geq 0\vd x^3\geq 0 $ 
}

\smallskip In both cases, we open two branches, one where $ x\geq 0 $, and the result is clear, the other \hbox{where $ x\leq 0 $}. For \Tsbf{Ato1} we deduce that $ 1\leq 0 $, then $ 1=0 $, then $ x=0 $. For \Tsbf{Ato2} we deduce that $ x^3\geq -xc\geq 0 $.

\smallskip Similarly, we prove \Tsbf{afr7} by examining separately the cases \gui{$ b\geq 0 $}, \gui{$ b\leq 0,\,a\geq 0 $} \hbox{and \gui{$ b\leq 0,\,a\leq 0 $}}. As a consequence of Formal \pstref{thfairecommesi} we obtain in classical mathematics the following representation theorem (as a special case of \thref{thcolsimralg}).

%: Corollary{corthfairecommesi}
\begin{corollaryc}[representation theorem] \label{corthfairecommesi}
Any $f$-ring $\gA$ is a subproduct of linearly ordered rings quotients of $\gA$. 
\end{corollaryc}
%--------- fin corollary ------------------------------- 

 The following theorem is of the same type as Item \textsl{1} of  Positivstellensatz~\ref{Pst1bis}. This result can be seen as a second form of Formal Positivstellensatz \ref{thfairecommesi} for $f$-rings. 
%We say that \textsl{a dynamic algebraic structure of type \sa{Afr} collapses when the rule $\vd 1=0$ is valid}. 
%: Theorem{thColsimafr}
\begin{theorem}[simultaneous collapse for \afrs]~\\
We consider \sads on the signature\\ 
\centerline{\sIgt{\AfR'}{\cdot=0,\cdot\geq 0\mathrel{;}\cdot+\cdot, \cdot\times\cdot,\cdot\vu\cdot,-\,\cdot,0,1} \label{thColsimafr}}\\  
The theories \Sa{Afr}, \Sa{Crcdsup} and all intermediate theories collapse simultaneously. 
\end{theorem}
%--------- end theorem -----------------------------------
%
\begin{proof} 
The theories \Sa{Afr} and \Sa{Atosup} collapse simultaneously according to Positivstellensatz~\ref{thfairecommesi}.
\\
The theories \Sa{Ato} and \Sa{Crcd} collapse simultaneously according to Item \textsl{1} of  Positivstellensatz \ref{Pst1bis}.
\\
Finally, the theories \Sa{Ato} and \Sa{Crcd} are essentially identical to the theories \Sa{Atosup} and \Sa{Crcdsup} respectively.
\end{proof}
%
%%%%%%%%%%%%%%%%%%%%%%%%%%%%%%%%%%%%%%%%%%%%%%%%%%%%%%%%%%%%%%%%%%%%
\Subsection{Localisations of $f$-rings}\label{locafr}
%: Subsection{Localisation of an $f$-ring}.

\paragraph{Generalities}~

Consider a monoid $S$ in an $f$-ring and construct the solution of the universal problem (in the category of $f$-rings) consisting in inverting the elements of $S$. 

To do this, we need only consider the usual localised ring $ S^{-1}\gA $ and define the law $\vu$ correctly. Since inverting $s$ or inverting $ s^2 $ amounts to the same thing, we can consider only fractions with denominator $\geq 0$. We then define 

\snic{ \dsp\frac{\,a\,}{s}\vu\frac{\,b\,}{t}:=\frac{\,at\vu bs\,}{st} \qquad (s,t\geq 0).}
 
\Note We have no choice, because since $s,t\geq 0$, we must have $st\,(\frac{a}{s}\vu\frac{b}{t})=st\,\frac{a}{s}\vu st\,\frac{b}{t}=at\vu bs$ in~$S^{-1}\gA$. It remains to be seen that the law is well-defined and that it continues to satisfy the required axioms. For example, let's check that it is well-defined. Suppose that $ \frac{a_1}{s_1}=\frac{a_2}{s_2} $, i.e.\ that $ a_1s_2s_3=a_2s_1s_3 $ for an~$ s_3\geq 0 $ in~$S$. Then we can easily check that the two elements $ \frac{a_i}{s_i}\vu\frac{b}{t} $ given by the definition above are equal in $ S^{-1}\gA $. This is the same calculation that was used to justify addition in $ S^{-1}\gA $ when we were young.\footnote{When we fell over in admiration of Claude Chevalley who dared to invert zerodivisors, and nothing awful resulted, quite the contrary.} Just replace $+$ by $\vu$, with the precaution of having denominators $\geq 0$.\eoe

%l
%:     Lemma{lemAfrqlg}
\begin{lemma} \label{lemAfrqlg}
An $f$-ring $\gA$ can always be considered as immersed in a $\QQ$-$f$ -algebra. 
\end{lemma}
%----------- fin lemma ----------------------------------- 
%
\begin{proof} Indeed, according to  \Grltn, the  \gui{integers} $n.1_\gA$ are regular and therefore $\gA$ injects itself into the $\QQ$-algebra $\QQ\otimes_\ZZ \gA$ which is an $f$-ring as a localisation of $\gA$.\footnote{This is true even if $\gA$ is trivial: the only case where the $\QQ$-algebra in question does not contain $\QQ$ as a subring.}
\end{proof}

%%%%%%%%%%%%%%%%%%%%%%%%%%%%%%%%%%%%%%%%%%%%%%%%%%%%%%%%%%%%%%%%%%%%
\paragraph{Gluing $f$-ring structures}
%: paragraph{Gluing $f$-rings}

\smallskip We assume that an \afr structure is given on a commutative ring $\gA$, locally. If these locally defined structures coincide two by two on the open sets where they are defined in common, then we can glue them together into an \afr structure defined on all $\gA$. This is similar to the covering principle for distributive lattices (Principle \ref{propRecouvTD}) and is stated in a similar precise manner as follows.

%: Plgc {plcc.frings}
\begin{plcc}[concrete gluing of $f$-rings structure over a ring]\label{plcc.frings}  
Let $ S_1 $, $ \dots $, $ S_n $ be comaximal monoids of a ring $\gA$. Let $ \gA_i $ denote $ \gA_{S_i} $, $ \gA_{ij} $ denote $ \gA_{S_iS_j} $, and assume that an $f$-ring structure with a $ \vu\!_i $ law is given on each $ \gA_i $. It is further assumed that the images in $ \gA_{ij} $ of the laws $ \vu\!_i $ and $ \vu\!_j $ coincide. Then there exists a unique $f$-ring structure on $\gA$ which induces by localisation in each $ S_i $ the structure defined on $ \gA_i $. This $f$-ring is identified with the projective limit of the diagram 
\[
\big(( \gA_i)_{i\in\lrbn},( \gA_{ij})_{i<j\in\lrbn};(\alpha_{ij})_{i\neq j\in\lrbn}\big),
\]
where $ \alpha_{ij} $ are localisation morphisms, in the category of $f$-rings.
\end{plcc}
%--- end-plgc-----------------------------------------
 
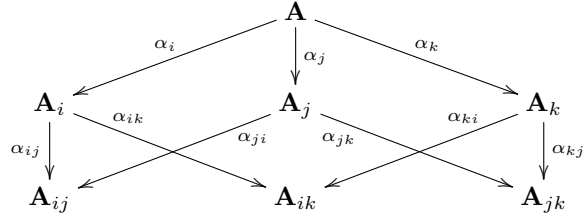
{\hspace*{10em}{
\xymatrix @R=2em @C=7em{
 & \gA \ar[rd]^{\alpha _{k}}\ar[d]^{\alpha _{j}}\ar[ld]_{\alpha _{i}}\\
 \gA _i\ar[d]_{\alpha _{ij}}\ar@/-0.75cm/[dr]^{\alpha _{ik}} &
 \gA _j\ar@/-1cm/[dl]^{\alpha _{ji}}\ar@/-1cm/[dr]_{\alpha _{jk}} &
 \gA _k\ar@/-0.75cm/[dl]_{\alpha _{ki}}\ar[d]^{\alpha _{kj}} &
\\
 \gA _{ij} & 
 \gA _{ik} & 
 \gA _{jk} 
}
}}
\begin{proof}
The ring $\gA$ is the limit of the projective system formed by $ \gA_i $ and $ \gA_{ij} $ in the category of commutative rings, and therefore also in the category of sets. It follows that there is a unique law $\vu$ on $\gA$ which gives the $ \vu_i $ on the $ \gA_i $ by the canonical maps $ \gA\to\gA_i $. It remains to check that it satisfies the axioms of the $\vu$ law for an $f$-ring. This follows from the fact that these axioms are given by equalities between terms, and from the fact that the natural map $ \varphi\colon \gA\to \prod_i\gA_i $, on the one hand preserves the laws of the $f$-ring structure, and on the other hand is injective.
\end{proof}
%

%paragraph{Localisation in a prime filter} 

\paragraph{Real schemes} 

%: Remark{remplcc.frings}
\begin{remark} \label{remplcc.frings} 
A corollary of the gluing Principle \ref{plcc.frings} is that the notion of a  Grothendieck $f$-scheme is well-defined. An $f$-scheme seems to be the most natural definition for the notion of a real scheme. Indeed, it allows nilpotents and therefore a good theory of multiplicities in real schemes.

\noindent More precisely, one could define a \gui{real scheme} as a scheme obtained by gluing together a finite number of affine schemes defined by $f$-rings which are finitely presented \RRlgs, or more generally finitely presented \Rlgs for a given ordered field with virtual roots $\gR$.

\noindent But this apparently remains unexplored.
\eoe\end{remark}
%----------- end remark ---------------------------------- 

%%%%%%%%%%%%%%%%%%%%%%%%%%%%%%%%%%%%%%%%%%%%%%%%%%%%%%%%%%%%%%%%%%%%
\Subsection{Rewriting terms in $f$-rings}
%: Subsection{Rewriting terms in $f$-rings}

Reference: \cite{Del86}. 

\smallskip Contrary to the theory of commutative rings in which the terms are rewritten in a unique normal form, we do not have such a satisfactory result for $f$-rings. We do, however, have a simplified form, similar to the conjunctive normal form in distributive lattices.

%: Lemma{lemAfrReecriture}
\begin{lemma} \label{lemAfrReecriture}
Let $\gA$ be an  $f$-ring and $t$ be a term written over indeterminates $ \xn $ and constants in $\gA$. This term can be rewritten as 
\[
\sup\nolimits_{i\in I}\big(\inf\nolimits_{j\in J_i}(f_{i,j}(\ux))\big)
\] 
for a suitable finite family of polynomials $ f_{i,j}\in\AXn $.
\end{lemma}
%--------- end lemma -----------------------------------
%
\begin{proof} Given the usual rewritings in distributive lattices and given that $ x\mapsto -x $ exchanges~$\vu$ and~$\vi$, it is sufficient to know how to rewrite $ a+(b\vu c) $ and $ a(b\vu c) $ in the desired form. This follows from the equality rules \Tsbf{grl}, \Tsbf{grl6} and \Tsbf{afr7}.
\end{proof}
%

%n
%: Notation{notaAFR}
\begin{definota} \label{notaAFR}~
\begin{enumerate}
\item   Let $ \gB=\big((G,R),\Sa{Afr}\big) $ be a dynamic algebraic structure of $f$-ring. Since the theory \Sa{Afr} is Horn, $\gB$ admits a \hyperref[Modelegnq]{\textsl{generic model}}, denoted $ \AFR(\gB) $, which is the usual $f$-ring defined by the generators $G$ and the relations $ R $. 
\item When $G=\so{1,\dots,n}$ and $R$ is empty, we get the $f$-ring $\AFR(\ZZ[\Xn])$ freely generated by $n$ \elts.
\end{enumerate}

 \end{definota}
%--------- end notation ------------------------------ 

Since the theory \Sa{Afr} is purely equational, Lemma \ref{lemAfrReecriture} is equivalent to its statement restricted to special cases where $\gA$ is an $f$-ring free over a finite set. 
%l
%: Lemma{lemafrgenerique}
\begin{lemma} \label{lemafrgenerique}~
\begin{enumerate}
 
\item The elements of the ring $ \AFR(\gB) $ can all be written in the form given in Lemma \ref{lemAfrReecriture} with the $ f_{ij}\in\ZG $.

\item If $ \gC $ is a commutative ring, take for $(G,R)$ the positive diagram of $ \gC $. Then $ \AFR(\gC) $ is the $f$-ring freely generated by the commutative ring $ \gC $, and the elements of $ \AFR(\gC) $ are written in the form $ \sup\nolimits_{i\in I}\! \big(\inf\nolimits_{j\in J_i}a_{ij})\big) $ with elements~$ a_{ij} $ of~$ \gC $. 
\end{enumerate}
 \end{lemma}
%----------- end lemma ----------------------------------- 

%%%%%%%%%%%%%%%%%%%%%%%%%%%%%%%%%%%%%%%%%%%%%%%%%%%%%%%%%%%%%%%%%%%%
\Subsection{$f$-rings of maps, semipolynomials}
%: Subsection{$f$-rings of maps, semipolynomials}

For any set $E$ and any $f$-ring $\gA$ the ring of maps $f\colon E\to\gA$ is provided with a natural structure of $f$-ring (it is the product structure).

%: Definota{defiSIPD}
\begin{definota} \label{defiSIPD}
Let $ \varphi\colon \gA\to\gB $ be a morphism of $f$-rings. 
The ring of  \textsl{$\gA$-semipolynomials in $n$ variables}\footnote{Semipolynomials are often called \gui{SIPD} or \gui{sup-inf-polynomially-defined maps}.} \textsl{on $\gB$} is the $f$-subring of maps $f\colon \gB^n\to\gB$ generated by the constants in $ \varphi(\gA)$ and the coordinate maps. It will be noted $\SIPD_n(\gA,\gB)$. We shorten $\SIPD_n(\gA,\gA)$ to $\SIPD_n(\gA)$.
\\
The definition extends to the case where $\gA$ and/or $\gB$ are linearly ordered rings, which are considered to be $f$-rings.
\end{definota}
%--------- fin definition -------------------------------- 

Note that it is not really restrictive to suppose that $\varphi$ is injective, which makes it possible to look at $\gA$ as an $f$-subring of $\gB$.

%: Lemma{lemSIPD}
\begin{lemma} \label{lemSIPD}
It is assumed that $ \gA\subseteq  \gB $. Any element of $ \SIPD_n(\gA,\gB) $ is rewritten as $ \sup_{i\in I}\!\big(\inf_{j\in J_i}(f_{i,j})\big) $ for a suitable finite family of polynomials $ f_{i,j}\in\Axn $.
\end{lemma}
%--------- end lemma -----------------------------------
%
\begin{proof} Very close to proof of Lemma \ref{lemAfrReecriture}.
\end{proof}
%

%e
%: Example{exaSIPD}
\begin{examples} \label{exaSIPD}~

\noindent 1. The two elements $ x\vu (1-x) $ and $ 1\vu x\vu (1-x) $ define the same map in $ \SIPD_1(\ZZ) $, but not in $ \SIPD_1(\QQ) $.

\smallskip\noindent 2. Let $ \gK=\QQ(\epsilon) $ with $ \epsilon $ infinitesimal positive and $\gR$ the real closure of $\gK$. 
\\
The semipolynomial $ f=x^+\vi-(x^2-\epsilon)(x^3-\epsilon) $ defines the null map on $\gK$ but does not define a null map on $\gR$: the interval $ [\epsilon^{1/2},\epsilon^{1/3}] $ is invisible on $\gK$. This example can be simplified by taking $\gK=\QQ[\epsilon]$ with  a suitable nilpotent $\epsilon>0$.
\eoe
\end{examples}
%--------- end example ---------------------------------------------- 

\section{Beyond purely equational theories} \label{secArftr}

%%%%%%%%%%%%%%%%%%%%%%%%%%%%%%%%%%%%%%%%%%%%%%%%%%%%%%%%%%%%%%%%%%%%
\Subsection{$f$-rings without zerodivisor} 
%:Subsection{$f$-rings without zerodivisor}

%: Lemma{lemAfrsdz}
\begin{lemma} \label{lemAfrsdz} 
An $f$-ring without zerodivisor is linearly ordered. In other words, if we add the axiom \Tsbf{ASDZ} to the theory \sa{Afr}, the rule \Tsbf{OT} is valid. In other words, the resulting theory \SA{Afrsdz} is essentially identical to the theory \Sa{Atonz} of linearly ordered rings without zerodivisor (see Item 3 of Lemma \ref{lemAtonz}).

\TwoRegles{\lab{ASDZ} {$ \,\,xy=0\vd x=0 \;\vou\;y=0 $}}
{\lab{OT} {$ \vd x\geq 0\;\vou\;x\leq 0 $}}
 
\end{lemma}
%----------- end lemma ----------------------------------- 
%
\begin{proof}
Since $ x^+x^-=0 $ (\Tsbf{afr1}), we obtain the valid rule

\UneRegle{~} {$ \vd x^+=0 \;\vou\;x^-=0 $}

\vspace{-1em}
\end{proof}

%%%%%%%%%%%%%%%%%%%%%%%%%%%%%%%%%%%%%%%%%
\Subsection{Local $f$-rings}
%: Subsection{Local $f$-rings}

%l
%: Lemma{lemAfrLoc0}
\begin{lemma} \label{lemAfrLoc0}
Let $\gA$ be a  local $f$-ring and $ x\in\Ati $, then $x$ is $\geq 0$ or $\leq 0$.
\end{lemma}
%----------- end lemma ----------------------------------- 
%
\begin{proof}
Given $ x\in\Ati $, we write $ x=x^+-x^- $, so $ x^+\in\Ati $ or $ x^-\in\Ati $. Now $ x^+x^-=0 $. In the first case we obtain $ x^-=0 $, in the second case $ x^+=0 $. \end{proof}
%

%%%%%%%%%%%%%%%%%%%%%%%%%%%%%%%%%%%%%%%%%%%%%%%%%%%%%%%%%%%%%%%%%%%%
%:Subsection{Strict $f$-ring}
\Subsection{Strict $f$-ring}
\rdb

The following theory merges the theories \sa{Afr} and \sa{Aso}. This theory is essentially identical to the one defined in the article \cite{LM2017}.
%: Definition{defiAsr}
\begin{definition} \label{defiAsr}
The language of the Horn theory \SA{Asr} of \textsl{strict $f$-rings} is given by the following signature.
\Sigt{\AsR}{\cdot=0,\cdot\geq 0,\cdot>0\mathrel{;}\cdot+\cdot, \cdot\times\cdot,\cdot\vu\cdot,-\,\cdot,0,1} \label{NOTASigAsr}
\noindent The axioms are as follows.
%i
\begin{itemize}
 
\item the rules of the purely equational theory \Sa{Afr},
 
\item the direct rules from \Tsbf{aso1} to \Tsbf{aso4}, (\paref{Axaso1})
 
\item the Horn rules \coligt, \Tsbf{Iv}, \Tsbf{Aso1} and \Tsbf{Aso2}, (\paref{AxAso1})
 
\item finally, we have the three rules \Tsbf{sup1}, \Tsbf{sup2} and \Tsbf{Sup} (\paref{Axsup1}) to link $ \cdot\geq 0 $ \hbox{ and $ \cdot\vu\cdot $}.
\end{itemize}

\end{definition}
%--------- end definition --------------------------------

We have put the predicate \gui{$\cdot\geq 0$} directly into the language rather than defining it from  $\cdot\vu\cdot$.
 
The meaning of $x>0$ is not fixed a priori by the axioms. It can range from \gui{$x$ is regular and $\geq 0$} to \gui{$x$ is invertible and $\geq 0$}.

%: Lemma{lemAsrs}
\begin{lemma} \label{lemAsrs}~
Consider the Horn theory of strictly lattice rings to which we add the axiom \Tsbf{OTF}. Then the rule \OTFx\ is also valid (these rules are recalled below). 

\TwoRegles{
\lab{OTF} $ \,\, x+y> 0 \vd x >0 \;\vou\; y>0 $ 
}
{
\lab{OTF $ \eti $} $ \,\, xy< 0 \vd x <0 \;\vou\; y<0 $ 
}
\end{lemma}
%--------- end lemma -----------------------------------
%
\begin{proof} 
Assume $ xy<0 $, hence $ x^2y^2>0 $, hence, by \Tsbf{Aso2}, $ x^2>0 $.
\\
Note that $ x^2=(x^+)^2+(x^-)^2 $. So by \Tsbf{OTF}, it is sufficient to treat the cases $ (x^+)^2>0 $ and $ (x^-)^2>0 $ separately. \\
If $(x^+)^2>0 $, we have $ xx^+=(x^+)^2>0 $, so by \Tsbf{Aso2}, $ x>0 $. And again by \Tsbf{Aso2}, we get $ y<0 $.
If $ (x^-)^2>0 $, we have $ -xx^-=(x^-)^2>0 $, so by \Tsbf{Aso2}, $ x<0 $.
\end{proof}
%

%%%%%%%%%%%%%%%%%%%%%%%%%%%%%%%%%%%%%%%%%%%%%%%%%%%%%%%%%%%%%%%%%%%%
%:Subsection Theory of reduced $f$-rings
\Subsection{Reduced $f$-rings}
\rdb

Here we examine the Horn theory \SA{Afrnz} of \textsl{reduced $f$-rings}. We therefore add to \Sa{Afr} the axiom \Tsbf{Anz} of reduced rings, which is a \rsim.

\Regles{
\laB{Anz} $\,\, a^2=0 \vd a =0$}%

In the sale way the Horn theory \SA{Asrnz} of \textsl{reduced strict $f$-rings} is the theory obtained from the theory \Sa{Asr} by adding as axiom the Horn rule \tsbf{Anz}.

%%%%%%%%%%%%%%%%%%%%%%%%%%%%%%%%%%%%%%%%%%%%%%%%%%%%%%%%%%%%%%%%%%%%
\paragraph{Some derived rules}~

\smallskip Let us prove in the theory \sa{Afrnz} the four rules \Tsbf{Afrnz1}, \Tsbf{Afrnz2}, \Tsbf{Aonz}, and \Tsbf{Aonz3} (these last ones were introduced \paref{AxAonz} and \pageref{corPst1bis}).

\Regles{
\Lab{Afrnz1} $ \,\, x^{3}\geq 0 \vd x\geq 0 $ %
}

\noindent We write $ x=x^+-x^- $. Since $ x^+\,x^-=0 $ we have $ x^{3}=(x^+)^{3}-(x^-)^{3}\geq 0 $. Multiplying by $ x^- $ gives $ (x^-)^{4}\leq 0 $, so $ (x^-)^{4}= 0 $. Now the ring is reduced: $ x^-=0 $ and $ x\geq 0 $.\eop

\smallskip Note that from \tsbf{Afrnz1} we deduce the same rule for an arbitrary odd exponent which replaces the exponent 3.

\smallskip We also have the following reciprocal of the rule \Tsbf{Afr1}.

\Regles{
\Lab{Afrnz2} $ \,\, ab=0 \vd \abs a \vi \abs b =0 $ %
}

\noindent Indeed if $ab=0$, then $ (\abs a\vi\abs b)^2\leq \abs a,\abs b=0 $, therefore $ \abs a\vi\abs b=0 $. \eop

\smallskip Thus, for $ a,b\geq 0 $, $ab=0$ is equivalent to $a\vi b=0$.

\Regles{
\laB{Aonz} $ \,\, c\geq 0\vet x(x^2+c)\geq 0 \vd x\geq 0\, $ 
}

\noindent Indeed, by \Tsbf{Ato2} we have $x^3\geq 0$, hence $x\geq 0$ by \Tsbf{Afrnz1}. And finally

\Regles{
\laB{Aonz3} $\,\, a\geq 0\vet b\geq 0\vet a^2=b^2 \vd a=b$%
}
 
\noindent Indeed $ \abS{a-b}^2\leq \abS{a-b}\, \abS{a+b} =\abs{a^2-b^2}$.

\eop

\medskip It is now easy to obtain the following result. 

%: Lemma{lemAfr-Afrnz}
\begin{lemma} \label{lemAfr-Afrnz}
In the theory \Sa{Afr} the rules \Tsbf{Afrnz1}, \Tsbf{Afrnz2},  \Tsbf{Aonz}, \Tsbf{Aonz3} and~\Tsbf{Anz} are equivalent. 
\end{lemma}
%--------- end lemma ----------------------------------- 

Here is a simple example of a consequence of \pstref{Pst2}.

%l
%: Lemma{lemsupdansAonz}
\begin{lemma} \label{lemsupdansAonz}
In a reduced $f$-ring, the element $c=a\vu b$ is characterised by the following equalities and inequalities
\[
c\geq a,\;c\geq b, \;(c-a)(c-b)=0.
\]
More precisely, the theory \sa{Afrnz} proves the following Horn rule

\Regles{\labu $ \,\,x\geq a\vet x\geq b\vet (x-a)(x-b)=0\vd x=a\vu b $.} 
\end{lemma}
%----------- end lemma ----------------------------------- 
%
\begin{proof}
In fact, as the rule is valid for the theory \Sa{Codsup}, this follows from Item~\textsl{3} of  \pstfref{Pst2}.
\end{proof}
%

%%%%%%%%%%%%%%%%%%%%%%%%%%%%%%%%%%%%%%%%%%%%%%%%%%%%%%%%%%%%%%%%%%%%
\paragraph{The rule \tsbf{FRAC} in \sa{Afrnz}}~

\smallskip Recall the rules \Tsbf{FRAC} and \FRACn.

\Regles{\lab{FRAC} $ \,\,0\leq a\leq b\vd \Exists z\; (zb=a^2\vet\,0\leq z\leq a) $ 
\lab{FRAC$ _n $} $ \,\,\abs u^n\leq \abs v^{n+1}\vd \Exists z\; (zv=u\vet\, \abs z^{n}\leq \abs v) $ \quad ($ n\geq 1 $) }

Note that the rule \tsbf{FRAC}, applied with $ a=1 $ implies the invertibility of any element $ \geq 1 $.

\smallskip We now recall for the theory \Sa{Afrnz} the analogue of Lemma \ref{lemCo0FRAC} for the theory \Sa{Co--}. 
%: Lemma{lemAfrnzFRAC}
\begin{lemma} \label{lemAfrnzFRAC}
The addition of the axiom \tsbf{FRAC} to the theory \sa{Afrnz} can be replaced by the introduction of a function symbol $ \Fr $ with the axioms \tsbf{fr1} and \tsbf{fr2} which we recall below

\TwoRegles{
\lab{fr1} $ \vd \Fr(a,b)\, \abs b=(\abs a\!\vi\! \abs b)^2 $ 
}
{
\lab{fr2} $ \vd 0\leq {\Fr(a,b)}\leq \abs a\!\vi\! \phantom{(\abs a\!\vi\abs! \abs b)^2} $ 
}
\end{lemma}
%----------- fin lemma ----------------------------------- 
%
\begin{proof} As in Lemma \ref{lemCo0FRAC}, we can see that it is a question of skolemising an existential rule. This gives an essentially identical theory if we have unique existence. The proof is that of Lemma \ref{lemUniqFRAC}. Assume $ yb=zb=a^2 $, $ 0\leq y\leq a $ and $ 0\leq z\leq a $. We have $ {(y-z)}\,b=0 $, $ \abs{y-z}\leq a\leq b $\footnote{We are in an $\ell$-group for addition, we can reason case by case, separately with $ 0\leq y\leq z\leq a $ and $ 0\leq z\leq y\leq a $. In both cases we obtain $ 0\leq \abs{z- y}\leq a $.} and thus $ \abs{y-z}^2\leq \abs{y-z}\,b=0 $. 
\end{proof}
%

%l
%: Lemma{lemfracfrac2}
\begin{lemma} \label{lemfracfrac2}
In the theory \sa{Afrnz} the rule $\tsbf{FRAC}_2$ is deduced from the rule \tsbf{FRAC} and from the rule asserting the existence of the sixth root $\geq 0$ of an element $\geq 0$.
\end{lemma}
%----------- end lemma ----------------------------------- 
%
\begin{proof}
Assume $u^2\leq v^3$ and we want to find a $z$ such that $zv=u$ and $z^2\leq v$. 

\smallskip\noindent 
First assume $u\geq 0$ and show that there is a $z$ such that $zv=u$. The rule \tsbf{FRAC} implies that the fraction $t=\frac{u^4}{v^3}$ is well-defined with $0\leq t\leq u^2$. Again \tsbf{FRAC} gives the fact that the fraction $w=\frac{t^2}{u^2}$ with $ 0\leq w\leq t\leq u^2 $ is well-defined. We then obtain $ u^2wv^6=t^2v^6=u^8 $. So $ u^2(wv^6-u^6)=0 $. Now $ w\leq u^2 $, so $ \abs{wv^6-u^6}\leq u^2(v^6+u^4) $, hence $ \abs{wv^6-u^6}^2\leq \abs{wv^6-u^6} u^2(v^6+u^4)=0 $. Given the rule \tsbf{Anz}, we get $ wv^6=u^6 $. Take $z=w^{\frac1 6}$ and $zv=u$. Furthermore, $z^6\leq u^2\leq v^3$ implies $z^2\leq v$.

\smallskip \noindent For an arbitrary $u$ we write $ u=u^+-u^- $; we have $ u^2=(u^+)^2+(u^-)^2\leq v^3 $. We obtain a $ z_1 $ such that $ z_1v=u^+ $ and a $ z_2 $ such that $ z_2v=u^- $, we put $ z=z_1-z_2 $ and we have $ zv=u $. We also get $ z^2\leq 2v $. Since $ z^2v^2=u^2\leq v^3 $, we have $ v^2(z^2-v)\leq 0 $. The inequalities $ 0\leq z^2\leq 2v $ imply $ \abs{z^2-v}\leq v $, hence $ (z^2-v)^2\leq v^2 $ and $ (z^2-v)^3\leq v^2(z^2-v)\leq 0 $, which implies $ z^2-v\leq 0 $.
\end{proof}
%

%

%%%%%%%%%%%%%%%%%%%%%%%%%%%%%%%%%%%%%%%%%%%%%%%%%%%%%%%%%%%%%%%%%%%%
%:Subsection Return to ordered fields
\section{Back to ordered fields}
\label{secCOG}

%%%%%%%%%%%%%%%%%%%%%%%%%%%%%%%%%%%%%%%%%%%%%%%%%%%%%%%%%%%%%%%%%%%%
%:Subsection real $f$-rings 
\Subsection{Real $f$-rings}

The real number field satisfies all the Horn rules of the theory of discrete real closed fields, but not all the \rdys. Recall the following \rdys satisfied by discrete ordered fields.

\TwoRegles{
\laB{IV} $\,\, x> 0 \vd \Exists y\, xy = 1$
\laB{OTF} $\,\, x+y> 0 \vd x > 0\;\vou\;y>0$
\laB{FRAC} $\,\,0\leq a\leq b\vd \Exists z\; (zb=a^2\vet 0\leq z\leq a)$
}
{
\lab{\coligt} $ \vd x^2> 0 \;\vou\; x= 0$
\laB{OT} $ \vd x \geq 0 \;\vou\; x\leq 0$
}

The real number field verifies \tsbf{IV}, \tsbf{OTF} and \tsbf{FRAC} but neither \tsbf{ED\ineq}, nor \tsbf{OT}.

\smallskip The following lemma prepares the definition of the dynamical theory \Sa{Aftr}.
%: Lemma{lemArftr0}
\begin{lemma} \label{lemArftr0} 
\label{i1lemArftr0} In the theory \Sa{Afrnz} to which we add the rule \Tsbf{FRAC}, if we define \gui{$ x>0 $} as an abbreviation of \gui{$ x\geq 0 \vii \exists z\; zx=1 $}, the predicate $ \cdot>0 $ satisfies all the axioms of \Sa{Asrnz} where it is present as well as the rule \Tsbf{IV}.
\end{lemma}
%----------- end lemma ----------------------------------- 

\begin{proof} Rules
\tsbf{IV}, \Tsbf{aso1}, \Tsbf{aso2}, \Tsbf{aso4}, \colneq\  are trivially valid. Let's look at the rule \Tsbf{aso3}: an element greater than or equal to a positive invertible element is invertible. This follows from the rule \tsbf{FRAC} because if $ b\geq a>0 $, we have a~$z$ such that $ zb=a^2 $, so $b$ is invertible. For \Tsbf{Aso1}, \Tsbf{Aso2} and \Tsbf{Iv}, we begin by validating the rule $ \,\,x>0\vet xu=1\vd u\geq0$: indeed $u=xu^2\geq 0$. The rest follows.
\end{proof}

%: Definition{defiAftr}
\begin{definition} \label{defiAftr} 
We now define the dynamical theory \SA{Aftr} of \textsl{strongly real $f$-rings} (or to abbreviate, \textsl{strongly real rings}) on the following signature.
\Sigt{\AftR}{\cdot=0,\cdot\geq 0,\cdot>0\mathrel{;}\cdot+\cdot, \cdot\times\cdot,\cdot\vu\cdot,-\,\cdot,\Fr(\cdot),0,1} 
\label{NOTASigAftr}
\noindent The axioms of the \sa{Aftr} theory are as follows:
\begin{itemize}
 
\item the axioms of \Sa{Afrnz};
 
\item the axioms which define $ x>0 $ as an abbreviation of \gui{$ x\geq 0\vii \exists z\,xz=1 $}; 
 
\item the axioms \Tsbf{fr1} and \Tsbf{fr2}. 
\end{itemize}
\index{strongly real $f$-ring}\index{f-ring@$f$-ring!strongly real ---}\index{strongly real ring}%.
%
%end{enumerate}
\end{definition}
%--------- fin definition --------------------------------

This definition is justified by the fact that the theory of \und{local} strongly real rings is essentially identical to the theory \sa{Co} of \ndsof: Item \textsl{3} of Lemma \ref{lemArftr}.
%l
%: Lemma{lemArftr01}
\begin{lemma} \label{lemArftr01}\label{i2lemArftr0} The dynamical theory \sa{Aftr} is essentially identical to the theory \sa{Asrnz} to which we add the axioms \Tsbf{IV} and \Tsbf{FRAC}\footnote{This implies that the theory {\ssa{Aftr}} defined here is slightly stronger than the one defined in \cite{LM2017}.}.
\end{lemma}
%----------- end lemma ----------------------------------- 
%
\begin{proof} 
The theory \sa{Asrnz} contains in its signature the predicate~\gui{$ \cdot>0 $} which is not present in \sa{Afrnz}. When we add the axiom \tsbf{IV}, we have $ x>0 $ if, and only if, $x$ is $\geq 0$ and invertible. Lemma \ref{lemArftr0} therefore implies that the theory \sa{Afrnz} to which we add the axiom \tsbf{FRAC} is essentially identical to the theory \sa{Asrnz} to which we add the axioms \tsbf{IV} and \tsbf{FRAC}. Finally, Lemma \ref{lemAfrnzFRAC} shows that adding the axiom \tsbf{FRAC} to \sa{Afrnz} is equivalent to adding the function symbol $ \Fr $ with the axioms \tsbf{fr1} and~\tsbf{fr2}.
\end{proof}

A strongly real ring is therefore a reduced $\QQ$-$f$-algebra in which any element greater than an invertible positive element is itself invertible. Moreover, the validity of the rule~\tsbf{FRAC} adds a little something. 

%: Lemma{lemArftr}
\begin{lemma} \label{lemArftr}~
\begin{enumerate}
 
\item 
 The Horn theory \Sa{Asrnz} to which we add as axioms the \rdys \Tsbf{IV} and \Tsbf{OTF} is essentially identical to the theory \Sa{Co--}.
 
\item The dynamical theory \Sa{Aftr} to which we add the axiom \Edineq\ is essentially identical to the theory \Sa{Codsup} of discrete ordered fields with sup (or to \Sa{Cod}).
 
\item The dynamical theory \Sa{Aftr} to which we add the axiom \Tsbf{OTF} is essentially identical to the theory \Sa{Co} of \emph{non} discrete ordered fields: a strongly real local ring is a \emph{non} discrete ordered field.\footnote{We could have avoided introducing the predicate $ x>0 $ in \ssa{Aftr} because it is defined as an abbreviation. The rule \Tsbf{OTF} should be replaced by the rule \Tsbf{AFRL} in Item 2. The theory would then be a Horn theory. This is hardly surprising since the theory of real closed rings is purely equational and a \ndrcf is a local real closed ring.}
\end{enumerate}
\end{lemma}
%----------- end lemma ----------------------------------- 
%
\begin{proof} \textsl{1.} Comparing the theories \sa{Co--} (Definition \ref{defiCo0}) and \sa{Asrnz}, we find in the first the additional axioms \tsbf{Aonz}, \tsbf{IV} and~\tsbf{OTF} and the collapse axiom is missing. But collapus follows from \tsbf{IV} and \tsbf{Aonz} is deduced from \tsbf{Anz} (Lemma \ref{lemAfr-Afrnz}). \\
NB: The axiom~\tsbf{IV} implies that $ x>0 $ is equivalent to \gui{$x$ is $\geq 0$ and invertible}. The axiom \tsbf{OTF} adds the fact that the ring is local. 

\smallskip\noindent \textsl{2}.
The new theory is an extension of \sa{Co--} from Item \textsl{1}. To move on to \sa{Codsup} we need only add \Edneq\ and \tsbf{OT}. Since every element is zero or invertible, we are dealing with a discrete field, and the rule \tsbf{ASDZ} is valid. Lemma~\ref{lemAfrsdz} then says that the rule \tsbf{OT} is also valid.

\smallskip\noindent \textsl{3}. Results from Item \textsl{1} and  Lemmas \ref{lemAfrnzFRAC}, \ref{lemArftr0} and \ref{lemAftrloc}.
\end{proof}
%

%l
%: Lemma{lemAftrloc}
\begin{lemma} \label{lemAftrloc}
In a reduced $f$-ring where the rule \tsbf{FRAC} is valid, the following rule \tsbf{AFRL} replaces the rule \Tsbf{OTF} when we define the predicate $ a >0 $ as an abbreviation of $ a\geq 0\vii \exists z\;az=1 $.

\Regles{\Lab{AFRL} $\;\;z(x+y)=1\vet x+y\geq 0\vd \Exists u\;(ux=1\vet x\geq 0 )\;\vou\;  \Exists v\;(vy=1\vet y\geq 0 )$ }

\end{lemma}
%----------- fin lemma -----------------------------------
%
\begin{proof} \textsl{Direct implication}. Assume $\gA$ is local and prove the rule \tsbf{AFRL}. Since $ x+y $ is invertible, $x$ or $y$ is invertible. For example~\hbox{$ x\in\Ati $}. By Lemma \ref{lemAfrLoc0} we have $ x\geq 0 $ or $ x\leq 0 $. If $ x\leq 0 $, then $ y\geq x+y\geq 0 $, therefore $ y\in \Ati $ (Lemma \ref{lemArftr0}).
\textsl{Reciprocal}. If $ x+y $ is invertible, then $ x+y\geq 0 $ or $ x+y\leq 0 $ (Lemma \ref{lemAfrLoc0}). In the first case, \tsbf{AFRL} shows that $x$ or $y$ is invertible. In the second case, $ -x $ or $ -y $ is invertible.
\end{proof}

\Subsection{Formal Posivitstellensätze with sup}
\rdb

For simultaneous collapse, we have already given  \thref{thColsimafr}.

\begin{pstf} [formal Positivstellensatz, 2] \label{Pst2} ~%
\index{Positivstellensatz!formal --- !for ordered fields, 2}\\ 
The following dynamical theories prove the same Horn rules.
\begin{enumerate}
 
\item \label{i1Pst2} On the signature \sigt{\Ao}{\cdot=0,\cdot\geq 0\mathrel{;}\cdot+\cdot, \cdot\times\cdot,-\,\cdot,0,1}: the theory \Sa{Aonz} of strictly reduced ordered rings (Definition \ref{defisaAor}), the theory \Sa{Crcdsup} of discrete real closed fields with sup and intermediate theories (for example \Sa{Afrnz}, \Sa{Atonz}, \Sa{Cod} or \Sa{Co--}).
 
\item \label{i2Pst2} On the signature \sigt{\Aso}{\cdot=0,\cdot\geq 0,\cdot>0\mathrel{;}\cdot+\cdot, \cdot\times\cdot,-\cdot, 0,1}: the \Sa{Asonz} theory of reduced strictly ordered rings, the \Sa{Crcdsup} theory and intermediate theories (for example~\Sa{Asrnz}, \Sa{Cod} or~\Sa{Co--}).
 
\item \label{i3Pst2} On the signature \sigt{\AfR'}{\cdot=0,\cdot\geq 0\mathrel{;}\cdot+\cdot, \cdot\times\cdot,\cdot\vu\cdot,-\,\cdot,0,1}: the \Sa{Afrnz} theory of reduced $f$-rings, the \Sa{Crcdsup} theory and intermediate theories (e.g. \Sa{Asrnz}, \Sa{Cod} or \Sa{Co--}). 
 
\item \label{i4Pst2} On the signature \sigt{\AsR}{\cdot=0,\cdot\geq 0,\cdot>0\mathrel{;}\cdot+\cdot, \cdot\times\cdot,\cdot\vu\cdot,-\,\cdot,0,1}: the theory \Sa{Asrnz} of reduced strict $f$-rings, the theory \Sa{Crcdsup} and intermediate theories (for example~\Sa{Co} or \Sa{Codsup}).
 
\item \label{i5Pst2} On the signature \sigt{\Co}{\cdot=0,\cdot\geq 0,\cdot>0\mathrel{;}\cdot+\cdot, \cdot\times\cdot,\cdot\vu\cdot,-\,\cdot,\Fr(\cdot,\cdot),0,1}: the \Sa{Aftr} theory of strongly real rings, the theory \Sa{Crcdsup} and intermediate theories (for example \Sa{Co} or \Sa{Codsup}).
\end{enumerate}
\end{pstf}
%--------- end pstf ---------------------------------------------- 

%
\begin{proof} \textsl{1} and \textsl{2}. The theory \sa{Crcdsup} is essentially identical to \sa{Crcd}. So Positivstellensätze \ref{Pst1bis} and \ref{Pst1} give Items \textsl{1} and \textsl{2}. 

\smallskip\noindent 
\textsl{3}. Theories \sa{Atonz} and \sa{Crcd} prove the same Horn rules (\ref{Pst1bis}). Theories \sa{Ato} et \sa{Afr} prove the same Horn rules (\pstref{thfairecommesi}), so also Theories \sa{Atonz} and \sa{Afrnz} (\thref{thEseqMemesfaitsbis}).
Finally  \sa{Crcd} and  \sa{Crcdsup} are \esid.

\smallskip\noindent 
\textsl{4}. Same reasoning as in the previous item.

\smallskip\noindent 
\textsl{5}. Note that \sa{Codsup} validates the rules \tsbf{IV} and \tsbf{FRAC}, which can be replaced by the introduction of $ \Fr $ with its two axioms (Lemma \ref{lemAfrnzFRAC}). On the other hand, the dynamical theory \sa{Aftr} is essentially identical to the theory \sa{Asrnz} to which we add the axioms \tsbf{IV} and \tsbf{FRAC} (Item \textsl{2} of Lemma \ref{lemArftr0}). Item \textsl{5} therefore results from Item \textsl{4}.
\end{proof}
%

%r
%: Remark{remPst2}
\begin{remark} \label{remPst2} 
In the theories \Sa{Aftr} and \Sa{Co}, we have added the function symbol $\Fr$ with the axioms~\Tsbf{fr1} and \Tsbf{fr2}, which increases the Horn rules formulable in these theories. Nevertheless the use of the symbol $\Fr$ can be eliminated in the Horn rules in favour of the axiom \Tsbf{FRAC} (see the addition of a function symbol \paref{skolemunique} and Lemma \ref{lemCo0FRAC}). As this axiom is satisfied in the stronger theory~\Sa{Cod}, \pstref{Pst2} is not affected by the presence of $\Fr$. We could also accept the presence of the function symbol $\Fr$ with the axioms \Tsbf{fr1} and \Tsbf{fr2} in the theory \Sa{Codsup}.
\eoe\end{remark}
%----------- end remark ---------------------------------- 

%c
%: Corollary{corPst2}
\begin{corollary} \label{corPst2} Consider the dynamical theory $ \Sa{Asrnz}= \Sa{Asrnz}(\QQ) $ of reduced strict $f$-rings.
\begin{enumerate}
 
\item Let $\gK$ be a discrete ordered field and $\gR$ its algebraic closure. Let $ \gA=\big((G,Rel),\Sa{Asrnz}(\gK)\big) $ be a dynamic algebraic structure with \hbox{$ G=(\xn) $} and $ Rel $ finite. We have an algorithm which decides whether $\gA$ collapses and which in case of a negative answer gives the description of a system $ (\xin) $ in $ \gR^n $ which satisfies the constraints given in the relations $ Rel $.
 
\item We have an algorithm that decides whether a Horn rule of $ \sa{Asrnz} $ is valid. If the answer is negative, the algorithm gives the description of a system $(\xin) $ in $ \gRa^n $ which contradicts the Horn rule.
 
\item The same results are valid with $ \Sa{Aftr}= \Sa{Aftr}(\QQ\cap\ClI{0,1}) $ instead of $ \sa{Asrnz} $. In this case we add the function symbol $ \Fr $ with the two accompanying axioms to \Sa{Crcdsup}.
\end{enumerate}
\end{corollary}
%--------- end corollary -------------------------------
%
\begin{proof} 
We have just seen (Formal \pstref{Pst2}) that the Horn theory \Sa{Asrnz} (resp. \Sa{Aftr}) proves the same Horn rules as \Sa{Crcdsup} (resp. by adding $ \Fr $). Moreover, we see that a Horn rule of \Sa{Crcdsup} (possibly by adding $ \Fr $) is equivalent to a family of Horn rules of \Sa{Crcd}. We can therefore conclude with Concrete Positivstellensatz \ref{thPstStengle}.
\end{proof}

The ring $ \SIPD_n(\gA) $ of semipolynomials over $\gA$ in $ n $ variables is explained in Definition \ref{defiSIPD}, the $f$-ring $ \AFR(\gA) $ generated by $\gA$ is defined in \ref{notaAFR}.

%: Theorem{thSIPDreduit}
\begin{theorem} \label{thSIPDreduit} Fix $ n $ and denote $ \Kux=\Kxn $.
\begin{enumerate}
 
\item Let $\gK$ be a discrete ordered field and $\gR$ its real closure. 
The ring $\SIPD_n(\gK,\gR)$ is identified with the $f$-ring generated by $\Kux$. More precisely: the structure of $\gK$ gives to $\Kux$ a dynamic algebraic structure of $f$-ring and the unique $\gK$-morphism of $f$-rings from $\AFR(\Kux)$ to $\SIPD_n(\gK,\gR) $ is an isomorphism. 
 
\item Let $\gK$ be a discrete ordered field and $\gR$ its real closure. If every semialgebraic open of $\gR^n$ contains points of $\gK^n$, the ring $\SIPD_n(\gK)$ is identified with $\AFR(\Kux)$.
 
\item (incomplete proof) If $\gK$ is a $\QQ$-algebra contained in $\RR$, the ring $\SIPD_n(\gK)$ identifies with $\AFR(\Kux)$.
\end{enumerate}
\end{theorem}
%--------- end theorem -----------------------------------
%
\begin{proof} It must be shown that if an expression of the form given in Lemma \ref{lemSIPD} defines the identically zero map, this can be proved using only the Horn rules of reduced $f$-rings.

\sni \textsl{1.}
By the Positivstellensatz, the fact that a semipolynomial is zero at any point in $ \gR^n $ has an algebraic certificate on $\gK$. Now \Sa{Afrnz} and \Sa{Crcdsup} prove the same Horn rules. (For more details on this kind of subject see \cite{GL93}). 

\sni \textsl{2.} Results from the previous item because under the considered hypothesis, a $\gK$-semipolynomial not zero everywhere on $ \gR^n $ is not zero everywhere on $ \gK^n $. 

\sni \textsl{3.} If $\gK$ is discrete, this follows from Item \textsl{1}, because a $\gK$-semipolynomial zero on $\gK$ is zero on~$\QQ$ and therefore on~$\RR$ and a fortiori on the real closure of the field of fractions of $\gK$. Apparently, it takes a bit of effort to obtain the result constructively in all generality, whereas it is clear in classical mathematics. It's the same kind of gymnastics as for the complete constructive proof of the solution of the 17th Hilbert problem on $\RR$, given in \cite{GL93}. The bonus is that the solution is then completely explicit, i.e.\ it does not use any sign test (or dependent choice axiom) on $\RR$.
\end{proof}

%section{Real lattice} 
\section{The real lattice and spectrum of a commutative ring} 
\label{subsecTRR}
%-----------------------------------------

A prime cone of the commutative ring $\gA$, i.e.\ an element of the real spectrum $ \Sper\,\gA $, can be given as a non-trivial integral quotient ring $ \gA/\fP $ with a linearly ordered ring structure, in other words as a minimal model of the $ \Sa{Aito}(\gA) $ theory of non-trivial integral linearly ordered rings based on~$\gA$ (see Definition \ref{defidiagramme}).

As the theory of nontrivial linearly ordered integral rings is a dynamical theory without existential axioms, the real spectrum is identified with the spectrum of the distributive lattice $ \Reel(\gA) $ obtained by \gui{recopying}\footnote{As we do below, in Definition \ref{defZarR}.} the axioms of the dynamical theory $\sa{Aito}(\gA) $.

To find the usual topology on the set $\Spec(\Reel\gA)$\footnote{According to the tradition established when the real spectrum was invented.} underlying the real spectrum $ \Sper\gA $, we must consider the lattice based on the only predicate $ x>0 $. This gives Definition \ref{defZarR}, which corresponds to the following valid \rdys in \sa{Cod}

\TwoRegles
{
\lab{Col\ineq} $ \,\,0> 0 \vd \Bot $ 
\lab{aso3} $ \,\, x > 0\vet y > 0 \vd x +y > 0 $ 
\lab{aso4} $ \,\, x > 0\vet y > 0 \vd x y > 0 $ 
}
{
\lab{aso1} $ \vd 1 > 0 $ 
\lab{OTF} $ \,\, x+y> 0 \vd x > 0\vou y>0 $ \lab{OTF$\eti$} $ \,\, xy< 0 \vd x <0 \vou y<0 $ 
}

%r
%: Remark{remdefZarR}
\begin{remark} \label{remdefZarR} 
If we base ourselves solely on the predicate $ x>0 $ and if we introduce the predicate $ x\geq 0 $ as the opposite of the predicate $ -x>0 $, and the predicate $ x=0 $ as the conjunction $ x\geq 0 \vii -x\geq 0 $, we obtain on the basis of the previous axioms alone a conservative extension which satisfies all the axioms of \sa{Aito}. The minimal models of the dynamical theory described by the 6 previous axioms are therefore (in classical mathematics) the integral quotients of $\gA$ with a relation of total order. This justifies the following definition: the spectrum of the lattice $ \Reel\gA $ is indeed identified with the real spectrum of $\gA$ (in classical mathematics).
\eoe
\end{remark}
%----------- end remark ---------------------------------- 

%--Definition{defZarR}---- 
\begin{definition} 
\label{defZarR} 
The \textsl{real lattice} of a commutative ring $\gA$, denoted $ \Reel\gA $, is the distributive lattice generated by $(\gA,\vdash)$ where $\vdash$ is the smallest entailment relation satisfying 
\begin{equation} 
\left.
\begin{array}{rclcrcl} 
 0 & \vdash & &\qquad \qquad &
 & \vdash & 1 ~\\
 x,\,y & \vdash & x+y &\quad & x+y & \vdash & x ,\, y \\
x,\,y & \vdash & xy & & -x y & \vdash & x ,\, y 
 \end{array}
\quad \right\}
\end{equation}
\end{definition}
%--- end-definition------------------------------------

This simple and constructive way of defining the real lattice goes back to \cite{CC00}, which was inspired by \cite[Section V-4.11]{Joh1986}. 

\newpage \thispagestyle{empty}

%%%%%%%%%%%%%%%%%%%%%%%%%%%%%%%%%%%%%%%%%%%%%%%%%%%%%%%%%%%%%%%%%%%%
%%%%%%%%%%%%%%%%%%%%%%%%%%%%%%%%%%%%%%%%%%%%%%%%%%%%%%%%%%%%%%%%%%%%
%%%%%%%%%%%%%%%%%%%%%%%%%%%%%%%%%%%%%%%%%%%%%%%%%%%%%%%%%%%%%%%%%%%%
%%%%%%%%%%%%%%%%%%%%%%%%%%%%%%%%%%%%%%%%%%%%%%%%%%%%%%%%%%%%%%%%%%%%
%%%%%%%%%%%%%%%%%%%%%%%%%%%%%%%%%%%%%%%%%%%%%%%%%%%%%%%%%%%%%%%%%%%%

\chapter{\textsl{Non} discrete real closed fields} \label{chapreelclos}
\Today
\minitoc

\section*{Introduction}
\addtocontents{toc}{\skip0.8em}
\addcontentsline{toc}{section}{Introduction}
\rdb

Section \ref{subsecclotrlRR} explains how to introduce square roots of the elements~\hbox{$\geq 0$} into an $f$-ring and in particular into a \ndsof. This is intended as an introduction to the more general notion of virtual roots. The case of discrete ordered fields was treated in \cite[section 3.2]{LR91}. The moral of this case is that we don't need to know whether a square root is already present in the ordered field in order to introduce it formally without any risk of contradiction. Here we see the superiority of the constructive point of view over the classical point of view (which usually uses \tsbf{LEM} to decide whether the coveted square root is already present or not).

Section \ref{secCoVR} explains how to add virtual root maps in  \ndsofs. Virtual roots were introduced in \cite{GLM98} for \dofs. The aim was to have, for a real monic polynomial, continuous maps of the coefficients which cover the real roots. In particular, this made it possible to have a constructive version of the intermediate value theorem in which no sign test was used. In fact, similar work can be done on any $f$-ring.  

Section \ref{secArc} deals with real closed rings and Section \ref{secCrc2} proposes a definition for \ndrcfs as local real closed rings. The theory of real closed rings is presented here in an elementary, purely equational form, in the style of \cite{Tre2007}. 

%%%%%%%%%%%%%%%%%%%%%%%%%%%%%%%%%%%%%%%%%%%%%%%%%%%%%%%%%%%%%%%%%%%%
%%%%%%%%%%%%%%%%%%%%%%%%%%%%%%%%%%%%%%%%%%%%%%%%%%%%%%%%%%%%%%%%%%%%
%:Subsection
\section{$2$-closed ordered field  (or euclidean field)} 
\label{subsecclotrlRR}
 As a starting point, let's look at the question of introducing the square roots of the elements $\geq 0$. For the case of a discrete ordered field we refer to \cite[section 3.2]{LR91}.
 
We are interested in the following rule which says that the elements $\geq 0$ are squares.

\Regles{
\Lab{sqr} $ \vd \Exists z\geq 0\;x^+=z^2 $ 
}

Another way to state the same thing is to postulate

\Regles{
\Lab{sqa} $\vd \Exists z\;\;(\abs x=z^2\vet z=\abs z)$
}

In an  \grl we have  in \gnl $x^+\neq \abs x$, but conditions $x=x^+$ and $x=\abs x$ are \eqve (they mean $x\geq 0$).

\vspace{-1em}
\Subsection{The case of a \codi} 

Référence: \cite[section 3.2]{LR91}.

%d
%:     Definition{deficdireel}
\begin{definition} \label{deficdireel}
A \cdi is said to be \textsl{real} if $\sum_{i=1}^nx_i^2=0$ implies $x_i=0$ ($1\leq i\leq n$). 
\end{definition}
%----------- fin definition -------------------------------- 

A \codi is real. 
In classical mathematics, any real discrete field can be ordered.
In constructive mathematics this simply means that if we formally add an order structure (with the axioms of discrete ordered fields) to a discrete field, and if the formal theory proves $1=0$, then we already had $1=0$ in the discrete field itself.

%:     Definition{deficdidc}
\begin{definition} \label{deficdidc}
A discrete ordered field is said to be \textsl{$d$-closed} if every polynomial $P$ of degree $\leq d$ that changes sign between $a$ and $b$ has a root on the interval with endpoints $a$ and $b$.
A $2$-closed discrete ordered field is also called an \textsl{Euclidean discrete field}.
\end{definition}
%----------- fin definition -------------------------------- 

The usual method of solving second-degree equations gives the following lemma. But this only works for a discrete field.

%:     lemma{lemmacodi2clos}
\begin{lemma} \label{lemmacodi2clos}  Soit $\gK$ un \codi.
\Propeq
\begin{enumerate}
\item Any \elt $\geq 0$ is a square.
\item $\gK$ is a $2$-closed \codi.
\end{enumerate}
 
\end{lemma}
%----------- fin lemma ----------------------------- 

%p
%:     Proposition{propcodi2clos}
\begin{proposition} \label{propcodi2clos} Let $\gK$ be a real \cdi \footnote{We use the reality axiom only for the case $\,\,x^2+y^2+z^2=0\vd x=y=z=0$.}. \Propeq
\begin{enumerate}
\item $\gK$  can be provided with an order (in the sense of discrete ordered fields) for which every element~$\geq 0$ is a square. In which case the possible order relation is unique.
\item Any square is a power $4$.
\end{enumerate}
 \end{proposition}
%----------- fin proposition ----------------------------- 
%
\begin{proof}
\textsl{1}  $\Rightarrow$ \textsl{2}. Let $x$ be equal to a square $y^2$. If $\gK$ is a \codi, we have $y\geq 0$ or $y\leq 0$. So there exists $z$ such that $y=z^2$ ou $-y=z^2$. In both cases $y^2=z^4$.

\smallskip \noindent \textsl{2} $\Rightarrow$ \textsl{1}. From $y^2=z^4$ we deduce that $y=z^2$ or $-y=z^2$. If $P$ is the set of squares, we therefore have $P\cup -P=\gK$.\\
Let us then show that for all $x,y$, there exists $z$ such that $x^2+y^2=z^2$.
Indeed $x^2+y^2=\pm z^2$ and if $x^2+y^2+ z^2=0$ then $x=y=z=0$ and  $x^2+y^2=z^2$. We therefore have $P+P\subseteq P$. We also have $P\cap -P=\so 0$. Indeed if $y^2=-y^2$ then $y^2+y^2=0$ therefore $y=0$. Finally $P\cdot P\subseteq P$ because $x^2y^2=(xy)^2$.
Thus the field can be ordered, in a unique way.
And every positive is a square.
\end{proof}

In \cite{LR91} if $\gK$ is a \codi and $a\geq 0$, the authors formally introduce a square root $\alpha\geq 0$ de $a$ and demonstrate that $\gK[\alpha]$ can be equipped with a discrete ordered field structure without needing to know whether $\alpha\in\gK$. In other words, we know the structure of $\gK[\alpha]$ as a \codi, but a priori we do not know whether $\gK[\alpha]$ is of dimension~$1$ or~$2$ as a \Kev.

This elementary construction is the basic building block for constructing the $2$-closure of a discrete ordered field.

Things are much more difficult for a \ndsof.

\Subsection{The case of an \afr} 

%: Remark{rem2closreduced}
\begin{remark} \label{rem2closreduit} 
In an $f$-ring, in the presence of nilpotents, two elements $z$ and $ y \geq 0 $ which have the same square are not necessarily equal, 
but if the ring is reduced, they are equal, by virtue of the \rsim \Tsbf{Afrnz3}.  So the rule \tsbf{sqr} is a simple existential rule with unique existence and if we slolemise this rule in the theory \sa{Afrnz} we get an essentially identical theory.
\eoe
\end{remark}
%----------- fin remark ---------------------------------- 

We now present a version in which a nonnegative square root of a nonnegative element is given as a unary law in the dynamical theory which extends \sa{Afr}

\smallskip 
\centerline{$ 
\Sqr\colon\gR\to\gR,\;x\mapsto \sqrt{ x^+} $ \quad  in case of a \rcf.} 

\smallskip\noindent This function symbol must obey the following  natural direct rules.

\TwoRegles{
\lAb{sqr$ _= $} $ \,\, y= 0 \vd \Sqr(x+y)=\Sqr(x) $ \label{Axsqreg}
\Lab{sqr0} $ \vd \Sqr(0)=0 $ 
\Lab{sqr1} $ \vd \Sqr(x)\geq 0 $ 
}
{
\Lab{sqr2} $ \vd \Sqr(x)=\Sqr(x^+) $ 
\Lab{sqr3} $ \vd \Sqr(x)^2=x^+ $ 
\Lab{sqr4} $ \vd \Sqr(x^+y^+)=\Sqr(x)\Sqr(y) $ 
}

Note that $ \Sqr(x)=0 $ when $ x\leq 0 $ and $ \Sqr(x)=\sqrt x $ when $ x\geq 0 $.
 
%d
%: Definition{defiAfr2c}
\begin{definition} \label{defiAfr2c}\label{defiAsr2c}~
\begin{itemize}
\item We denote \SA{Afr2c} the purely equational theory of \textsl{$2$-closed} $f$-rings: it is obtained from \Sa{Afr} by adding the unary function symbol $\Sqr$ with the six preceding axioms: \sIgt{\AfRdc}{\cdot=0,\cdot\geq 0,\cdot>0\mathrel{;}\cdot+\cdot,\cdot\times\cdot,\cdot\vu\cdot,-\,\cdot,\Sqr(\cdot),0,1}.\label{NOTASigAfr2c}\index{f-ring@$f$-ring!$2$-closed ---}\index{$2$-closed!$f$-ring}
\item We denote \SA{Asr2c} the dynamical theory of \textsl{$2$-closed strict $f$-rings} obtained from \Sa{Asr} in the same way that \Sa{Afr2c} was obtained from \Sa{Afr}:\\
\sIgt{\AsRdc}{\cdot=0,\cdot\geq 0,\cdot>0\mathrel{;}\cdot+\cdot, \cdot\times\cdot,\cdot\vu\cdot,-\,\cdot,\Sqr(\cdot),0,1}.\label{NOTASigAsr2c}\index{$2$-closed!strict $f$-ring}

\item We denote \SA{Aftr2c} the dynamical theory of \textsl{$2$-closed \aftrs} obtained from \Sa{Aftr} in the same way that \Sa{Afr2c} was obtained from \Sa{Afr}:\\
\sIgt{\AftRdc}{\cdot=0,\cdot\geq 0,\cdot>0\mathrel{;}\cdot+\cdot, \cdot\times\cdot,\cdot\vu\cdot,-\,\cdot,\Fr(\cdot),\Sqr(\cdot),0,1}.\label{NOTASigAftr2c}\index{strongly real!$2$-closed --- ring}\index{$2$-closed!strongly real ring}
\item We denote \SA{Co2c} the dynamical theory of $2$-closed \ndsofs obtained from \Sa{Co} in the same way: same signature as \Sa{Aftr2c}.\label{NOTASigCo2c}\index{ordered field!2-closed@$2$-closed ---}\index{$2$-closed!non disc@\ndsof}%.

\end{itemize}
\end{definition}
%----------- end definition -------------------------------- 

%:     Remark{remSqa}
\begin{remark} \label{remSqa} 
We have chosen the function $\Sqr(x)$ which verifies  \fbox{$\Sqr(x)\geq 0$ and $\Sqr(x)^2=x^+$} because it corresponds to the second virtual root of the polynomial $Y^2-x$ in the case of a discrete ordered field. But we could also use \fbox{$\Sqa(x):=\Sqr(\abs x)$} which satisfies the equality
\fbox{$\Sqa(x^+)=\Sqr(x)$} and which can be characterised by \fbox{$\Sqa(x)\geq 0$ and $\Sqa(x)^2=\abs x$}.
We easily verify that the function $\Sqa$ can be introduced with the following axioms.

\DeuxRegles{
\lAb{sqa$_=$} $\,\, y= 0 \vd \Sqa(x+y)=\Sqa(x)$\label{Axsqaeg}
\Lab{sqa1} $ \vd \Sqa(x)= \abS{\Sqa(x)}$
\Lab{sqa3} $ \vd \Sqa(x)^2=\abs x$
}
{
\Lab{sqa0} $ \vd \Sqa(0)=0$
\Lab{sqa2} $ \vd \Sqa(x)=\Sqa(\abs x)$
\Lab{sqa4} $ \vd \Sqa(x y)=\Sqa(x)\Sqa(y)$
}

\end{remark}
%----------- fin remark ---------------------------------- 

%l
%: Lemma{lemAr2creduit}
\begin{lemma} \label{lemAr2creduit}
A  $ 2 $-closed $f$-ring is reduced. 
\end{lemma}
%----------- end lemma ----------------------------------- 
%
\begin{proof}
On the one hand $ \abs{a}^2=\abs{a^2} $, and on the other hand for a $ x\geq 0 $ such that $ x^2= 0 $, we have the equalities $ 0=\Sqr(x^2)=\Sqr(x)^2=x^+=x $.
\end{proof}

\Subsubsection{Some derived rules in \sa{Afr2c}}

\DeuxRegles{
\Lab{Aonz0} $\,\, x^2+y^2+z^2=0 \vd  x=0$
\Lab{Afr21} $\,\, (x^2+y^2)^2= z^4 \vd  x^2+y^2 = z^2$
}
{
\Lab{AFR2} $ \vd \Exists y\; x^2=y^4$
}

Let us note that the rule \tsbf{Aonz0} is valid in \Sa{Aonz} (strictly reduced ordered rings).

\Subsubsection{Theories which are \esid to $\sa{Afr2c}$}

%l
%: Lemma{lemSqr}
\begin{lemma} \label{lemSqr}
The following five extensions of the \sa{Afr} theory are essentially identical.
\begin{enumerate}
\item The purely equational theory \sa{Afr2c}.
\item We add the function symbol $\Sqa$  and the 6 axioms indicated in Remark \ref{remSqa}.
\item We add as axioms the \rdys \Tsbf{Anz} and \Tsbf{sqr}.
\item We add as axioms the \rdys \Tsbf{Anz} and \Tsbf{sqa}.
\item We add as axioms the \rdys \Tsbf{Anz} and \Tsbf{AFR2}.

\end{enumerate}
\end{lemma}
%----------- end lemma -----------------------------------

In Items \textsl{3}, \textsl{4}, \textsl{5}, we don't change the language. 

\begin{proof} 
We show that the theory of Item \textsl{3} is \esid to \sa{Afr2c}. 
First, in the theory \sa{Afrnz} the simple existential rule \tsbf{sqr} is with unique existence by virtue of Remark \ref{rem2closreduit}.
Then, we check that the function $\Sqr$ obtained by skolemising the existential axiom \tsbf{sqr} satisfies the 6 desired axioms.
\\
We also show that the theory of Item \textsl{5} is essentially identical to that of Item \textsl{4}.
In the latter we have $\Sqa(x)\geq 0$ and $\Sqa(x)^4= {\abs x} ^2=x^2$,
so the axiom \tsbf{AFR2} is verified.
In the theory of Item \textsl{5} we have $x^2={\abs x} ^2=y^4$, so by \Tsbf{Aonz3} $\abs x=y^2={\abs y}^2$. Thus $\abs y$ holds for the $z$ in \Tsbf{sqa}.
\\
The rest is left to the reader.
\end{proof}

The following lemma can be seen as a generalisation to the \nds case of the fact that on a  $ 2 $-closed discrete ordered field, the commutative ring structure completely determines the order structure.

%: Lemma{lemAr2cunique}
\begin{lemma} \label{lemAr2cunique}
On a commutative ring, if there is a  $2$-closed $f$-ring structure, it is unique. More generally, a ring morphism between two  $2$-closed $f$-rings is a $2$-closed $f$-ring morphism. 
\end{lemma}
%----------- end lemma ----------------------------------- 
%
\begin{proof}
Let $ \varphi\colon \gA\to\gB $ be a ring morphism where $\gA$ and $\gB$ are $2$-closed $f$-rings. Since $ x\geq 0 $ are squares, the order relation is respected. Now in a $2$-closed $f$-ring (or more generally in a reduced $f$-ring) the element $c=a\vu b$ is characterised by the equalities and inequalities $c\geq a$, $c\geq b$ and $(c-a)(c-b)=0$ (Lemma \ref{lemsupdansAonz}). We deduce that the ring morphism is also a morphism for $\vu$ laws. Finally, since in a reduced $f$-ring, two elements $\geq 0$ which have the same square are equal (Remark \ref {rem2closreduit}), the law $\Sqr$ is also respected by the ring morphism.
\end{proof}

Since the theory \Sa{Afr2c} is purely equational, any $f$-ring $\gA$ freely generates a  $2$-closed $f$-ring: its 2-closure $\AFRdC(\gA)$. The question then arises: what does the 2-closure of an $f$-ring look like? Here's the first clue. \index{2-closure!of an $f$-ring}\label{notaAFR2C} 

%l
%: Lemma{lemAfr2ccloture}
\begin{lemma} \label{lemAfr2ccloture}
Any reduced $f$-ring injects into its 2-closure. 
\end{lemma}
%----------- end lemma ----------------------------------- 
%
\begin{proof}
The theory \Sa{Afr2c} proves the same Horn rules as \Sa{Afrnz}: this follows from Item \textsl{3} of \pstfref{Pst2}, because the map~$ \Sqr $ added to the theory \Sa{Crcd} gives an essentially identical theory. We therefore do not obtain any new equality between elements of the original $f$-ring after formally adding the square roots of the elements $\geq 0$. 
\end{proof}
This generalises the fact that a discrete ordered field is injected into its 2-closure (\cite{LR91,LR90}), which is a discrete ordered field. For the \nds case arises the natural question \ref{quest2cloture}.

\Subsubsection{Axioms in order that a commutative ring be a \afrdc}
This paragraph clarifies Lemma \ref{lemAr2cunique}. It generalizes to the case of a commutative ring what was done for a discrete ordered field in order to make it $2$-closed.
It was enough to impose the axioms \tsbf{Aonz0} (reality) and \tsbf{AFR2} (any square is a power of 4, see proposition \ref{propcodi2clos}).
We propose the following system of axioms, which added to the theory of commutative rings, gives a theory essentially identical to \Sa{Afr2c}.
We must take $x\geq 0$ as an abbreviation of $\Exists z \; x=z^2$, and $x\geq y$ as an abbreviation of $x-y\geq 0$.

\DeuxRegles{
\laB{Aonz0} $\,\, x^2+y^2+z^2=0 \vd  x=0$
\laB{Afr21} $\,\, (x^2+y^2)^2= z^4 \vd  x^2+y^2 = z^2$
\Lab{Afr23} $\,\,  z\geq x\vet z\geq -x\vet x^2=y^4\vd z\geq y^2 $
}
{
\laB{AFR2} $ \vd \Exists y\; x^2=y^4$
\Lab{Afr22} $\,\, x^2= y^4\vet (y^2+x)^2=z^4 \vd  y^2+x = z^2$
}

Here are some explanations. 
Axiom \tsbf{Afr21} admits as a special case $u^4=v^4\vd u^2=v^2$.
This allows us to see that in \tsbf{AFR2} the element $y^2$ is uniquely determined from $x$ and can therefore be skolemised under the name of $\abs x$.
We see that the definition for $x\geq 0$ is equivalent to $x=\abs x$.
Then we must see that this function $\abs {\,\cdot\,}$ satisfies the axioms that we proposed for the definition of a lattice group structure on a given abelian group on \paref{grl-abs}.
We must introduce a constant~$\frac 1 2$ to have the axioms \Tsbf{2div1} and \Tsbf{2div2}. We thus obtain a lattice group structure on the additive group of the ring.\footnote{In fact $\gA$ is replaced by 
$\gA/\sqrt{\gen{0}}$ and if $nx=0$ for an integer $n\geq 1$, we must also cancel $x$. A more comfortable situation would be to suppose that we start from a reduced \QQlg. In this case $\gA$ is injected into the  \afrdc that we construct.}
We then simply need to check the validity of the axioms \Tsbf{ao1}, \Tsbf{ao2} and \Tsbf{afr6b} (\paref{ao1}), which is immediate.

\smallskip \rem We would like to be able to demonstrate that axioms \tsbf{Afr22} and \tsbf{Afr23}, which correspond to \Tsbf{abs3} and \Tsbf{Abs1}, are consequences of the other axioms.
\eoe

\Subsection{The case of a \ndsof} 

%l
%: Lemma{lemdefiAsr2c}
\begin{lemma} \label{lemdefiAsr2c} The \sa{Co2c} theory is essentially identical to the following two theories.
\begin{enumerate} 
 
\item On the signature $(\cdot=0\mathrel{;}\cdot+\cdot, \cdot\times\cdot,\cdot\vu\cdot,-\,\cdot,\,\Fr(\cdot),\,\Sqr(\cdot),0,1)$ the theory obtained from \sa{Afr2c} by adding the function symbol $\Fr$ with  axioms \Tsbf{fr1}, \Tsbf{fr2}  and  \Tsbf{AFRL}.
 
\item On the signature $\sIgt{\AsR}{\cdot=0,\cdot\geq 0,\cdot>0\mathrel{;}\cdot+\cdot, \cdot\times\cdot,\cdot\vu\cdot,-\,\cdot,0,1}$ 
 the theory obtained by adding as axioms to \Sa{Asr} the rules \Tsbf{IV}, \Tsbf{OTF}, \Tsbf{FRAC}, \Tsbf{Anz} and \Tsbf{sqr}.
\end{enumerate}
\end{lemma}
%----------- end lemma ----------------------------------- 

%%%%%%%%%%%%%%%%%%%%%%%%%%%%%%%%%%%%%%%%%%%%%%%%%%%%%%%%%%%%%%%%%%%%

\section{Virtual roots}\label{secCoVR}

%:Subsection Definition and first properties
\Subsection{Definition and first properties} \label{subsecrappelsvr}

References: \cite{GLM98,CLLR06,AG2012,BG2011,Gal2013}.

\smallskip The idea which guided the introduction of virtual roots was to have, for a real monic polynomial, continuous maps of the coefficients which cover the real roots. When a real root disappears in the complex plane, it can be replaced by the root of the derivative that coincides with the double real root when it disappears.

For example, the virtual square roots of an arbitrary real $a$ (i.e.\ $-\Sqr(a)$ and $+\Sqr(a)$) are equal to $-\sqrt a$ and $\sqrt a$ when $a\geq 0$, otherwise they are zero: this is the value they had when they disappeared (imagine the polynomial $X^2-a$ varying continuously with $a\in\gR$).

\smallskip First, let's recall a purely algebraic version of the mean value theorem in case of polynomials.

%: Lemma{lemaccrfinis}
\begin{lemma}[algebraic mean value theorem] \emph{\cite{LR90,LR91}} \label{lemaccrfinis} \\
We can construct two families $(\lambda_{i,j})_{1\leq i\leq j\leq n}$ and $(r_{i,j})_{1\leq i\leq j\leq n}$ in $\QQ\;\cap\; (0,1)\,$,  with $\sum_{i=1}^nr_{i,n}=1$ for all $n\geq 1$ and such that, for any polynomial $f\in\QQ[X]$ of degree $\leq n$, we have 
 in $\QQ[a,b]$:
\[ 
 f(b)-f(a)=(b-a)\times \som_{i=1}^nr_{i,n}\cdot f'(a+\lambda_{i,n}(b- a)).
\] 
The result applies to any $\QQ$-algebra $\gA$ (in particular to \emph{non}-\dofs). If $\gA$ is a strictly ordered $\QQ$-algebra, this shows that a polynomial whose derivative is $>0$ on an open interval $(a,b)$ is a strictly increasing map on the closed interval $\ClI{a,b}$. 
\end{lemma}
%--------- end lemma ----------------------------------- 

%e
%: Example{exaACF}
\begin{example} \label{exaACF}
For example, for polynomials of degree $ \leq 4 $ we have 
\[\ndsp 
\frac {f(1)-f(-1)}2= \frac1 3\,f'(-\frac2 3 )+
\frac1 6\,f'(-\frac1 3 )+
\frac1 6\,f'(\frac1 3 )+ 
\frac1 3\,f'(\frac2 3 ), 
\] 
and more generally, with $ \Delta=b-a $ 
\[\ndsp
f(b)-f(a)=\Delta\cdot\big(
\frac1 3\,f'(a+\frac1 6 \Delta)+
\frac1 6\,f'(a+\frac1 3 \Delta)+
\frac1 6\,f'(a+\frac2 3 \Delta)+ 
\frac1 3\,f'(a+\frac5 6 \Delta)\big). 
\]
\eoe
\end{example}
%--------- fin example ---------------------------------------------- 

%: lemBasicVirtualRoots
\begin{lemma}[slight variation on {\cite[Proposition 1.2]{GLM98}}] \label{lemBasicVirtualRoots}~
\begin{enumerate}
 
\item Let $ \sigma=\pm1 $ and $ f\colon \ClI{a,b}\to\RR $ ($ a\leq b \in \RR $) be a continuously differentiable map such that $ \sigma\,f'(x)>0 $ on $\; \ClI{a,b} \,$. Then $\abs f$ reaches its  minimum at a unique $x\in\ClI{a,b}$. We call this real $\rR (a,b,f,\sigma)$. \\
We have $(x-a)(x-b)f(x)=0$ and $x$ is the only real number satisfying the following system of inequalities. 
%-----------------begin item------------------

\TwoRegles{
\labu $ a \leq x \leq b $ 
\labu $ \sigma\, (x - a) f(a) \leq 0 $ 
\labu $ \sigma\, (b-x) f(b) \geq 0 $ 
}
{
\labu $ \sigma\, (x-a) f(x)\leq 0 $ 
\labu $ \sigma\, (b-x) f(x) \geq $0 
}

\item %[2.] 
Let $\sigma=\pm1$ and $ f\colon [\,a,+\infty) \to\RR$ be a continuously differentiable map such that $\sigma\,f'(x)>0$ on $\;(a,+\infty)\,$. It is assumed that there is a $ b>a $ such that $\sigma\,f(b)>0$. 

\noindent Then $\abs f$ reaches its  minimum at a single $ x\in[\,a,+\infty) $. We denote $ \rR (a,+\infty,f,\sigma) $ this real. We have $ (x-a)f(x)=0 $ and $x$ is the only real verifying the following system of inequalities: 
%-----------------begin item------------------

\TwoRegles{
\labu $ a \leq x $ 
\labu $ \sigma\,(x - a) f(a) \leq 0 $ 
}
{
\labu $ \sigma\,(x - a) f(x) \leq 0 $ 
\labu $ \sigma\, f(x) \geq 0 $ 
}

\item %[3.] 
A statement similar to the previous one, left to the reader, for the interval $ \;(-\infty,a\,] $.

\item %[4.]%
This lemma is also valid for a discrete real closed field $\gR$ if $f$ is a continuous semialgebraic map continuously derivable on an interval $ \ClI{a,b} $.
\end{enumerate}

%-----------------end item------------------
\end{lemma}

%r
%: Remark{remlemBasicVirtualRoots}
\begin{remark} \label{remlemBasicVirtualRoots} 
1) In the article \cite{GLM98}, when $f$ is a monic polynomial of degree $d$, the hypothesis is formulated in the form $ \sigma\,f'(x)\geq 0 $ on $ \ClI{a,b} $, which implies that the set of parameters ($a$, $b$ and the coefficients of $f$) satisfying the hypothesis is a semialgebraic closed subset of $\RR^{d+2}$. We then show that the map $\rR (a,b,f\sigma)$ is semialgebraically continuous on this closed set.
This will be the case here or \ref{prdfVirtualRoots} and \ref{thVirtualRoots}.

\smallskip \noindent 
2) Note that Items \textsl{2} and \textsl{3} are offset from Item \textsl{1}.
\smallskip \noindent 
3) We probably can give explicitely a \mcu for $ \rR $ if we give certain details about the continuous semialgebraic map $ f' $ (details available when $f$ is a monic polynomial). 
\eoe
 \end{remark}
%----------- end remark ---------------------------------- 

\smallskip From this lemma we obtain the construction of \textsl{virtual roots} for a monic polynomial of degree $d$: firstly they \gui{cover} all the real roots, and secondly they vary continuously as a function of the coefficients of the polynomial.

For a monic polynomial $f$ of degree $d$, we note $f^{[k]}$ the $k$-th derivative of $f$ divided by its leading coefficient $(0\leq k< d) $: it is a monic polynomial of degree $ d-k $.

%: PropDef{prdfVirtualRoots}
\begin{propdef} \label{prdfVirtualRoots} Let $\gR$ be a discrete real closed field or the field~$\RR$. For any monic polynomial 

%vspace{-.2em}
\snic{f(X) = X^{d} - ( a_{d-1} X^{d-1} + \cdots +a_1X+ a_0) \quad (d\geq 1)}

\noindent it is correct to define the maps \emph{virtual roots of $f$} 

\snic{\rho_{d,j}(f) = \rho_{d,j}(a_{d-1}, \ldots, a_0)}

\noindent for $1 \leq j \leq d$ by induction on $d$ in the following way (we abbreviate below $\rho_{\delta ,j}(f^{[d-\delta ]})$ to $\rho_{\delta,j}$). 
%-----------------begin item------------------

\vspace{.1em}
\begin{itemize}
\item $ \rho_{1,1}(X - a) = \rho_{1,1}(a) := a $;
\item $ \rho_{d,j} := \rR (\rho_{d-1,j-1}, \rho_{d-1,j}, f) $ 
for $ 1 \leq j \leq d $\quad  (when $ d\geq 2 $) ;
\end{itemize}
%-----------------end item------------------

\vspace{.2em}
\noindent By convention we have set $ \rho_{\delta,0}=(-1)^d\infty $ and $ \rho_{\delta,\delta+1}=+\infty $ for all $\delta\geq 1$, and the map $\rR$ is the one defined in Lemma \ref{lemBasicVirtualRoots}.

\smallskip \noindent In other words, the induction is correct because when $\rho_{d-1,j-1}< \rho_{d-1,j}$ we have $(-1)^{d-j}\,f'(x)>0$ on the corresponding open interval. 

\end{propdef}
%--------- end propdef ------------------------------

This proposition can be proved simultaneously with the Items \textsl{\ref{ivrvariation}} and \textsl{\ref{ivrsigne}} of the following theorem, using Lemmas \ref{lemaccrfinis} and \ref{lemBasicVirtualRoots}.

%: Theorem{thVirtualRoots}
\begin{theorem}[some properties of virtual roots] \label{thVirtualRoots} \emph{\cite{GLM98,CLLR06}}\\
Let $\gR$ be a discrete real closed field or the field $\RR$. Let $ \xi $ be an arbitrary element of the field.  We consider a  monic polynomial $f$ of degree $d$.
\begin{enumerate}
 
\item The $ \frac{d(d+1)}2 $ elements $ \rho_{\delta,j}(f^{[d-\delta]}) $ which have been introduced in Lemma \ref{lemBasicVirtualRoots} are characterised by a system of large inequalities.
 
\item Each map $ \rho_{d,j}:\gR^d\to\gR $ is uniformly continuous on any ball\footnote{$\mathrm{B}_{d,M}:=\sotQ{(a_{d-1},\dots,a_0)\,}{\,\sum_ia_i^2\leq M}$, ($M>0$). Uniform continuity can be given in fully explicit form à la {\L}ojasiewicz.}  $\mathrm{B}_{d,M} $.

\item 
We note \fbox{$\wi f=\prod_{j=1}^d(X-\rho_{d,j})$} and \fbox{$f\sta=\prod_{j=0}^{d-1} f^{[j]}$}, 
 \fbox{$\rho_{\delta,j}=\rho_{\delta,j}(f^{[d-\delta]})$}, and we set  conventions \fbox{$\rho_{\delta,0}=(-1)^\delta\infty$} et \fbox{$\rho_{\delta,\delta+1}=+\infty$}.
\begin{enumerate}
 
\item We have $ \rho_{d,1}\leq \rho_{d-1,1}\leq \dots\leq \rho_{d-1,j}\leq \rho_{d,j+1}\leq \dots\leq \rho_{d-1,d-1}\leq \rho_{d,d} $.
 
\item If $ d\geq 2 $ and $ f=X^d-a $, then $ \rho_{d,d}=\sqrt[d]{a^+} $, 
 $ \rho_{d,j}=0 $ for $ 1<j<d $, $ \rho_{d,1}+\rho_{d,d}=0 $ if $d$ is even and $ \rho_{d,1}+\rho_{d,d}=\sqrt[d]{a} $ if $d$ is odd.
 
\item \label{ivrsupinf} If $ f=\prod_{i=1}^d(X-\xi_i) $ for $ \xi_i\in\gR $, then $ \wi f=f $. Consequently $ \rho_{d,1}=\Vi_{i}\,\xi_i $, $ \rho_{d,d}=\Vu_{\!i}\,\xi_i $ and $ \rho_{d,\delta}=\Vi_{J\subseteq \lrb{1..d}, \#J=\delta}(\Vu_{i\in J}\,\xi_i) $.
 
\item \label{ivrvariation} 
\begin{itemize}
 
\item If $ \rho_{d-1,j}<\rho_{d-1,j+1} $, $ (0\leq j\leq d-1) $, then $f$ is strictly monotonic on the interval, increasing if $ d-j $ odd, decreasing otherwise.
 
\item For $ 0\leq j\leq d-1 $, we have $ (-1)^{d-j}\big(f(\rho_{d-1,j+1})-f(\rho_{d-1,j})\big)\leq 0 $.\footnote{This implies that in the system of large inequalities which characterises the $\rho_{\delta,j}$ for $1\leq j\leq\delta\leq d$, the sign $\sigma$ in Lemma \ref{lemBasicVirtualRoots}  can be given directly, as in the example which follows the theorem. This simplifies things a little: the $\sigma$ \gui{disappear}.}
\end{itemize}
 
\item \label{ivrsigne} 
 If $ \rho_{d,j}<\xi<\rho_{d,j+1} $, $ (0\leq j\leq d) $, then $ (-1)^{d-j}f(\xi)>0 $.
 
\item The zeros of $f$ are zeros of $ \wi f $, with multiplicity greater than or equal to $ \wi f $. More precisely 
%i
\begin{itemize}\itemsep.2em
 
\item If $ f(\xi)=0 $, then $ \wi f(\xi)=0 $;
 
\item If $ \abS{\wi f(\xi)} > 0 $, then $ \abs{f(\xi)} > 0 $;
 
\item If $ f^{[j]}(\xi)=0 $ for $ j\in\lrbk $, then 
 $ {\wi f}^{[j]}(\xi)=0 $ for $ j\in\lrbk $;
  
\item If $f^{[j]}(\xi)=0$ for $ j\in\lrbk $ and $ \abS{{\wi f}^{[\delta+1]}(\xi)}> 0 $, then $ \abs{f^{[\delta+1]}(\xi)}> 0 $.
\end{itemize}
Furthermore, if the multiplicities are known, the difference in multiplicities for $f$ and $ \wi f $ is even (for example, a non-zero $ \rho_{d,j} $ of $f$ is of even multiplicity in $\wi f$). 
 
\item The real zeros of $f\sta$ are exactly the $\rho_{\delta,j}$. More precisely
\begin{itemize}
 
\item each $\rho_{\delta,j}$ is a zero of $f\sta$,
 
\item if all $|\xi-\rho_{\delta,j}|$ are $>0$, then $|f^\star(\xi)|>0$, 
 
\item the polynomial $\wi f$ divides $(f\sta)^d$. 
\end{itemize}
 
\item \label{ivrBudan} \emph{(Budan Fourier count)} Let $a\in\gR$ be such that the $\abs{f^{[\delta]}(a)} > 0$ for $0\leq\delta\leq d$, and let~$r$ be the number of sign changes in the sequence of $f^{[\delta]}(a),\, (\delta=d,\dots,0) $, $ (0\leq r\leq d) $. \\
Then $ \rho_{d,d-r}<a<\rho_{d,d-r+1} $. 
 
\item \label{ivrTVI} \emph{(Intermediate Value Theorem)}\\ 
If $ a<b $ and $ f(a)f(b)<0 $, we have \smash{$ \prod_{j=1}^{d}f(\mu_j)=0 $, where \fbox{$ \mu_j=a\vu(b\vi\rho_{d,j}) $}.}. \\
Special cases.
\vspace{.1em}
\begin{itemize}
 
\item If $d$ is odd, then $ \prod_{j=1}^{d}f(\rho_{d,j})=0 $.
 
\item If $ 0\leq \delta<\ell\leq d $ and $ f(\rho_{d-1,\delta})f(\rho_{d-1,\ell})<0 $, then $ \prod_{j=\delta}^{\ell-1}f(\rho_{d,j})=0 $.
\vspace{.2em}
\item If, according to Item \ref{ivrBudan} we have $ \rho_{d,\delta}<a<\rho_{d,\delta+1}<b<\rho_{d,\delta+2} $, then $ f(\rho_{d,\delta+1})=0 $. 
\end{itemize}
 
\item \emph{(Extrema values)} \label{ivrExtrema} The monic polynomial $f$ \gui{attains its upper bound and its lower bound on any closed bounded interval} in the following precise sense: if $a<b$, we have
\vspace{-.9em} 
\[ 
\begin{array}{rcl} 
\sup_{\xi\in\ClI{a,b}}f(\xi) & = & f(a)\vu f(b) \vu \sup_{j=1}^{d-1}f(\nu_j) \quad \hbox{where}\;\; \fbox{$ \nu_j=a\vu(b\vi\rho_{d-1,j}) $}\,, \oups.3em] 
\inf_{\xi\in\ClI{a,b}}f(\xi) & = & f(a)\vi f(b) 
\vi \inf_{j=1}^{d-1}f(\nu_j) \,. 
 \end{array}
\]

\vspace{-.6em} 
If $f$ has a constant strict sign $ \sigma=\pm1 $ on $ \ClI{a,b} $, we have $ \inf_{\xi\in\ClI{a,b}}\big(\sigma\,f(\xi)\big)>0 $. 
 
\item \label{ivrMinAbs} \emph{(Minimum in absolute value)}~\\ If $ a<b $, we have

\vspace{-.2em}
\snic{\inf_{\xi\in\ClI{a,b}}\abs{f(\xi)}=
\abs{f(a)} \vi \abs{f(b)} \vi \inf_{j=1}^{d}\abs{f(\mu_j)}.
\qquad\qquad \phantom{a}}

Furthermore, if the second member is $ >0 $, then $f$ has a constant sign on $ \ClI{a,b} $.

\item \emph{(A bound)} If $ f(x)=x^d+\sum_{\delta=0}^{d-1}a_\delta x^\delta $ we have $ \abs{\rho_{d,j}}\leq \sup_{\delta=0}^{d}(1+\abs{a_\delta}) $ ($ 1\leq j\leq d $). 
 
\item \emph{(Change of variable)} Let $ f(x)=x^d+\sum_{\delta=0}^{d-1}a_\delta x^\delta $ and $ g(x)=x^d+\sum_{\delta=0}^{d-1}c^{d-\delta}a_\delta x^\delta $ (formally $ g(x)=c^df(x/c) $).
\begin{itemize}
 
\item If $ c\geq 0 $, we have $ \rho_{d,j}(g)=c\rho_{d,j}(f) $ ($ 1\leq j\leq d $).
 
\item If $ c\leq 0 $, we have $ \rho_{d,j}(g)=c\rho_{d,d+1-j}(f) $ ($ 1\leq j\leq d $). 
 
\item In all cases, $ \prod_{1\leq j\leq d}(x-c\rho_{d,j}(f))=\prod_{1\leq j\leq d}(x-\rho_{d,j}(g)) $.
\end{itemize}

\end{enumerate}
\end{enumerate} 
\end{theorem}
%--------- fin theorem ----------------------------------- 

%: Example{exavr}
\begin{example} \label{exavr} We explain here the inequalities mentioned in Item~\textsl{1} of \thref{thVirtualRoots} leading to $ \rho_{4,3}(f) $ for a polynomial $ f(X)=X^4-(a_3X^3+a_2X^2+a_1X+a_0) $, written here in the form of direct rules without hypotheses. We use the conventions of Item \textsl{3} of \thref{thVirtualRoots}. 
Thus, let 
$\rho_{1,1}=\rho_{1,1}(\frac {a_3} 4)$, 
$\rho_{2,j}=\rho_{2,j}(\frac {a_3} 2, \frac {a_2} 6)$,
$\rho_{3,j}=\rho_{3,j}(\frac {3a_3} 4, \frac {a_2} 2, \frac {a_1} 4)$,
$\rho_{4,j}=\rho_{4,j}(a_3,{a_2},{a_1},{a_0})$. 
The inequalities characterising $\rho_{1,1}$, $\rho_{2,2}$, $\rho_{3,2}$ and $\rho_{4,3}$ are given. 
Note that in the definition of virtual roots,
the sign $\sigma=\pm1$ before $x-a$ or $b-x$ in  \lemref{lemBasicVirtualRoots} is known because of Item \textsl{\ref{ivrvariation}} of \thref{thVirtualRoots}, which explains why
this sign does not appear in the inequalities below.
 
\Regles{\lab{vr$_{1,1}$} $\vd \rho_{1,1} = \frac {a_3} 4$}

\vspace{-.8em}
\DeuxRegles{
\lab{vr$_{2,1,0} $} $ \vd \rho_{2,1}\leq \rho_{1,1}
\phantom{(x^2-a^2)} $ 
\lab{vr$_{2,1,1}$} $ \vd (\rho_{2,1} - \rho_{1,1}) \, f^{[2]}(\rho_{1,1}) \leq 0 $ 
}
{
\lab{vr$_{2,1,2}$} $\vd (\rho_{2,1} - \rho_{1,1})\, f^{[2]}(\rho_{2,1}) \geq 0$
\lab{vr$_{2,1,3}$} $\vd f^{[2]}(\rho_{2,1}) \geq 0$
}

\vspace{-.8em}
\DeuxRegles{
\lab{vr$_{2,2,0} $} $ \vd \rho_{1,1}\leq \rho_{2,2}
\phantom{(x^2-a^2)} $ 
\lab{vr$_{2,2,1}$} $ \vd (\rho_{2,2} - \rho_{1,1}) \, f^{[2]}(\rho_{1,1}) \geq 0 $ 
}
{
\lab{vr$_{2,2,2}$} $\vd (\rho_{2,2} - \rho_{1,1})\, f^{[2]}(\rho_{2,2}) \leq 0$
\lab{vr$_{2,2,3}$} $\vd f^{[2]}(\rho_{2,2}) \geq 0$
}

\vspace{-.8em}
\TwoRegles{ 
\lab{vr$_{3,3,0}$} $\vd   \rho_{2,2}\leq \rho_{3,3}
\phantom{f^{[1]}(\rho_{3,3}}$
\lab{vr$_{3,3,1}$} $\vd  (\rho_{3,3} - \rho_{2,2})\, f^{[1]}(\rho_{1,1})\geq 0$
}
{
\lab{vr$ _{3,3,2} $} $ \vd (\rho_{3,3} - \rho_{2,2})\, f^{[1]}(\rho_{3,3}) \leq 0 $ 
\lab{vr$ _{3,3,3} $} $ \vd f^{[1]}(\rho_{3,3}) \geq 0 $ 
}

\vspace{-.8em}
\TwoRegles{ 
\lab{vr$ _{3,2,0} $} $ \vd \rho_{2,1}\leq \rho_{3,2}\leq \rho_{2,2} 
\phantom{f^{[1]}(\rho_{3,2})} $ 
\lab{vr$ _{3,2,1} $} $ \vd (\rho_{3,2} - \rho_{2,1})\, f^{[1]}(\rho_{2,1})\, \geq 0 $ 
\lab{vr$ _{3,2,2} $} $ \vd (\rho_{3,2} - \rho_{2,2})\, f^{[1]}(\rho_{2,2})\, \geq 0 $ 
}
{
\lab{vr$ _{3,2,3} $} $ \vd (\rho_{3,2} - \rho_{2,1})\, f^{[1]}(\rho_{3,2})\, \geq 0 $ 
\lab{vr$ _{3,2,4} $} $ \vd (\rho_{3,2} - \rho_{2,2})\, f^{[1]}(\rho_{3,2})\, \geq 0 $ 
}

\vspace{-.8em}
\TwoRegles{
\lab{vr$ _{4,3,0} $} $ \vd \rho_{3,2}\leq \rho_{4,3}\leq \rho_{3,3}
\phantom{f(\rho_{4,3})} $ 
\lab{vr$ _{4,3,1} $} $ \vd (\rho_{4,3} - \rho_{3,2})\, f(\rho_{3,2})\, \geq 0 $ 
\lab{vr$ _{4,3,2} $} $ \vd (\rho_{4,3} - \rho_{3,3})\, f(\rho_{3,3})\, \geq 0 $ 
}
{
\lab{vr$ _{4,3,3} $} $ \vd (\rho_{4,3} - \rho_{3,2})\, 
f(\rho_{4,3})\geq 0 $ 
\lab{vr$ _{4,3,4} $} $ \vd (\rho_{4,3} - \rho_{3,3})\, f(\rho_{4,3})\,\geq 0 $ 
}
 
\eoe
\end{example}
%--------- fin example ---------------------------------------------

\Subsection{A result à la Pierce-Birkhoff}
%: Subsection{à la Pierce-Birkhoff}

We call \textsl{polyroot map} a map $ \gR^m\to \gR $ which can be written in the form $ \rho_{d,j}(f_1, \dots, f_d) $ for integers $ 1\leq j\leq d $ and polynomials $ f_j\in\gR[\xm] $.%.
\index{polyroot map}

\smallskip The following theorem à la Pierce-Birhoff is worth noting. It looks like a Nusllstellensatz: it expresses that there is a purely algebraic reason for a map being semialgebraic continuous and integral over the ring of polynomials. 

%: Theorem{thPBpolyroots}
\begin{theorem} \label{thPBpolyroots} \emph{(\cite[Theorem 6.4]{GLM98})}
Let $\gR$ be a discrete real closed field and let $ g\colon \gR^m\to\gR $ be an continuous semialgebraic map integral on the ring $ \Rxm $ (seen as a ring of functions). Then $g$ is a combination by $\vu$, $\vi$ and $ \,+\, $ of polyroot maps $ \gR^m\to\gR $. 
More precisely, if $ g(\ux) $ is a root of the  $Y$-monic polynomial $  P(Y,\ux)$ of degree $d$, it is expressed as a sup-inf combination of maps of the form

\vspace{-.8em}
\begin{equation} \label {eqthPBpolyroots}
\rho_{d,j}(P)+\sqrt[r]{ R_\ell^+ \cdot\bigg(1+{\som_{i=1}^nx_i^2}
\bigg)^s}
\end{equation}

\vspace{-.4em}
\noindent for $ R_\ell\in\Rxm $ (the second term in the sum \pref{eqthPBpolyroots} is also a polyroot map, see Item {3b} of \thref{thVirtualRoots}). 
\end{theorem}
%--------- end theorem -----------------------------------

\rem
When the map $g$ is piecewise polynomial, it cancels a monic polynomial $P(Y)=\prod_{i=1}^d(Y-f_i)$ for $f_i\in\Rxm$. In the expression obtained by \ref{eqthPBpolyroots} for $g$, it is the 
{\L}ojasiewicz inequality which is responsible for the extraction of the $r$-th root in the formula. As for the $\rho_{d,j}(P)$ they are sup-inf combinations of the $f_i$ (Item \textsl{\ref{ivrsupinf}} of \thref{thVirtualRoots}).
\eoe

\Subsection{$f$-rings with virtual roots}\label{subsecafrvr}

\begin{example} \label{exavr2} 
We take again Example \ref{exatotordnonreduced} of the $\QQ$-linearly ordered algebra $\QQ[\alpha]$, with $\alpha>0$ and $\alpha^6=0$. We will see that the constraints imposed on $\rho_{2,2}(f)$, when $f=X^2-a^2$ and $a\geq 0$, do not necessarily imply that $\rho_{2,2}=a$. The constraints are as follows for $x=\rho_{2,2}$ (note that $\rho_{1,1} = 0 $):

\TwoRegles{
\lab{vr$_{2,2,0}$} $\vd   0\leq x
\phantom{(x^2-a^2)}$
\lab{vr$_{2,2,1}$} $\vd  - x \, a^2 \leq 0$
}
{
\lab{vr$ _{2,2,2} $} $ \vd x \, (x^2-a^2) \leq 0 $ 
\lab{vr$ _{2,2,3} $} $ \vd (x^2-a^2) \geq 0 $ 
}

\smallskip\noindent 
If we take $ a=\alpha $, all $ x\geq 0 $ such that $ x^2=a^2 $ fit, and therefore all $ \alpha+y\alpha^5 $ for $ y\in\QQ[\alpha] $ are solutions. If we take $ a^2=0 $ the constraints are equivalent to \gui{$ x\geq 0 $ and $ x^3\leq 0 $} and any element of the interval $  \ClI{0,\alpha^2} $ is a solution, including $ \zeta=\alpha^2 $ whereas $ \zeta^2> 0 $.
\eoe
\end{example}
%--------- fin example ---------------------------------------------- 

%: Lemma{lem2afrvrreduced}
\begin{lemma} \label{lem2afrvrreduced}
On an $f$-ring, the system of inequalities satisfied by the virtual roots $\rho_{k,j}$ ($1\leq j\leq k\leq d$) for a given monic polynomial of degree $d$, if they exist, defines these elements unambiguously. 
\end{lemma}
%--------- end lemma ----------------------------------- 
%
\begin{proof}
The uniqueness in question is expressed by means of Horn rules. The theory \Sa{Afrnz} proves the same Horn rules as the theory of discrete real closed fields with sup (Formal \pstref{Pst2}). In the latter theory, uniqueness is guaranteed (Item \textsl{1} of \thref{thVirtualRoots}).
\end{proof}

Example \ref{exavr2} and Lemmas \ref{lem2afrvrreduced} and \ref{lemafrvrreduced} justify the following definition. 

\begin{definition} \label{defiAfrvr} \index{f-ring@$f$-ring!with virtual roots} ~
\begin{enumerate}
 
\item 
 The purely equational theory \SA{Afrrv} of \textsl{$f$-rings with virtual roots} is obtained as follows from the purely equational theory \Sa{Afr}.

\vspace{.1em} \begin{itemize}
 
\item For $1\leq j\leq d$ in $\N$, we add a function symbol $\rho_{d,j}$ of arity $d$;
 
\item as axioms we add the inequalities described in Item \textsl{1} of \thref{thVirtualRoots};
  
\item we add the following rule \tsbf{vrsup} 

\regles{\Lab{vrsup} $ \vd \rho_{2,2}(a+b,-ab)=a\vu b $.}
\end{itemize} 
The signature is therefore as follows:
\sIgt{\AfRrv}{\cdot=0 \mathrel{;}\cdot+\cdot, \cdot\times\cdot,\cdot\vu\cdot,-\,\cdot,(\rho_{d,j})_{1\leq j\leq d},0,1}.\label{NOTASigAfrrv}
 
\item 
In the same way, we define the Horn theory \SA{Asrrv} of \textsl{strict $f$-rings with virtual roots} from the Horn theory \Sa{Asr}:\\
\sIgt{\AsRrv}{\cdot=0,\cdot\geq 0,\cdot>0 \mathrel{;}\cdot+\cdot, \cdot\times\cdot,\cdot\vu\cdot,-\,\cdot,(\rho_{d,j})_{1\leq j\leq d},0,1}.\label{NOTASigAsrrv}
\end{enumerate}

\end{definition}
%--------- fin definition --------------------------------

%: Lemma{lemafrvrreduced}
\begin{lemma} \label{lemafrvrreduced}
An $f$-ring with virtual roots is reduced. 
\end{lemma}
%--------- end lemma ----------------------------------- 
%
\begin{proof}
Given the rule \Tsbf{vrsup}, if $ a^2=0 $, we have with $ b=-a $: 
\[
0=\rho_{2,2}(0,0)=\rho_{2,2}(a+b,-ab)=\abs a.
\]
\vspace{-2em}

\end{proof}

%: paragraph{Domain variant}
\paragraph{Domain variant}

%: Definition{defiAfrvrbis}
\begin{definition}[$f$-ring with virtual roots, domain variant] \label{defiAfrvrbis} ~
\\ The Horn theory \SA{Aitorv} of \textsl{linearly ordered domains with virtual roots} is obtained from the Horn theory \Sa{Aito} of linearly ordered domains by adding the virtual roots in the same way as the theory \sa{Afrrv} is obtained from the theory \sa{Afr} in Definition \ref{defiAfrvr}. 
\end{definition}
%--------- end definition -------------------------------- 

Note that we don't need to put the rule \tsbf{vrsup} in the axioms.

%l
%: Lemma{lemAitorv}
\begin{lemma} \label{lemAitorv}
A linearly ordered domain with virtual roots is integrally closed and its field of fractions is discrete real closed. Reciprocally, an integrally closed domain whose fraction field is discrete real closed is an integrally closed domain with virtual roots.
\end{lemma}
%----------- end lemma ----------------------------------- 
%
\begin{proof}
Let $\gA$ be the domain and $\gK$ its field of fractions, which is discrete.

\noindent 
\textsl{Direct implication}. A monic polynomial $ f\in\AX $ satisfies \RCFn\ because of Item \textsl{3i} of \ref{thVirtualRoots} and the fact that $\gK$ is discrete. For an arbitrary polynomial of $ \KX $ we use the change of variables in Item \textsl{3m} to reduce to a monic polynomial of $ \AX $. So $\gK$ is discrete real closed. Finally $\gA$ is integrally closed due to Item \textsl{3f}.

\noindent \textsl{Reciprocal implication}. The order on $\gK$ induces a total order on $\gA$. It must be shown that for a monic polynomial $ f\in\AX $ the $ \rho_{d,j}(f) $ are in $\gA$. Now they are zeros of $ f^\star $, a monic polynomial of $ \AX $, so they are in $\gK$, and $\gA$ is integrally closed, so they are in $\gA$. Thus the maps $ \rho_{d,j}(a_0,\dots,a_n) $ defined from $ \gK^n $ to $\gK$ are restricted to maps $ \gA^n\to\gA $. 
\end{proof}
%
%: paragraph{Rings of continuous integer semialgebraic maps}
\paragraph{Rings of integral continuous  semialgebraic maps}

\smallskip \thref{thPBpolyroots} (for discrete real closed fields) legitimates the following definition.

%: Definota{defiSacem}
\begin{definota} \label{defiSacem}
Let $\gR$ be an $f$-ring with virtual roots (special cases: an \hyperref[defiCorv]{ordered field with virtual roots} or a real closed ring). The families $ \Sace_m(\gR) $ ($ m\in\N $) of \textsl{integral continuous semialgebraic maps} are defined as the families of maps  $\gR^m\to\gR$ stable by composition, containing the polynomial maps (with coefficients in~$\gR$) and the virtual root maps. 
In other words, an element of $\Sace_m(\gR)$ is a map $\gR^m\to\gR$ defined by a term of the language of $\Sa{Afrrv}(\gR)$ with the $m$ variables $\xm$ (some of which may be absent). 
\end{definota}
%--------- fin definota -------------------------------- 

%%%%%%%%%%%%%%%%%%%%%%%%%%%%%%%%%%%%%%%%%%%%%%%%%%%%%%%%%%%%%%%%%%%%
\paragraph{Pierce-Birkhoff rings}
%: paragraph{Pierce-Birkhoff rings}

%: Definota{defiSacembis}
\begin{definotas} \label{defiSacembis}
Let $\gA$ be a ring, or more generally a dynamic algebraic structure of an $f$-ring.
\begin{enumerate}
 
\item The ring $\AFRNZ(\gA)$ is the \textsl{reduced $f$-ring generated by $\gA$}. 
 
\item The ring $\AFRRV(\gA)$ is the $f$-ring with virtual roots generated by $\gA$.
 
\item The ring $\PPM(\gA)$ is defined as the $f$-subring of $\AFRRV(\gA)$ formed by the elements~$x$ which cancel a polynomial $\prod_{i=1}^k(X-a_i)$ for $a_i\in\gA$. 
 
\item A ring $\gA$ is called a \textsl{Pierce-Birkhoff ring} when the natural morphism $ \AFRNZ(\gA)\to\PPM(\gA) $ is an isomorphism.
\end{enumerate} 
\end{definotas}
%--------- end definition --------------------------------

See Question \ref{questpbring}. 

%%%%%%%%%%%%%%%%%%%%%%%%%%%%%%%%%%%%%%%%%%%%%%%%%%%%%%%%%%%%%%%%%%%%
\section{Real closed rings}\label{secArc}

%%%%%%%%%%%%%%%%%%%%%%%%%%%%%%%%%%%%%%%%%%%%%%%%%%%%%%%%%%%%%%%%%%%%
\Subsection{Constructive definition and variants}\label{subsecArcconst}
%: Subsection{Constructive definition}

%: Definition{defiArc}
\begin{definition}[real closed rings] \label{defiArc} \index{real closed! ring} 
 The purely equational theory \SA{Arc} of \textsl{real closed rings} is obtained by adding to the theory \Sa{Afrrv} the function symbol $\mathrm{Fr}$ and the axioms \Tsbf{fr1} and \Tsbf{fr2}.
\\
The signature is therefore as follows:
\Sigt{\ArC}{\cdot=0\mathrel{;}\cdot+\cdot, \cdot\times\cdot,\cdot\vu\cdot,-\,\cdot,(\rho_{d,j})_{1\leq j\leq d},\mathrm{Fr}(\cdot,\cdot),0,1}\label{NOTASigArc}
\end{definition}

\vspace{-1em}

%: Lemma{lemArcunique}
\begin{lemma} \label{lemArcunique}
On a commutative ring, if there is a real closed ring structure, it is unique. More generally, a ring morphism between two real closed rings is a real closed ring morphism. 
\end{lemma}
%----------- end lemma ----------------------------------- 
%
\begin{proof}
Results from the lemmas \ref{lemAr2cunique} and \ref{lem2afrvrreduced}.
\end{proof}

 \CAdre{.9}{In the following, when we do not specify otherwise, a \textsl{real closed ring} always designates a ring defined in \ref{defiArc}.}

%l
%: Lemma{lemArcbis}
\begin{lemma}[variants for \sa{Arc}] \label{lemArcbis} ~
\begin{enumerate}
 
\item The theory \sa{Arc} can also be obtained from the Horn theory \Sa{Aftr} by adding the virtual roots in the same way that the theory \sa{Afrrv} is obtained from the theory~\sa{Afr} in Definition \ref{defiAfrvr}. Moreover, given Lemma \ref{lemafrvrreduced}, the axiom \Tsbf{Anz} of the theory \sa{Aftr} can be omitted. A real closed ring can therefore be seen as a \emph{strongly real ring with virtual roots}.
 
\item The theory \sa{Arc} is essentially identical to the theory \Sa{Asrrv} of strict $f$-rings with virtual roots to which we add the function symbol $ \mathrm{Fr} $ and the axioms \Tsbf{fr1} and \Tsbf{fr2}. NB. The predicate $ x>0 $ must be added to \sa{Arc} as an abbreviation of \gui{$x$ \hbox{is $\geq 0$} and invertible}.
\end{enumerate}
\end{lemma}
 %----------- end lemma ----------------------------------- 
%
\begin{proof} Item \textsl{1} is clear. We deduce Item \textsl{2} by recalling Lemma \ref{lemArftr0}.
\end{proof}
%

%: paragraph{Rings of continuous semialgebraic functions}
\paragraph{Continuous semialgebraic maps}

\smallskip 
We now take Definition \ref{defiFSAGC2} (legitimised by \thref{thParamcontFsagc0}) and extend it to real closed rings. Note that every real closed ring contains a conformal copy of $\RRa$.

We also assume that we have proved \thref{thArc3} and its corollaries.
%d
%: Definota{defiFSAGC2+}
\begin{definota} \label{defiFSAGC2+} 
Let $\gR$ be a real closed ring and  a map $ f\colon \gR^n\to \gR $.\index{continuous semialgebraic map!on a real closed ring}.
\begin{enumerate}
 
\item (Elementary case) The map $f$ is said to be \textsl{semialgebraic continuous} (in an elementary way) if there exists a semialgebraic continuous map $ g\colon \RRa^{n}\to\RRa $ expressed by a term $ t(\xn) $ of \sa{Arc} and if $f$ coincides with the map defined by this term.
 
\item (General case) The map $f$ is \textsl{semialgebraic continuous} if there exists an integer $ r\geq 0 $, elements $ y_1,\dots,y_r\in \gR $ and a map $ h\colon \gR^{r+n}\to \gR $ which belongs to the previous elementary case such that 
\[
\forall \xn\in\gR\;\; f(\xn)=h(\yr,\xn).
\]
\end{enumerate}
We denote $ \Sac_n(\gR) $ the ring of these maps (it is a real closed ring for the natural order relation).\label{Sacn} \thref{thParamcontFsagc0} shows that for a discrete real closed field $\gR$ we find the usual definition of continuous semialgebraic maps.
\end{definota}
%----------- fin definota -------------------------------- 

For a comparison of $ \Sac_n(\gR) $ with $ \Sace_n(\gR) $ see the question \ref{Qu-polrootsafrvr}.

%%%%%%%%%%%%%%%%%%%%%%%%%%%%%%%%%%%%%%%%%%%%%%%%%%%%%%%%%%%%%%%%%%%%
%: paragraph{An example}
\paragraph{An example}
%p
%: Proposal{propfsagccovrsup}
\begin{proposition} \label{propfsagccovrsup}
Let $\gR$ be an $f$-ring with virtual roots and let \hbox{$ f\colon \gR^n\times \gR^p\to \gR $} be a continuous semialgebraic map. The map 
\[
g\colon \gR^p\to \gR,\,\ux \mapsto \sup\nolimits_{\uz\in \ClI{0,1}^n}f(\uz,\ux)
\] 
is well-defined and continuous semialgebraic. 
\end{proposition} 
%----------- end{proposition} ----------------------------- 
%
\begin{proof}
Given the ad hoc definition of the rings $ \Sac_{m}(\gR) $ we are immediately reduced to the case where $ \gR=\RRa $.
\end{proof}

See also the questions \ref{questArc2}.
%%%%%%%%%%%%%%%%%%%%%%%%%%%%%%%%%%%%%%%%%%%%%%%%%%%%%%%%%%%%%%%%%%%% 
%:Subsection Ordered fields with virtual roots
\Subsection{Ordered fields with virtual roots}\label{subsecCo0rv}

%: Definition{defiCorv}
\begin{definition} \label{defiCorv} (Compare with Definition \ref{defiAfrvr}, see Lemma \ref{lem2afrvrreduced}).
\begin{enumerate}
 
\item The dynamical theory \SA{Corv} of \textsl{ordered fields with virtual roots} is obtained as follows from the dynamical theory \Sa{Co}
of \ndsofs.\index{ordered field!with virtual roots}

\vspace{.1em} \begin{itemize}
 
\item For $ 1\leq j\leq d $ in $ \N $, we add a function symbol $ \rho_{d,j} $ of arity $d$;
 
\item As axioms, we add the inequalities described in Item \textsl{1} of \thref{thVirtualRoots}.
\end{itemize} 
The signature is therefore as follows:
\Sigt{\CoRv}{\cdot=0,\cdot\geq 0,\cdot>0\mathrel{;}\cdot+\cdot, \cdot\times\cdot,\cdot\vu\cdot,-\,\cdot,(\rho_{d,j})_{1\leq j\leq d},\mathrm{Fr}(\cdot,\cdot),0,1}\label{NOTASigCorv}

\vspace{-1em}

\item The dynamical theory \SA{Co--rv} is obtained in the same way from the theory \Sa{Co--}.
 
\item The dynamical theory \SA{Codrv} is obtained in the same way from the theory \Sa{Cod}.
\end{enumerate}
\end{definition}
%----------- end definition -------------------------------- 
%
\rem The theory \Sa{Codrv} is essentially identical to the theory obtained by adding to \Sa{Co0rv} the axiom \gui{of third excluded} \Edineq\ (see remark \ref{remCo0}).

%:Subsection Positivstellensatz formel
\Subsection{Formal Positivstellensatz}\label{subsecPst3}

%: pstf {Pst3}
\begin{pstf}[formal Positivstellensatz, 3] \label{Pst3}  \index{Positivstellensatz!formal --- !for ordered fields, 3} ~
\begin{enumerate}
 
\item The theories \Sa{Codrv}, \Sa{Crcd} and \Sa{Crcdsup} are essentially identical.
 
\item 
The following dynamical theories prove the same Horn rules (written in the language of \Sa{Afrrv}).
\begin{enumerate}
 
\item The theory \Sa{Afrrv} of $f$-rings with virtual roots.
 
\item The theory \Sa{Arc} of real closed rings.
 
\item The theory \Sa{Codrv} of discrete ordered fields with virtual roots.
\end{enumerate}
\item 
The following dynamical theories prove the same Horn rules (written in the language of \Sa{Asrrv}).
\begin{enumerate}
 
\item The theory \Sa{Asrrv} of strict $f$-rings with virtual roots.
 
\item The theory \Sa{Corv} of ordered fields with virtual roots.
 
\item The theory \Sa{Codrv} of discrete ordered fields with virtual roots.
\end{enumerate}
 
\item \label{i4Pst3} \thref{thVirtualRoots} is entirely valid for the theory \sa{Asrrv} (thus also for \sa{Corv}). The same is true for the purely equational theory \sa{Afrrv} (so also for \sa{Arc}) if the points which use the predicate~$ \cdot> 0 $ are deleted or suitably reformulated \hbox{with $ \cdot\geq 0 $}. 
\end{enumerate}
\end{pstf}
%%%%%%%%%%%%%%%%%%%%%%%%%%%%%%%%%%%%%%%%%%%%%%%%%%%%%%%%%%%%%%%%%%%% 
%
\begin{proof} The first point is clear. Items \textsl{2} and \textsl{3} are therefore variants of Formal \pstref{Pst2} taking into account \thref{thEseqMemesfaits}. 
\\
Finally, for Item \textsl{4}, it follows from previous Items that the assertions of Items \textsl{2} and \textsl{3} of \thref{thVirtualRoots} can be written in the form of Horn rules.
\end{proof}

In classical mathematics given the general representation theorem \ref{thcolsimralg}, the formal Positivstllensätze stated so far give the following results, which can be seen in classical mathematics as characterising the dynamical theories under consideration. 
%c
%: Corollary{corPst3}
\begin{corollaryc} \label{corPst3}
On their respective signatures, the following objects are all isomorphic to sub-\sa{T}-structures of products of discrete real closed fields (considered with the predicate $ x>0 $ and the maps $ \sup $, $ \mathrm{Fr} $ and~$ \rho_{d,j} $).
\begin{itemize}
 
\item A reduced $f$-ring (theory \Sa{Afrnz}).
 
\item A reduced strict $f$-ring (theory \Sa{Asrnz}). 
 
\item A strongly real ring (theory \Sa{Aftr}). 
 
\item A local $f$-ring (theory \Sa{Co}). 
 
\item An $f$-ring with virtual roots (theory \Sa{Afrrv}). 
 
\item A real closed ring (theory \Sa{Arc}). 
 
\item A strict $f$-ring with virtual roots (theory \Sa{Asrrv}). 
 
\item An ordered field with virtual roots (theory \Sa{Corv}).
%%
%item 
\end{itemize}

\end{corollaryc}
%--------- fin corollary ------------------------------- 
%

\Subsection{Quotient, localisation and gluing of real closed rings}
%: Subsection{Quotient, location and gluing of real closed rings}

%l
%: Lemma{lem1Afrc}
\begin{lemma}[quotient structure] \label{lem1Afrc} 
Let $\gA$ be a real closed ring and $ I $ a radical ideal. Then $ \gA/I $ is a real closed ring. 
\end{lemma}
%----------- end lemma ----------------------------------- 
%
\begin{proof} Let us first show that the radical ideal $ I $ is solid. We must first show that if $ x\in I $, then $ \abs x \in I $: indeed $ {\abs x}^2=x^2\in I $. Then if $ 0\leq \abs x\leq y $ \hbox{with $ y\in I $}, we must show that $ x\in I $. Now, by \tsbf{FRAC}, $y$ divides $ \abs x^2=x^2 $, so $ x^2\in I $, then $ x\in I $. The quotient~$ \gA/I $ is therefore an $f$-ring. Next we need to see that the virtual root maps $ \rho_{d,j} $ and the map \gui{fraction} $ \mathrm{Fr} $ \gui{pass to the quotient}. Now these maps, when they exist in a reduced $f$-ring, are well-defined by the systems of inequalities they satisfy (Lemma \ref{lem2afrvrreduced}). As these inequalities pass to the quotient, everything is in order. 
\end{proof}
%

%l
%: Lemma{lem2Afrc}
\begin{lemma}[localisation] \label{lem2Afrc}
Let $\gA$ be a real closed ring and $S$ be a monoid. Then $ S^{-1}\gA $ is a real closed ring. 
\end{lemma}
%----------- end lemma ----------------------------------- 
%
%
\begin{proof} We already know that $ S^{-1}\gA $ has a $\vu$ law which makes it an $f$-ring. Let's see what happens to the virtual roots. Let's take Example \ref{exavr} with a polynomial $ f(X)=X^4-(\frac{a_3}{s}X^3+\frac{a_2}{s}X^2+\frac{a_1}{s}X+\frac{a_0}{s}) $ with $ a_i\in\gA $ and $ s\in S^+ $. In $ S^{-1}\gA $ with $ Y=sX $ we have 
\[
s^4f(X)=Y^4-(a_3Y^3+sa_2Y^2+s^2a_1Y+s^3a_0)=Y^4-(b_3Y^3+b_2Y^2+b_1Y+b_0)=g(Y)
\] 
and therefore also $ s^4g(\frac Y s)=f(X) $. 
Consider the virtual roots $ \rho_{i,j} $ for the monic polynomial~$g$ of $ \gA[Y] $ with $ (i,j) $ equal to $ (1,1) $, $ (2,1) $, $ (2,2) $, $ (3,3) $, $ (3,2) $, $ (4,3) $, 
and finally the $ \rho'_{i,j}=\frac{\rho_{i,j}}s $. 
We see that the inequalities in Example \ref{exavr}, just as they are satisfied for the $ \rho_{i,j} $ with respect to the polynomial $g$ in $ \gA[Y] $, are ipso facto satisfied for the $ \rho'_{i,j} $ with respect to the polynomial $f$ in $ S^{-1}\gA[X] $. These inequalities completely characterise the virtual roots when they exist (Lemma \ref{lem2afrvrreduced}).
\\
A similar reasoning works for the map $ \mathrm{Fr}(\cdot,\cdot) $.
\end{proof}
%

%: Plgc {plcc.arc}
\begin{plcc}[concrete gluing of real closed rings]\label{plcc.arc} ~
\\
Let $ S_1 $, $ \dots $, $ S_n $ be comaximal monoids of a ring $\gA$. Let $ \gA_i $ denote $ \gA_{S_i} $, $ \gA_{ij} $ denote $ \gA_{S_iS_j} $, and assume that a structure of type \Sa{Arc} is given on each $ \gA_i $. It is further assumed that the images in $ \gA_{ij} $ of the laws of $ \gA_i $ and $ \gA_j $ coincide. Then there exists a unique structure of real closed ring on $\gA$ which induces by localisation in each $ S_i $ the structure defined on $ \gA_i $. This real closed ring is identified with the projective limit of the diagram
\[
\big(( \gA_i)_{i<j\in\lrbn},( \gA_{ij})_{i<j\in\lrbn};(\alpha_{ij})_{i< j\in\lrbn}\big),
\]
where $ \alpha_{ij} $ are localisation morphisms, in the category of real closed rings.
\end{plcc}
%--- end-plgc-----------------------------------------
%
\begin{proof}
We copy, mutatis mutandis, the proof of the concrete local-global principle~\ref{plcc.frings} for $f$-rings.
\end{proof}

\rems 1) This implies that the notion of a real closed scheme is well-defined.

\noindent 2) The analogous concrete local-global principles, with the same proof, are valid for reduced $f$-rings, for strongly real rings, and for $f$-rings with virtual roots.
\eoe

%%%%%%%%%%%%%%%%%%%%%%%%%%%%%%%%%%%%%%%%%%%%%%%%%%%%%%%%%%%%%%%%%%%%
%:Subsection
\Subsection{Comparison with the definition in classical mathematics}

References: \cite{Sch84,PN2002,Tre2007}.

\smallskip The structure of \textsl{real closed ring} is defined by N. Schwartz in a very abstract way in his Phd~\cite[Schwartz, 1984]{Schthesis}. An axiomatisation as a coherent theory was proposed in \cite[Prestel \& Schwartz, 2002]{PN2002} (see Definition \ref{defiPrScArc} and Proposition \ref{propPrScArc}). 

The aim of N. Schwartz was to give an abstract description of the rings of continuous semialgebraic maps on semialgebraic closed subsets for a fixed real closed field $\gR$, and to define abstract \gui{real closed spaces}.

%:paragraph{An axiomatic of Niels Schwartz
\paragraph{An axiomatic of Niels Schwartz}~\label{subsecAxScw}

\smallskip Here is the definition of real closed rings in classical mathematics given in \cite[Schwartz, 1986]{Sch84}.

%d
%: Definition{defiArcclama}
\begin{definitionc} \label{defiArcclama} 
 A ring \textsl{real closed} is a reduced ring $\gA$ satisfying the following properties.
\begin{enumerate}
 
\item The set of squares of $\gA$ is the set of $\geq 0$ elements of a partial order which makes $\gA$ an $f$-ring.
 
\item If $0\leq a\leq b$, there exists $z$ such that $zb=a^2$ (\textsl{convexity axiom}).
 
\item For any prime ideal $\fp$, the residual ring $\gA/\fp$ is integrally closed and its field of fractions is a real closed field.
\end{enumerate}
\end{definitionc}
%----------- end definition -------------------------------- 

%l

%: Proposal{propAfrc}
\begin{propositionc} \label{propAfrc} 
In classical mathematics, Definitions \ref{defiArc} and \ref{defiArcclama} are equivalent. 
\end{propositionc}
%----------- end proposition -----------------------------
%
\begin{proof}
\textsl{Direct.} For a real closed ring $\gA$ of Definition \ref{defiArcclama}, the virtual root maps are well-defined, as we know that all continuous semialgebraic maps defined on $\RRa$ are defined on $\gA$. The same applies to the map \gui{fraction} $\mathrm{Fr}$. 

\smallskip\noindent \textsl{Reciprocal.} For a real closed ring $\gA$ of Definition \ref{defiArc}, we must show that Item \textsl{3} of Definition~\ref{defiArcclama} is satisfied. Given a prime ideal $\fp$, the residual ring $\gA/\fp$ has no zerodivisors  and is therefore linearly ordered (Lemma \ref{lemAfrsdz}). It is also a real closed ring by Lemma \ref{lem1Afrc}. Lemma \ref{lemAitorv} tells us that $\gA/\fp$ is integrally closed and that its field of fractions is a real closed field.
\end{proof}
%

%%%%%%%%%%%%%%%%%%%%%%%%%%%%%%%%%%%%%%%%%%%%%%%%%%%%%%%%%%%%%%%%%%%%
%:paragraph{The Prestel-Schwartz axiomatics
\paragraph{The axiomatics of Prestel-Schwartz}~\label{subsecAxPrSc}

\smallskip\noindent The article \cite[Prestel \& Schwartz, 2002]{PN2002} shows in classical mathematics that the real closed ring structure of Definition \ref{defiArcclama} is described by a coherent theory. The existential axioms proposed by the authors to replace Item \textsl{3} of \ref{defiArcclama} are very sophisticated and the proof is also an astonishing tour de force.
%d
%: Definition{defiPrScArc}
\begin{definition} \label{defiPrScArc} \textsl{(Prestel-Schwartz real closed rings)}
A commutative ring is said to be real closed if it satisfies the following axioms. 
\begin{itemize}
 
\item [i-iv)] The commutative ring $\gA$ is reduced, the elements $\geq 0$ are exactly the squares and the order relation makes $\gA$ a   convex $f$-ring (axiom \Tsbf{CVX}) 
 
\item [v)] For each $ d\geq 1 $, let $ f(x)=x^{2d+1}+\sum_{k=0}^{2d}a_kx^k $, $ \delta=\mathrm{discr}_x(P) $ its discriminant, and $ g(x)=x^{2d+1}+\sum_{k=0}^{2d}\delta^{2(2d+1-k)}a_kx^k $, we pose the axiom 
\[
\vd \Exists z \;g(z)=0
\]
 
\item [vi)] For each $ d\geq 1 $ we pose the axiom
\[
\,\, x^d+\som_{k=0}^{d-1}a_kx^ky^{d-k}=0\vd \Exists z_1,\dots,z_d \;y(x-z_1y)\cdots(x-z_dy)=0
\]
\end{itemize}
\end{definition}
%p
%: Proposal{propPrScArc}
\begin{proposition} \label{propPrScArc}
In the theory \sa{Arc} the axioms of Definition \ref{defiPrScArc} are valid \rdys.
\end{proposition}
%----------- end{proposition} ----------------------------- 
%
\begin{proof}
Recall that \thref{thVirtualRoots} is fully valid in the theory \sa{Arc} (Item~\textsl{\ref{i4Pst3}} of \ref{Pst3}).
 
\noindent Let's look at the axiom \textsl{vi)}. Let $ f=x^d+\som_{k=0}^{d-1}a_kx^k $ and $ g=x^d+\som_{k=0}^{d-1}a_kx^ky^{d-k} $. 
\\
We denote $ \wi f=\prod_{1\leq j\leq d}(x-\rho_{d,j}(f)) $ and $ \wi g=\prod_{1\leq j\leq d}(x-\rho_{d,j}(g)) $. Item \textsl{3m} of \thref{thVirtualRoots} gives the equality 
\[
\prod\nolimits_{1\leq j\leq d}(x-y\rho_{d,j}(f))=\prod\nolimits_{1\leq j\leq d}(x-\rho_{d,j}(g)).
\]
Moreover, Item \textsl{3f} of \thref{thVirtualRoots} for the polynomial $g$ gives
\[
\,\, x^d+\som_{k=0}^{d-1}a_kx^ky^{d-k}=0\vd (x-\rho_{d,1}(g))\cdots(x-\rho_{d,d}(g))=0.
\]
We therefore obtain in the \sa{Arc} theory, by taking $ z_k=\rho_{d,k}(f) $, the valid rule 
\[
\,\, x^d+\som_{k=0}^{d-1}a_kx^ky^{d-k}=0\vd \;(x-z_1y)\cdots(x-z_dy)=0.
\]
\noindent Let's look at the axiom \textsl{v)}. We will show that the element $z$ whose existence is asserted can be chosen as a continuous semialgebraic map of the parameters $ a_k $. Since this map is cancelled by the monic polynomial $ Q $ we then conclude by the theorem \gui{ à la Pierce-Birkhoff}~\ref{thPBpolyroots}. Given Formal Positivstellensatz \ref{Pst3} (Item \textsl{2}) we need only prove the validity of the rule in the theory \sa{Codrv}. Let us therefore consider a discrete real closed field and, in the parameter space, a connected component of the open $ \so{delta\neq 0} $. On this connected component, the real zeros of $f$ are simple (there is at least one because the degree is odd) and vary continuously as a function of the parameters. Those of $g$ are simply multiplied by $ \delta^2 $. So on this connected component we have the element $z$ sought as a continuous semialgebraic function of the parameters by choosing the largest of the real zeros. As we approach an edge of a related component, these zeros tend towards $0$ (they are zeros of $f$ multiplied by $ \delta^2 $). So these continuous semialgebraic maps join together to form a global continuous semialgebraic map. 
\end{proof}

In classical mathematics, the reciprocal implication is demonstrated: the Prestel-Schwartz axioms imply the existence of virtual roots (because they are continuous semialgebraic maps). This gives the equivalence in classical mathematics of our axiomatics and that of Prestel-Schwartz. 

%:paragraph{Tressl's axiomatics
\paragraph{The axiomatics of Marcus Tressl}~\label{subsecAxTr}

\smallskip\noindent 
A more elementary version, similar to the one we propose, for the theory of real closed rings can be found in \cite[Tressl, 2007]{Tre2007} (see also \cite{Sch84,Sch84b,Sch97,Sch89}). In this paper, a real closed ring is an $f$-ring $\gR$ on which are given all continuous semialgebraic maps defined on $\RRa$, and for which all algebraic identities linking these maps on $\RRa$ are satisfied in $\gR$. 

A good analysis of the classical mathematical articles on real closed rings should allow us to understand why it is enough to add the fractions allowed by the rule \Tsbf{FRAC} to an $f$-ring with virtual roots to be able to capture all the continuous semialgebraic maps $ {\RRa}^m\to\RRa $. This is the subject of the following concrete results, which are valid in classical mathematics, but for which we would like a constructive proof. See in particular the question \ref{questArc3}.

Recall that according to the finiteness theorem (\cite[Theorem~2.7. 1]{BCR}) the graph $ G_f=sotq{(\ux,y)}{\ux\in\gR^n, y=f(\ux)} $ of a continuous semialgebraic map $ f\colon \RRa^n\to\RRa $ is a semialgebraic closed subset of~$ \RRa^{n+1} $ which can be described as the zero set a \textsl{semipolynomial map} $ F: \RRa^{n+1}\to \RRa $, \textsl{i. e.} a map written in the form
\[
\sup\nolimits_i\,(\inf\nolimits_{1\leq j\leq k_i}\,p_{ij}) \quad \hbox{ where }p_{ij}\in\RRa[\xn,y]
\]
We can decide whether such a graph is that of a continuous semialgebraic map. The following theorem means that in such a case we can prove the existence of $y$ depending on $x_i$ directly in the purely equational theory \sa{Arc}. 

%: theorem{thArc3}
\begin{theorem} \label{thArc3} 
Any continuous semialgebraic map $ \RRa^n\to\RRa $ can be defined by a term of the theory~\sa{Arc}.
\end{theorem}
%----------- fin theorem ----------------------------- 
%c
%: Corollary{corthArc3}
\begin{corollary} \label{corthArc3}
The axiomatisation proposed in \ref{defiArc} for real closed rings is equivalent to that proposed by Tressl \cite{Tre2007}. 
\end{corollary}
%--------- fin corollary ------------------------------- 

%: Corollary{cor2thArc3}
\begin{corollary} \label{cor2thArc3}
Let $\gR$ be a real closed ring. Any continuous semialgebraic map $ \gR^n\to\gR $ (Definition \ref{defiFSAGC2+}) is defined by a term of $ \sa{Arc}(\gR) $ with $ n $ free variables. 
\end{corollary}
%--------- end corollary ------------------------------- 

When the axiom \Tsbf{OTF} is added, \thref{thArc3} gives the following corollary.
%: corollary{corCrc1Corv}
\begin{corollary} \label{corCrc1Corv} 
The theories \Sa{Crc1} and \Sa{Corv} are essentially identical.
 \end{corollary}
%----------- end corollary --------------

The following corollary is more problematic, can we return to the case $ \gR=\RRa $ ?

%: corollary{cor3thArc3}
\begin{corollary} \label{cor3thArc3} 
Consider a real closed ring $\gR$, a continuous semialgebraic map  $ g\colon \gR^n \to \gR $ (an element of $ \SaC_n(\gR) $) and a polynomial $ p\in\gR[\xn] $ with at least one invertible coefficient. 
We assume that, on the set $ \sotq{\uxi\in\gR^n}{\abS{p(\uxi)}> 0} $, the fraction $ f=g/p $ satisfies a \mcu on all bounded susbsets à la {\L}ojasiewicz (as in Lemma \ref{factfsagcLoja}). 
Then there exists a unique continuous semialgebraic map \hbox{$ h\colon \gR^n \to \gR $} such that $ hp=g $.
\end{corollary}
%----------- end corollary ----------------------------- 
%
\begin{proof}
\fbox{?? give precisions.}
\end{proof}
%

%%%%%%%%%%%%%%%%%%%%%%%%%%%%%%%%%%%%%%%%%%%%%%%%%%%%%%%%%%%%%%%%%%%%

\section{\textsl{Non} discrete real closed fields }\label{secCrc2}

%:Subsection Definition \label{subsecdefiCrc2}
\Subsection{A reasonable definition} \label{subsecdefiCrc2}%
\index{real closed!non dis@\textsl{non} discrete --- field}

%: Lemma{lemArcloc}
\begin{lemma} \label{lemArcloc}
A real closed ring is local if, and only if, it satisfies the rule \Tsbf{AFRL}.
\end{lemma}
%----------- end lemma ----------------------------------- 
%
\begin{proof} See Lemma \ref{lemAftrloc}.
\end{proof}

We now propose for the theory of \ndrcfs a formulation essentially identical to \Sa{Corv}, but almost purely equational. The rule \tsbf{AFRL} is preferred to the rule \Tsbf{OTF} because we do not introduce the predicate $\cdot>0$ which would take us out of the purely equational theories for \Sa{Arc}.

%d
%: Definition{defiCrc2}
\begin{definition} \label{defiCrc2} 
 The \textsl{dynamical theory of \ndrcfs}, denoted \SA{Crc2}, is the extension of the purely equational theory \sa{Arc} obtained by adding the rule \tsbf{AFRL}. In other words, a \ndrcf is nothing other than a \fbox{local real closed ring}.\index{non discrete@\nds!real closed field}.
\end{definition}
%----------- end definition -------------------------------- 

%p
%: Proposition{propCrc2}
\begin{proposition} \label{propCrc2}
The theories \Sa{Corv}, \Sa{Crc1} and \Sa{Crc2} are essentially identical (we must define the predicate $ >0 $ which we add to \sa{Crc2}). 
\end{proposition}
%----------- end{proposition} ----------------------------- 
%
\begin{proof} 
Corollary \ref{corCrc1Corv} compares \sa{Corv} and \sa{Crc1}. Lemma \ref{lemArftr} tells us that a \ndsof is none other than a local strongly real ring. In other words, the theory \sa{Co} is essentially identical to the theory \sa{Aftr} to which we add the axiom \tsbf{AFRL}. Let's start with \sa{Aftr}. If we add the virtual roots then \tsbf{AFRL} we pass to \sa{Arc} (Lemma \ref{lemArcbis} Item~\textsl{1}) then to \sa{Crc2}. If we add \tsbf{AFRL} then the virtual roots we go to \sa{Co} then to \sa{Corv}.
\end{proof}
%

%r
%: Remark{rem}
\begin{remarks} \label{remCrc}~

\noindent 1) The field $\RR$ is a constructive model of the theory \sa{Crc2}.

\smallskip\noindent 2) The theory \sa{Crcd} of discrete real closed fields is essentially identical to the theory obtained by adding to \sa{Crc2} the axiom~\Edineq\ which says that equality is decidable.

\smallskip\noindent 3) The theory \sa{Crc2} is nothing other than the theory of local real closed rings. However, there are local real closed rings which are not fields in Heyting's sense. For example, consider the ring $\gA$ of continuous semialgebraic maps on $\RRa$, and let $ \gB=S^{-1}\gA $ where $S$ is the monoid of maps~$f$ such that $ f(0)\neq 0 $. It is the ring of germs at $ (0) $ of maps~\hbox{$ f\in\gA $}. An element $ f\in\gA $ is $ >0 $ in $\gB$ (resp. $\leq 0$ in $\gB$) if, and only if, $ f(0)>0 $ in $\RRa$ (resp.~\hbox{$ f(x)\leq 0 $} in the neighbourhood of $0$). This shows that \Tsbf{HOF} is not satisfied in $\gB$, because it is not enough for $ f(0)\leq 0 $ for $f$ to be $\leq 0$ in the neighbourhood of $0$. Note that this locally real closed ring admits two minimal prime ideals, with the respective locals being the germs of maps to the right (or left) of $0$. 

\smallskip\noindent 4) The theory \Sa{Crc2} can be used to prove the existence of a square root for a complex number of modulus 1. The unit circle $ \so{x^2+y^2=1} $ is covered by the open areas $ \so{x>-1} $ and $ \so{x<1} $, on each of which the existence is guaranteed by a continuous map. However, this existence cannot be proved in \Sa{Arc}, because in this purely equational theory, every existence is certified by a term, and every term defines a continuous semialgebraic map.
\eoe
\end{remarks}
%----------- end remark ---------------------------------- 

\Subsection{Real closure of a reduced $f$-ring?}
%: Subsection{Real closure of a reduced $f$-ring}

Given a reduced $f$-ring $\gA$, we know (\pstref{Pst2}) that the theory $ \Sa{Crcdsup}(\gA) $ proves the same Horn rules as $ \Sa{Afrnz}(\gA) $. The same applies to all intermediate theories, in particular to the theories \Sa{Afrrv} and \Sa{Arc}. 

As the latter are purely equational theories, the reduced $f$-ring $\gA$ gives rise to an $f$-ring with virtual roots $ \AFRRV(\gA) $ and a real closed ring $ \ARC(\gA) $.\label{notaARC} 

Since the theories \sa{Afrnz}, \sa{Afrrv} and \sa{Arc} prove the same Horn rules, $\gA$ is a substructure (of $f$-ring) of $ \AFRRV(\gA) $ which is itself a substructure (of $f$-ring with virtual roots) of $ \ARC(\gA) $. In other words, adding the symbols for virtual roots and fractions (with their axioms) does not change $\gA$ as an $f$-ring. 

These two constructions of \gui{real fences} are without mystery, and unique to within a single isomorphism.

We are in the same situation as for the construction of the real closure of a discrete ordered field (\cite{LR90,LR91}), but here the result seems completely obvious whereas it requires a non-negligible effort in the articles quoted. The main reason for this (very small) miracle is that we are relying here on a constructive proof of the Positivstellensatz. The secondary reason is that we are dealing here only with Horn theories (instead of dynamical theories). 

%r
%: Remark{remcloturereelle}
\begin{remark} \label{remcloturereelle} 
A construction of the real closure of a discrete ordered field $\gK$ can also be obtained according to the following argument. We begin by establishing the simultaneous collapse of the theory of discrete ordered fields and that of discrete real closed fields (as in \cite[Theorem 3.6]{CLR01}). This is a variant of the formal Positivstellensatz. Then we dynamically evaluate~$\gK$ as a discrete real closed field. This forces us to introduce the real zeros of any polynomial, with a Thom coding for each of them (for a polynomial which cancels this zero). Since no ambiguity is possible, the dynamic algebraic structure constructed is in fact a usual algebraic structure of a real closed field. This construction is admittedly less detailed than the one explained in \cite{LR91}, but it is essentially equivalent. In fact, in the other direction, we could probably deduce Theorem 3.6 of \cite{CLR01} from the construction given in \cite{LR91}. What improves \cite{Lom91} and \cite{CLR01} on the previous result is, on the one hand, that the formal Positivstellensatz is more general (Theorem 3.8 in \cite{CLR01}), and on the other hand, and above all, that the concrete Positivstellensatz is demonstrated. 
\eoe
\end{remark}
%----------- end remark ---------------------------------- 

\Subsection{Real closure of a \ndsof}
%: Subsection{Real closure of a \ndsof}
 
Let us consider a discrete ordered field, \textsl{i.e.} a model $\gK$ of the theory \Sa{Co}. We know that \sa{Corv} proves the same Horn rules as \sa{Co}.

Let us denote $\gR$ the dynamic algebraic structure $ \sa{Corv}(\gK) $. 

All the closed terms of the dynamic algebraic structure $\gR$ are constructed on elements of $\gK$ by means of the function symbols given in the signature (polynomials, virtual roots, legitimate fractions).

The dynamic structure $\gR$ is a natural candidate to be \und{the} (usual) algebraic structure of type \Sa{Corv} generated by~$\gK$, if one exists. However, the problem is that $\gR$ is a dynamic algebraic structure of type \sa{Corv}, but not necessarily a model of this theory, because this dynamical theory is defined with non-algebraic axioms. 

We can first consider the usual real closed ring algebraic structure $ \ARC(\gK) $ which is identified with the dynamic algebraic structure $ \sa{Arc}(\gK) $. The question is: is the axiom \Tsbf{AFRL} a valid rule in $ \ARC(\gK) $? In other words, is $ \ARC(\gK) $ a model of~\sa{Corv}? In which case we can identify $\gR$ (dynamic algebraic structure) and $ \ARC(\gK) $ (usual algebraic structure). 
\label{pagevirtualclosure}

The answer is not obvious (see Question \ref{questCloturevirtuelle}). 

%%%%%%%%%%%%%%%%%%%%%%%%%%%%%%%%%%%%%%%%%%%%%%%%%%%%%%%%%%%%%%%%%%%%
%%%%%%%%%%%%%%%%%%%%%%%%%%%%%%%%%%%%%%%%%%%%%%%%%%%%%%%%%%%%%%%%%%%%

\section{A non-archimedean \ndrcf} \label{crcndna}

 \newcommand{\Pu}[1]{\RRa\big[\big[\vep^{1/#1}\big]\big]}

In this section we describe what we believe\footnote{The proof of \thref{propPuiseux} is not so clear.} to be an example of
a non-archimedean \ndrcf.

\noindent Let  $ \vep $ be an indeterminate. In Section \ref{condna} we introduced the ordered non-archimedean \ndsof $ \gQ=\gZ[1/\vep] $ where $ \gZ=\QQ[[\vep]] $ is the ring of formal series in $ \vep $ with rational coefficients where $ \vep $ is a strictly positive infinity.

In fact, the coefficients of the series under consideration could have been taken from any discrete ordered field, in particular from the field $\RRa$ of algebraic real numbers. We will denote $\gR_0=\RRa[[\vep]]$ the analogue of $\gZ$ and $ \gR=\gR_0[1/\vep] $ the analogue of $\gQ$.

\smallskip We now extend these constructions to the field $\gP$ of Puiseux series with real algebraic coefficients. 

First we have the rings of series $\gP_{0,d}=\Pu{d}$ for the integers $d\geq 1$, all isomorphic to $\gR_0$, with the inclusion morphisms $\gP_{0,d}\to \gP_{0,dd'}$. This forms an inductive system whose limit $\gPo$ (the Puiseux series of valuation~\hbox{$\geq 0$}) can be seen as the union of $\gP_{0,d}$. 

Finally, the Puiseux series themselves form the ring defined as $\gP:=\gPo[1/\vep]$.

\smallskip Note that $ \gP_{j,d}=\sotq{\alpha\in\gP_{0,d}[1/\vep]}{ \alpha/\vep^{j/d}\in \gP_{0,d}} $. We have $ \gP=\Vu_{j,d}\gP_{j,d} $.

We introduce notations which generalise to $ \gP_{j,d} $ those already given for $\gR$. These notations are consistent with the inclusions $ \gP_{j,d}\subseteq \gP_{jd',dd'} $.

Let $ \alpha=\sum_{k=j}^\infty a_{k/d}\vep^{k/d}\in \gP_{j,d}\subseteq \gP_{0,d}[1/\vep] $ ($ j,d\in\ZZ, d\geq 1 $). We define:
\begin{itemize}
 
\item $ \rc_{\ell/d}(\alpha)=
\formule 
{0&\hbox{ if }\ell<j,\\  
a_{\ell/d} &\hbox{ if } \ell\geq j.
} $ 
 
\item $ \kappa_{\ell/d}(\alpha)=s_{\ell/d}\in\so{-1,0,1} $ is defined by recurrence as follows:\\
 $ s_{\ell/d}=\formule {\hbox{ if }\ell<j, & \hbox{then }
0&
\\
\hbox{if }\ell\geq j,& 
\formule{\hbox{if }s_{(\ell-1)/d}\neq 0,& \hbox{alors } s_{(\ell-1)/d},\\
\hbox{if }s_{(\ell-1)/d}= 0,& \hbox{then }\hbox{sign of }a_{\ell/d}.} 
} $ 
 
\item $ \alpha>0 $ means $ \exists \ell\geq j\; \kappa_{\ell/d}(\alpha)=1 $. 
 
\item $ \alpha\geq 0 $ means $ \forall \ell\geq j\; \kappa_{\ell/d}(\alpha)\geq 0 $. 
 
\item $ \rv(\alpha)>k/d $ means $ \kappa_{k/d}(\alpha)=0 $.
 
\item $ \rv(\alpha)\leq k/d $ means $ \kappa_{k/d}(\alpha)=\pm1 $.
 
\item $ \rv(\alpha)\geq k/d $ means $ \kappa_{(k-1)/d}(\alpha)=0 $.
 
\item $ \rv(\alpha)=k/d $ means $ \rv(\alpha)\geq k/d $ and $ \rv(\alpha)\leq k/d $.
%
%item 
%
\end{itemize}
 
\smallskip From the previous study in Section \ref{condna} which led to Proposition \ref{propQ[[T]]} for the ring $\gZ$ and to \thref{propQ((T))} for the ring $\gQ$, we deduce analogous results for the rings $ \gP_{0,d} $ then for $ \gPo $, then for~$\gP$. 
%: Proposition{propP_0}
\begin{proposition} \label{propP_0}~
\begin{enumerate}
 
\item 
The ring $ \gPo $ is a reduced strict $f$-ring which satisfies the following properties. 
\begin{itemize}
 
\item It satisfies the rules \Tsbf{OTF}, \OTFx, \Tsbf{FRAC} and \Tsbf{Val2}. In particular (Lemma \ref{lemAfrnzFRAC}) the \fsagc~$\Fr$ satisfying the rules \Tsbf{Fr1} and \Tsbf{Fr2} is well-defined and the corresponding function symbol can be considered as part of the signature. 
 
\item This is a residually discrete henselian local ring. 
 
\item Its residual field is isomorphic to $\RRa$. \\
We have $ \gPo\eti=\sotq{\xi\in\gPo}{\kappa_0(\xi)=\pm1} $ and $ \Rad(\gPo)=\sotq{\xi\in\gPo}{\kappa_0(\xi)=0} $.

\item The valuation group is isomorphic to $ (\QQ,+,\geq) $ (the class of $ \vep $ corresponds to the element~$ 1 $ of~$\QQ$).
 
\item The elements $\geq 0$ are squares: the ring $ \gPo $ is a  $ 2 $-closed $f$-ring (theory \Sa{Asr2c}).
 
\item More generally, the elements $\geq 0$ are powers $k$-th of elements $\geq 0$. Since we are dealing with unique existence, we can introduce the corresponding function symbols in the signature.
 
\item Furthermore, the ordered Heyting axiom $ \lnot(\xi>0)\Rightarrow\xi\leq 0 $ is satisfied.
\end{itemize}
 
\item 
The ring $\gP$ is a reduced strict  $f$-ring which satisfies the following properties. 
\begin{itemize}
 
\item An element is $ >0 $ if, and only if, it is $\geq 0$ and invertible.
 
\item The rules \Tsbf{OTF}, \tsbf{OTF $ \eti $}, \Tsbf{FRAC} and \Tsbf{IV} (a fortiori \Tsbf{Val2}) are satisfied. In particular (Lemma \ref{lemAfrnzFRAC}) the continuous semialgebraic map~$ \Fr $ satisfying the rules \tsbf{Fr1} and \tsbf{Fr2} is well-defined.
 
\item It is a local ring with $ \Rad(\gP)=0 $ (a Heyting field in the terminology of \cite{CACM} or \cite{MRR}).
 
\item The elements $\geq 0$ are squares of elements $\geq 0$: the ring $\gP$ is a  $2$-closed strict $f$-ring (theory \Sa{Asr2c}).
 
\item More generally, the elements $\geq 0$ are powers $k$-th of elements $\geq 0$. Since we are dealing with unique existence, we can introduce the corresponding function symbols in the signature.
 
\item The ordered Heyting axiom $ \lnot(\xi>0)\Rightarrow\xi\leq 0 $ is satisfied. 
\end{itemize}
\end{enumerate}

\end{proposition} 
%----------- end of proposal ----------------------------- 
%
\begin{proof}
Only the fact that the elements $\geq 0$ are powers $k$-th of elements $\geq 0$ is a really new point which requires a proof. This is left to the reader.
\end{proof}

We denote $ \gPal $ the integral closure of $ \RRa(\vep) $ in $\gP$: this is the ring of Puiseux series which are integral over the discrete ordered subfield $ \QQ(\vep) $.

In the following we will use the notion of extension by continuity. To talk about extension by continuity, we need to define the notion of a convergent sequence, and check that the usual rules for boundary crossing work for this notion. 
%d
%: Definition{deficonvinP}
\begin{definition}[convergent sequences in $\gP$ ] \label{deficonvinP}
We will say that \textsl{the sequence~$ (\alpha_n)_{n\in\N} $ converges towards $\alpha$ in~$\gP$} if there exist $ j $ and $ d\in \ZZ $ with $ d\geq 1 $ such that 
\begin{itemize}
 
\item $\alpha$ and the $ \alpha_n $ are all in $ \gP_{j,d} $,
 
\item $ \lim_n\rv(\alpha_n-\alpha)=+\infty $, i.e.\ again:

\smallskip 
\centerline{$ \forall k\geq j \, \exists N \, \forall m\geq N\, \forall \ell\in \lrb{j.. k}\; \rc_{\ell/d}(\alpha_m)=\rc_{\ell/d}(\alpha) $. } 
\end{itemize}
We will then write $ \alpha=\lim_n\alpha_n $. 
 
\end{definition}
%----------- end definition -------------------------------- 

We can easily establish the following properties.
%p
%: Proposition{propconvinP}
\begin{proposition} \label{propconvinP}~
\begin{enumerate}
 
\item $ \lim_n \alpha_n=0 $ if, and only if, $ \lim_n \rv(\alpha_n)=+\infty $.
 
\item If $ \lim_n \alpha_n=\alpha $ then $\alpha$ is invertible if, and only if, $ \exists N\,\forall n>N \;\rv(\alpha_n)=\rv(\alpha_N)<infty $. In this case $ \alpha^{-1}=\lim_{n\geq N}\alpha_n^{-1} $.
 
\item If $ \alpha=\lim_n \alpha_n $, $ \beta=\lim_n \beta_n $ and $ a\in\RRa $, then 
\begin{itemize}
 
\item $ a\alpha=\lim_na\alpha_n $,
 
\item $ \alpha+ \beta=\lim_n(\alpha_n+ \beta_n) $,
 
\item $ \alpha \beta=\lim_n(\alpha_n \beta_n) $,
 
\item $ \alpha\vu \beta=\lim_n(\alpha_n\vu \beta_n) $,
 
\item $ \Fr(\alpha,\beta)=\lim_n\Fr(\alpha_n,\beta_n) $ and
 
\item $ ({\alpha^+})^q=\lim_n(\alpha_n^+)^q $ ($ q\in\QQ,q>0 $).
\end{itemize}
 
\item All $ \alpha\in \gP_{j,d} $ is the limit of the sequence of Laurent polynomials $ \pi_{m}(\vep^{1/d}) $ for $ m\geq j $ obtained by truncation of the series $\alpha$, defined precisely by 
\[
\pi_{m}\in\RRa[\vep^{1/d}][1/\vep],\;\pi_{m}=\som_{k:j\leq k<m}\rc_{k/d}(\alpha)\vep^{k/d}
\] 
We also note that $ \RRa[\vep^{1/d}][1/\vep]\subseteq \gPal $. 
\end{enumerate}
 \end{proposition} 
%----------- end statement ----------------------------- 

What we hope is to prove the following theorem.

%: theorem{propPuiseux}
\begin{theorem} \label{propPuiseux}
The ring $ \gP=\gPo[1/\vep] $ satisfies all the axioms of the theory \Sa{Crc2}. It is therefore a discrete Heyting non-archimedean real closed field.
\end{theorem}
%----------- end of proposition ----------------------------- 

%
\begin{proof}[First proof] We are going to generalise the passage to the limit properties described in Proposition \ref{propconvinP} to all continuous semialgebraic maps defined on $\RRa$.

\noindent The paper \cite{CM2013} shows that $ \gPal $ is a discrete real closed field. It is therefore a real closure of $ \RRa(\vep) $, constructed in a very different way from that proposed in~\cite{LR91}. Now consider a cube $ [-a,a]^r=K\subseteq \RRa^r $ and a continuous semialgebraic map $ f\colon K\to\RRa $. Since $ \gPal $ is a discrete real closed field, $f$ extends uniquely into a continuous semialgebraic map $ f_1\colon K_1\to\gPal $, where $ K_1\subseteq \gPal^r $ is defined by the same system of inequalities as $ K $. We will show that $ f_1 $ extends by continuity into a map $ f_2\colon K_2\to\gP $, where $ K_2\subseteq \gP^r $ is defined by the same system of inequalities as $ K $. This will suffice to show that $\gP$ is a model of \sa{Crc1}.\footnote{The details of this statement are left to the reader.}
%l
%: propdef{lemconvfsagc}
\begin{propdef} \label{lemonvfsagc} We apply the previous notations for $ K\subseteq K_1\subseteq K_2 $. Let $ f\colon K\to\RRa $ be a continuous semialgebraic map and $ f_1\colon K_1\to \gPal $ be its extension to $ \gPal $. Then for any sequence $ (\alpha_n) $ in $ \RRa[\vep,\vep^{-1}]^r\cap K_2 $ which converges to a $ \alpha\in\gP $, the sequence $ f_1(\alpha_n) $ converges in $\gP$. The limit depends only on $\alpha$ and is denoted $ f_2(\alpha) $. 
\end{propdef}
%----------- end propdef ----------------------------------- 
\begin{proof} Not so simple! First we have to see that the $f_1(\alpha_n)$'s belong to a given $\gP_{j,D}$; next a \L ojasievicz inequality could be used for the convergence. \fbox{?? Precisions}
\end{proof}
\end{proof}
\begin{proof}[Second proof]
Given Proposition \ref{propP_0} it suffices to prove the existence property of virtual roots for the ring $\gP$.
One solution is perhaps to take up proposition \ref{lemonvfsagc} by limiting oneself to extending by continuity the virtual root functions, which could be done by induction.
\fbox{?? Precisions} 
\end{proof}

%%%%%%%%%%%%%%%%%%%%%%%%%%%%%%%%%%%%%%%%%%%%%%%%%%%%%%%%%%%%%%%%%%%%
%%%%%%%%%%%%%%%%%%%%%%%%%%%%%%%%%%%%%%%%%%%%%%%%%%%%%%%%%%%%%%%%%%%%

\section{Use of virtual roots in constructive real algebra}\label{secRVARC}
The results stated in this subsection for the real number field also seem valid in the dynamical theory \Sa{Corv}. Some may require only~\Sa{Co0rv} or~\Sa{Arc}.

%--- SubSubsection{TDY}---- 
\Subsection{Basic semialgebraic subsets of the real line}
%: Subsection{Basic semialgebraic subsets of the real line}
%-----------

Let us define a \textsl{basic semialgebraic closed subset} of the real line as a subset of the form $\rF\!_f=\sotq{x\in\RR}{f(x)\geq 0}$ for an $ f\in\RRX $.

\smallskip\noindent \textsl{First example}. Consider the polynomials $ f(X)=X^2-b $ and $ g=-f $.
%i
\begin{itemize}
 
\item If $ b<0 $, we have $ \rF\!_f=\RR $ and $ \rF\!_g=\emptyset $.
 
\item If $ b>0 $, we have $ \rF\!_f=\,]-\infty,-\sqrt b]\cup[\sqrt b,+\infty[ $ and $ \rF\!_g=[-\sqrt b,\sqrt b] $.
  
\item If $ b=0 $, we have $ \rF\!_f=\RR $ and $ \rF\!_g=\so 0 $.
\end{itemize}
To obtain such a precise description of these semialgebraic closed subsets it is absolutely necessary to know the sign of $ b=-f(0) $. 
\\
If we denote $\alpha$ and $\beta$ the virtual roots of $f$, we have the following alternative description. 
%i
\begin{itemize}
 
\item If $\alpha<\beta$, i.e.\ if $ f\big(\frac {\alpha+\beta} 2\big)>0 $ we have $ \rF\!_f=\,]-\infty,\alpha]\cup[\beta,+\infty[ $ and $ \rF\!_g=[\alpha,\beta] $.
 
\item If $\alpha=\beta$ and $f(\alpha)<0$, i.e.\ if $f\big(\frac {\alpha+\beta} 2\big)<0 $, we have $\rF\!_f=\RR$ and $\rF\!_g=\emptyset$.
 
\item If $\alpha=\beta$ and $f(\alpha)=0$, that is if $f\big(\frac {\alpha+\beta} 2\big)=0$, we have $\rF\!_f=\RR$ and $\rF\!_g=\so \alpha$.
\end{itemize}

\smallskip\noindent \textsl{Second example}. \\
The case of a monic polynomial $f$ of degree $ \delta>2 $. Let us denote $\mathrm{Vr}_f$ the list of its virtual roots. \thref{thVirtualRoots} allows us to describe the adherence of $ \rF\!_f\cup \mathrm{Vr}_f $ exactly as the adherence of the union of the following intervals 
%i
\begin{itemize}
 
\item $ (-\infty,\rho_{\delta,1}\,]\;\hbox{if }\delta\equiv0 \mod 2 $ 
 
\item $  \ClI{\rho_{\delta,k}, \rho_{\delta,k+1}}  \hbox{ for } k\in\lrb{0..\delta-2},\,k\equiv\delta\mod 2$ 
  
\item $ [\,\rho_{\delta,\delta}, +\infty) $ 
\end{itemize}
In imprecise imagery: \gui{we know $\rF\!_f$ to the nearest $\mathrm{Vr}_f$}. 
\\
Generally speaking, the problem with a polynomial of known degree arises from the fact that \thref{thVirtualRoots} asserts something precise about the sign of the polynomial on an interval $\ClI{\rho_{\delta,j},\rho_{\delta,j+1}}$ only when $\rho_{\delta,j}< \rho_{\delta,j+1}$. The result is as follows.

%l
%: Lemma{lemrFf}
\begin{lemma} \label{lemrFf}
Let $f\in\RRX$ be a polynomial of degree $\delta$ known and $g=f/c_\delta$ the corresponding monic polynomial ($c_\delta$ is the leading coefficient, $>0$ or $<0$). Let us note $\rho_{\delta,k}=\rho_{\delta,k}(g)$.
\begin{enumerate}
 
\item The adherence of $ \rF\!_f\cup \mathrm{Vr}_f $ is equal to the adherence of an explicit finite union of closed intervals whose bounds are $\rho_{\delta,k}$ or $+\infty $, or $-\infty$.
 
\item When we know the signs of $(\rho_{\delta,k+1}-\rho_{\delta,k})$ and $g(\rho_{\delta,k})$, we have an exact description of the closed $\rF\!_f$ in the form of a finite union of disjoint closed intervals. The information required is equivalent to giving the signs of $g(\frac{\rho_{\delta,k}+\rho_{\delta,k+1}} 2)$.
\end{enumerate}

\end{lemma}
%----------- fin lemma ----------------------------------- 

\noindent \textsl{When the degree of $f$ is not known}, we lose control of the situation {in $ +\infty $ and $ -\infty $}. The fuzziest situation, in which we have no control at all, arises when we don't know whether the polynomial is identically zero or not.

\smallskip Similar results hold for a basic open $ \rU_f=\sotq{x\in\RR}{f(x)> 0} $. 

%--- SubSubsection{TDY}---- 
\Subsection{Sign and variation tables}
%: Subsection{Sign and variation tables}
%-----------

Let $\gR$ be a constructive model of \Sa{Co--}. Two elements $a$ and $b$ are said to be \gui{distinct} if $a\neq b$, i.e.\ $a-b$ is invertible.

%l
%: Lemma{lemMinoration}
\begin{lemma} \label{lemMinoration}
Given a list $L$ of $k$ elements and a list $L'$ of $k+\ell$ distinct elements in $\gR $, at least $\ell$ elements of $L'$ are distinct from all elements of $L$. 
\end{lemma}
%----------- end lemma ----------------------------------- 
%
%begin{proof}
%
%end{proof}
%

\thref{thVirtualRoots}, Items \textsl{\ref{ivrvariation}} and \textsl{\ref{ivrsigne}}, almost gives a complete table of signs and variations for the monic polynomial $f$.

For the complete table of signs of $f$, any hesitations concern the signs of $f$ in the virtual roots~$\xi_k$ of $f'$. The same applies to the table of variations of $f$, with the signs of $f'$ at the virtual roots of~$f''$.

This leads to the following result.

%f
%: Proposition{factCompleteTable}
\begin{proposition} \label{factCompleteTable} Let $\gR$ be an ordered field with virtual roots.
\begin{enumerate}
 
\item Let $f(x)\in\Rx$ be a monic polynomial of degree $k\geq 2$ and  $k+\ell-1$  distinct elements  $a_i\in\gR$. For at least $\ell$ of these elements, the polynomial $f(x)+a_i$ has a known strict sign at each of the virtual roots of $f'$, and its complete sign table is known exactly. 
\\
If $k=2$ then the complete table of signs and variations is known exactly.
 
\item Let $f(x)\in\Rx $ be a monic polynomial of degree $k\geq 3$,  $k+\ell-1$  distinct elements  $a_i\in\gR$, and $ k+\ell-2 $  distinct  elements  $b_j\in\gR$. For at least $\ell^2$ of the pairs $(a_i,b_j)$, we have a complete table of signs and variations known exactly for the polynomial $f(x)+b_jx+a_i$.
\end{enumerate} 
\end{proposition}

\begin{remarks} \label{remfactCompleteTable} ~

\noindent 1) We probably have a perturbation result of the same style which says that \textsl{for almost all perturbations} of a finite number of monic polynomials~$f_i$, we know with certainty the strict equalities and inequalities between all the virtual roots of~$f_i$ and all their derivatives, as well as the signs of~$f_i$ in each of these virtual roots, which gives a complete table of signs and variations for the family of~$f_i$ and their derivatives.

\smallskip\noindent 2) If we want a result analogous to Proposition \ref{factCompleteTable} for a continuous semialgebraic map, we will have to place ourselves in the theory \Sa{Corv} and restrict the table of signs and variations sought to a bounded closed interval. 
\eoe
\end{remarks}
%----------- end remark ---------------------------------- 

%--- SubSubsection{TDY}---- 
\Subsection{An approximate cylindrical algebraic decomposition ({CAD})?}
%: Subsection{An approximate cylindrical algebraic decomposition?}
%-----------

 The problem arises of giving an approximate CAD for a finite family of polynomials of $ \RXn $ where $\gR$ is a constructive model of \Sa{Co0rv} (or of \Sa{Corv}). This would be a result that cleverly generalises Lemma \ref{lemrFf} or Proposition \ref{factCompleteTable}.

In piano-mover terms, instead of deciding whether \gui{this passes} or \gui{that doesn't pass}, we'd get approximate results of the following kind: given the data of the problem and a desired precision $ \epsilon $, we'd compute \und{uniformly} a $ \alpha\in\gR $ such that:
%i
\begin{itemize}
 
\item if $ \alpha>0 $, there is a way of passing at a distance $ >\epsilon $ from the obstacles, and we'll tell you how,
 
\item si $ \alpha<1 $ il n'y a pas moyen de passer en respecter un distance $ >2\epsilon $.
 \end{itemize}
Naturally, the piano must be a well-defined semialgebraic compact, as must the obstacles, and as must the space in which the piano is moved.

\smallskip In general, since it is impossible to control, even in an approximate way, the behaviour at infinity of a polynomial whose coefficients are all close to $0$, we must necessarily limit ourselves to calculating an approximate CAD for a finite family of polynomials on a well-defined compact of the style $ \ClI{0,1}^n $. If we try to reproduce a usual CAD (for a discrete real closed field) on $\RR$, we can see that the coefficients of a sub-resultant polynomial may well all be very close to $0$. But \textsl{a priori} virtual roots are only effective for monic polynomials.

On this kind of subject, we're still in our infancy.

%%%%%%%%%%%%%%%%%%%%%%%%%%%%%%%%%%%%%%%%%%%%%%%%%%%%%%%%%%%%%%%%%%%% 
\Subsection{Stratifications}
%: Subsection{Stratifications}
%-----------

It seems that stratifications, when assumed, are a restful framework in which many results valid for discrete real closed fields can be extended without too much difficulty to the  \nds case. 
 
%%%%%%%%%%%%%%%%%%%%%%%%%%%%%%%%%%%%%%%%%%%%%%%%%%%%%%%%%%%%%%%%%%%% 
\Subsection{The Fundamental Theorem of Algebra (\tsbf{FTA})}
%: Subsection{The Fundamental Theorem of Algebra}
%-----------

 For a treatment of \tsbf{FTA} without the axiom of dependent choice, see \cite{Ric2000}.

 Since the virtual roots are continuous maps, and since it is impossible to follow by continuity the zeros of a complex polynomial (monic of  fixed degree $m$ and with variable coefficients), we certainly cannot obtain one of these zeros expressed as an element of $ \SaCe_m(\RR) $. Nevertheless, we can cover the zeros of a complex polynomial of degree $ \delta $ by a finite number of expressions in $ \Sace_{\delta^2}(\RR) $.

What we'd like to do here is to do it in a fairly optimal way.

\paragraph{1. The square roots of a complex number $ c=a+ib $.} 

\noindent 
The zeros of the polynomial $ f(Z)=Z^2-c $ are given in the form $ x+iy $ by the real solutions of the system \gui{$ x^2-y^2=a $, $ 2xy=b $} and are calculated as follows:
%i
\begin{itemize}
 
\item $ (x^2+y^2)^2=a^2+b^2 $, so $ x=\pm u $ with $ u=\sqrt{\frac1 2 (a+\sqrt{a^2+b^2})}\in\Sace_{4}(\RR) $ 
 
\item $ y=\pm v $ with $ v=\sqrt{\frac1 2 (-a+\sqrt{a^2+b^2})} $, with the constraint $ xyb\geq 0 $.
\end{itemize}

\noindent If we denote $ z_1=u+iv $, $ z_2=-z_1 $, $ f_1(Z)=(Z-z_1)(Z-\ov{z_1}) $, $ f_2(Z)=(Z-z_2)(Z-\ov{z_2}) $ and $ g(Z)=(Z-c)(Z-\ov c) $ we obtain the equality
% equation label {eqsqrC}
\vspace{-.5em}
\begin{equation} \label {eqsqrC}
g(Z^2)=f(Z)\ov f(Z)=f_1(Z)f_2(Z)
\end{equation}
% end-equation
The polynomials $f_1$, $f2 $, $g$ and $ f\ov f$ are real, everywhere $\geq 0$, each with a simple algebraic certificate for its character $\geq 0$ when the variable is real. When $c\neq0$, the zeros of $f$ are divided between the zeros of $f_1$ and those of $f_2$.

We can estimate that we have thus obtained the optimal solution for the square roots of a complex number in the context of the $\RR$-algebra of maps generated by the maps  \gui{virtual square roots} $\rho_{2,2}$, and more generally the optimal solution in the context of the algebras $\Sace_m(\RR)$.

Note that $x$ and $y$ being roots of real polynomials of degree $4$, there were  16 possible choices for $x+iy$. 

\paragraph{2. The general case}. ~\\
We have the following non-optimal result.

%: Proposition{propTFA}
\begin{proposition}[\tsbf{FTA} via the virtual roots]\label{propTFA}~\\
Let $f$ be a complex monic polynomial of degree $\delta$.
There exist $\delta^4$ polynomials $q_\ell$ positive quadratic\footnote{Precisely: monic polynomials of degree $2$ everywhere $\geq0$.} having their coefficients constructed over $\rho_{\delta^2,k}(\dots)$ (polyroots in the real and imaginary parts of the coefficients of $f$) such that $f\ov f$ divides the product of $q_\ell$. \\
If the real closed field under consideration is discrete, the polynomial $ f(z) $ decomposes into a product of $(z-\zeta_j)$ factors explicit on $\gC$, with the $\zeta_j$ whose real and imaginary parts are roots of monic  real polynomials of degree $\delta^2$, whose coefficients are $\QQ$-polynomials in the real and imaginary parts of the coefficients of $f$.
\end{proposition} 
%--------- end statement ----------------------------- 
%
\begin{proof}
The real part of a zero $\zeta_j$ of $f$ is written $(\zeta_j+\ov{\zeta_j})/2$. The $(\zeta_j+\ov{\zeta_k})/2$ are $ \delta^2 $ and are the zeros of a real polynomial $ h_1 $ of degree $\delta^2$ whose coefficients are expressed as $\QQ$-polynomials in the real and imaginary parts of the coefficients of $f$. Among the real zeros of~$ h_1 $ are the $\frac1 2(\zeta_j+\ov{\zeta_j})$. These are therefore virtual roots of$ h_1$. Similar reasoning applies to the imaginary part, with a real polynomial $ h_2 $ of degree $\delta^2$. If $\alpha$ is a virtual root of~$ h_1 $ and~$ \beta $ a virtual root of~$h_2$, we associate the polynomial 
\[
q=(z-\alpha+i\beta)(z-\alpha+i\beta)=(z-\alpha)^2+\beta^2
\] 
which is one of the $ q_\ell $ in the statement.
\end{proof}

%r
%: Remark{rem}
\begin{remark} \label{remPR2019} 
In the paper \cite{PR2019} the authors prove that a discrete ordered field $ \delta^2 $-closed (i.e.\ satisfying the intermediate value theorem for polynomials of degree $ \leq \delta^2 $) satisfies the fundamental theorem of algebra for polynomials of degree $ \leq \delta $. We can deduce this result from Proposition \ref{propTFA} using the formal Positivstellensatz as follows. Assume that the real closed field is discrete. Then the fact that $ f\ov f $ divides the product of $ q_\ell $ implies that $f$ admits at least one complex zero, among the zeros of $ q_\ell $.\footnote{We have a little better. The product of $ q_\ell $ decomposes into a product of linear, and therefore irreducible, factors in $ \gC[Z] $. Since $ f(Z) $ divides this product, and since $ \gC[Z] $ is a gcd domain, it is in fact a by-product.} Moreover, the virtual roots of $ h_1 $ and $ h_2 $ are characterised by systems of large inequalities. A Horn rule on the language of ordored fields states that, for a discrete real closed field, if we put these systems of large inequalities into hypotheses, we obtain as a valid consequence the fact that the product of $ f(\alpha\pm i\beta) $ suitable is zero. According to Item \textsl{2} of Formal Positivstellensatz \ref{Pst1}, this Horn rule is valid for any ordered field (discrete or not) as well as for real closed rings, since it is valid in the theory \Sa{Asonz}. And if the field satisfies the \tsbf{TVI} for polynomials of degree $ \leq \delta^2 $, the hypotheses are satisfied by the virtual roots of $ h_1 $ and~$ h_2 $. In the same way, if the language of ordered fields has been enriched by introducing virtual root maps for polynomials of degree $ \leq \delta^2 $, with the corresponding axioms, we will also obtain for the corresponding dynamic algebraic structures the fact that the product of $ f(\alpha\pm i\beta) $ suitable is zero.\eoe 
\end{remark}
%----------- end remark ---------------------------------- 

%\vspace{-1.5em}
\paragraph{3. The general case in terms of multisets.} \label{TFA3}
Reference: the \tsbf{FTA} in \cite[Richman]{Ric2000}. 
\\
\textsl{A priori}, the \gui{\tsbf{FTA} version multisets} seems difficult to formulate correctly without having the metric space of $ n $-multisets of complex numbers.

To get around this, we can reduce the \gui{\tsbf{FTA} version multisets} to a set of dynamically valid rules giving an essentially equivalent formulation that uses counting the number of zeros inside rectangles in the style of~\cite[Eisermann]{Eis2012}. The article \cite{PR2019} seems to us to give all the necessary details.

Since we do not assume that the ordered field is discrete, we must use only rectangles on whose edges we are certain that there are no complex zeros of the polynomial under consideration. 

An explicit test shows that we must avoid at most $ \delta $ horizontal lines and at most $ \delta $ vertical lines for our rectangles. This is formulated by saying that if we consider $ \delta+m $ distinct horizontal lines, we are certain that at least $m$ of them are good (the same goes for the verticals).

For these rectangles, counting the zeros inside works and always gives a well-defined integer. 

We have a valid rule which ensures that no complex zero lies outside an explicitly large enough rectangle. For this sufficiently large rectangle the count gives the expected number~$ \delta $. And a valid rule says that when a rectangle is cut in half, the sum of the two counts equals the previous count.

If we also want to deal with the non-archimedean case, we need to establish Horn rules saying that we can enclose the $ \delta $ zeros in a union of rectangles of arbitrarily small size. 

Further study of the paper \cite{PR2019} should lead to the desired results, which are more precise than the \tsbf{FTA} considered in Item 2, results which can be considered to be \underline{the} satisfactory constructive form of \tsbf{FTA}, and which will be valid for \sa{Corv} theory, formulable as valid Horn rules in that theory. But these Horn rules would not be valid in the theory of real closed rings.

% Subsection
\section{Some questions}

\Subsection{Continuous semialgebraic maps}

%: Question{Qu-continuousRV}
\begin{question} \label{Qu-continuiteRV} Make more explicit the (constructive) result
of continuity of virtual root maps: Item \emph{2} of \thref{thVirtualRoots}. \textsl{Each map $ \rho_{d,j}:\gR^d\to\gR $ is uniformly continuous on any ball $\mathrm{B}_{d,M}:=\sotQ{(a_{d-1},\dots,a_0)\,}{\,\sum_ia_i^2\leq M}$, ($M>0$).} Continuity should be given in fully explicit form à la {\L}ojasiewicz.
\end{question}
%--------- end fact ---------------------------------------------- 

%: Question{questArc3}
\begin{question} \label{questArc3} 
Give a constructive proof of \thref{thArc3}.
\end{question}
%----------- end question ----------------------------- 
%: Question{Qu-polrootsafrvr}
\begin{question} \label{Qu-polrootsafrvr} Let $\gR$ be a real closed ring. Is any element of $ \SaC_n(\gR) $ an integer on the ring of polynomials $ \Rxn $ an element of $ \SaCe_n(\gR) $ ? 
\end{question}
%--------- end fact ---------------------------------------------- 

%: Question{Qu-polrootsRR}
\begin{question} \label{Qu-polrootsRR} Is every continuous map $ \RR^m\to\RR $ which is integral on the ring of polynomials an element of $ \SaCe_m(\RR) $ ? The answer in classical mathematics is yes, because we can apply \thref{thPBpolyroots} to $\RR$. 
\end{question}

Questions \ref{questArc2} have to do with the o-minimal character of the \ndrcf structure. The word \gui{compact} below is used to mean \gui{closed bounded subset}.
%q
%: Question{questArc2}
\begin{questions} \label{questArc2} \textsl{(remember Proposition \ref{propfsagccovrsup})}\\
Consider an ordered field with virtual roots $\gR$. 
\begin{enumerate}
 
\item Show that a continuous semialgebraic map which is everywhere $ >0 $ on the compact $ \ClI{0,1}^n\subseteq \gR^n $ is minorized (on this compact) by an element $ >0 $. And that the lower bound is an element of $\gR$.
 
\item Extend the result to an arbitrary \gui{well-defined} semialgebraic compact: by this we mean a bounded semialgebraic closed subset $K$ for which the function \gui{distance to $ K $} is a continuous semialgebraic map (an element of $ \Sac_n(\gR) $).
\end{enumerate}
\end{questions}

\Subsection{Real closure}

%: Question{quest2cloture}
\begin{question} \label{quest2cloture} 
If $\gK$ is a model of \sa{Co} (or of \sa{Co--}), is its 2-closure $ \gL $ as an $f$-ring still a model of \sa{Co} (or of \sa{Co--})? 
\end{question}
%----------- end of question ----------------------------- 

We repeat the previous question (adding some details) for the real closure.

\begin{question}\label{questCloturevirtuelle}
Let  $\gK$ be a model of the theory \Sa{Co} and $\gR$ be the dynamic algebraic structure $ \Sa{Corv}(\gK) $ (as \paref{pagevirtualclosure}). Is $\gR$ a constructive model of \sa{Corv}? 
\\
In particular, is the following metatheorem satisfied? Given two closed terms $\alpha$ and $ \beta $ of $\gR$ such that the rule $ \vd \alpha+\beta>0 $ is valid, is it true that one of the two rules $ \vd \alpha>0\, $, $ \vd \beta>0 $ is valid?
\\
We can ask the same question in the following form: if $\gK$ is a model of \sa{Co}, does the (usual) algebraic structure $ \ARC(\gK) $ satisfy the rule \Tsbf{AFRL}? 
\end{question}

%: Question{questEis2012}
\begin{question} \label{questEis2012} 
Can the article \cite[Eisermann]{Eis2012} be reread for a \ndrcf, i.e.\ in the theory \Sa{Crc2}? This question requires a detailed development of the ideas in Item 3, \paref{TFA3} in the paragraph concerning \tsbf{FTA}.
\end{question}

%: Question{questTVICrc}
\begin{question} \label{questTVICrc} 
Show that the intermediate value theorem, stated in the form of the rule \RCFn\ \paref{RCFn}, is not valid in the theory \Sa{Crc2}. Note that a slightly weakened form is valid: see Item \textsl{\ref{i4Pst3}} of Formal Positivtellensatz \ref{Pst3} and Item \textsl{\ref{ivrTVI}} of \thref{thVirtualRoots}.
\end{question}

%: Question{questTVICrc2}
\begin{question} \label{questTVICrc2} 
Show that the theorem which states that every complex number has a square root is not a valid rule in the theory \Sa{Crc2}. Compare with Proposition \ref{propTFA} which might seem to assert the opposite.
\end{question}

\Subsection{Pierce-Birkhoff}

%: Question{questpbring}
\begin{questions} \label{questpbring} ~

\noindent 1) Does the definition of a Pierce-Birkhoff ring given in \ref{defiSacembis} coincide in classical mathematics with the notion defined in \cite[Madden, 1989]{Mad89}?

\smallskip\noindent 2) If this is indeed the case, the question arises of giving constructive proofs for sophisticated results, such as the fact that a regular Noetherian coherent ring of dimension $ \leq 2 $ is a Pierce-Birkhoff ring \cite{LMSS2012}.

\smallskip\noindent 3) Recall that the usual Pierce-Birkhoff conjecture is proved in \cite{Mahe83} for $ \gR[x,y] $ when $\gR$ is a discrete real closed field but it is not so clear that there is a constructive proof for $ \RR[x,y] $. 
\end{questions}
%----------- end question ----------------------------- 

\Subsection{The 17th Hilbert problem}

%: Question{quest17H}
\begin{question} \label{quest17H} 
To what extent does the constructive solution of the 17th Hilbert problem for~$\RR$ (see \cite[section 6.1]{GL93}) apply to any strict $f$-ring with virtual roots? If not, what stronger theory would do: \Sa{Crc2}, \Sa{Arc}, \Sa{Crca}, \Sa{Crc3} (\paref{theorieCrc3})?
\end{question}

\Subsection{The Grail?}

The question arises of a theorem analogous to \thref{thTarski}, but now for the non-discrete case. 

\smallskip Formal Positivstellensatz \ref{Pst3} implies that the theory \sa{Arc} is the Horn theory generated by $\gRa$, by $\RR$ or by $\RR_{\tsbf{PR}}$ on the signature of \sa{Arc}.

%: Question{questGraal1}
\begin{question} \label{questGraal1} 
Is the theory \sa{Arc} skolemised from the cartesian theory generated by $ \gRa $, by $\RR$ or by $ \RR_{\tsbf{PR}} $ on the signature of commutative rings? 
\end{question}

%: Question{questGraal2}
\begin{question} \label{questGraal2} 
In what sense could we say that the theory \Sa{Corv} is the dynamical theory generated by $\RR$ \gui{without axiom of dependent choice} on the signature of \sa{Arc}? Same question with $ \RR_{\tsbf{PR}} $.

\noindent NB: this question seems impossible to formulate in classical mathematics, and in constructive mathematics, we would need to have a clear idea of $\RR$ \gui{without an axiom of dependent choice}. 
\end{question}
%----------- end question ----------------------------- 

\newpage \thispagestyle{empty}

\chapter{The axiom of archimedianity}\label{secGeomReelsArchi}
\Today
\minitoc

In this chapter, in order to better describe the algebraic properties of $\RR$, we make an attempt which consists in not leaving the dynamical theories while preserving the essence of the non-\rdy~\Tsbf{HOF}.

However, the language remains essentially that of ordered rings.

In the third part, we will make a much more ambitious attempt using a much richer language, which will essentially show us a geometric theory of the reals as a precursor of the theory of o-minimal structures. 

% section
\section{Archimedean \ndrcfs}

 The following rule, which means that the field is archimedean, is satisfied on $\RR$ 

\Regles{\Lab{AR1} $ \dsp\vd \Vou_{n\in\N} \abs x \leq n $ \qquad \qquad\qquad ({\bf Archimedes 1})}

%d
%: Definition{defiCrca}
\begin{definition} \label{defiCrca}
We define the theory \SA{Crca} of real closed archimedean fields as the geometric theory obtained by adding the axiom \Tsbf{AR1} to the theory \Sa{Crc2}. 
\end{definition}
%----------- end definition -------------------------------- 

The example given in Item 3 of Remark \ref{remCrc} (a local real closed ring with zerodivisors, model of the theory \sa{Crc2}) remains a model of \sa{Crca}. Examples \ref{exacorpsnondiscret} are also models of the theory \sa{Crca}: in general the subrings of $\RR$ stable for virtual root maps, the fraction $ \mathrm{Fr} $ and the inverses of invertible elements, are models of \sa{Crca}. 

%%%%%%%%%%%%%%%%%%%%%%%%%%%%%%%%%%%%%%%%%%%%%%%%%%%%%%%%%%%%%%%%%%%%
%%%%%%%%%%%%%%%%%%%%%%%%%%%%%%%%%%%%%%%%%%%%%%%%%%%%%%%%%%%%%%%%%%%%

% Subsection
\section{Some questions}

\Subsection{Axiom of archimedianity}

%: Question{questRarchi}
\begin{questions} \label{questRarchi} ~\\
We know that we cannot express the fact that $\RR$ is archimedian in a finitary way. We express it with the infinite rule \Tsbf{AR1}.
\begin{itemize}
 
\item One question that arises is whether the theory \Sa{Crca} obtained by adding the rule \Tsbf{AR1} is a conservative extension of \Sa{Crc2}, this seems likely.
 
\item We could start by showing that \Sa{Crca} proves the same Horn rules as~\Sa{Crc2}. 
 
\item On the other hand, for the corresponding formal theory in which we allow the introduction of predicates for $ P\Rightarrow Q $ and $ \forall x P $ (with Gentzen's natural deduction rules) it could be that a statement like \Tsbf{HOF} becomes provable. 
\end{itemize}

\end{questions}
%----------- end question -----------------------------

%%%%%%%%%%%%%%%%%%%%%%%%%%%%%%%%%%%%%%%%%%%%%%%%%%%%%%%%%%%%%%%%%%%%
\paragraph{The principle of omniscience \tsbf{LPO} is safe in real algebra?}

%: Question{questRLPO}
\begin{question} \label{questRLPO} ~\\ 
The following rule is not satisfied on $\RR$, because it implies the \Edineq\ rule. 

\Regles{\Lab{AR2} $ \dsp\vd x= 0 \; \vou \; \Vou
_{n\in\N} \abs x > 1/2^n $ \qquad ({\bf Archimedes 2})
}

\noindent But no doubt it is \gui{admissible}, in the sense that adding it to \Sa{Crca} would provide a conservative extension of \Sa{Crcd}. 

\noindent This result would be a kind of realisation of Hilbert's programme for \tsbf{LPO}, restricted to the theory~\Sa{Crca}. Indeed, the theory \Sa{Crcd} is itself harmless compared to \Sa{Crc2} because it proves the same Horn rules.
\end{question}
%----------- end question -----------------------------

%%%%%%%%%%%%%%%%%%%%%%%%%%%%%%%%%%%%%%%%%%%%%%%%%%%%%%%%%%%%%%%%%%%%
\paragraph{Convergent series in real algebra?}

%: Question{questRLPO}
\begin{question}
Let $[x]^n=\frac 1{2^n}\vi(x\vu - \frac 1{2^n}) $. The following rule is not a \rdy

\UneRegle{Cauchy} {$ \dsp\vd \Exists x\,\Vii_{n\in\N}\, 
\abs {x -\som_{p=0}^n[x_p]^p} \leq 1/2^n $ 
}

A function symbol $ \sum_{n=0}^\infty [x_n]^n $ should be introduced for these infinite sums. This would replace the illegitimate rule \tsbf{Cauchy} by an infinite number of legitimate Horn rules. But is such a function symbol legitimate?
\end{question}
%----------- end question -----------------------------

%%%%%%%%%%%%%%%%%%%%%%%%%%%%%%%%%%%%%%%%%%%%%%%%%%%%%%%%%%%%%%%%%%%%
\Subsection{Schmüdgen's Positivstellensatz}\index{Positivstellensatz!Schmüdgen's}
%: Subsection{Le Positivstellensatz de Schmüdgen}

References: \cite{Schm91,Schw2002,Schw2003}

%q
%: Question{questSchmudgen}
\begin{question} \label{questSchmudgen} 
Is geometric theory sufficient to develop theorems of the Schmüdgen type?
 \end{question}
%----------- end question ----------------------------- 
\newpage \thispagestyle{empty}

%%%%%%%%%%%%%%%%%%%%%%%%%%%%%%%%%%%%%%%%%%%%%%%%%%%%%%%%%%%%%%%%%%%%

\chapter*{Conclusion}
\addstarredchapter{Conclusion}

The most important questions that remain to be resolved for this 2nd part seem to us to be the following.

\begin{enumerate}
 
\item Question \ref{questthParamcontFsagc0}. Give a constructive proof of \thref{thParamcontFsagc0}.
\\
\item Let $\gR$ be a discrete real closed field and $ f\colon \gR^n\to\gR $ be a continuous semialgebraic map. There exists an integer $ r\geq 0 $, a continuous semialgebraic map $ g\colon \gR^{r+n}\to\gR $ defined on~$\RRa$, and an element $ \uy\in\gR^r $ such that
\[
\forall \xn\in\gR\;\; f(\xn)=g(\yr,\xn).
\]

\item Question \ref{Qu-continuiteRV}. Make explicit Item \emph{2} of \thref{thVirtualRoots} asserting the uniform continuity of maps $ \rho_{d,j}:\gR^d\to\gR $ on any closed ball.

\item Question \ref{questArc3}. Give a constructive proof of \thref{thArc3} \textsl{Any continuous semialgebraic map $ \RRa^n\to\RRa $ can be defined by a term of the theory~\sa{Arc}}. This will make it possible to clarify definitively the constructive Definition \ref{defiArc} of real closed rings and its relationship in classical mathematics with various constructive characterisations of real closed rings.

\item Question \ref{questCloturevirtuelle} concerning the possibility of constructing a real closure of a non-discrete ordered field. Let $\gK$ be a model of the theory \Sa{Co} and $\gR$ the dynamical algebraic structure $ \Sa{Corv}(\gK) $ (as \paref{pagevirtualclosure}). Is $\gR$ a constructive model of \sa{Corv}? 

\end{enumerate}

\part[Improved version of the theory of \ndrcfs]{An improved version of the theory of \ndrcfs and an attempt at a constructive version of o-minimal structures}

\chapter*{Introduction}
\addstarredchapter{Introduction}

In this third part we explore the possibility of better describing the algebraic properties of~$\RR$ by extending the language through the introduction of sorts for continuous semialgebraic maps on compact cubes.

\smallskip Indeed, we note that the general situation became clearer when we introduced the maps $\cdot\vu\cdot$, $\cdot\vi\cdot$ and the virtual root maps. These natural extensions to the language used have gone a long way towards overcoming the obstacles that the notion of \nds order seems to offer to a formalisation in finitary dynamical theory. 

\smallskip However, from a constructive point of view, it is not natural to be interested in the real zeros of polynomials whose degree is fixed. The good reason for this with $\RR$ is that we don't control the zeros in the neighbourhood of infinity when the degree is not clearly fixed. By replacing $\RR$ by the real interval $ \II_\RR=\ClI{-1,1}\subseteq \RR $ this so-called good reason disappears by itself. 

\smallskip The idea is that you control things constructively only within the compact framework. We need to detox from $ \pm\infty $ and go back to Greek mathematics! Consequently, we must drop $\RR$ in favour of the interval $ \II_\RR $, for example by replacing $ + $ by the half-sum. This requires us to go back to the axiomatics, but the benefit will be that it will be easier to formulate certain properties linked to the fact that from the constructive point of view $\RR$ \textsl{is not} discrete. 

\smallskip Note that up to now, we have been rather dry concerning some of the desirable properties stated in \ref{prptaCrcdesirable}: indeed we have not been able to correctly state the principles of extension by continuity or the gluing principles with sufficient generality. We could only talk about uniform continuity from outside the dynamical theory. Indeed, uniform continuity requires an alternation of quantifiers of the type $ \forall n \exists m \forall x,y $ which requires a priori to leave the framework of geometric theories. This is also due to the fact that we had no sort of continuous semialgebraic maps. 

In this section we try to make up for this lack. And we must remember that from a constructive point of view, a continuous map on a compact does not exist without a \mcu. The gamble we take here is to internalise the question of uniform continuity. This means that, for the moment, we remain within a finitary dynamical theory framework. 

Moreover, the extended framework that we propose with the introduction of these new sorts seems to be a correct framework for approaching a constructive treatment of o-minimal structures.

\smallskip Here is a brief description of the contents of the third part. 

\smallskip Chapter \ref{chapomin} recalls the fascinating properties of o-minimal structures in classical mathematics. These are finiteness properties exactly similar to those of the algebraic geometry of discrete real closed fields, and yet devoid of algorithmic character by the use of the sign test on real numbers in classical theory. Constructing an algorithmic theory of o-minimal structures is a crucial challenge in the \gui{constructive Hilbert programme}, which aims to uncover hidden constructions in contemporary classical mathematics and to reformulate purely ideal theorems into constructive statements. This programme avoids the use of the formal theory \sa{ZFC}, which describes an hypothetical set universe that does not correspond to any proven mathematical construction. 

\smallskip Chapter \ref{chapfnfrb} proposes a first finitary dynamical theory for sorts describing uniformly continuous real maps with values in $ \II_\RR $. 

\smallskip Chapter \ref{chapfnbg} gives a general framework to describe the properties of uniformly continuous maps defined on $ \II_\RR^{\,m} $ with values in $ \II_\RR^{\,n} $. A decisive aspect is to take into account the fact that a \mcu of a map $f$ can be seen as another uniformly continuous map $g$ attached to the map $f$. 

\smallskip Chapter \ref{chapaxiomesomin} proposes new axioms which are a priori satisfied for the algebraic geometry of real closed fields and which seem decisive for approaching an hypothetical and highly desirable constructive theory of o-minimal structures. We are nevertheless very far from having formalised in a dynamical theory what would be a constructive version of o-minimal structures.

\newpage \thispagestyle{empty}

%%%%%%%%%%%%%%%%%%%%%%%%%%%%%%%%%%%%%%%%%

\chapter{O-minimal structures} \label{chapomin}
\Today
\minitoc

\Subsection{Definition, definable parts}

References: \cite{cos99}, \cite[\hbox{1998}]{vdD}, \cite{vdD86,DD88}. 

\smallskip 
The definition of an o-minimal structure over a real closed field $\gR$ in classical mathematics is given by a collection $(S_n)_{n\in\N}$, where each $S_n$ is a set of parts of $\gR^n$. 

%:2025 début réécrit
The elements of $S_n$ are called the \textsl{definable parts} of the o-minimal structure under consideration.

As defining axioms we ask the following stability properties.
\begin{enumerate}
 
\item The semi-algebraic subsets of $ \gR^n $ are in $ S_n $.
 
\item Every $ S_n $ is a Boolean algebra of sets (stability by finite intersection and reunion, and complementary passage).
 
\item If $A\in S_n$ and $B\in S_m$ then $ A\times B\in S_{m+n} $.
 
\item If $ A\in S_{n+1} $ and $ p_n:\gR^{n+1}\to\gR^n $ is the projection onto the first subspace $ \gR^n $ of coordinates (forgetting the last coordinate), then $ p_n(A)\in S_n $.

%:2025 réécrit
\item Any element of $S_1$ is a finite union of definable open intervals and points.
\end{enumerate}

\Subsection{Definable maps, outstanding results}
A map $ A\to B $ between definable sets is said to be \textsl{definable} if its graph is definable. 

\smallskip Let's recall some key results.

\smallskip 
\begin{itemize}
\labu The domain of definition and the image set of a definable map are definable.
\labu The composite of two definable maps is definable. 
\labu Any definable part is a Boolean combination of definable closed parts. More precisely, we have a definable cylindrical decomposition of $ \gR^n $ adapted to any finite family of definable parts (analogously to the CAD in the case of semi-algebraic parts for a discrete real closed field). The cells of the decomposition are homeomorphic to open simplexes, with definable homeomorphisms. 
\labu If $ A\in S_n $ is closed (for the Euclidean distance of $ \gR^n $) and non-empty, then the function \gui{distance to~$A$} 
\[
d_A:\gR^n\to \gR,\,x\mapsto \inf\nolimits_{y\in A}\norme{x-y}
\] 
is (continuous and) definable.
\labu If $ f\colon \gR^n\to \gR $ is continuous and definable, the zeros of $f$ form a definable closed part. Conversely, according to the previous item, any definable closed part of $ \gR^n $ is the zero set a definable continuous map. 
\labu If $ I=(a,b)\subseteq \gR $ (with $a,b\in\gR_\infty:=\gR\cup\so{-\infty,+\infty} $) and if $ f\colon I\to\gR $ is a definable map, then
\begin{itemize}
 
\item there is a subdivision of $ I $ 
\[
a=a_0<a_1<\dots<a_k=b
\] 
such that on each open interval of the subdivision, $f$ is either constant, or strictly monotone and continuous,
 
\item we also have a subdivision such that on each open interval of the subdivision, $f$ is derivable with definable derivative, continuous and of constant sign ($ =0 $ or $ >0 $ or $ <0 $). 
\end{itemize}
\labu If $ A\in S_n $ is a definable closed subset  and $ f\colon A\to \gR $ is definable continuous, it can be extended into a definable continuous map on $ \gR^n $.
\labu Any definable continuous map $ (-1,1)\to(-1,1) $ extends by continuity into a definable continuous map $ \ClI{-1,1}\to \ClI{-1,1} $. 
\labu Any continuous map $\ClI{-1,1}\to \ClI{-1,1} $ is bounded. 
\labu If $ f\colon \ClI{-1,1}^{n+1}\to \gR $ is continuous and definable, the map $ g\colon \ClI{-1,1}^{n}\to \gR $ defined by 
\[
g(\xn):=\sup\nolimits_{y\in\ClI{-1,1}}f(\xn,y)
\] 
is continuous and definable.
Note that in particular if $f$ is everywhere $\leq 0$ and if $A$ is the zero set $f$, then $ p_n(A) $ is the zero set $g$. If $A$ is a definable closed set $ \subseteq \ClI{-1,1}^{n+1} $, we can take for $f\,$ the map $ -d_A:\ClI{-1,1}^{n+1}\to \gR $. 
\end{itemize}

\Subsection{Variant} 
All this implies that we could just as easily define the considered o-minimal structure on $\gR$ by giving the following objects.
\begin{enumerate}
 
\item Definable continuous maps 
 $ \ClI{-1,1}^{n}\to \ClI{-1,1} $.
 
\item The bicontinuous increasing bijection (definable in any o-minimal structure) 
\[
(-1,1)\to \gR,\;x\mapsto x/(1-x^2)
\] 
and the reciprocal bijection
\[
\gR\to (-1,1),\;x\mapsto \big(\sqrt{4x^2+1}-1\big)/2x
\]
\end{enumerate}

In fact, using the coding given in Item 2, to get the definable continuous maps $ \gR^n\to \gR$ we just need to know how to describe the definable continuous maps $ f\colon (-1,1)^{n}\to(-1,1) $. To do this, all we need to know is how to describe the continuous definable maps $ g\colon\ClI{-1,1}^n\to\ClI{-1,1} $.

Let us note $ \norme{x}=\sup_{i\in\lrbn}\abs{x_i} $ for $ x=(\xn)\in\gR^n $.

In the case where the growth to infinity of any definable map $f$ from $ \gR^n $ to $\gR$ is bounded by a polynomial, for such a map $f$, we have a continuous definable map $ g\colon \ClI{-1,1}^{n}\to \ClI{-1,1} $ written in the form $ g(x)=\big(1-\norme{x}^2\big)^kf(x) $, and the map $f$ can be encoded by the pair $ (g,k) $. The map $g$ tends uniformly towards $0$ when $x$ tends towards the edge of $ \ClI{-1,1}^n $.

In the general case, we can replace $ h(x):=1-\norme{x}^2 $ in $ g(x) $ by a map $ \varphi\circ h $ where $ \varphi\colon \ClI{0,1}\to \ClI{0,1} $ is continuous definable and strictly positive on $ [0, 1)\, $. 

\newpage \thispagestyle{empty}

\chapter{Rings of bounded real maps}
\label{SubsecIcrftr}
\label{chapfnfrb}
\Today
\minitoc

%section*{Introduction}
%addtocontents{toc}{\skip0.8em}
%addcontentsline{toc}{section}{Introduction}
%\rdb
%
%\fbox{TO BE WRITTEN}

\section{Some reminders of the second part}

The Horn theory \Sa{Aftr} of strongly real rings is the theory of reduced $f$-rings to which we add the relation symbol $\cdot >0$ as an abbreviation of \gui{$ x\geq 0\vii \exists z\,xz=1 $} and the function symbol $\Fr$ with the axioms \Tsbf{fr1} and \Tsbf{fr2} (Definition \ref{defiAftr}).

A strongly real ring is therefore a reduced  $\QQ$-$f$-algebra in which any element greater than an invertible positive element is itself invertible, and in which the rule~\tsbf{FRAC} is valid.

Finally, the dynamical theory \sa{Co} of \ndsofs can be described as the theory of local strongly real rings, which amounts to adding the axiom \Tsbf{OTF} (Lemma \ref{lemArftr}, Item \textsl{3}) to the theory \sa{Atfr}.

\begin{lemma} \label{lemIFR}\label{NOTAuplus}
Let $\gA$ be a strongly real ring. Let $\gI=\sotq{x\in\gA}{-1\leq x\leq 1} $. We define on $\gI$ the law $ x\uplus y=\frac 1 2 (x+y) $. The structure obtained on $\gI $ for the signature 
\Sigt{\Icro}{\cdot= \cdot,\cdot\geq \cdot, \cdot>\cdot\mathrel{;}\cdot\,\uplus\,\cdot, \cdot\times \cdot, \cdot\vu\cdot, \Fr(\cdot,\cdot),\,- \cdot,0} \label{NOTASigIcro}
\noindent allows us to reconstruct, in a unique way, the structure of $\gA$ as a strongly real ring.
\end{lemma}
%----------- end lemma ----------------------------------- 
%
%
\begin{proof}
This is essentially because any element $ z\in\gA $ can be written in the form $ x/y $ with $ -1<x<1 $ and $ 0<y<1 $ (for example $ x=\frac{z}{2+\abs z} $ and $ y=\frac{1}{2+\abs z} $).
\end{proof}
%

%%%%%%%%%%%%%%%%%%%%%%%%%%%%%%%%%%%%%%%%%%%%%%%%%%%%%%%%%%%%%%%%%%%%
\section[Dynamical theory of rings of bounded real maps]{Dynamical theory of rings of bounded real maps}
\label{secAfrb}

We are going to use a more complete signature which corresponds better to the intuition of an interval as a convex subset.
\rdb

\smallskip We denote $ x\oplus y $ the following composition law in an $f$-ring: $ (x,y)\mt -1\vu(1\vi(x+y)) $.\footnote{This is the addition, put back into the interval $\gI $ if it comes out of it.} \label{oplus}

\smallskip We denote $ \mathrm{Cb} $ the set of \textsl{systems of barycentric coefficients}, defined precisely as follows: 
\[\ndsp
\mathrm{Cb}=\sotq{(r_k)_{k\in\lrbn}}{n\geq 2,\,r_1,\dots,r_n\in\QQ,\,r_1,\dots,r_n\geq 0, \,\sum_{k=1}^n r_k=1}. 
\]
We note $\II_\QQ=\sotq{r\in\QQ}{-1\leq r\leq 1}$.

For each $ \rho=(r_k)_{k\in\lrbn} $ in $ \mathrm{Cb} $, $ \mathrm{Brc}_{\rho} $ is a function symbol of arity $ n $ corresponding to the map: $\gI^n\to\gI,\;(x_k)_{k\in\lrbn}\mapsto\sum_{k=1}^nr_kx_k $. \label{notacb} The language of the dynamical theory of \textsl{rings of bounded real maps} \SA{Afrb} is defined by the following signature. There is only one sort, denoted~$ \AfrB$ 
\Sigt{\AfrB}{\cdot=0,\cdot\geq 0,\cdot>0\mathrel{;}(\mathrm{Brc}_{\rho})_{\rho\in\mathrm{Cb}}, \cdot\oplus\cdot, \cdot\times\cdot,-\,\cdot,\cdot\vu\cdot,\mathrm{Fr}(\cdot,\cdot), (r)_{r\in\II_\QQ}} \label{NOTASigAfrb}

\noindent {\bf Abbreviations}

\smallskip\noindent \textsl{Function symbols}

\vspace{-1em} \TwoCols{
\begin{itemize}
\labu $ x\vi y $ means $ - (-x\vu -y) $ 
\labu $ \abs{x} $ means $ x \vu -x $ 
\end{itemize}
}
{
\begin{itemize}
\labu $ {x}^+ $ means $ x \vu 0 $ 
\labu $ {x}^- $ means $ -x \vu 0 $ 
\end{itemize}
}

\medskip\noindent \textsl{Predicates}

\vspace{-1em} \TwoCols{
\begin{itemize}
\labu $ x = y $ means $ x - y = 0 $ 
\labu $ x \geq y $ means $ x - y \geq 0 $ 
\labu $ x > y $ means $ x - y > 0 $ 
\end{itemize}
}
{\begin{itemize}
\labu $ x \perp y $ means $ \abs x \vi \abs y =0 $ 
\labu $ x \leq y $ means $ y\geq x $ 
\labu $ x < y $ means $ y> x $ 
\end{itemize}
} 

\smallskip\noindent {\bf Axioms}\label{AxiomesAfrb}

\smallskip\noindent 
\textsl{The axioms are all the \rdys stated in the language of \sa{Afrb} which are valid for the interval $\gI=\ClI{-1,1}$ in the theory $\Sa{Aftr}(\QQ)$ of strongly real $\QQ$-algebras}. 

%: Lemma{lemAxiomesAfrb}
\begin{lemma} \label{lemAxiomesAfrb}
Valid Horn rules in \sa{Afrb} are decidable.
\end{lemma}
%----------- end lemma ----------------------------------- 

%
\begin{proof}
Consequence of Item \textsl{3} of Corollary \ref{corPst2}.
\end{proof}

Note that it is not known whether valid \rdys are decidable. The same question arises in the local case for the dynamical theory \sa{Co}. This question does not seem very important insofar as we are essentially interested in the case of the theories \sa{Crc1} and \sa{Crc2}, where the problem remains mysterious and is added to that of knowing whether we have captured all the algebraic properties of the field $\RR$. 

\section{Dynamical theory of compact real intervals}\label{secIcr}

The dynamical theory \SA{Icr} of \textsl{compact real intervals} has a single sort, denoted $\Icr$. Its language is defined by the following signature.
\Sigt{\Icr}{\cdot=0,\cdot>0,\cdot\geq0\mathrel{;}(\mathrm{Brc}_{\rho})_{\rho\in\mathrm{Cb}}, \cdot\oplus\cdot, \cdot\times\cdot,-\,\cdot,\cdot\vu\cdot,(\rT_n)_{n\in\NN},\mathrm{Fr}(\cdot,\cdot), (r)_{r\in\II_\QQ}}\label{NOTASigIcr}
\noindent The dynamical theory \sa{Icr} is obtained from the theory \Sa{Afrb} described in Section \ref{secAfrb} by adding
\begin{itemize}
\item The function symbols $ \rT_n $ for Chebyshev polynomials, with the axioms\\
$\vd T_0(x) =1$,  $\vd T_1(x) =x$,  $\vd T_{n}(x) = 2xT_{n-1}(x) - T_{n-2}(x)$ ($n\geq 2$).\\ 
For the main properties of Chebyshev polynomials we refer to the book \cite[Chebyshev Polynomials]{MH2003}
\item The axiom \Tsbf{OTF} (valid for the \ndsof structure) reformulated as follows:

\regles{\lAb{OTF$'$} $ \,\, x\oplus y>0 \vdi_{x,y:\Icr}\; x>0 \;\;\vou\;\;y>0 $}
 
\end{itemize}

%: Question{questIcr}
\begin{remark} \label{questIcr}
Theories \Sa{Icr} and \Sa{Co} are probably essentially identical. Otherwise it would be necessary to add axioms to \sa{Icr} to make it true.
\eoe
\end{remark}
%----------- end question ----------------------------- 
\newpage \thispagestyle{empty}

\stepcounter{chapter}

\chapter{A reinforced language and the first corresponding axioms}
\label{chapfnbg}
\Today
\minitoc

\section*{Introduction}
\addtocontents{toc}{\skip0.8em}
\addcontentsline{toc}{section}{Introduction}
\rdb
We now introduce the sorts of continuous semialgebraic maps in order to obtain a more expressive dynamical theory than \Sa{Crc2} for \ndrcfs.

This new dynamical theory, which we shall call \sa{Crc3}, attempts here to summarise what we are entitled to expect from an o-minimal structure for uniformly continuous definable maps on $ \II_\RR $.%
\index{real closed!non dis@\textsl{non} discrete --- field}
 
 As we have already indicated, we restrict ourselves to uniformly continuous bounded maps, in much the same spirit as Bishop.

%%%%%%%%%%%%%%%%%%%%%%%%%%%%%%%%%%%%%%%%%%%%%%%%%%%%%%%%%%%%%%%%%%%%
% section{The sorts of reinforced language}
\section{The sorts of reinforced language}
\begin{enumerate}
 
\item The sort $\Icr$, for the compact interval $\gI=\ClI{-1,1}$.
 
\item For each $ m\geq 0,n\geq 1 $, a sort $ \Df_{m,n} $ for uniformly continuous definable maps\footnote{Continuous semialgebraic maps for the theory of real closed fields.} $\gI^m\to \gI^n $, the sort $ \Df_{m,1} $ is noted~$\Df_m$. In particular $ \Icr=\Df_0=\Df_{0,1} $.
  
\item A sort  $ \Li_n $ seen as a subset of $ \Df_n $, for certain smooth maps given at the start (at least Chebyshev polynomials).
 
\item A sort $ \Mc $ for \mcus. They are seen as particular objects of sort $ \Df_1 $.
 
\item A sort $ \Dfmc_n $ for pairs formed by an object of sort $ \Df_n $ and by a \mcu that fits it.
\end{enumerate}

%%%%%%%%%%%%%%%%%%%%%%%%%%%%%%%%%%%%%%%%%%%%%%%%%%%%%%%%%%%%%%%%%%%%
\section{An abstraction principle}
%: Subsection A principle of abstraction

For any term $ t(\xn) $ of type $ \Icr^n\to \Icr $ from the theory (where the $x_i$ cover all the free variables present in the term), a term which provides a map $ \gI^n\to\gI $ in a model, we must do what is necessary so that there exists a term $ \dot t $ in $ \Df_n $ which \gui{evaluates as $t$}. In other words, we need to put in place what we need to mimic, within our geometric theory, the $\lambda$-abstraction of the $\lambda$-calcul.

To do this, the signature
\begin{itemize}
 
\item symbols of type $ \Df_n\times \, \Icr^n\to\,\Icr $ for the evaluation of a $ u:\Df^n $ into $ x_i:\Icr $;
 
\item symbols for the composition of maps (with suitable axioms); 
 
\item symbols which give a name to the maps given in the signature (for example $ \cdot\times \cdot $ must have a name as an object so $ \Df_2 $); 
 
\item \dots
\end{itemize}

\smallskip This approach is essential if we are to be able to talk uniformly, and not just occasionally, about the properties of continuous definable maps.\footnote{This is reminiscent of what Kleene does when he defines (uniformly) primitive recursive maps.} 

%: Remark{remsymbolsfunctions}
\begin{remark} \label{remsymbolsfunctions} 
 One might think that some function symbols introduced \textsl{a priori} to mimic \hbox{$\lambda$-abstraction} could have been added \textsl{a posteriori} by virtue of the possibility of adding a function symbol in the case of unique existence, thus providing a dynamical theory essentially identical to the previous one. But the existence (in the unique existence in question) of a well-defined map from the sort $ \Icr\times \Icr $ to the sort $\Icr$ does not mean the existence of a corresponding object in~$ \Df_2 $, or even its uniqueness (because the extensionality axioms introduced later are too weak). What we mean by introducing \textsl{a priori} these maps as objects of sort $ \Df_{m,n} $, is that all sufficiently simple maps, in particular those described in the signatures, are indeed continuous and definable. \eoe
\end{remark}
%----------- end remark ---------------------------------- 

\section{First structures on sorts $ \Df_{m,n} $}

\Subsection{Sorts $ \Df_{m,n} $}

The sort $\Icr$ of compact real intervals ($f$-rings) has the structure  described in Section \ref{secIcr}. 

\smallskip Each sort $\Df_m$ ($ m\geq 0,n\geq 1 $) is accompanied by function symbols and predicates as well as axioms of rings of bounded real maps (dynamical theory \Sa{Afrb}). 

\smallskip \rem The axiom \tsbf{OTF'} \paref{secIcr} is not valid for the sorts $ \Df_m=\Df_{m,1} $ for $ m\geq 1 $.
\eoe

\paragraph{Identification of $ \Df_{m,n} $ and $ ({\Df_m})^n $}~

\smallskip For each $ i\in\lrbn $ we have a function symbol $ \pi_{m,n,i} $ of type $ \Df_{m,n}\to \Df_m $ corresponding to the $i$-th coordinate. We also give a function symbol of type $ (\Df_m)^n\to \Df_{m,n} $ for the bijection; we will note it $ (\varphi_1,\dots,\varphi_n) $ in an admittedly somewhat ambiguous way. With the appropriate axioms, this allows us to identify $ \Df_{m,n} $ and $ (\Df_m)^n $.
 
 \Regles{ 
 \lab{ } 
 $ \vdi_{\varphi_1,\dots,\varphi_n:\Df_m}; \varphi_i=\pi_{m,n,i}((\varphi_1,\dots,\varphi_n)) $ \quad $ (i\in\lrbn) $ %
 \lab{ }
 $ \vdi_{\varphi\colon \Df_{m,n}}\; \varphi=(\pi_{m,n,1}(\varphi), \dots,\pi_{m,n,n}(\varphi)) $ 
 } 

We give the axioms that $ \pi_{m,n,i} $ is a morphism for the ring structures of bounded real maps of $ \Df_{m,n} $ and $\Df_m$.
 
%%%%%%%%%%%%%%%%%%%%%%%%%%%%%%%%%%%%%%%%%%%%%%%%%%%%%%%%%%%%%%%%%%%%
\paragraph{Composition of maps}~

\smallskip We have function symbols $ \rC_{m,n,p} $ of type $ \Df_{n,p}\times \Df_{m,n}\to \Df_{m,p} $ corresponding to the composition of maps. For $ \varphi\colon \Df_{m,n} $ and $ \psi:\Df_{n,p} $, we write $ \psi\circ \varphi:=\rC_{m,n,p}(\psi,\varphi) $. We have constants of sort $ \Df_{n,n} $ for the \gui{identity maps} $ \Id_n\colon \gI^n\to\gI^n $.

\smallskip The axioms for the associativity of composition are given.

\smallskip We give the axioms which say that for $\varphi$ fixed of sort $\Df_{m,n}$, the map $\psi\mt\psi\circ \varphi$ is a morphism for the ring structures of bounded real maps of  $\Df_{n,p}$ and  $\Df_{m,p}$.

\smallskip We abbreviate $\rC_{n,m}$ to the term of 
type $\Df_n\times\, (\Df_m)^n\to\Df_m$, defined by 
\[
\rC_{n,m}(\varphi,\eta_1,\dots,\eta_n)\eqdef \varphi\circ(\eta_1,\dots,\eta_n).
\]
%%%%%%%%%%%%%%%%%%%%%%%%%%%%%%%%%%%%%%%%%%%%%%%%%%%%%%%%%%%%%%%%%%%%
\paragraph{Evaluation of maps}~

\smallskip For $ \Icr=\Df_0 $, the function symbol $ \rC_{n,0} $ of type $ \Df_n\times \,\Icr^n\to \Icr $ defines the evaluation of a map in variables taken from $\gI $. The associativity of composition is then naturally related as follows with $x_i$ of sort $\Icr$, $\varphi$ of sort $ \Df_n $ and $ \psi $ of sort $ \Df_1 $ 
\[
(\psi\circ \varphi)(\xn)=(\psi\circ \varphi)\circ (\xn)=\psi\circ (\varphi\circ (x_1,\dots,x_n))=\psi(\varphi(\xn)).
\]

%%%%%%%%%%%%%%%%%%%%%%%%%%%%%%%%%%%%%%%%%%%%%%%%%%%%%%%%%%%%%%%%%%%%
\paragraph{Constant maps}~

\smallskip We have a function symbol $ \jmath_n=\jmath_{0,n} $ of type $ \Icr\to\Df_n $ for constant maps.

Axioms are given which say that these are morphisms for ring structures of bounded real maps\footnote{For example for $ r\in\II_\QQ $, an axiom says that $ \jmath_n(r) $ is equal, as an object of sort $ \Df_n $, to the~$r$ given in the ring structure of bounded real maps.} and that the evaluation of a constant map in any arguments is indeed the desired constant.

\smallskip More generally, if $ 0\leq m<n $ we have a function symbol $\jmath_{m,n}$ of type $\Df_m\to\Df_n$ for objects of sort $\Df_n$ corresponding to maps which depend only on the $m$ first variables and which can therefore be expressed from objects of sort $\Df_m$. The axioms are analogous to those given for the case $m=0$.

%%%%%%%%%%%%%%%%%%%%%%%%%%%%%%%%%%%%%%%%%%%%%%%%%%%%%%%%%%%%%%%%%%%%
\paragraph{Rearrangement of variables}~

\smallskip For $ m,n>0 $ and a map $ \kappa\colon \lrbm\to\lrbn $ we have an object $ \wi\kappa $ of sort $ \Df_{n,m} $ with the axiom 

\Regles{
\lab{c$_\kappa$} $\vdi_{\xn:\Icr}\; \wi\kappa(\xn)=(x_{\kappa_1},\dots,x_{\kappa_m})$
}

 We also give the associated natural axioms: $\wi{\kappa\circ \tau}=\wi\kappa\circ \wi\tau$.
 
\smallskip So for $ m<n $ we have the equality $ \jmath_{m,n}(\varphi)=\varphi\circ \wi\kappa $, where $ \kappa\colon \lrbm\to\lrbn $ verifies $ \kappa(i)=i $ for $ i\in\lrbm $. This equality means that we do not need to introduce the symbol $ \jmath_{m,n} $.

\smallskip We also have, for $ \tau_{n,i}\colon \lst1 \to \lrbn $ defined by $ \tau_{n,i}(1)=i $ and $ \psi $ so that $ \Df_{m,n} $, the equality $ \pi_{m,n,i}(\psi)=\wi{\tau_{n,i}}\circ \psi $. 
%%%%%%%%%%%%%%%%%%%%%%%%%%%%%%%%%%%%%%%%%%%%%%%%%%%%%%%%%%%%%%%%%%%%

\Subsection{Gluing of elements of $ \Df_1 $ on consecutive intervals}

%%%%%%%%%%%%%%%%%%%%%%%%%%%%%%%%%%%%%%%%%%%%%%%%%%%%%%%%%%%%%%%%%%%%
\paragraph{Restrict an element of $ \Df_1 $ to an interval}~

\smallskip If $f$ is of sort $ \Df_1 $, we want to have a name for the map $g$ obtained from the restriction of $f$ to an interval $ \ClI{a,b}\subseteq\gI $. 

This is done using a function symbol $ \Rs\colon\Df_1 \times \Icr \times \Icr \to \Df_1 $. 

When $ a\leq b $, we extend $g$ with $ g(x)=f(b) $ if $ x\geq b $ and $ g(x)=f(a) $ if $ x\leq a $. When $ a\geq b $, we permute $a$ and $b$. We therefore have the following axioms

\Regles
{\labu $ \vdi_{a,b:\Icr,f:\Df_1} \Rs(f,a,b)=\Rs(f,a\vi b,a\vu b) $ 
\labu $ \,\,x\leq a\vi b\vdi_{a,b,x:\Icr,f:\Df_1} \Rs(f,a,b)(x)=f(a\vi b) $ 
\labu $ \,\,x\geq a\vu b\vdi_{a,b,x:\Icr,f:\Df_1} \Rs(f,a,b)(x)=f(a\vu b) $ 
\labu $ \,\,a\vi b \leq x\leq a\vu b\vdi_{a,b,x:\Icr,f:\Df_1} \Rs(f,a,b)(x))=f(x)) $ 
}

\paragraph{Gluings} 

\smallskip If $ f_0,\dots,f_n $ are of sort $ \Df_1 $, and if $ 0\leq a_1\leq \dots\leq a_n\leq 1 $ we want to have a name for the map which glues the $f_i$ restricted to $ [a_i,a_{i+1}] $, possibly shifted vertically to ensure continuity. 

This is done using a function symbol $ \Rc_n\colon (\Df_1)^{n+1} \times (\Icr)^{n} \to \Df_1 $. 

We have the following basic axioms (let $ a_0=0 $ and $ a_{n+1}=1 $) 

\Regles{\lab{Rc$_{n,j}$} $ \Vii_{i} (a_i\leq a_{i+1}, f_{i}(a_{i+1})=f_{i+1}(a_{i+1}))
\vdi_{a_i:\Icr,f_i:\Df_1} \Rs(\Rc(\uf,\ua),a_j,a_{j+1})=\Rs(f_j,a_j,a_{j+1}) $} 

We add the appropriate axioms to force the assumptions of \tsbf{Rc$ _{n,j} $}. 

%r
\Subsection{Axioms of weak extensionality}
%: Subsection{Axioms of weak extensionality}

For each sort $ \Df_n $ with $ n\geq 1 $ we have the following axiom of weak extensionality.

\UneRegle{EXT$_n$} {$\,\,a> 0 \vdi_{a:\Icr,\varphi:\Df_n}\;  
\abs {\varphi} < \jmath_n(a) \;\vou\; \Exists x\, \abs{\varphi(x)}>\frac {a} {2}$  }

As a consequence, a map which is everywhere null is \gui{almost} null: it is increased in absolute value by any constant $ >0 $. To conclude that it is null, we would have to invoke \Tsbf{HOF}, a non-geometric axiom which we do not want, or a dubious axiom of archimedeanity such as~\Tsbf{AR2} in an infinitary geometric theory.\footnote{A very unsound solution to this weakness of dynamical theory would be to consider as models only those where objects of sort $\Df_n$ are $\geq 0$ (resp. $>0 $) exactly when they are evaluated $\geq 0$ (resp.\ $ >0 $) at any point of $\gI $.} 

\smallskip Note that the axiom \tsbf{EXT$_0$} simply says that for $ x,a:\Icr $ and $ a>0 $, we have $ \abs x<a $ or $ \abs x>\frac a 2 $, which is a variant of \Tsbf{OTF}.

\smallskip Finally, note that the rule \tsbf{EXT$ _n $} follows from \tsbf{OTF} and the upper bound axioms in Section~\ref{AxBsup} (with $ m=0 $).

\newpage \thispagestyle{empty}

%%%%%%%%%%%%%%%%%%%%%%%%%%%%%%%%%%%%%%%%%%%%%%%%%%%%%%%%%%%%%%%%%%%%
%%%%%%%%%%%%%%%%%%%%%%%%%%%%%%%%%%%%%%%%%%%%%%%%%%%%%%%%%%%%%%%%%%%%
\chapter{Decisive axioms}\label{chapaxiomesomin}

\Today
\minitoc

%section*{Introduction}
%addtocontents{toc}{\skip0.8em}
%addcontentsline{toc}{section}{Introduction}
%\rdb

%%%%%%%%%%%%%%%%%%%%%%%%%%%%%%%%%%%%%%%%%%%%%%%%%%%%%%%%%%%%%%%%%%%%
%: Subsection{Upper bound axioms}{label{AxBsup}}
\section{Upper bound axioms}\label{AxBsup}

The upper bound axioms replace \textsl{a priori} the projection axiom for definable parts in o-minimal structures.

\smallskip 
For $m>0$ we have a function symbol $ \sup_{m} $ of type $\Df_m\to\Icr$ for the lub. It satisfies the axioms describing the lub, namely

\regles{
\lab{sup$^{\Df}_{m}$}  $ \vdi_{\varphi:\Df_m}
\;\varphi\leq \jmath_m(\sup_{m}(\varphi)) $
\lab{SUP$^{\Df}_{m}$}  $\,\,\epsilon>0\vdi_{\epsilon:\Icr;\varphi:\Df_m}
\;\Exists \uy \;\;\varphi(\uy)+\epsilon> \sup_{m}(\varphi)$
}

More generally for $n\geq 0$ and $m>0$ we have a function symbol $\sup_{m+n,n}$ of type $\Df_{m+n}\to \Df_{n}$ for the upper bound on the $m$ last variables (in $\gI^m $) with the following axioms (so $\sup_{m,0}$ is none other than $\sup_m$).

\regles{
\lab{sup$^{\Df}_{m+n,n}$} $ \vdi_{\varphi:\Df_{n+m}}
\jmath_{n,m+n}(\sup_{m+n,n}(\varphi)) $ 
\lab{SUP$^{\Df}_{m+n,n}$} $ \,\epsilon>0\vdi_{\epsilon,\xn:\Icr;\varphi:\Df_{n+m}}
\sup_{m+n,n}(\varphi(\ux,\uy)+\epsilon> \sup_{m+n,n}(\varphi)(\ux) $ 
}

\eoe

%%%%%%%%%%%%%%%%%%%%%%%%%%%%%%%%%%%%%%%%%%%%%%%%%%%%%%%%%%%%%%%%%%%%
\Subsection{Axioms of uniform continuity}
%: Subsection{The axioms of uniform continuity}

We now explain how a suitable system of axioms can translate the fact that any definable continuous map admits a \mcu, while remaining within the framework of a geometric theory. This is possible because definable continuous maps admit \mcus that are themselves particular continuous definable maps. The sorts $\Mc$ and $\Dfmc_n$ with their axioms are crucial here.

\smallskip We start by giving a function symbol $\jmath_\Mc$ for an injection of type $\Mc\to\Df_1$. An axiom specifies that $\jmath_\Mc$ is injective.

\smallskip We have a predicate $\mathrm{Mcu}_{n}$ on $\Df_n\times\,\Mc$ which expresses that $\mu$ is a modulus for $\varphi$ by means of the following abbreviation.

\Regles
{
\labu $\,\,\mathrm{Mcu}_n(\varphi,\mu)$ is an abbreviation
for:  $\mu \big(\abs{\varphi(\ux)-\varphi(\ux')}\big)\leq \norme{(\ux)-(\ux') }$ 
}

%space{-.2em}
\noindent where $ \norme{(\uz)}=\sup_i\abs{z_i} $. Here the inequality seems to be written, in the form of evaluated maps, with $x_i$ and~$x'_i$ of sort $\Icr$. But in fact, this inequality should be read as linking two objects of sort $\Df_{2n}$. This avoids the use of the universal quantifier on $x_i$ and $ x'_j $ in the definition of uniform continuity! Dynamic theories do not allow the creation of new formulas using universal quantifiers, so we get round the difficulty by mimicking $\lambda$-abstraction! 

\medskip The following axioms specify constraints on objects $\mu$ of sort $\Mc$.

\DeuxRegles
{
\lab{Mc$_1$} $\,\, 0<b<c\vdi_{\mu:\Mc;b,c:\Icr}
\;0<\mu(b)< \mu(c)$
\lab{mc$_1$} $\vdi_{\mu:\Mc} 
\;\mathrm{Mcu}_1(\jmath_\Mc(\mu),\mu)$}
{\lab{Mc$_2$} $\,\, a\leq 0\vdi_{\mu:\Mc;a:\Icr}
\;\mu(a)=0$
}

\smallskip 
\rem In the case where we consider only continuous semialgebraic maps, 
{\L}ojasievicz assures us that any \mcu can be taken from the only maps $ \epsilon>0 \mt c \,\epsilon^n $ (with some $ c>0 $) 
\eoe

\smallskip 
The not very intuitive axiom $ \tsbf{mc}_1 $ will be a valid rule if we require in another axiom that any object of sort $ \Mc $ corresponds to a convex map. 
\eoe

\smallskip The sort $\Dfmc_n$ is defined as a subsort of the product sort $\Df_n\times \Mc$. It is accompanied by two function symbols $\mathrm{df}_n$ and $ \mathrm{mc}_n$, with the appropriate axioms, which mean that an object of sort $\Dfmc_n$ can be considered as a pair of objects $ (\varphi,\mu)$ of respective sorts\footnote{It is not necessary to create the product type as such. The following axiom will suffice: $ \mathrm{df}_n(\theta)=\mathrm{df}_n(\theta'), \mathrm{mc}_n(\theta)=\mathrm{mc}_n(\theta')\vdi_{\theta,\theta':\Dfmc_n} \theta=\theta'$.} $\Df_n$ and $\Mc$. The axiom $ \tsbf{mcu}_n $ says how the subsort is defined: if $ (\varphi,\mu) $ is of sort $ \Dfmc_n $, then the predicate $ \mathrm{Mcu}_n(\varphi,\mu) $ is satisfied. 

\regles
{\lab{mcu$_n$} $
\vdi_{\psi:\Dfmc_n}\;
\mathrm{Mcu}_n(\mathrm{df}_n(\psi), \mathrm{mc}_n(\psi))$
}

\smallskip Finally, we have the axiom $ \tsbf{DFMC}_n $ which says that any object $\varphi$ of sort $ \Df_n $ is the image by $ \mathrm{df}_n $ of an object $ \theta $ of sort $ \Dfmc_n $.

\regles{\lab{DFMC$ _n $} $ \vdi_{\varphi:\Df_n}\; \Exists \theta;\; \mathrm{df}_n(\theta)=\varphi $} 

All this machinery explicitly guarantees the uniform continuity of maps represented by objects of sort $ \Df_n $.
 
%r
%: Remark{remMcupartout}
\begin{remark} \label{remMcupartout} 
Each time we introduced a constant of sort $ \Df_n $, we actually had to introduce a constant \gui{above it} of sort $ \Dfmc_n $. This is not difficult because in each case a \mcu is obvious.
\eoe
\end{remark}
%----------- end remark ---------------------------------- 

%%%%%%%%%%%%%%%%%%%%%%%%%%%%%%%%%%%%%%%%%%%%%%%%%%%%%%%%%%%%%%%%%%%%
\section{Axioms for smooth maps}
%: Subsection{The axioms of smoothness}

Objects of sort $ \Li_n $ are seen as objects of sort $ \Df_n $ which define certain smooth maps (i.e.~$ \Cin $). The signature includes a function symbol $ \jmath_{\Li_n} $ for the corresponding injection, with the axioms which say that it is an injective morphism for suitable laws (those which preserve the smooth maps). 

Constant maps and coordinate maps are given as objects of sort $ \Li_n $. 

The sort $ \Li_1 $ contains the Chebyshev polynomials. 

It might be possible to introduce other Nash maps into $ \Li_n $; this should not change the dynamical theory but could facilitate certain proofs. 

In the case where we are aiming at a particular o-minimal structure (other than that provided by the continuous semialgebraic maps), other maps can be given which will serve as a basis for the definition of the structure. 

\Subsection{Density axiom} 

The zeros of a non-zero smooth map (in an o-minimal structure) form a closed $F$ with an empty interior. We can express (at least partially) this density property (for the complementary of $ F $) by means of the following axiom

\Regles{
\lab{Dens$ _n $}
 $ \,\abs{\varphi(\an)}>0,\; \varphi\times \psi=0 \vdi_{a_1,\dots,a_n:\Icr;\varphi:\Li_n;\psi:\Df_n}\;\psi=0
 $ 
}

\Subsection{The derivation}

We think it is convenient to introduce the derivative (or partial derivative) following the Bridger-Stolzenberg definition (see \cite{AD1994} and \cite{BS1999}). A map $ \varphi\colon \II\to\RR $ is continuously derivable if the map \gui{rate of increase} can be extended by continuity, i.e.\ if there exists a uniformly continuous map $ \psi\colon \II^2\to\RR $ satisfying the identity \fbox{$ \varphi(x_1)-\varphi(x_2)=\psi(x_1,x_2)\times (x_1-x_2) $}. The derivative of $\varphi$ is then given by $ \varphi'(x)=\psi(x,x) $.

As we only want maps  $\gI^n \to \gI$, we must use an implicit coding $(x,p)$ with $x\in\gI$ and $ p\in \N $ for the real $px$. 

The map $\psi$ is uniquely determined by $\varphi$ (see below the valid rule \Tsbf{Der}) so in our dynamical theory we can introduce it by means of a function symbol $ \Delta=\Delta_{1,1} $ of type $ \Li_1\to\Li_2 $ which satisfies the axiom

\Regles{
\Lab{der}
 $ \vdi_{\varphi:\Li_1}\;\varphi(x_1)-\varphi(x_2)=\Delta(\varphi)(x_1,x_2)\times (x_1-x_2) $ 
}

\noindent \rem This equality appears to be written in the form of maps evaluated as linking two objects of sort~$ \rR $, but in fact \textsl{it should be read as linking two objects of sort $\Li_2$}, which are evaluated in $(x_1,x_2)$ in the form indicated in the axiom as it appears to be written.

\smallskip In fact we have to use the implicit coding alluded to above and the rule \tsbf{der} must in fact be written in the form

\Regles{
\lab{der}
 $ \vdi_{\varphi:\Li_1}\;\frac 1 {2p} (\varphi(x_1)-\varphi(x_2))=\Delta_p(\varphi)(x_1,x_2)\times (x_1-x_2) $ 
}

To avoid complicating the presentation, in the following we pretend that $ \Delta(\varphi) $ is the real map \gui{rate of increase}. 

\smallskip The following uniqueness rule follows from the axiom \tsbf{Dens$ _2 $}: in the first member we must read an equality between objects of sort $ \Df_2 $ and the smooth map is $ x_1-x_2 $ seen as an element of $~\Li_2 $. 

\Regles{
\Lab{Der}
 $ \varphi(x_1)-\varphi(x_2)=\psi(x_1,x_2)\times (x_1-x_2) \vdi_{\varphi:\Li_1;\psi:\Df_2}\;\Delta(\varphi)=\psi $ 
}
\smallskip 
In the same way, for several variables, analogous axioms are required for each partial derivative. In particular, for $ n\geq 2 $ and $ i\in\lrbn $ we have a function symbol $ \Delta_{n,i} $ of type $ \Li_n\to\Li_{n+1} $ which satisfies the axiom

\Regles{
\lab{der$_{n,i}$} {\mathrigid 1.7mu
$\vdi_{\varphi:\Li_n}\;\varphi(x_1,\dots,x_i,\dots,x_n)-\varphi(x_1,\dots,x'_i,\dots,x_n)=\Delta_{n,i}(\varphi)(x_1,\dots,x_i,x'_i,\dots,x_n)\times (x_i-x_i')$}
}

\noindent \textsl{This equality must be read as linking two objects of sort $ \Li_{n+1} $}.

\smallskip 
\rem Using the upper bound axiom, we obtain that smooth maps are lipschitzian, which gives a particularly simple \mcu.
\eoe

\Subsection{What other axioms for derivation?} 

Here we need to consider which axioms need to be introduced corresponding to the usual properties of derivation. Most of these properties should result from the definition (axioms $ \tsbf{der}_{n,i} $) and the axioms $ \tsbf{Dens}_{n} $. 
%%%%%%%%%%%%%%%%%%%%%%%%%%%%%%%%%%%%%%%%%%%%%%%%%%%%%%%%%%%%%%%%%%%%
\Subsection{Axioms of virtual roots}

Virtual roots can be defined a priori for any smooth map whose derivative of order $k $ is $ >0 $ (on $ \II $), by virtue of Lemma \ref{lemBasicVirtualRoots} and the uniform constructive version of the mean value theorem. We then obtain most of the properties described in Definition \ref{prdfVirtualRoots} and \thref{thVirtualRoots}. The polynomial $ f(X) = X^{d} - ( a_{d-1} X^{d-1} + \cdots +a_1X+ a_0) $ which depends on $ d+1 $ variables can be replaced by any smooth $\varphi$ map of $ d+1 $ variables $ X, a_1,\dots,a_{d} $ whose $k$-th partial derivative  with respect to $ X $ is $ >0 $ as an object of sort $ \Li_{d+1} $.

If $ \inf(\varphi^{(k)})=\phi(a_1,\dots,a_{d}) $, we can treat the map $ \psi_k=\varphi+(c-\phi)^+X^k/k! $, for a constant \hbox{$ c>0 $}. Its $k$-th derivative with respect to $ X $ is $ \geq c $, and it is equal to $\varphi$ if $ \phi\geq c $. 
We can then introduce the $k$ virtual roots of $\varphi$ on $\gI $ as objects of sort $ \Df_d $ as in Definition \ref{prdfVirtualRoots} and \thref{thVirtualRoots}, but using our $\lambda$-abstraction. More precisely, we have  \gui{virtual roots} function symbols  $\mathrm{Rv}_{d,k,j}$ of type $\Li_{d+1}\to\Df_d$. And we have the corresponding axioms, direct translations of Definition \ref{prdfVirtualRoots} and of \thref{thVirtualRoots} (by replacing $-\infty$ and $+\infty $ by $-1$ and $+1$).

%%%%%%%%%%%%%%%%%%%%%%%%%%%%%%%%%%%%%%%%%%%%%%%%%%%%%%%%%%%%%%%%%%%%
\section{Axioms of real closure or o-minimal closure}
%: Subsection{Axioms of real closure}

From now on we deal with axioms that correspond to the general idea of real closure and o-minimal structure.
%%%%%%%%%%%%%%%%%%%%%%%%%%%%%%%%%%%%%%%%%%%%%%%%%%%%%%%%%%%%%%%%%%%%

%%%%%%%%%%%%%%%%%%%%%%%%%%%%%%%%%%%%%%%%%%%%%%%%%%%%%%%%%%%%%%%%%%%%
\Subsection{Finiteness axioms}

 The virtual root axioms are already axioms of finiteness, but independent of any o-minimal structure. 

We should have an analogue to Proposition \ref{factCompleteTable} (table of signs and variations) for continuous semialgebraic maps, and this should also work for o-minimal structures. In classical mathematics, tables of signs and variations exist for definable maps of an o-minimal structure, and Proposition~\ref{factCompleteTable} shows how to transform the classical statement into a constructive one. Here again, the problem is to formulate dynamical axioms that capture this type of result. One solution would be to have an infinite dynamical theory with axioms that say roughly that a continuous definable map is \gui{piecewise smooth monotone} in a statement to be specified, similar to Item \textsl{2} of Proposition~\ref{factCompleteTable}. 

%%%%%%%%%%%%%%%%%%%%%%%%%%%%%%%%%%%%%%%%%%%%%%%%%%%%%%%%%%%%%%%%%%%%
\Subsection{Gluing of maps defined on an open covering}

A finite cover of $\gI^n $ by definable opens is given here in the form 
\[ 
V_i=\sotq{\ux\in\gI^n}{g_i(\ux)>0}\quad i\in\lrbp 
\]
where $g_i$ are of sort $\Df_n$ and satisfy \fbox{$\sum_{i}g_i^+>0$ (1)}. 
Functions $h_i$ of sort $\Df_n$ are considered, for which a priori only the restrictions $h_i\frt{V_i}$ are relevant. 
The fact that $h_i$ and $h_j$ coincide on $V_i\cap V_j$ results in the equality \fbox{$h_ig_i^+g_j^+=h_jg_j^+g_i^+$ (2)}. 
Under hypotheses (1) and (2) we ask for the existence and uniqueness of an $f$ of sort $\Df_n$ verifying $ fg_i^+=h_ig_i^+ $ for each $i$ (which means that $f\frt{V_i}=h_i\frt{V_i}$). 
A priori we must have $f=(\sum_{i}h_ig_i^+)(\sum_{i}g_i^+)^{-1}$ (hence the uniqueness). And we get 
\[
fg_k^+=\frac{\sum_{i}h_ig_i^+g_k^+}{\sum_{i}g_i^+}=
\frac{\sum_{i}h_kg_i^+g_k^+}{\sum_{i}g_i^+}=\frac{(h_kg_k^+)\,\sum_{i}g_i^+}{\sum_{i}g_i^+}=h_kg_k^+.
\] 

%%%%%%%%%%%%%%%%%%%%%%%%%%%%%%%%%%%%%%%%%%%%%%%%%%%%%%%%%%%%%%%%%%%%
\Subsection{Gluing of maps defined on a closed covering}

A finite covering of $\gI^n $ by definable closed subsets is given here in the form 
\[ 
F_i=\sotq{\ux\in\gI^n}{g_i(\ux)\geq 0}\quad i\in\lrbp 
\] 
where $ g_i $ are of sort $ \Df_n $ and satisfy \fbox{$ \sup_{i}g_i\geq 0 $}. Functions $ h_i $ of sort $ \Df_n $ are considered, for which a priori only the restrictions $ h_i\frt{F_i} $ are relevant. The fact that $ h_i $ and $ h_j $ coincide on $ F_i\cap F_j $ results in the validity of the rules ($ i,j\in\lrbp $) 

\Regles{
\lab{} $ \,\,g_i(\ux)\geq 0\vet g_j(\ux)\geq 0\vdi_{\xn:\Icr;g_i,h_i,g_j,h_j:\Df_n)}\;h_i(\ux)=h_j(\ux) $ 
}

A uniform algebraic version of this validity can be stated as follows

\Regles{
\labu $ \vdi_{g_i,h_i,q_{ij},q_{ji}:\Df_n} \;({h_i-h_j})^2+g_iq_{ij}^++g_jq_{ji}^+=0 $ 
}

\noindent where the $ q_{k\ell} $ are of sort $ \Df_n $. Let's abbreviate the second member as $ E_{ij} $. Under the hypothesis of the equalities $ E_{ij} $, we want to have a map $f$ (an object $f$ of sort $ \Df_n $) satisfying an identity which means that $ f\frt{F_i}=h_i\frt{F_i} $. This can be expressed in the form of the following rule

\Regles{
\lab{RCVF} $ \,\,\sup_ig_i\geq 0\vet E_{1,2}\vet\dots\vet E_{p-1,p}\vdi_{g_i,h_i,q_{ij},q_{ji}:\Df_n}\;\Exists f,q_1,\dots,q_p\;\Vi_i(f-h_i)^2+g_iq_i^+=0 $ 
}

\noindent All the (free or dummy) variables in this rule are of sort $ \Df_n $.

In classical mathematics, this type of rule is valid for o-minimal structures. However, from a constructive point of view, we may have to restrict ourselves to coverings by \textsl{located closed subsets}.\footnote{A closed subset is said to be \textsl{located} when the distance to it is a well-defined map from a constructive point of view.
It seems necessary to add an axiom saying that the distance map to the sero set of a continuous definable map is itself definable.} This will complicate the writing of the axioms.

Note that the object $f$ whose existence is postulated is provably unique by virtue of a classical calculation for Positivstellensätze: we use the identity $(a+b)^2+(a-b)^2=2(a^2+b^2)$.
%%%%%%%%%%%%%%%%%%%%%%%%%%%%%%%%%%%%%%%%%%%%%%%%%%%%%%%%%%%%%%%%%%%%
\Subsection{Axioms of extension by continuity}

Typically, the \FRACn\ rules are special continuity extension axioms. The aim here is to state different rules that apply more generally (without the continuity extension giving $0$ to the disputed values) but with a smooth denominator.

For example, a map that is definable outside the zeros of a smooth (non-zero) map and continuous on its domain of definition is uniquely extended by continuity if it is uniformly continuous.

The problem is to formulate this in the context of our dynamical theory.

\smallskip It will be good enough to be able to formulate it for a quotient $f/g$ (well-defined outside the zeros of $g$) with $g$ smooth.

The fact that $f$ cancels at the zeros of $g$ can be put as an hypothesis in the following strong form: there exists an $\alpha$ of such a sort $ \Mc $ that \fbox{$\alpha\circ \abs f\leq \abs g$}.

The uniform continuity of $ f/g $ outside the zeros of $g$ seems to be stated using the reciprocal bijection of an object of sort $ \Mc $ on the interval $\ClI{0,1}$. In fact, we want to write something like
\[ 
\mu\left(\abs{\frac{f(\ux)g(\ux')-f(\ux')g(\ux)}{g(\ux)g(\ux')}}\right)\leq \norme
{(\ux)-(\ux')}
\] 
for $ \norme{(\ux)-(\ux')}>0 $, which could be rewritten without the assumption $ \norme{(\ux)-(\ux')}>0 $ in the framework of geometric theory as
\[ 
\abs{f(\ux)g(\ux')-f(\ux')g(\ux)}\leq \abs{g(\ux)g(\ux')},\nu(\norme{(\ux)-(\ux')})
\] 
with the reciprocal bijection $ \nu $ (on $ [0,\mu(1)] $) of $ \mu\frt{\ClI{0,1}} $.

%%%%%%%%%%%%%%%%%%%%%%%%%%%%%%%%%%%%%%%%%%%%%%%%%%%%%%%%%%%%%%%%%%%%
\Subsection{Conclusion: the improved real closed field structure}
%: Subsection{Conclusion: the structure of \Crc3}

The theory \SA{Crc3} will be obtained once all the axioms have been worked out. We have seen that the theory \Sa{Icr} can be considered as a variant of the theory \Sa{Co}. The theory \sa{Crc3}, which could also be called the theory of \textsl{compact real closed intervals}, is an improved variant of \Sa{Crc2}, in which o-minimal structures (which are enriched structures of real closed fields) could have a place as particular dynamic algebraic structures. 

%t
%: Theorem{thicrc1}
\begin{theorem} \label{thicrc1}
In constructive mathematics, the real interval $\II=\II_\RR$ and the continuous semialgebraic maps ${\II}^n\to\II$, provides a model of the dynamical theory \sa{Crc3}.
\end{theorem}
%----------- fin theorem ----------------------------- 
%
\begin{proof}
{\tt It seems that the ad hoc definition of continuous semialgebraic maps adopted in \ref{defiFSAGC2+} reduces this theorem to a theorem concerning essentially $\RRa$. But we need to check all the details and this may lead us to change the formulation of some axioms}.
\end{proof}

This theory \sa{Crc3} should make it possible to demonstrate constructive results which escape the more elementary theory \sa{Crc2} for the simple reason that they do not correspond to \rdys of \sa{Crc2}. Moreover, the same question arises for the \rdys of \sa{Crc2} themselves. 

%%%%%%%%%%%%%%%%%%%%%%%%%%%%%%%%%%%%%%%%%%%%%%%%%%%%%%%%%%%%%%%%%%%%
\section{O-minimal structures}

It seems that the axioms proposed here for the structure of compact real closed intervals are almost correct for constructively describing certain o-minimal structures defined in classical mathematics: those generated by the restrictions to the compact cube $ \II^n $ of certain smooth maps in the neighbourhood of $ \II^n $. 

The weakest point seems to be stability by projection. A priori, the current system of axioms only guarantees this stability for definable closed bounded parts. 

The resulting structure depends on the smooth maps given at the outset in the $ \Li_n $ sorts.

We are primarily interested in the structure obtained by taking the real analytic maps in the vicinity of the cube as the starting smooth maps. In dimension 1, this probably works well with Chebyshev series. 

It is a real challenge to give a constructive version of the classical theory, for example starting from the presentations given in \cite{vdD86} and \cite{DD88}. It would at least be necessary to demonstrate constructively that real analytic maps in the neighbourhood of the cube give rise to a structure which is a model of the dynamical theory \Sa{Crc3}.

Note also that from a strictly computational point of view, we are a priori more interested in the enumerable field $ \RR_{\tsbf{PR}} $ of real numbers computable in primitive recursive time, or in the enumerable field $ \RR_{\tsbf{Ptime}} $ of real numbers computable in polynomial time (see Example \ref{exacorpsnondiscret}). As for the definable continuous maps corresponding to these fields (for a fixed o-minimal structure), they too can no doubt be enumerated using Chebyshev series. 

Finally, it should be pointed out that, as things stand, the system of axioms envisaged does not seem sufficient to really describe o-minimal structures, since it only guarantees stability by projection for bounded closed definable parts.

%%%%%%%%%%%%%%%%%%%%%%%%%%%%%%%%%%%%%%%%%%%%%%%%%%%%%%%%%%%%%%%%%%%%
\section{Some questions}
%q
%: Question{questRRmodele}
\begin{question} \label{questRRmodele}~

\noindent Does the theory \Sa{Crc3} prove more \rdys than the theory of the interval $\ClI{-1,1} $ for a real closed field described by \Sa{Crc1}, or by \Sa{Crc2}?
\end{question}
%----------- end question ----------------------------- 
%q
%: Question{questRROmin}
\begin{questions} \label{questRROmin}~

\noindent 1) 
Does $\RR$ (for the sort $ \rR $), with $ \II=\II_\RR $ (for the sort $\Icr$) and for $ \Li_n $ the maps $ {\II}^n\to\II $ which are analytic in a neighbourhood of $ {\II}^n $, support a model of the dynamical theory \Sa{Crc3}? 

\smallskip\noindent 2)
If so, do the objects of sort $ \Df_n $ correspond exactly to the elements of the strongly real ring generated by the maps associated with the objects of sort $ \Li_n $?
\end{questions}
%----------- end question ----------------------------- 

\newpage \thispagestyle{empty}

\chapter*{General conclusion}
\addstarredchapter{General conclusion}

This dissertation, and the unanswered questions it contains, is a measure of our ignorance of real algebra.

\newpage \thispagestyle{empty}
 
\rdb
%%%%%%%%%%%%%%%%%%%%%%%%%%%%%%%%%%%%%%%%%%%%%%%%%%%%%%%%%%%%%%%%%%%%%

\part*
{Références et index}
\addstarredpart{References and index}

\rdb
\addcontentsline{toc}{section}{References. Books}

\setlength\labelnumberwidth{4em}
\printbibliography[title=References. Books,keyword=book,resetnumbers=true]

\rdb
\addcontentsline{toc}{section}{References. Articles}
\setlength\labelnumberwidth{2em}
\printbibliography[title=References. Articles,keyword=paper,resetnumbers=true]

\newpage \thispagestyle{empty}

\normalsize

\catcode`\@=11
%\@makesindexhead{Index des notations}

\newbox\toto
\newbox\tata
\newlength\largeurtoto
\newlength\largeurtata

\newcommand \ttt[3]{\setbox\tata=\hbox{#1}%
\ifdim\wd\tata<.17\textwidth\relax%
\setlength{\largeurtoto}{.80\textwidth}%
\setlength{\largeurtata}{.15\textwidth}%
\else%
\setlength{\largeurtoto}{.95\textwidth}%
\addtolength{\largeurtoto}{-\wd\tata}%
\setlength{\largeurtata}{\wd\tata}%
\fi%
\setbox\toto=\hbox{\parbox[b]{\largeurtoto}{\leftskip10pt\parindent-\leftskip\strut#2\dotfill\par}}%
\smallskip\noindent\mbox{%
\parbox[b][\ht\toto][t]{\largeurtata}{\strut#1}%
\parbox[b]{\largeurtoto}{\leftskip10pt\parindent-\leftskip\strut#2\dotfill\par}%
\hspace{.01\textwidth}%
\parbox[b]{.04\textwidth}{\hfill\strut#3}%
}%
\par}

\newcommand\NOTA[3]{\ttt{\smash{#1}}{\smash{#2}}{\pageref{#3}}}
\newcommand\NOTAx[2]{\ttt{\smash{\tsbf{#1}}}{\smash{#2}}{\pageref{Ax#1}}}
\newcommand\NOTAt[2]{\ttt{\smash{\sa{#1}}}{\smash{#2}}{\pageref{theorie#1}}}

\newpage
\chapter*{Notations index}
\addcontentsline{toc}{section}{Notations index}
\markboth{Notations}{Notations index}

%\begin{flushright}{page}\end{flushright}%
\vspace{.5em}
\subsection*{Logic}

\NOTA{$ \vd $}{deduction rule}{NOTAvou}
\NOTA{$\vou$}{open branches of computation}{NOTAvou}
\NOTA{$ \Exists u $}{introduce a fresh variable $u$}{NOTAExists}
\NOTA{$\Bot$}{collapse symbol}{NOTABot}
\NOTA{$ \vii $}{logical \gui{and}}{NOTAvii}
\NOTA{$\vuu$}{logical \gui{or}}{NOTAvuu}
\NOTA{$ \exists $}{logical \gui{there exists}}{NOTAexists}
\subsection*{Function symbols}

\NOTA{$ a\vu b $}{$ \sup(a,b) $}{deficodisup}
\NOTA{$ \Fr(a,b) $}{$ a/b $ (supposed well-defined)}{lemCo0FRAC}
\NOTA{$ \fsac_F $}{continuous semialgebraic map of graph $ F $}{Notafsa}
\NOTA{$ a\vi b $}{$ \inf(a,b) $}{secgrl}
\NOTA{$ \Sqr $}{$ \Sqr(x)=\sqrt{ x^+} $}{subsecclotrlRR}
\NOTA{$ \rho_{d,j} $}{virtual root}{prdfVirtualRoots}
\NOTA{$ \uplus $}{$ {\frac 1 2} \,(x+y) $ on interval $ [-1,+1] $}{NOTAuplus}
\NOTA{$ \oplus $}{forced addition on $ [-1,+1] $}{oplus}
\NOTA{$ \mathrm{Cb} $}{barycentric coefficients}{notacb}
\NOTA{$ \mathrm{Brc} $}{barycenters}{notacb}
\NOTA{$ \rT_n $}{Chebyshev polynomial}{NOTASigIcr}

\subsection*{Theories}
\NOTAt{Cd}{discrete fields}
\NOTAt{Al}{local rings}
\NOTAt{Ac}{commutative rings}
%:2025 ajout anneaux locaux, sans diviseurs de zéro, intègres
\NOTAt{Asdz}{without zerodivisor rings}
\NOTAt{Ai}{integral rings, domains}
\NOTAt{Cod}{discrete ordered fields}
\NOTAt{Crcd}{discrete real closed fields}
\NOTAt{Apo}{preordered rings}
\NOTAt{Ao}{ordered rings}
\NOTAt{Aonz}{strictly reduced ordered rings}
\NOTAt{Ato}{linearly ordered rings}
\NOTAt{Atonz}{reduced linearly ordered rings}
\NOTAt{Apro}{proto-ordered rings}
\NOTAt{Aso}{strictly ordered rings}
\NOTAt{Asto}{linearly, strictly ordered rings}
\NOTAt{Asonz}{reduced strictly ordered rings}
\NOTAt{Aito}{linearly ordered domains}
\NOTAt{Codsup}{discrete ordered fields with sup}
\NOTAt{Atosup}{linearly ordered rings with sup}
\NOTAt{Astosup}{strict $f$-rings with sup}
\NOTAt{Crcdsup}{real closed fields with sup}
\NOTAt{Co--}{minimal theory of \ndsofs}
\NOTAt{Co}{\ndsofs}
\NOTAt{Crc1}{\ndrcfs}
\NOTAt{Tr0}{(bounded) lattices}
\NOTAt{Tr}{nontrivial lattices}
\NOTAt{Trdi}{distributive lattices}
\NOTAt{Grl}{$\ell$-groups}
\NOTAt{Gtosup}{linearly ordered groups with sup}
\NOTAt{Afr}{$f$-rings}
\NOTAt{Afrsdz}{$f$-rings without zerodivisor}
\NOTAt{Asr}{strict $f$-rings}
\NOTAt{Afrnz}{reduced $f$-rings}
\NOTAt{Asrnz}{reduced strict $f$-rings}
\NOTAt{Aftr}{strongly real $f$-rings}
\NOTAt{Afr2c}{$2$-closed $f$-rings}
\NOTAt{Asr2c}{$2$-closed strict $f$-rings}
\NOTAt{Co2c}{$2$-closed \ndsofs}
\NOTAt{Afrrv}{$f$-rings with virtual roots}
\NOTAt{Asrrv}{strict $f$-rings with virtual roots}
\NOTAt{Aitorv}{linearly ordered domains with virtual roots}
\NOTAt{Arc}{real closed rings}
\NOTAt{Corv}{\ndsofs with virtual roots}
\NOTAt{Co--rv}{}
\NOTAt{Crc2}{\ndrcfs, 2; essentially identical to \Sa{Corv} and to \Sa{Crc1}}
\NOTAt{Crca}{archimedean \ndrcfs}
\NOTAt{Afrb}{rings of bounded real maps}
\NOTAt{Icr}{compact real intervals, analog to \Sa{Co}}
\NOTAt{Crc3}{\ndrcfs, 3}
%: axiomes
%\NOTA{}{}{}
\subsection*{Signatures}

\NOTA{$\Sigma_\Ac$}{$(\cdot=0\mathrel{;}\cdot+\cdot,\cdot\times \cdot,-\,\cdot,0,1)$\quad commutative rings }{NOTASigAc}

\NOTA{$\Sigma_\Ai$}{$(\cdot=0,\cdot\neq0\mathrel{;}\cdot+\cdot,\cdot\times \cdot,-\,\cdot,0,1)$\quad integral rings (domains) }{NOTASigAi}

\NOTA{$\Sigma_\Alu$}{$(\cdot=0,\U(\cdot)\mathrel{;}\cdot+\cdot,\cdot\times \cdot,-\,\cdot,0,1)$\quad local rings with units }{NOTASigAlu}

\NOTA{$\Sigma_\Alrd$}{$(\cdot=0,\U(\cdot),\Rn(\cdot)\mathrel{;}\cdot+\cdot,\cdot\times \cdot,-\,\cdot,0,1)$ residually discrete local rings }{NOTASigAlrd}

\NOTA{$\Sigma_\Aso$}{$(\cdot=0,\cdot\geq 0,\cdot>0\mathrel{;}\cdot+\cdot, \cdot\times\cdot,-\cdot, 0,1)$\quad strictly ordered rings}{NOTASigAso}

\NOTA{$\Sigma_\Ao$}{$(\cdot=0,\cdot\geq 0\mathrel{;}\cdot+\cdot, \cdot\times\cdot,-\,\cdot,0,1)$\quad ordered rings }{NOTASigAo}

\NOTA{$\Sigma_{\CoO}$}{$(\cdot=0,\cdot\geq 0,\cdot>0\mathrel{;}\cdot+\cdot, \cdot\times\cdot,\cdot\vu\cdot,-\,\cdot,0,1)$\quad \ndsof, 0 }{NOTASigCoO}

\NOTA{$\Sigma_\Co$}{$(\cdot=0,\cdot\geq 0,\cdot>0\mathrel{;}\cdot+\cdot, \cdot\times\cdot,\cdot\vu\cdot,-\,\cdot,\mathrm{Fr}(\cdot,\cdot),0,1)$\;\ndsof}{NOTASigCo}

\NOTA{$\Sigma_\TR$}{$(\cdot=\cdot\mathrel{;}\cdot \vi \cdot, \cdot \vu \cdot, 0,1)$\quad bounded lattices }{NOTASigTr}

\NOTA{$\Sigma_\Gao$}{$(\cdot=0,\cdot\geq 0\mathrel{;}\cdot+\cdot,-\,\cdot,0)$\quad ordered abelian groups }{NOTASigGao}

\NOTA{$\Sigma_\GRL$}{$(\cdot=0\mathrel{;}\cdot+\cdot,-\,\cdot,\cdot\vu\cdot,0)$\quad $\ell$-groups }{NOTASigGrl}

\NOTA{$\Sigma_\AfR$}{$(\cdot=0\mathrel{;}\cdot+\cdot, \cdot\times\cdot,\cdot\vu\cdot,-\,\cdot,0,1)$\quad $f$-rings }{NOTASigAfr}

\NOTA{$\Sigma_{\AfR'}$}{$(\cdot=0,\cdot\geq 0\mathrel{;}\cdot+\cdot, \cdot\times\cdot,\cdot\vu\cdot,-\,\cdot,0,1)$\; $f$-rings, bis}{NOTASigAfr'}

\NOTA{$\Sigma_\AsR$}{$(\cdot=0,\cdot\geq 0,\cdot>0\mathrel{;}\cdot+\cdot, \cdot\times\cdot,\cdot\vu\cdot,-\,\cdot,0,1)$\quad strict $f$-rings }{NOTASigAsr}

\NOTA{$\Sigma_\AftR$}{$(\cdot=0,\cdot\geq 0,\cdot>0\mathrel{;}\cdot+\cdot, \cdot\times\cdot,\cdot\vu\cdot,-\,\cdot,\Fr(\cdot),0,1)$\quad strongly real rings }{NOTASigAftr}

\NOTA{$\Sigma_\AfRdc$}{$(\cdot=0,\mathrel{;}\cdot+\cdot, \cdot\times\cdot,\cdot\vu\cdot,-\,\cdot,\Sqr(\cdot),0,1)$\quad $2$-closed $f$-rings }{NOTASigAfr2c}

\NOTA{$\Sigma_\AsRdc$}{$(\cdot=0,\cdot\geq 0,\cdot>0\mathrel{;}\cdot+\cdot, \cdot\times\cdot,\cdot\vu\cdot,-\,\cdot,\Sqr(\cdot),0,1)$\quad $2$-closed strict $f$-rings }{NOTASigAsr2c}

\NOTA{$\Sigma_\AftRdc$}{$\cdot=0,\cdot\geq 0,\cdot>0\mathrel{;}\cdot+\cdot, \cdot\times\cdot,\cdot\vu\cdot,-\,\cdot,\Fr(\cdot),\Sqr(\cdot),0,1)$ }{NOTASigAftr2c}

\NOTA{$\Sigma_\AfRrv$}{$(\cdot=0 \mathrel{;}\cdot+\cdot, \cdot\times\cdot,\cdot\vu\cdot,-\,\cdot,(\rho_{d,j})_{1\leq j\leq d},0,1)$\quad $f$-rings with virtual roots }{NOTASigAfrrv}

\NOTA{$\Sigma_\ArC$}{$(\cdot=0\mathrel{;}\cdot+\cdot, \cdot\times\cdot,\cdot\vu\cdot,-\,\cdot,(\rho_{d,j})_{1\leq j\leq d},\mathrm{Fr}(\cdot,\cdot),0,1)$\quad real closed rings }{NOTASigArc}

\NOTA{$\Sigma_\CoRv$}{$(\cdot=0,\cdot\geq 0,\cdot>0\mathrel{;}\cdot+\cdot, \cdot\times\cdot,\cdot\vu\cdot,-\,\cdot,(\rho_{d,j})_{1\leq j\leq d)},\mathrm{Fr}(\cdot,\cdot),0,1)$  }{NOTASigCorv}

\NOTA{$\Sigma_{\Icro}$}{$(\cdot= \cdot,\cdot\geq \cdot, \cdot>\cdot\mathrel{;}\cdot\,\uplus\,\cdot, \cdot\times \cdot, \cdot\vu\cdot, \mathrm{Fr}(\cdot,\cdot),\,- \cdot,0)$\quad compact real interval, 0}{NOTASigIcro}

\NOTA{$\Sigma_\AfrB$}{$(\cdot=0,\cdot\geq 0,\cdot>0\mathrel{;}(\mathrm{Brc}_{\rho})_{\rho\in\mathrm{Cb}}, \cdot\oplus\cdot, \cdot\times\cdot,-\,\cdot,\cdot\vu\cdot,\mathrm{Fr}(\cdot,\cdot), (r)_{r\in\II_\QQ})$ }{NOTASigAfrb}

\NOTA{$\Sigma_\Icr$}{$(\cdot=0,\cdot>0,\cdot\geq0\mathrel{;}(\mathrm{Brc}_{\rho})_{\rho\in\mathrm{Cb}}, \cdot\oplus\cdot, \cdot\times\cdot,-\,\cdot,\cdot\vu\cdot,\mathrm{Fr}(\cdot,\cdot), (r)_{r\in\II_\QQ}, (\rT_n)_{n\in\NN})$ }{NOTASigIcr}

%\NOTAx{}{}{}

\subsection*{Some axioms and dynamical rules}
\NOTAx{AL}{$\,\, (x+y)z=1 \vd \Exists u \; xu=1 \;\vou\;\Exists v\;yv=1$ }
\NOTAx{Anz}{$\,\, x^2= 0 \vd x = 0$ }
\NOTAx{ASDZ}{$\,\,xy=0 \vd x=0 \;\;\vou\;\;y=0$ }
\NOTAx{CD}{$\vd x=0\;\;\vou\;\;\Exists {y}\;\;xy=1$ }
\NOTA{$\tsbf{CL$_{\Ac}$}$}{$\,\,1=_\Ac0\vd \Bot$ }{AxCLnqAc}
\NOTA{$\tsbf{ED\inq}$}{$\vd x=0 \vou x\neq 0$ }{AxEdnq}
\NOTA{$\tsbf{col\inq}$}{$\,\,0\neq 0 \vd 1=0$ } {Axcolnq}
\NOTAx{Uv}{$\,\, xy = 1 \vd \U(x) $ }
\NOTAx{UV}{$\,\, \U(x) \vd \Exists y\, xy = 1$ }
\NOTAx{AL1}{$\,\, \U(x+y) \Vd \U(x) \vou \U(y)$ }
\NOTAx{NIL}{$\,\,\mathrm{Z}(x)\vd\Vou_{n\in\NN^+} x^n=0$ }
\NOTA{\tsbf{aso1}  -  \tsbf{aso4}}{}{Axaso1}
\NOTAx{Iv}{$\,\, xy = 1 \vd x\neq 0 $ }
\NOTAx{IV}{$\,\, x\neq 0 \vd \Exists y\, xy = 1$ }
\NOTA{$\tsbf{col\igt}$}{$\,\,0> 0 \vd 1=0$ }{Axcoligt}
\NOTAx{OT}{$\vd x \geq 0 \;\vou\; x\leq 0$ }
\NOTA{\tsbf{RCF$_n$}}{$\,\, a< b\vet P(a)P(b)<0 \vd \Exists x\, 
\big(P(x)=0,\,a<x<b\big)$ \quad ($P(x)=\sum_{k=0}^na_kx^k$) }{RCFn}\NOTAx{Aonz}{$\,\, c\geq 0\vet x(x^2+c)\geq 0 \vd x\geq 0$ }
\NOTAx{Aso1}{$\,\, x> 0\vet xy\geq 0 \vd y\geq 0 $ }
\NOTAx{Aso2}{$\,\, x\geq 0\vet xy> 0 \vd y> 0$ } 
\NOTAx{OTF}{$\,\, x+y> 0 \vd x > 0\;\vou\;y>0$ }
\NOTA{\tsbf{OTF}$\eti$}{$\,\, xy< 0 \vd x <0 \;\vou\; y<0 $ }{AxOTFx}
\NOTAx{Ato1}{$\,\, y\geq 0 \vet xy=1\vd x\geq 0$ }
\NOTAx{Ato2}{$\,\, c\geq 0\vet x(x^2+c)\geq 0 \vd x^3\geq 0$ }
\NOTAx{Aonz3}{$\,\, a\geq 0\vet b\geq 0\vet a^2=b^2 \vd a=b$ }
\NOTAx{CVX}{$\,\,0\leq a\leq b\vd \Exists z\; zb=a^2$ }\NOTAx{FRAC}{$\,\,0\leq a\leq b\vd \Exists z\; (zb=a^2\vet 0\leq z\leq a)$ }\NOTAx{fr1}{$\vd \mathrm{Fr}(a,b)\, \abs b=(\abs a\!\vi\! \abs b)^2$ }\NOTAx{fr2}{$\vd 0\leq {\mathrm{Fr}(a,b)}\leq \abs a\!\vi\! \abs b$ }
\NOTA{\tsbf{FRAC$_n$}}{\smash{$\,\,\abs u^n\leq \abs v^{n+1}\vd \Exists z\; (zv=u\vet \abs z^{n}\leq \abs v)$ \quad ($n\geq 1$)}}{AxFRACn}
\NOTAx{Gao}{$\,\, x\geq 0\vet x\leq 0 \vd x = 0$ }
\NOTAx{grl}{$\vd x+(y\vu z)=(x+y)\vu(x+z)$ }
\NOTAx{sup1}{$\vd x\vu y\,\geq x$ }
\NOTAx{sup2}{$\vd x\vu y\,\geq y$ }
\NOTAx{Sup}{$\,\, z\geq x\vet z\geq y\vd z\geq x\vu y$ }
\NOTAx{afr}{$\vd x^+\, (y\vu z)=(x^+\, y)\vu(x^+\, z)$ }
\NOTAx{sup}{$\vd ((x\vu y)- x)\,((x\vu y)- y)=0$ }
\NOTA{\tsbf{afr1} - \tsbf{afr7} }{}{Axafr1}
\NOTA{\tsbf{Afr1} - \tsbf{Ato2} }{}{AxAfr1}
\NOTA{\tsbf{Afrnz1} - \tsbf{Afrnz2} }{}{AxAfrnz1}
\NOTAx{AFRL}{$\,\, z(x+y)=1\vet x+y\geq 0\vd \Exists u\;(ux=1\vet x\geq 0 )\;\vou\; \Exists v\;(vy=1\vet y\geq 0 )$ }
\NOTAx{sqr0}{$\vd \Sqr(0)=0$ }
\NOTAx{sqr1}{$\vd \Sqr(x)=\Sqr(x)^+ $ }
\NOTAx{sqr2}{$\vd \Sqr(x)=\Sqr(x^+)$ }
\NOTAx{sqr3}{$\vd \Sqr(x)^2=x^+ $ }
\NOTAx{sqr4}{$\vd \Sqr(x^+y^+)=\Sqr(x)\Sqr(y)$ }
\NOTA{\tsbf{vr$ _{i,j,k} $}}{axioms for virtual roots, examples}{exavr} 
%\NOTAx{ }{$ $}

\subsection*{Algebraic structures, generic models}

\NOTA{$\RRa$}{the field of real algebraic numbers}{NOTARa}
\NOTA{$ \RR_{\tsbf{PR}} $}{the field of primitive recursive real numbers}{exacorpsnondiscret}
\NOTA{$ \RR_{\tsbf{Ptime}} $}{the field of real numbers computable in polynomial time}{exacorpsnondiscret}
\NOTA{$ \RR_{\tsbf{Rec}} $}{the field of recursive real numbers}{exacorpsnondiscret}
\NOTA{$ \Sac_n(\RR) $}{\ref{defiFSAGC2}: semialgebraic continuous fonctions in $n$ variables, see also \ref{defiFSAGC2+}}{defiFSAGC2}

\smallskip\noindent In the following $\gA$ is a commutative ring with a suitable algebraic structure over a suitable dynamical theory in some context, $\gR$ is an $f$-ring with virtual roots.
 
\NOTA{$ \AFR(\gA) $}{$f$-ring freely generated by $\gA$}{notaAFR}
\NOTA{$ \SIPD_n(\gA) $}{ring of semipolynomials (or sipd maps) in $n$ variables over $\gA$}{defiSIPD}
\NOTA{$ \SIPD_n(\gA,\gB) $}{ring of $\gA$-semipolynomials in $n$ variables over $\gB$}{defiSIPD}
\NOTA{$ \AFRdC(\gA) $}{$2$-closure of an $f$-ring}{notaAFR2C}
\NOTA{$ \Sace_n(\gR) $}{ring of integral continuous semialgebraic maps in $n$ variables}{defiSacem}
\NOTA{$\Sac_n(\gR)$}{ring of continuous semialgebraic maps in $n$ variables}{defiFSAGC2+}
\NOTA{$\AFRNZ(\gA)$}{reduced $f$-ring generated by $\gA$}{defiSacembis}
\NOTA{$\AFRRV(\gA)$}{$f$-ring with virtual roots generated by $\gA$}{defiSacembis}
\NOTA{$\PPM(\gA)$}{ring of piecewise polynomial elements of $\AFRRV(\gA)$}{defiSacembis}
\NOTA{$\ARC(\gA)$}{real closed ring generated by $\gA$}{notaARC}
%\NOTA{$ $}{}{}

%\newpage
\rdb
\addcontentsline{toc}{section}{Terms index}
\printindex

\end{document}